\documentclass[letter,10pt]{report}
\usepackage{amssymb,amsmath,amsfonts,amsthm}
\usepackage{mathrsfs}
\usepackage{subfigure}
\usepackage{float}
\usepackage{graphicx}
\usepackage[left=1 in, right=1 in,top=1 in, bottom=1 in]{geometry}

\usepackage{hyperref}

\newtheorem{theorem}{Theorem}[section]
\newtheorem{lemma}[theorem]{Lemma}

\newtheorem{remark}[theorem]{Remark}

\makeatletter
\renewcommand \theequation {%
\ifnum \c@chapter>\z@ \@arabic\c@chapter.%
\fi\@arabic\c@equation} \@addtoreset{equation}{chapter}
\makeatother

\def\XXint#1#2#3{{\setbox0=\hbox{$#1{#2#3}{\int}$ }
\vcenter{\hbox{$#2#3$ }}\kern-.6\wd0}}

\providecommand{\abs}[1]{\left\vert#1\right\vert}
\providecommand{\nm}[1]{\left\Vert#1\right\Vert}

\providecommand{\br}[1]{\left\langle #1 \right\rangle}
\providecommand{\tm}[2]{\left\Vert#1\right\Vert_{L^2(#2)}}
\providecommand{\im}[2]{\left\Vert#1\right\Vert_{L^{\infty}(#2)}}

\providecommand{\lnnm}[1]{{\left\Vert#1\right\Vert}_{L^{\infty}_{\eta}L^{\infty}_{\phi,\psi}}}
\providecommand{\tnnm}[1]{{\left\Vert#1\right\Vert}_{L^{2}_{\eta}L^2_{\phi,\psi}}}
\providecommand{\ltnm}[1]{{\left\Vert#1\right\Vert}_{L^{\infty}_{\eta}L^{2}_{\phi,\psi}}}

\providecommand{\lnm}[1]{\left\Vert#1\right\Vert_{L^{\infty}_{\phi,\psi}}}
\providecommand{\tnm}[1]{\left\Vert#1\right\Vert_{L^{2}_{\phi,\psi}}}

\providecommand{\lnmp}[1]{\left\Vert#1\right\Vert_{L^{\infty}_{+,\phi,\psi}}}
\providecommand{\tnmp}[1]{\left\Vert#1\right\Vert_{L^{2}_{+,\phi,\psi}}}

\providecommand{\lss}[2]{\left\Vert#1\right\Vert_{L^{\infty}_{#2}}}

\def\ud{\mathrm{d}}
\def\dt{\partial_t}
\def\p{\partial}
\def\ls{\lesssim}

\def\half{\dfrac{1}{2}}
\def\rt{\rightarrow}
\def\r{\mathbb{R}}
\def\no{\nonumber}
\def\ue{\mathrm{e}}
\def\ui{\mathrm{i}}
\def\ds{\displaystyle}

\def\u{U}
\def\ub{\mathscr{U}}
\def\bu{\bar U}
\def\bub{\bar{\mathscr{U}}}
\def\uf{\mathfrak{U}}
\def\buf{\bar{\mathfrak{U}}}
\def\uu{\mathcal{U}^B}
\def\buu{\bar{\mathcal{U}}^B}
\def\ui{\mathcal{U}^I}
\def\bui{\bar{\mathcal{U}}^I}

\def\e{\epsilon}
\def\s{\mathbb{S}}
\def\vx{\vec x}
\def\vw{\vec w}

\def\vt{\vec\varsigma}
\def\vr{\vec r}
\def\nx{\nabla_{x}}
\def\vn{\vec\nu}

\def\l{\lambda}
\def\ll{\mathcal{L}}

\def\kk{\kappa}
\def\q{Q}
\def\qi{\mathcal{Q}}
\def\qb{\mathscr{Q}}
\def\qf{\mathfrak{Q}}

\def\v{\mathscr{V}}
\def\w{\mathscr{W}}
\def\d{\delta}

\def\vn{\vec\nu}

\def\t{\mathcal{T}}
\def\k{\mathcal{K}}
\def\a{\mathscr{A}}

\def\id{{\bf{1}}}

\def\rr{\mathscr{R}}

\def\gb{\mathscr{G}}
\def\gf{\mathfrak{G}}

\def\rp{\r^+}

\def\ss{S}
\def\h{h}
\def\g{g}

\def\tf{\tilde F}
\def\tv{\tilde V}

\begin{document}

\title{Asymptotic Analysis of Transport Equation in Bounded Domains}
\author{Lei Wu\footnote{L. Wu is supported by NSF grant DMS-1853002. Email: lew218@lehigh.edu}\\
Department of Mathematics, Lehigh University }
\date{}


\maketitle

\begin{abstract}
\par Consider neutron transport equations in 3D convex domains with in-flow boundary. We mainly study the asymptotic limits as the Knudsen number $\e\rt 0^+$. Using Hilbert expansion, we rigorously justify that the solution of steady problem converges to that of the Laplace's equation, and the solution of unsteady problem converges to that of the heat equation. The proof relies on a detailed analysis on the boundary layer effect with geometric correction.

This problem can be formulated in many different settings, and the above one is probably the most physically significant and most mathematically challenging. We have to utilize almost all methods and techniques we developed in a series of papers \cite{AA003,AA005,AA007,AA009,AA012,AA014} in the past decade, and bring novel ideas to treat the new complications.

The difficulty mainly comes from three sources: 3D domain, boundary layer regularity, and time dependence. To fully solve this problem, we introduce several techniques: (1) boundary layer with geometric correction; (2) remainder estimates with $L^2-L^{2m}-L^{\infty}$ framework; (3) boundary layer decomposition.
\end{abstract}

\tableofcontents

\newpage


\pagestyle{myheadings} \thispagestyle{plain} \markboth{LEI WU}{ASYMPTOTIC ANALYSIS OF TRANSPORT EQUATION}

\chapter{Introduction}

\section{Problem Presentation}

We consider the steady neutron transport equation in a
three-dimensional bounded convex domain with in-flow boundary. In the spacial
domain $\vx=(x_1,x_2,x_3)\in\Omega$ where $\p\Omega\in C^3$ and the velocity domain
$\vw=(w_1,w_2,w_3)\in\s^2$, the neutron density $u^{\e}(\vx,\vw)$
satisfies
\begin{align}\label{transport}
\left\{
\begin{array}{l}\displaystyle
\e \vw\cdot\nabla_x u^{\e}+u^{\e}-\bar u^{\e}=0\ \ \text{in}\ \ \Omega\times\s^2,\\\rule{0ex}{2.0em}
u^{\e}(\vx_0,\vw)=g(\vx_0,\vw)\ \ \text{for}\
\ \vw\cdot\vn<0\ \ \text{and}\ \ \vx_0\in\p\Omega,
\end{array}
\right.
\end{align}
where
\begin{align}\label{average}
\bar u^{\e}(\vx)=\frac{1}{4\pi}\int_{\s^2}u^{\e}(\vx,\vw)\ud{\vw},
\end{align}
$\vn$ is the outward unit normal vector, with the Knudsen number $0<\e<<1$.

Also, we consider the unsteady counterpart. Taking additional temporal domain $t\in\rp$ into account, the neutron density $u^{\e}(t,\vx,\vw)$
satisfies
\begin{eqnarray}\label{transport.}
\left\{
\begin{array}{l}\displaystyle
\e^2\dt u^{\e}+\e \vw\cdot\nabla_x u^{\e}+u^{\e}-\bar u^{\e}=0\ \ \text{in}\ \ \rp\times\Omega\times\s^2,\\\rule{0ex}{2.0em}
u^{\e}(0,\vx,\vw)=h(\vx,\vw)\ \ \text{in}\ \ \Omega\times\s^2,\\\rule{0ex}{2.0em}
u^{\e}(t,\vx_0,\vw)=g(t,\vx_0,\vw)\ \ \text{for}\
\ t\in\rp,\ \ \vw\cdot\vn<0\ \ \text{and}\ \ \vx_0\in\p\Omega,
\end{array}
\right.
\end{eqnarray}
where
\begin{eqnarray}\label{average.}
\bar u^{\e}(t,\vx)=\frac{1}{4\pi}\int_{\s^2}u^{\e}(t,\vx,\vw)\ud{\vw}.
\end{eqnarray}
The initial and boundary data satisfy the
compatibility condition
\begin{eqnarray}\label{compatibility condition.}
h(\vx_0,\vw)=g(0,\vx_0,\vw)\ \ \text{for}\ \ \vx_0\in\p\Omega\ \ \text{and}\ \ \vw\cdot\vn<0.
\end{eqnarray}
\ \\
In both cases, we intend to study the behavior of $u^{\e}$ as $\e\rt0$. 
\ \\
Based on the flow direction, we can divide the boundary $\Gamma=\{(\vx,\vw): \vx\in\p\Omega\ \ \text{and}\ \ \vw\in\s^2\}$ into
the in-flow boundary $\Gamma^-$, the out-flow boundary $\Gamma^+$
and the grazing set $\Gamma^0$ as
\begin{align}
\Gamma^{-}=\{(\vx,\vw): \vx\in\p\Omega,\ \vw\cdot\vn<0\}\\
\Gamma^{+}=\{(\vx,\vw): \vx\in\p\Omega,\ \vw\cdot\vn>0\}\\
\Gamma^{0}=\{(\vx,\vw): \vx\in\p\Omega,\ \vw\cdot\vn=0\}
\end{align}
It is easy to see that $\Gamma=\Gamma^+\cup\Gamma^-\cup\Gamma^0$.
Hence, the boundary condition is only prescribed for $\Gamma^{-}$. This is usually called the in-flow or absorbing boundary condition.

\section{Background and Methods}

We have been involved in this long-term project to justify the diffusive limits for neutron transport equations(see \cite{AA003}, \cite{AA005}, \cite{AA006}, \cite{AA007}, \cite{AA009}, \cite{AA012} and \cite{AA014}) and the hydrodynamic limits for nonlinear Boltzmann equation (see \cite{AA004} and \cite{AA013}). In the past decade, we push the results in many aspects: geometry of domain, dimension, time-dependence, etc. Here, we present a brief review of the background and the progress we have made.

\subsection{Asymptotic Analysis}

Diffusive limits, or more general hydrodynamic limits, are central to connecting kinetic theory and fluid mechanics. The basic idea is to consider the asymptotic behaviors of the solutions to Boltzmann equation, transport equation, or Vlasov systems. Since the early $20^{th}$ century, this type of problems have been extensively studied in many different settings: steady or unsteady, linear or nonlinear, strong solution or weak solution, etc.

Among all these variations, one of the simplest but most important models --- neutron transport equation in bounded domains, has attracted a lot of attention since the dawn of the atomic age. Besides its significance in nuclear sciences and medical imaging, the neutron transport equation is usually regarded as a linear prototype of the more complicated nonlinear Boltzmann equation, and thus, is an ideal starting point to develop new theories and techniques.
The early development focuses on formal expansion, explicit solution and numerical methods. We refer to
\cite{Larsen1974=}, \cite{Larsen1974}, \cite{Larsen1975}, \cite{Larsen1977}, \cite{Larsen.D'Arruda1976}, \cite{Larsen.Habetler1973}, \cite{Larsen.Keller1974}, \cite{Larsen.Zweifel1974}, \cite{Larsen.Zweifel1976} for more details.

We start from 2D steady neutron transport equation, i.e. $\vx\in\Omega\subset\r^2$ and $\vw\in\s^1$. Generally speaking, the solution $u^{\e}$ varies smoothly and slowly in the interior of $\Omega$, and behaves like $u^{\e}-\bar u^{\e}=0$ which ignores $\vw\cdot\nabla_xu^{\e}$. However, its value changes severely when approaching the boundary $\p\Omega$ and in this regime, $\vw\cdot\nabla_xu^{\e}$ plays a crucial role. The smaller $\e$ is, the more violently $u^{\e}$ changes.

This phenomenon indicates that $u^{\e}$ can actually be described in two distinct regimes with different scalings, namely, the interior solution $\u$ and the boundary layer $\ub$. 
The interior solution satisfies certain fluid equations or thermodynamic equations, and the boundary layer satisfies a half-space kinetic equation, which decays rapidly when it is away from the boundary.

\begin{figure}[H]
\begin{minipage}[t]{0.5\linewidth}
\centering
\includegraphics[width=3in]{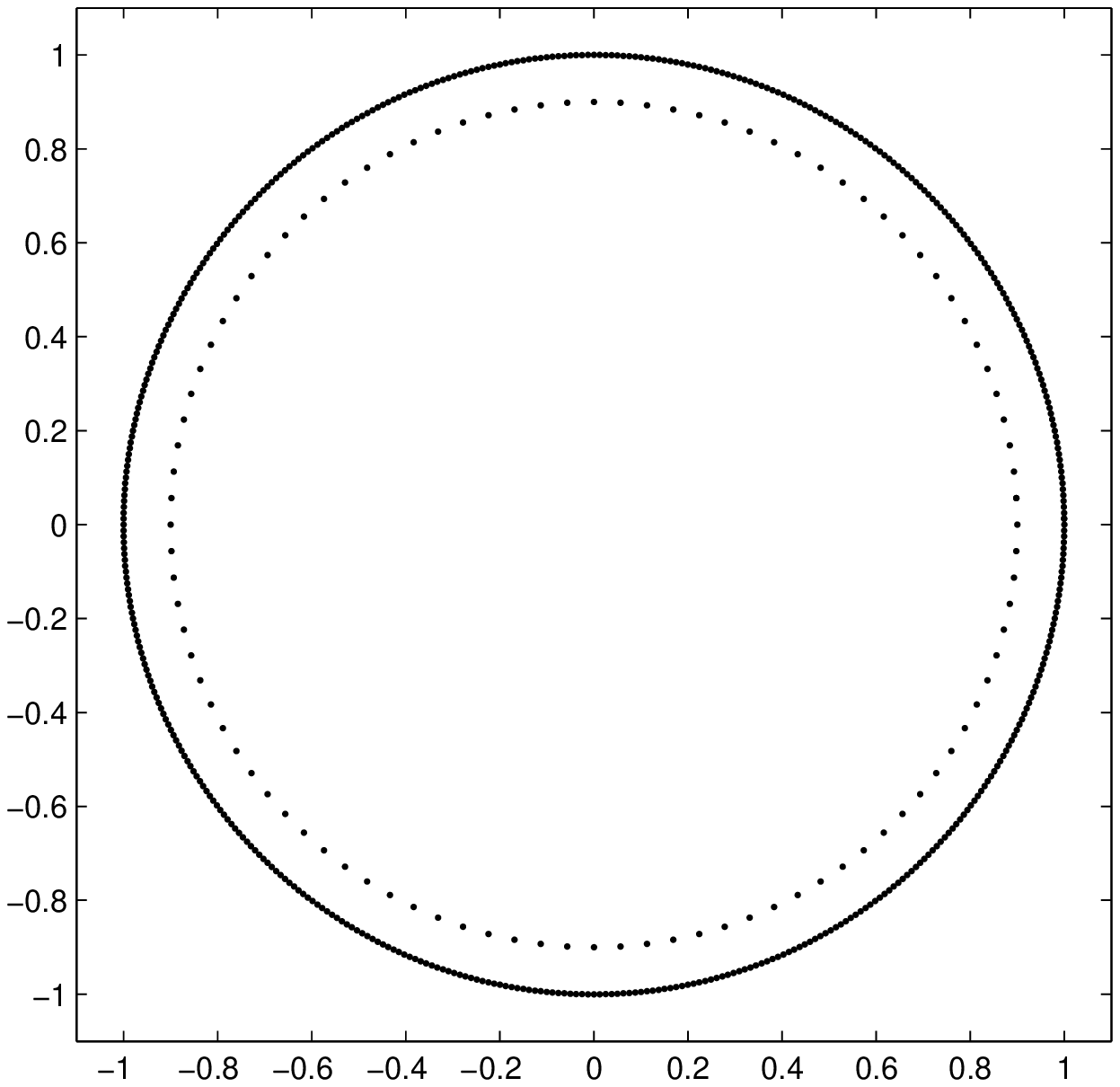}
\caption{Boundary Layer in a Disk} \label{fig 3}
\end{minipage}%
\begin{minipage}[t]{0.5\linewidth}
\centering
\includegraphics[width=3in]{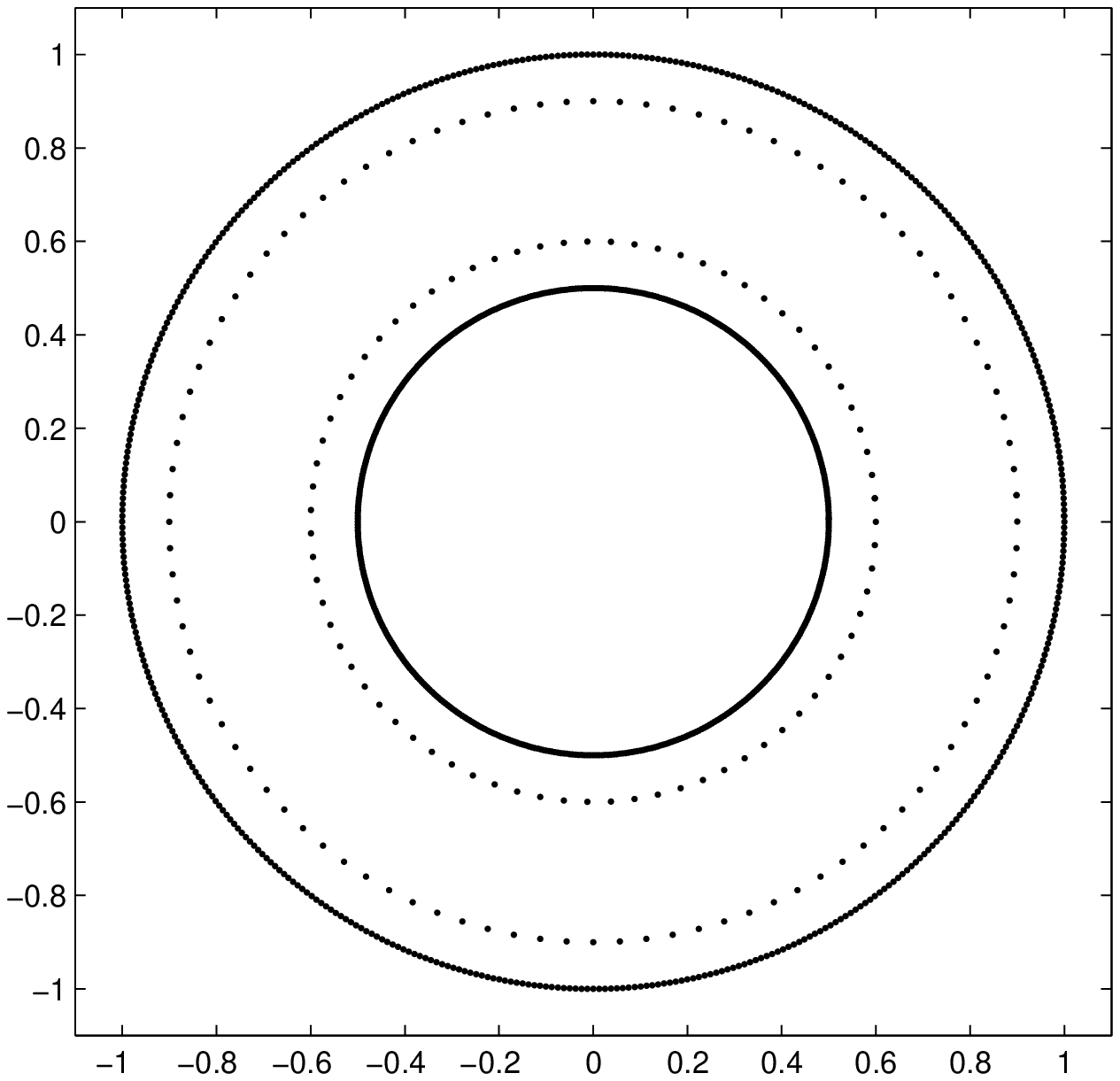}
\caption{Boundary Layers in an Annulus} \label{fig 4}
\end{minipage}
\end{figure}
In Figure \ref{fig 3} and Figure \ref{fig 4}, the solid circles represent the physical boundary $\p\Omega$. The regions between the solid and dotted circles are the regime of boundary layers. Here we exaggerate the thickness of boundary layer regions for clarity; it is actually very thin and depends on $\e$.

The justification of this approximation, i.e. the so-called diffusive limit usually involves two steps:
\begin{enumerate}
\item
Hilbert expansion: expanding $\u=\ds\sum_{k=0}^{\infty}\e^k\u_k$ and $\uu=\ds\sum_{k=0}^{\infty}\e^k\uu_k$ as power series of $\e$ and proving the coefficients $\u_k$ and $\uu_k$ are well-defined. \\
Traditionally, the estimates of the interior solutions $\u_k$ are relatively straightforward. On the other hand, boundary layers $\uu_k$ satisfy one-dimensional half-space problems which lose some key structures of the original equations. The well-posedness of boundary layer equations are sometimes extremely difficult and it is possible that they are actually ill-posed (e.g. certain type of Prandtl layers \cite{Guo.Nguyen2011}).
\item
Remainder estimates: proving that $R=u^{\e}-\u_0-\uu_0=o(1)$ as $\e\rt0$. \\
Ideally, this should be done just by expanding to the leading-order level $\u_0$ and $\uu_0$. However, in singular perturbation problems, the estimates of the remainder $R$ usually involve negative powers of $\e$, which requires an expansion to higher-order terms $\u_N$ and $\uu_N$ for $N\geq1$ such that we have a sufficient power of $\e$. In other words, we define $R=u^{\e}-\ds\sum_{k=0}^{N}\e^k\u_k-\ds\sum_{k=0}^{N}\e^k\uu_k$ for $N\geq1$ instead of $R=u^{\e}-\u_0-\uu_0$ to get better estimates of $R$.
\end{enumerate}

\subsection{Classical Approach}

The construction of kinetic boundary layers
has long been believed to be satisfactorily solved since Bensoussan, Lions and Papanicolaou published their remarkable paper \cite{Bensoussan.Lions.Papanicolaou1979} in 1979. Their formulation, based on the flat Milne problem, was later extended to treat the nonlinear Boltzmann equation (see \cite{Sone2002} and \cite{Sone2007}).

In detail, in $\Omega\subset\r^2$, let $\eta\in[0,\infty)$ denote the rescaled normal variable with respect to the boundary, $\iota$ the tangential variables, and $\phi$ the velocity variable in $\s^1$.
The classical boundary layer $\uu_0$ satisfies the flat Milne problem,
\begin{align}\label{correct 1}
\sin\phi\frac{\p \uu_0}{\p\eta}+\uu_0-\bar\uu_0=&0.
\end{align}

Unfortunately, in \cite{AA003}, we demonstrated that both the proof and result of this formulation are invalid due to a lack of regularity in estimating $\dfrac{\p\uu_0}{\p\iota}$. Also, this glitch was further captured by numerical tests in \cite{Li.Lu.Sun2017}. This pulls the whole research back to the starting point, and any later results based on this type of boundary layers should be reexamined.

To be more specific, the remainder estimates require $\uu_1\in L^{\infty}$ which needs $\dfrac{\p\uu_0}{\p\iota}\in L^{\infty}$. However, though \cite{Bensoussan.Lions.Papanicolaou1979} shows that $\uu_0\in L^{\infty}$, it does not necessarily mean that $\dfrac{\p \uu_0}{\p\eta}\in L^{\infty}$. Furthermore, this singularity $\dfrac{\p \uu_0}{\p\eta}\notin L^{\infty}$ will be transferred to $\dfrac{\p \uu_0}{\p\iota}\notin L^{\infty}$. A careful construction of boundary data justifies this invalidity, i.e. the chain of estimates
\begin{align}
R=o(1)\ \Leftarrow\ \uu_1\in L^{\infty}\ \Leftarrow\ \dfrac{\p\uu_0}{\p\iota}\in L^{\infty}\ \Leftarrow\ \dfrac{\p \uu_0}{\p\eta}\in L^{\infty},
\end{align}
is broken since the rightmost estimate is wrong.

Note that the difficulty of the above classical approach is purely due to the geometry of the curved boundary $\p\Omega$. When $\p\Omega$ is flat, i.e. when $\Omega$ is the half space $\r\times\r^+$, the flat Milne problem \eqref{correct 1} provides the correct description of the kinetic boundary layer.

\subsection{Geometric Correction (\cite{AA003} and \cite{AA006})}

While the classical method fails, a new approach with geometric correction to the boundary layer construction has been developed to ensure regularity in the cases of disk and annulus in \cite{AA003} and \cite{AA006}. The new boundary layer $\uu_0$ satisfies the $\e$-Milne problem with geometric correction,
\begin{align}\label{itt 01}
\sin\phi\frac{\p \uu_0}{\p\eta}+\frac{\e}{R_{\kappa}-\e\eta}\cos\phi\frac{\p
\uu_0}{\p\phi}+\uu_0-\bar\uu_0=&0,
\end{align}
where $R_{\kappa}$ is the radius of curvature of the boundary. We proved that the solution recovers the well-posedness and exponential decay as in the flat Milne problem, and the regularity in $\iota$ is indeed improved, i.e. $\dfrac{\p\uu_0}{\p\iota}\in L^{\infty}$. A similar formulation was introduced in \cite{Cercignani.Marra.Esposito1998} to describe Boltzmann equation. Our analysis provides a rigorous justification of its implementation in the construction of kinetic boundary layers.

However, this new method fails to treat more general smooth convex domains. Roughly speaking, we have two contradictory goals to achieve:
\begin{itemize}
\item
To prove diffusive limits, the remainder estimates require higher-order regularity estimates of the boundary layer.
\item
The geometric correction $\dfrac{\e}{R_{\kappa}-\e\eta}\cos\phi\dfrac{\p
\uu_0}{\p\phi}$ in \eqref{itt 01} is related to the curvature of the boundary curve, which prevents higher-order regularity estimates.
\end{itemize}
In other words, the improvement of regularity is still not enough to close the proof.

\subsection{Diffusive Boundary (\cite{AA007} and \cite{AA009})}

In \cite{AA007}, for the 2D case of diffusive boundary, we push the above argument from both sides, i.e. improvements in remainder estimates and boundary layer regularity.
%

In detail, consider the boundary layer expansion
\begin{align}
\uu(\eta,\iota,\vw)\sim \uu_0(\eta,\iota,\vw)+\e\uu_1(\eta,\iota,\vw).
\end{align}
The diffusive boundary condition
\begin{align}
u^{\e}(\vx_0,\vw)=\frac{1}{2}\int_{\vw\cdot\vn>0}u^{\e}(\vx_0,\vw)(\vw\cdot\vn)\ud{\vw}+\e g(\vx_0,\vw),
\end{align}
leads to an important simplification: $\uu_0=0$. Then the next-order boundary layer $\uu_1$ must formally satisfy
\begin{align}
\sin\phi\frac{\p \uu_1}{\p\eta}+\frac{\e}{R_{\kappa}-\e\eta}\cos\phi\frac{\p
\uu_1}{\p\phi}+\uu_1-\bar\uu_1=&0.
\end{align}
The proof of diffusive limit requires an estimate of $\dfrac{\p \uu_1}{\p\iota}$, which we do not have in hand. Here, a key observation is that $W=\dfrac{\p \uu_1}{\p\iota}$ satisfies
\begin{align}
\sin\phi\frac{\p W}{\p\eta}+\frac{\e}{R_{\kappa}-\e\eta}\cos\phi\frac{\p
W}{\p\phi}+W-\bar W=&-\frac{\p_{\iota}R_{\kappa}}{R_{\kappa}-\e\eta}\bigg(\frac{\e}{R_{\kappa}-\e\eta}\cos\phi\frac{\p
\uu_1}{\p\phi}\bigg).
\end{align}
Note that the term in the large parenthesis of the right-hand side is part of the $\uu_1$ equation and its estimate depends on $\sin\phi\dfrac{\p \uu_1}{\p\eta}$. In other words, the estimate of $\dfrac{\p \uu_1}{\p\iota}$ depends on $\sin\phi\dfrac{\p \uu_1}{\p\eta}$, not just $\dfrac{\p \uu_1}{\p\eta}$ which is possibly unbounded. The $\sin\phi$ is crucial to eliminate the singularity. This forms the major proof in \cite{AA007}, i.e. the weighted regularity of $\uu_1$.

Our main idea is
to delicately track $\uu_1$ along the characteristics in the mild formulation, and prove the weighted $W^{1,\infty}$ estimates of the boundary layer.
In particular, we showed that $\dfrac{\p\uu_1}{\p\iota}$ is bounded even when $R_{\kappa}$ is not constant for general convex domains.
Furthermore, with a novel $L^{2m}-L^{\infty}$ framework, we prove a new remainder estimate, which does not require any higher regularity estimates of the boundary layer.

In summary, in \cite{AA007}, we proved the diffusive limit: $u^{\e}$ converges to the solution of a Laplace's equation with Neumann boundary condition. Such result is extended in \cite{AA009} to treat 3D case with diffusive boundary.

\subsection{In-Flow Boundary (\cite{AA014} and Current Monograph)}

It is notable that, for the case of in-flow boundary as equation \eqref{transport}, the situation is much worse. The leading-order boundary layer $\uu_0$ is no longer zero, i.e.
\begin{align}
\sin\phi\frac{\p \uu_0}{\p\eta}+\frac{\e}{R_{\kappa}-\e\eta}\cos\phi\frac{\p
\uu_0}{\p\phi}+\uu_0-\bar\uu_0=&0,\\
\sin\phi\frac{\p \uu_1}{\p\eta}+\frac{\e}{R_{\kappa}-\e\eta}\cos\phi\frac{\p
\uu_1}{\p\phi}+\uu_1-\bar\uu_1=&-\cos\phi\frac{\p\uu_0}{\p\iota}.
\end{align}
The remainder contains the term $\dfrac{\p\uu_1}{\p\iota}$, which depends on the estimate of $\dfrac{\p^2\uu_0}{\p\iota^2}$. Then we must prove $W^{2,\infty}$ estimates in the boundary layer equation. In principle, this is impossible for general kinetic equations as \cite{Guo.Kim.Tonon.Trescases2013} pointed out.

Here, we have a key observation that actually the singularity that prevents higher-order regularity concentrates in the neighborhood of the grazing set, so it is natural to isolate the singular part from the whole solution and tackle them in different approaches.

Inspired by \cite{Li.Lu.Sun2017}, we introduce a new regularization argument in \cite{AA014}. Instead of trying different weighted norms, we may also modify the boundary data and smoothen the boundary layer in this modified problem.

To be precise, we decompose the boundary data $g=\gb+\gf$, such that
\begin{itemize}
\item
the boundary layer $\ub$ with data $\gb$, which we call regular boundary layer, attains second-order regularity in the tangential direction, i.e. $\dfrac{\p^2\ub}{\p\iota^2}\in L^{\infty}$; $\gb=g$ in most of the region except a small neighborhood of the grazing set;
\item
the boundary layer $\uf$ with data $\gf$, which we call singular boundary layer, attains only first-order regularity in the tangential direction i.e. $\dfrac{\p\uf}{\p\iota}\in L^{\infty}$, but the support of $\gf$ is restricted to a very small neighborhood of the grazing set with diameter $\e^{\alpha}$ for some $0<\alpha<1$.
\end{itemize}
In other words, for the remainder estimates, the extra power of $\e$ comes from two sources: $\ub$ gains power by expanding to the higher order, and $\uf$ gains power through a small support $\e^{\alpha}$.

However, this method cannot be directly adapted to 3D case with in-flow boundary. The difficulty is comes from the failure of balance between remainder estimates and boundary layer regularity:
\begin{itemize}
\item
The above boundary layer decomposition argument contains a delicate balance of regular and singular part, i.e. $\alpha$ cannot be too large (near $1$) or too small (near $0$). A detailed analysis and maximization argument reveal that the best power we can gain from this argument is $\e^{\frac{1}{2}-}$.
\item
The 3D remainder estimate is much worse in $L^{2m}-L^{\infty}$ framework. Due to the restriction of Sobolev embedding and trace theorem, we need an additional $\e^{\frac{2}{3}}$ to close the proof.
\end{itemize}
Certainly, $\e^{\frac{1}{2}}$ is not enough to fill the gap of $\e^{\frac{2}{3}}$ needs. Roughly speaking, the regularization is still not enough in 3D.

In the current monograph, we introduce a key improvement to this argument. By using mixed derivative estimates, we greatly improve the bounds of regular boundary layer $\ub$ and push the maximization procedure to $\alpha=1$. Then we gain an extra $\e^{1-}$ from the boundary layer decompose which eventually results in the diffusive limit with $\e^{\frac{1}{3}-}$ convergence.

\subsection{Time-Dependence (\cite{AA005} and \cite{AA012} and Current Monograph)}

Compared with steady problem, there are much less work on unsteady problem \eqref{transport.}. Our previous research only covers the 2D case with either in-flow \cite{AA005} or diffusive \cite{AA012} boundary. In the current monograph, we will focus on the 3D case.

Typically, as a 1D variable, time is not a major difficulty in the asymptotic analysis of linear equations. However, the regularity issue in kinetic equations has serious consequences in its evolutionary counterpart:
\begin{itemize}
\item
The exact solution can be approximated by the interior solution and some additional correction layers. Besides the known boundary layer, the time derivative will introduce a new correction term --- initial layer, which varies rapidly near $t=0$. Moreover, these two types of layers may have interaction with each other and we have to introduce a mixed-type correction --- the notorious initial-boundary layer to compensate for such effect. Unfortunately, the well-posedness of such equation is not available even in 1D, so this is a dead end so far.
\item
The remainder estimate is much worse than that of steady problems. Besides the similar difficulty in 3D as in steady problems, the time derivative requires an additional $\e^{\frac{1}{3}}$ to close the proof. This is an extremely sad news. The improvement of boundary layer decomposition argument can provide additional $\e^{1-}$, but now we need in total $\e^{\frac{2}{3}}\times \e^{\frac{1}{3}}=\e$, which is not enough.
\end{itemize}
In order to resolve such difficulties, we introduce several new techniques:
\begin{itemize}
\item
Analyzing the compatibility condition \eqref{compatibility condition.}.\\
This provides additional information between the interaction of initial and boundary layers. In particular, we deduce that the leading-order initial-boundary layer vanishes, so we do not need to discuss its well-posedness. Then the next-order expansion is already unnecessary from proof viewpoint.
\item
$L^2-L^{2m}-L^{\infty}$ framework.\\
This is an improvement that is particularly designed for unsteady problems. It is also the highlight of the current monograph. In addition to the usual $L^2$ energy estimates, we justify its $L^{2m}$ version with delicately chosen test functions. Then an intricate interpolation argument and $L^{2m}$ kernel estimates help to reduce the effect of time derivatives. Eventually, this step only requires $\e^{\frac{5}{6}}$ rather than $\e$, which yields an asymptotic convergence of $\e^{\frac{1}{6}}-$.
\end{itemize}

\section{Main Theorem}

\subsection{Steady Problem}

\begin{theorem}[Steady Diffusive Limit]\label{main theorem}
Assume $g(\vx_0,\vw)\in C^3(\Gamma^-)$. Then for the steady neutron
transport equation \eqref{transport}, there exists a unique solution
$u^{\e}(\vx,\vw)\in L^{\infty}(\Omega\times\s^2)$. Moreover, for any $0<\d<<1$, the solution obeys the estimate
\begin{align}
\im{u^{\e}-\u-\uu}{\Omega\times\s^2}\leq C(\d)\e^{\frac{1}{3}-\d},
\end{align}
where $\u(\vx)$ satisfies the Laplace equation with Dirichlet boundary condition
\begin{align}
\left\{
\begin{array}{l}
\Delta_x\u(\vx)=0\ \ \text{in}\
\ \Omega,\\\rule{0ex}{1.5em}
\u(\vx_0)=D(\vx_0)\ \ \text{on}\ \
\p\Omega,
\end{array}
\right.
\end{align}
and $\uu(\eta,\iota_1,\iota_2,\phi,\psi)$ satisfies the $\e$-Milne problem with geometric correction
\begin{align}
\left\{
\begin{array}{l}
\sin\phi\dfrac{\p \uu }{\p\eta}+F(\e;\eta,\psi,\iota_1,\iota_2)\cos\phi\dfrac{\p
\uu }{\p\phi}+\uu -\buu =0,\\\rule{0ex}{1.5em}
\uu (0,\iota_1,\iota_2,\phi,\psi)=g(\iota_1,\iota_2,\phi,\psi)-D(\iota_1,\iota_2)\ \ \text{for}\ \
\sin\phi>0,\\\rule{0ex}{1.5em}
\uu (L,\iota_1,\iota_2,\phi,\psi)=\uu (L,\iota_1,\iota_2,\rr[\phi],\psi),
\end{array}
\right.
\end{align}
for $L=\e^{-\frac{1}{2}}$, $\rr[\phi]=-\phi$, $\eta$ the rescaled normal variable, $(\iota_1,\iota_2)$ the tangential variable, and $(\phi,\psi)$ the velocity variables (see Section \ref{substitution}).
\end{theorem}

\begin{remark}
Note that the effects of the boundary layer decays very fast when it is away from the boundary. Roughly speaking, this theorem states that for $\vx$ not very close to the boundary, $u^{\e}(\vx,\vw)$ can be approximated by the solution of a Laplace equation with Dirichlet boundary condition.
\end{remark}

\subsection{Unsteady Problem}

\begin{theorem}[Unsteady Diffusive Limit]\label{main theorem.}
Assume $h(\vx,\vw)\in C(\Omega\times\s^2)$ and $g(t,\vx_0,\vw)\in C^3(\rp\times\Gamma^-)$. Then for the unsteady neutron
transport equation \eqref{transport.}, there exists a unique solution
$u^{\e}(t,\vx,\vw)\in L^{\infty}(\rp\times\Omega\times\s^2)$. Moreover,
\begin{align}
\lim_{\e\rt0}\nm{\ue^{K_0t}\Big(u^{\e}-\u-\ui-\uu\Big)}_{L^{\infty}(\rp\times\Omega\times\s^2)}=0,
\end{align}
where $\u(\vx)$ satisfies the heat equation with Dirichlet boundary condition
\begin{align}
\left\{
\begin{array}{l}
\dt\u-\Delta_x\u(\vx)=0\ \ \text{in}\
\ \rp\times\Omega,\\\rule{0ex}{1.5em}
\u(0,\vx)=\ds\frac{1}{4\pi}\int_{\s^2}h(\vx,\vw)\ud{\vw}\ \ \text{in}\ \ \Omega,\\\rule{0ex}{1.5em}
\u(t,\vx_0)=D(t,\vx_0)\ \ \text{on}\ \
\p\Omega,
\end{array}
\right.
\end{align}
$\ui(\tau,\vx,\vw)$ satisfies
\begin{align}
\ui(\tau,\vx,\vw)=\ue^{-\tau}\left(h(\vx,\vw)-\frac{1}{4\pi}\int_{\s^2}h(\vx,\vw)\ud{\vw}\right),
\end{align}
for $\tau$ the rescaled time variable, and $\uu(t,\eta,\tau,\phi)$ satisfies the $\e$-Milne problem with geometric correction
\begin{align}
\left\{
\begin{array}{l}
\sin\phi\dfrac{\p \uu }{\p\eta}-\bigg(\dfrac{\sin^2\psi}{R_1(\tau)-\e\eta}+\dfrac{\cos^2\psi}{R_2(\tau)-\e\eta}\bigg)\cos\phi\dfrac{\p
\uu}{\p\phi}+\uu -\buu =0,\\\rule{0ex}{1.5em}
\uu (t,0,\tau,\phi,\psi)=g(t,\tau,\phi,\psi)-D(t,\tau)\ \ \text{for}\ \
\sin\phi>0,\\\rule{0ex}{1.5em}
\uu (t,L,\tau,\phi,\psi)=\uu (t,L,\tau,\rr[\phi],\psi),
\end{array}
\right.
\end{align}
for $L=\e^{-\frac{1}{2}}$, $\rr[\phi]=-\phi$, $\eta$ the rescaled normal variable, $\tau$ the tangential variable, and $\phi$, $\psi$ the velocity variable (see Section \ref{substitution..} and \ref{substitution.}).
\end{theorem}

\begin{remark}
Note that the effects of the boundary layer decays very fast when it is away from the boundary. Roughly speaking, this theorem states that for $\vx$ not very close to the boundary, $u^{\e}(t,\vx,\vw)$ can be approximated by the solution of a heat equation with Dirichlet boundary condition.
\end{remark}

\section{Notation and Convention}

Throughout this paper, $C>0$ denotes a constant that only depends on
the domain $\Omega$, but does not depend on the data or $\e$. It is
referred as universal and can change from one inequality to another.
When we write $C(z)$, it means a certain positive constant depending
on the quantity $z$. We write $a\ls b$ to denote $a\leq Cb$.

This paper is organized as follows: in Chapter 2, we study the steady problem, and in Chapter 3, we study the unsteady problem. Chapter 4 focuses on the analysis of boundary layer equation, i.e. the $\e$-Milne problem with geometric correction.


\chapter{Steady Neutron Transport Equation}

In this chapter, we prove the diffusive limit of the steady neutron transport equation \eqref{transport}.

\section{Asymptotic Expansions}

\subsection{Interior Expansion}

We define the interior expansion as follows:
\begin{align}\label{interior expansion}
\u(\vx,\vw)\sim\u_0(\vx,\vw)+\e\u_1(\vx,\vw)+\e^2\u_2(\vx,\vw),
\end{align}
where $\u_k$ can be determined by comparing the order of $\e$ via
plugging \eqref{interior expansion} into the equation
\eqref{transport}. Thus we have
\begin{align}
\u_0-\bu_0=&0,\label{expansion temp 1}\\
\u_1-\bu_1=&-\vw\cdot\nx\u_0,\label{expansion temp 2}\\
\u_2-\bu_2=&-\vw\cdot\nx\u_1.\label{expansion temp 3}
\end{align}
Plugging \eqref{expansion temp 1} into \eqref{expansion temp 2},
we obtain
\begin{align}
\u_1=\bu_1-\vw\cdot\nx\bu_0.\label{expansion temp 4}
\end{align}
Plugging \eqref{expansion temp 4} into \eqref{expansion temp 3},
we get
\begin{align}\label{expansion temp 5}
\u_2-\bu_2=&-\vw\cdot\nx(\bu_1-\vw\cdot\nx\bu_0)\\
=&-\vw\cdot\nx\bu_1+\Big(w_1^2\p_{x_1x_1}\bu_0+w_2^2\p_{x_2x_2}\bu_0+w_3^2\p_{x_3x_3}\bu_0\Big)\no\\
&+2\Big(w_1w_2\p_{x_1x_2}\bu_0+w_1w_3\p_{x_1x_3}\bu_0+w_2w_3\p_{x_2x_3}\bu_0\Big).\no
\end{align}
Integrating \eqref{expansion temp 5} over $\vw\in\s^2$, we have
the final form
\begin{align}
\Delta_x\bu_0=0,
\end{align}
where all cross terms vanish due to the symmetry of $\s^2$. Hence, $\u_0(\vx,\vw)$ satisfies the equation
\begin{align}\label{interior 1}
\left\{ \begin{array}{l} \u_0=\bu_0,\\\rule{0ex}{1.0em}
\Delta_x\bu_0=0.
\end{array}
\right.
\end{align}
In a similar fashion, for $k=1,2$, we may define that $\u_k$ satisfies
\begin{align}\label{interior 2}
\left\{ \begin{array}{l} \u_k=\bu_k-\vw\cdot\nx\u_{k-1},\\
\Delta_x\bu_k=\displaystyle-\int_{\s^2}\vw\cdot\nx\u_{k-1}\ud{\vw}.\end{array}
\right.
\end{align}
It is easy to see that $\bu_k$ satisfies an elliptic equation. However, the boundary condition of $\bu_k$ is unknown at this stage, since generally $\u_k$ does not necessarily satisfy the in-flow boundary condition of \eqref{transport}. Therefore, we have to resort to boundary layer analysis.

\subsection{Quasi-Spherical Coordinate System}\label{substitution}

While the interior solution can be well-defined using the standard Cartesian coordinate system, the boundary layer requires a local description in a neighborhood of the physical boundary $\p\Omega$. We call it a quasi-spherical coordinate system. The construction can be divided into the following substitutions:\\
\ \\
\textbf{Substitution 1: Spacial Substitution:}\\
Consider the three-dimensional transport operator $\vw\cdot\nx$. By standard differential geometry, in a neighborhood of $\vx_0\in\p\Omega$, we can always define an orthogonal curvilinear coordinates system $(\iota_1,\iota_2)$ such that at $\vx_0$ the coordinate lines coincide with the principal directions (if $\vx_0$ is not a umbilical point, then we can extend such properties to the whole neighborhood; note that the equation we derive is invariant under the change of coordinate system since it only depends on normal vector and tangential plane, so we only need to choose the simplest one). The boundary surface is $\vr=\vr(\iota_1,\iota_2)$. In addition, $\p_1\vr$ and $\p_2\vr$ denote two orthogonal tangential vectors. Then the outward unit normal vector is
\begin{align}\label{coordinate 1}
\vn=\frac{\p_1\vr\times\p_2\vr}{\abs{\p_1\vr\times\p_2\vr}}.
\end{align}
Here $\abs{\cdot}$ denotes the length and $\p_i$ denotes the derivative with respect to $\iota_i$. Let
\begin{align}\label{coordinate 2}
P=\abs{\p_1\vr\times\p_2\vr}=\abs{\p_1\vr}\abs{\p_2\vr}=P_1P_2,
\end{align}
for $P_i=\abs{\p_i\vr}$ with the unit tangential vectors
\begin{align}\label{coordinate 14}
\vt_1=\frac{\p_1\vr}{P_1},\ \ \vt_2=\frac{\p_2\vr}{P_2}.
\end{align}
Then consider the new coordinate system $(\mu,\iota_1,\iota_2)$, where $\mu$ denotes the normal distance to boundary surface $\p\Omega$, i.e.
\begin{align}
\vx=\vr-\mu\vn.
\end{align}
The transport operator becomes
\begin{align}\label{coordinate 8}
\vw\cdot\nx=&-\frac{\bigg(\left(\p_1\vr-\mu\p_1\vn\right)\times\left(\p_2\vr-\mu\p_2\vn\right)\bigg)\cdot\vw}
{\bigg(\left(\p_1\vr-\mu\p_1\vn\right)\times\left(\p_2\vr-\mu\p_2\vn\right)\bigg)\cdot\vn}\frac{\p f}{\p\mu}\\
&+\frac{\bigg(\left(\p_2\vr-\mu\p_2\vn\right)\times\vn\bigg)\cdot\vw}{\bigg(\left(\p_1\vr-\mu\p_1\vn\right)\times\left(\p_2\vr-\mu\p_2\vn\right)\bigg)\cdot\vn}\frac{\p f}{\p\iota_1}-\frac{\bigg(\left(\p_1\vr-\mu\p_1\vn\right)\times\vn\bigg)\cdot\vw}{\bigg(\left(\p_1\vr-\mu\p_1\vn\right)\times\left(\p_2\vr-\mu\p_2\vn\right)\bigg)\cdot\vn}\frac{\p f}{\p\iota_2}.\no
\end{align}
As usual, at $\vx_0$, define the first fundamental form $(E,F,G)=\Big(\p_1\vr\cdot\p_1\vr, \p_1\vr\cdot\p_2\vr, \p_2\vr\cdot\p_2\vr\Big)$ and second fundamental form $(L,M,N)=\Big(\p_{11}\vr\cdot\vn, \p_{12}\vr\cdot\vn, \p_{22}\vr\cdot\vn\Big)$. Then we have $F=M=0$ due to the orthogonality and the principal curvatures are
\begin{align}
\kk_1=\frac{L}{E},\ \ \kk_2=\frac{N}{G}.
\end{align}
Note that $\kk_1$ and $\kk_2$ depend on $(\iota_1,\iota_2)$ and can change from point to point on $\p\Omega$. Also,
\begin{align}\label{coordinate 3}
\p_1\vn=\kk_1\p_1\vr,\ \ \p_2\vn=\kk_2\p_2\vr.
\end{align}
Hence, direct computation using \eqref{coordinate 1}, \eqref{coordinate 2} and \eqref{coordinate 3} reveals that
\begin{align}
\bigg(\left(\p_1\vr-\mu\p_1\vn\right)\times\left(\p_2\vr-\mu\p_2\vn\right)\bigg)\cdot\vn=&(1-\kk_1\mu)(1-\kk_2\mu)(\p_1\vr\times\p_2\vr)\cdot\vn=(1-\kk_1\mu)(1-\kk_2\mu)P,
\label{coordinate 4}\\
\label{coordinate 5}\\
\bigg(\left(\p_1\vr-\mu\p_1\vn\right)\times\left(\p_2\vr-\mu\p_2\vn\right)\bigg)\cdot\vw=&(1-\kk_1\mu)(1-\kk_2\mu)(\p_1\vr\times\p_2\vr)\cdot\vw
=(1-\kk_1\mu)(1-\kk_2\mu)P(\vw\cdot\vn),\no
\end{align}
and
\begin{align}
\bigg(\left(\p_2\vr-\mu\p_2\vn\right)\times\vn\bigg)\cdot\vw=&(1-\kk_2\mu)(\p_2\vr\times\vn)\cdot\vw=(1-\kk_2\mu)P_2(\vw\cdot\vt_1),\label{coordinate 6}\\
\bigg(\left(\p_1\vr-\mu\p_1\vn\right)\times\vn\bigg)\cdot\vw=&(1-\kk_1\mu)(\p_1\vr\times\vn)\cdot\vw=-(1-\kk_1\mu)P_1(\vw\cdot\vt_2).\label{coordinate 7}
\end{align}
Hence, plugging \eqref{coordinate 4}, \eqref{coordinate 5}, \eqref{coordinate 6} and \eqref{coordinate 7} into \eqref{coordinate 8}, we have the transport operator
\begin{align}
\vw\cdot\nx=-(\vw\cdot\vn)\frac{\p}{\p\mu}-\frac{\vw\cdot\vt_1}{P_1(\kk_1\mu-1)}\frac{\p}{\p\iota_1}-\frac{\vw\cdot\vt_2}{P_2(\kk_2\mu-1)}\frac{\p}{\p\iota_2}.
\end{align}
Above computation is valid in a neighborhood $\Sigma\subset\p\Omega$ of the boundary surface. Let
\begin{align}
R_{\min}=\min_{\iota_1,\iota_2}\{R_1(\iota_1,\iota_2),R_2(\iota_1,\iota_2)\},
\end{align}
where $R_1(\iota_1,\iota_2)=\dfrac{1}{\kk_1(\iota_1,\iota_2)}$ and $R_2(\iota_1,\iota_2)=\dfrac{1}{\kk_2(\iota_1,\iota_2)}$. Therefore, under the substitution $(x_1,x_2,x_3)\rt(\mu,\iota_1,\iota_2)$ for $0\leq\mu<R_{\min}$, the equation \eqref{transport} is transformed into
\begin{align}\label{coordinate 9}
\left\{
\begin{array}{l}\displaystyle
\e\bigg(-(\vw\cdot\vn)\frac{\p u^{\e}}{\p\mu}-\frac{\vw\cdot\vt_1}{P_1(\kk_1\mu-1)}\frac{\p u^{\e}}{\p\iota_1}-\frac{\vw\cdot\vt_2}{P_2(\kk_2\mu-1)}\frac{\p u^{\e}}{\p\iota_2}\bigg)+u^{\e}-\bar u^{\e}=0\ \ \text{in}\ \ (0,R_{\min})\times\Sigma\times\s^2,\\\rule{0ex}{2.0em}
u^{\e}(0,\iota_1,\iota_2,\vw)=g(\iota_1,\iota_2,\vw)\ \ \text{for}\
\ \vw\cdot\vn<0.
\end{array}
\right.
\end{align}
\ \\
\textbf{Substitution 2: Velocity Substitution:}\\
Define the orthogonal velocity substitution
\begin{align}\label{coordinate 11}
\left\{
\begin{aligned}
-\vw\cdot\vn=&\sin\phi,\\
\vw\cdot\vt_1=&\cos\phi\sin\psi,\\
\vw\cdot\vt_2=&\cos\phi\cos\psi,
\end{aligned}
\right.
\end{align}
for $\phi\in\left[-\dfrac{\pi}{2},\dfrac{\pi}{2}\right]$ and $\psi\in[-\pi,\pi]$. Then using chain rule, we may directly compute
\begin{align}
\frac{\p}{\p\iota_1}\rt \frac{\p}{\p\iota_1}+\frac{\p}{\p\phi}\frac{\p\phi}{\p\iota_1}+\frac{\p}{\p\psi}\frac{\p\psi}{\p\iota_1},\quad
\frac{\p}{\p\iota_2}\rt \frac{\p}{\p\iota_2}+\frac{\p}{\p\phi}\frac{\p\phi}{\p\iota_2}+\frac{\p}{\p\psi}\frac{\p\psi}{\p\iota_2}.
\end{align}
We first compute the derivatives of $\phi$. Taking $\iota_i$ derivative on both sides of the first equation of \eqref{coordinate 11}, we have
\begin{align}
-\vw\cdot\p_i\vn=\cos\phi\frac{\p\phi}{\p\iota_i},
\end{align}
which implies
\begin{align}
\frac{\p\phi}{\p\iota_i}=-\frac{\vw\cdot\p_i\vn}{\cos\phi}=-\frac{\kk_i(\vw\cdot\p_i\vr)}{\cos\phi}=-\frac{\kk_iP_i(\vw\cdot\vt_i)}{\cos\phi}
\end{align}
Then using the second and third equations of \eqref{coordinate 11}, we obtain
\begin{align}\label{coordinate 12}
\frac{\p\phi}{\p\iota_1}=-\kk_1P_1\sin\psi,\quad\frac{\p\phi}{\p\iota_2}=-\kk_2P_2\cos\psi.
\end{align}
Then we consider the derivatives of $\psi$. There seems to be multiple ways to achieve this in the second and third equations of \eqref{coordinate 11}. However, in order to avoid the singularity caused by $\phi$ and $\psi$ in the denominator of final expression, we must be very careful. Taking $\iota_2$ derivative on both sides of the second equation of \eqref{coordinate 11}, we obtain
\begin{align}
\vw\cdot\p_2\vt_1=-\sin\phi\sin\psi\frac{\p\phi}{\p\iota_2}+\cos\phi\cos\psi\frac{\p\psi}{\p\iota_2},
\end{align}
which implies
\begin{align}\label{coordinate 13}
\frac{\p\psi}{\p\iota_2}=\frac{1}{\cos\phi\cos\psi}\bigg(\vw\cdot\p_2\vt_1+\sin\phi\sin\psi\frac{\p\phi}{\p\iota_2}\bigg).
\end{align}
We need to compute the two terms on the right-hand side separately. Using \eqref{coordinate 14}, we obtain
\begin{align}\label{coordinate 15}
\vw\cdot\p_2\vt_1=&\vw\cdot\p_2\left(\frac{\p_1\vr}{\abs{\p_1\vr}}\right)=\vw\cdot\frac{\p_{12}\vr\abs{\p_1\vr}^2-\p_1\vr(\p_1\vr\cdot\p_{12}\vr)}{\abs{\p_1\vr}^3}
=\vw\cdot\frac{\p_{12}\vr-\vt_1(\vt_1\cdot\p_{12}\vr)}{P_1}=\vw\cdot\frac{\vt_1\times(\p_{12}\vr\times\vt_1)}{P_1}.
\end{align}
We can easily see that $\p_{12}\vr-\vt_1(\vt_1\cdot\p_{12}\vr)$ denotes the projection of $\p_{12}\vr$ into the subspace spanned by $\vn$ and $\vt_2$. In addition, the second fundamental form $M=\p_{12}\vr\cdot\vn=0$ indicates that $\p_{12}\vr$ is orthogonal to $\vn$. Hence, we may decompose
\begin{align}
\vw=\vn(\vw\cdot\vn)+\vt_1(\vw\cdot\vt_1)+\vt_2(\vw\cdot\vt_2),
\end{align}
where $\vn(\vw\cdot\vn)$ and $\vt_1(\vw\cdot\vt_1)$ term vanish in \eqref{coordinate 15}. Therefore, using \eqref{coordinate 11}, we know
\begin{align}\label{coordinate 16}
\vw\cdot\p_2\vt_1=&(\vw\cdot\vt_2)\frac{\vt_2\cdot\Big(\vt_1\times(\p_{12}\vr\times\vt_1)\Big)}{P_1}
=\cos\phi\cos\psi\frac{\vt_2\cdot\Big(\vt_1\times(\p_{12}\vr\times\vt_1)\Big)}{P_1}.
\end{align}
On the other hand, using \eqref{coordinate 12}, we can directly compute
\begin{align}\label{coordinate 17}
\sin\phi\sin\psi\frac{\p\phi}{\p\iota_2}=-\kk_2P_2\sin\phi\sin\psi\cos\psi.
\end{align}
Inserting \eqref{coordinate 16} and \eqref{coordinate 17} into \eqref{coordinate 13}, we get
\begin{align}\label{coordinate 18}
\frac{\p\psi}{\p\iota_2}=&\frac{1}{\cos\phi\cos\psi}
\bigg(\cos\phi\cos\psi\frac{\vt_2\cdot\Big(\vt_1\times(\p_{12}\vr\times\vt_1)\Big)}{P_1}
-\kk_2P_2\sin\phi\sin\psi\cos\psi\bigg)\\
=&\frac{\vt_2\cdot\Big(\vt_1\times(\p_{12}\vr\times\vt_1)\Big)}{P_1}
-\kk_2P_2\tan\phi\sin\psi.\no
\end{align}
In a similar fashion, taking $\iota_1$ derivative on both sides of the third equation of \eqref{coordinate 11}, we obtain
\begin{align}
\vw\cdot\p_1\vt_2=-\sin\phi\cos\psi\frac{\p\phi}{\p\iota_1}-\cos\phi\sin\psi\frac{\p\psi}{\p\iota_1},
\end{align}
which implies
\begin{align}\label{coordinate 13'}
\frac{\p\psi}{\p\iota_1}=-\frac{1}{\cos\phi\sin\psi}\bigg(\vw\cdot\p_1\vt_2+\sin\phi\cos\psi\frac{\p\phi}{\p\iota_1}\bigg).
\end{align}
We need to compute the two terms on the right-hand side separately. Using \eqref{coordinate 14}, we obtain
\begin{align}\label{coordinate 15'}
\vw\cdot\p_1\vt_2=&\vw\cdot\p_1\left(\frac{\p_2\vr}{\abs{\p_2\vr}}\right)=\vw\cdot\frac{\p_{12}\vr\abs{\p_2\vr}^2-\p_2\vr(\p_2\vr\cdot\p_{12}\vr)}{\abs{\p_2\vr}^3}
=\vw\cdot\frac{\p_{12}\vr-\vt_2(\vt_2\cdot\p_{12}\vr)}{P_2}=\vw\cdot\frac{\vt_2\times(\p_{12}\vr\times\vt_2)}{P_2}.
\end{align}
Here $\p_{12}\vr-\vt_2(\vt_2\cdot\p_{12}\vr)$ denotes the projection of $\p_{12}\vr$ into the subspace spanned by $\vn$ and $\vt_1$. Still, $\p_{12}\vr$ is orthogonal to $\vn$. Hence, we may decompose
\begin{align}
\vw=\vn(\vw\cdot\vn)+\vt_1(\vw\cdot\vt_1)+\vt_2(\vw\cdot\vt_2),
\end{align}
where $\vn(\vw\cdot\vn)$ and $\vt_2(\vw\cdot\vt_1)$ term vanish in \eqref{coordinate 15'}. Therefore, using \eqref{coordinate 11}, we know
\begin{align}\label{coordinate 16'}
\vw\cdot\p_1\vt_2=&(\vw\cdot\vt_1)\frac{\vt_1\cdot\Big(\vt_2\times(\p_{12}\vr\times\vt_2)\Big)}{P_2}=
\cos\phi\sin\psi\frac{\vt_1\cdot\Big(\vt_2\times(\p_{12}\vr\times\vt_2)\Big)}{P_2}.
\end{align}
On the other hand, using \eqref{coordinate 12}, we can directly compute
\begin{align}\label{coordinate 17'}
\sin\phi\cos\psi\frac{\p\phi}{\p\iota_1}=-\kk_1P_1\sin\phi\sin\psi\cos\psi.
\end{align}
Inserting \eqref{coordinate 16'} and \eqref{coordinate 17'} into \eqref{coordinate 13'}, we get
\begin{align}\label{coordinate 18'}
\frac{\p\psi}{\p\iota_1}=&-\frac{1}{\cos\phi\sin\psi}
\bigg(\cos\phi\sin\psi\frac{\vt_1\cdot\Big(\vt_2\times(\p_{12}\vr\times\vt_2)\Big)}{P_2}
-\kk_1P_1\sin\phi\sin\psi\cos\psi\bigg)\\
=&-\frac{\vt_1\cdot\Big(\vt_2\times(\p_{12}\vr\times\vt_2)\Big)}{P_2}
+\kk_1P_1\tan\phi\cos\psi.\no
\end{align}
Assembling \eqref{coordinate 12}, \eqref{coordinate 18} and \eqref{coordinate 18'}, we have
\begin{align}
\frac{\p}{\p\iota_1}\rt&\frac{\p}{\p\iota_1}-\kk_1P_1\sin\psi\frac{\p}{\p\phi}
+\bigg(-\frac{\vt_1\cdot\Big(\vt_2\times(\p_{12}\vr\times\vt_2)\Big)}{P_2}
+\kk_1P_1\tan\phi\cos\psi\bigg)\frac{\p}{\p\psi},\label{coordinate 19}\\
\frac{\p}{\p\iota_2}\rt&\frac{\p}{\p\iota_2}-\kk_2P_2\cos\psi\frac{\p}{\p\phi}
+\bigg(\frac{\vt_2\cdot\Big(\vt_1\times(\p_{12}\vr\times\vt_1)\Big)}{P_1}
-\kk_2P_2\tan\phi\sin\psi\bigg)\frac{\p}{\p\psi}.\label{coordinate 20}
\end{align}
Then inserting \eqref{coordinate 12}, \eqref{coordinate 19} and \eqref{coordinate 20} into \eqref{coordinate 9}, we obtain the transport operator
\begin{align}
\vw\cdot\nx=&\sin\phi\frac{\p}{\p\mu}-\bigg(\frac{\sin^2\psi}{R_1-\mu}+\frac{\cos^2\psi}{R_2-\mu}\bigg)\cos\phi\frac{\p}{\p\phi}\\
&+\bigg(\frac{\cos\phi\sin\psi}{P_1(1-\kk_1\mu)}\frac{\p}{\p\iota_1}+\frac{\cos\phi\cos\psi}{P_2(1-\kk_2\mu)}\frac{\p}{\p\iota_2}\bigg)\no\\
&+\Bigg(\frac{\sin\psi}{1-\kk_1\mu}\bigg(\cos\phi\Big(\vt_1\cdot\Big(\vt_2\times(\p_{12}\vr\times\vt_2)\Big)\Big)
-\kk_1P_1P_2\sin\phi\cos\psi\bigg)\no\\
&
+\frac{\cos\psi}{1-\kk_2\mu}\bigg(-\cos\phi\Big(\vt_2\cdot\Big(\vt_1\times(\p_{12}\vr\times\vt_1)\Big)\Big)
+\kk_2P_1P_2\sin\phi\sin\psi\bigg)\Bigg)\frac{1}{P_1P_2}\frac{\p}{\p\psi}.\no
\end{align}
Hence, under substitution $(w_1,w_2,w_3)\rt(\phi,\psi)$ for $\phi\in\left[-\dfrac{\pi}{2},\dfrac{\pi}{2}\right]$ and $\psi\in[-\pi,\pi]$,
the equation \eqref{transport} is transformed into
\begin{align}\label{coordinate 21}
\left\{
\begin{array}{l}\displaystyle
\e\sin\phi\dfrac{\p u^{\e}}{\p\mu}-\e\bigg(\dfrac{\sin^2\psi}{R_1-\mu}+\dfrac{\cos^2\psi}{R_2-\mu}\bigg)\cos\phi\dfrac{\p u^{\e}}{\p\phi}\\\rule{0ex}{2.0em}
+\e\bigg(\dfrac{\cos\phi\sin\psi}{P_1(1-\kk_1\mu)}\dfrac{\p u^{\e}}{\p\iota_1}+\dfrac{\cos\phi\cos\psi}{P_2(1-\kk_2\mu)}\dfrac{\p u^{\e}}{\p\iota_2}\bigg)\\\rule{0ex}{2.0em}
+\e\Bigg(\dfrac{\sin\psi}{1-\kk_1\mu}\bigg(\cos\phi\Big(\vt_1\cdot\Big(\vt_2\times(\p_{12}\vr\times\vt_2)\Big)\Big)
-\kk_1P_1P_2\sin\phi\cos\psi\bigg)\\\rule{0ex}{2.0em}
+\dfrac{\cos\psi}{1-\kk_2\mu}\bigg(-\cos\phi\Big(\vt_2\cdot\Big(\vt_1\times(\p_{12}\vr\times\vt_1)\Big)\Big)
+\kk_2P_1P_2\sin\phi\sin\psi\bigg)\Bigg)\dfrac{1}{P_1P_2}\dfrac{\p u^{\e}}{\p\psi}\\
+ u^{\e}-\bar u^{\e}=0\ \ \text{in}\ \ (0,R_{\min})\times\Sigma\times\left[-\dfrac{\pi}{2},\dfrac{\pi}{2}\right]\times[-\pi,\pi],\\\rule{0ex}{2.0em}
u^{\e}(0,\iota_1,\iota_2,\phi,\psi)=g(\iota_1,\iota_2,\phi,\psi)\ \
\text{for}\ \ \sin\phi>0.
\end{array}
\right.
\end{align}
\ \\
\textbf{Substitution 3: Scaling Substitution:}\\
Define the scaled variable $\eta=\dfrac{\mu}{\e}$, which implies $\dfrac{\p}{\p\mu}=\dfrac{1}{\e}\dfrac{\p}{\p\eta}$. Then, under the substitution $\mu\rt\eta$, the equation \eqref{transport} is transformed into
\begin{align}\label{coordinate 22}
\left\{
\begin{array}{l}\displaystyle
\sin\phi\dfrac{\p u^{\e}}{\p\eta}-\e\bigg(\dfrac{\sin^2\psi}{R_1-\e\eta}+\dfrac{\cos^2\psi}{R_2-\e\eta}\bigg)\cos\phi\dfrac{\p u^{\e}}{\p\phi}\\\rule{0ex}{2.0em}
+\e\bigg(\dfrac{\cos\phi\sin\psi}{P_1(1-\e\kk_1\eta)}\dfrac{\p u^{\e}}{\p\iota_1}+\dfrac{\cos\phi\cos\psi}{P_2(1-\e\kk_2\eta)}\dfrac{\p u^{\e}}{\p\iota_2}\bigg)\\\rule{0ex}{2.0em}
+\e\Bigg(\dfrac{\sin\psi}{1-\e\kk_1\eta}\bigg(\cos\phi\Big(\vt_1\cdot\Big(\vt_2\times(\p_{12}\vr\times\vt_2)\Big)\Big)
-\kk_1P_1P_2\sin\phi\cos\psi\bigg)\\\rule{0ex}{2.0em}
+\dfrac{\cos\psi}{1-\e\kk_2\eta}\bigg(-\cos\phi\Big(\vt_2\cdot\Big(\vt_1\times(\p_{12}\vr\times\vt_1)\Big)\Big)
+\kk_2P_1P_2\sin\phi\sin\psi\bigg)\Bigg)\dfrac{1}{P_1P_2}\dfrac{\p u^{\e}}{\p\psi}+u^{\e}-\bar u^{\e}=0\\
+ u^{\e}-\bar u^{\e}=0\ \ \text{in}\ \ \left(0,\dfrac{R_{\min}}{\e}\right)\times\Sigma\times\left[-\dfrac{\pi}{2},\dfrac{\pi}{2}\right]\times[-\pi,\pi],\\\rule{0ex}{2.0em}
u^{\e}(0,\iota_1,\iota_2,\phi,\psi)=g(\iota_1,\iota_2,\phi,\psi)\ \
\text{for}\ \ \sin\phi>0.
\end{array}
\right.
\end{align}

\subsection{Boundary Layer Expansion}

We define the boundary layer expansion as follows:
\begin{align}\label{boundary layer expansion}
\uu(\eta,\iota_1,\iota_2,\phi,\psi)\sim\uu_0(\eta,\iota_1,\iota_2,\phi,\psi)+\e\uu_1(\eta,\iota_1,\iota_2,\phi,\psi),
\end{align}
where $\uu_k$ can be defined by comparing the order of $\e$ via
plugging \eqref{boundary layer expansion} into the equation
\eqref{coordinate 22}. Thus, in a neighborhood of the boundary, we have
\begin{align}
\sin\phi\frac{\p\uu_0}{\p\eta}-\e\bigg(\dfrac{\sin^2\psi}{R_1-\e\eta}+\dfrac{\cos^2\psi}{R_2-\e\eta}\bigg)\cos\phi\dfrac{\p \uu_0}{\p\phi}+\uu_0-\buu_0=&0,\label{expansion temp 6}\\
\sin\phi\frac{\p\uu_1}{\p\eta}-\e\bigg(\dfrac{\sin^2\psi}{R_1-\e\eta}+\dfrac{\cos^2\psi}{R_2-\e\eta}\bigg)\cos\phi\dfrac{\p \uu_1}{\p\phi}+\uu_1-\buu_1=&-G[\uu_0],\label{expansion temp 7}
\end{align}
where
\begin{align}\label{coordinate 23}
G[\uu_0]=&\bigg(\dfrac{\cos\phi\sin\psi}{P_1(1-\e\kk_1\eta)}\dfrac{\p \uu_0}{\p\iota_1}+\dfrac{\cos\phi\cos\psi}{P_2(1-\e\kk_2\eta)}\dfrac{\p \uu_0}{\p\iota_2}\bigg)\\
&+\Bigg(\dfrac{\sin\psi}{1-\e\kk_1\eta}\bigg(\cos\phi\Big(\vt_1\cdot\Big(\vt_2\times(\p_{12}\vr\times\vt_2)\Big)\Big)
-\kk_1P_1P_2\sin\phi\cos\psi\bigg)\no\\
&+\dfrac{\cos\psi}{1-\e\kk_2\eta}\bigg(-\cos\phi\Big(\vt_2\cdot\Big(\vt_1\times(\p_{12}\vr\times\vt_1)\Big)\Big)
+\kk_2P_1P_2\sin\phi\sin\psi\bigg)\Bigg)\dfrac{1}{P_1P_2}\dfrac{\p \uu_0}{\p\psi},\no
\end{align}
and
\begin{align}
\buu_k(\eta,\iota_1,\iota_2)=\frac{1}{4\pi}\int_{-\pi}^{\pi}\int_{-\frac{\pi}{2}}^{\frac{\pi}{2}}\uu_k(\eta,\iota_1,\iota_2,\phi,\psi)\cos\phi\ud{\phi}\ud{\psi}.
\end{align}
We call this type of equations the $\e$-Milne problem with geometric correction.

\subsection{Decomposition and Modification}

In this section, we prove the important decomposition of boundary data, which can greatly improve the regularity. The idea is adapted from \cite{Li.Lu.Sun2017} for the flat Milne problem.

Consider the $\e$-Milne problem with geometric correction with $L=\e^{-\frac{1}{2}}$ and $\rr[\phi]=-\phi$,
\begin{align}\label{deco}
\left\{
\begin{array}{l}
\sin\phi\dfrac{\p f}{\p\eta}-\e\bigg(\dfrac{\sin^2\psi}{R_1-\e\eta}+\dfrac{\cos^2\psi}{R_2-\e\eta}\bigg)\cos\phi\dfrac{\p f}{\p\phi}+f-\bar f=0,\\\rule{0ex}{1.5em}
f(0,\phi,\psi)=g(\phi,\psi)\ \ \text{for}\ \
\sin\phi>0,\\\rule{0ex}{1.5em}
f(L,\phi,\psi)=f(L,\rr[\phi],\psi).
\end{array}
\right.
\end{align}
We assume that $g(\phi,\psi)$ is not a constant and $0\leq g(\phi,\psi)\leq 1$. This is always achievable and we do not lose the generality since the equation is linear. For some $\alpha>0$ which will be determined later, define two $C^{\infty}$ auxiliary functions for $\phi\in\left[0,\dfrac{\pi}{2}\right]$ and $\psi\in[-\pi,\pi]$,
\begin{align}
g_1(\phi,\psi)=\left\{
\begin{array}{ll}
0&\ \ \text{for}\ \ \phi\in(0,\e^{\alpha}],\\
g(\phi,\psi)&\ \ \text{for}\ \ \phi\in\left[2\e^{\alpha},\dfrac{\pi}{2}\right],\\
\end{array}
\right.
\end{align}
and
\begin{align}
g_2(\phi,\psi)=\left\{
\begin{array}{ll}
1&\ \ \text{for}\ \ \phi\in(0,\e^{\alpha}],\\
g(\phi,\psi)&\ \ \text{for}\ \ \phi\in\left[2\e^{\alpha},\dfrac{\pi}{2}\right].\\
\end{array}
\right.
\end{align}
A standard construction using mollifier justifies the existence of $g_i$ for $i=1,2$. Also, we can easily obtain $\lss{\dfrac{\p g_i}{\p\phi}}{\phi,\psi}\leq C\e^{-\alpha}$ and $\lss{\dfrac{\p^2 g_i}{\p\phi^2}}{\phi,\psi}\leq C\e^{-2\alpha}$. In addition, it is easy to verify that $\lss{\dfrac{\p g_i}{\p\psi}}{\phi,\psi}\leq C$. Let $f_1(\eta,\phi,\psi)$ and $f_2(\eta,\phi,\psi)$ be the solutions to the equation \eqref{deco} with in-flow data $g_1(\phi,\psi)$ and $g_2(\phi,\psi)$ respectively. Then by Theorem \ref{Milne theorem 1}, we know $f_1$ and $f_2$ are well-defined in $L^{\infty}_{\eta,\phi,\psi}$. By Theorem \ref{Milne theorem 3}, they satisfy the maximum principle, which means
\begin{align}
f_1(0,0^+,\psi)-\bar f_1(0)=&-\bar f_1(0)<0,\\
f_2(0,0^+,\psi)-\bar f_2(0)=&1-\bar f_2(0)>0.
\end{align}
Therefore, there exists a constant $0<\l<1$ such that
\begin{align}
\l\Big(f_1(0,0^+,\psi)-\bar f_1(0)\Big)+(1-\l)\Big(f_2(0,0^+,\psi)-\bar f_2(0)\Big)=&0.
\end{align}
Let $g_{\l}(\phi)=\l g_1(\phi)+(1-\l)g_2(\phi)$ and the corresponding solution to the equation \eqref{deco} is $f_{\l}(\eta,\phi)$. We have
\begin{align}
f_{\l}(0,0^+,\psi)-\bar f_{\l}(0)=0.
\end{align}
Since for $\phi\in(0,\e^{\alpha}]$, $g_{\l}=1-\l$ is a constant, we naturally have $\dfrac{\p g_{\l}}{\p\phi}=0$. We may solve from the equation \eqref{deco} that at $\eta=0,\phi\in[0,\e^{\alpha}],\psi\in[-\pi,\pi]$,
\begin{align}
\dfrac{\p f_{\l}}{\p\eta}&=\frac{1}{\sin\phi}\Bigg(\e\bigg(\dfrac{\sin^2\psi}{R_1-\e\eta}+\dfrac{\cos^2\psi}{R_2-\e\eta}\bigg)\cos\phi\dfrac{\p
g_{\l}}{\p\phi}-\Big(f_{\l}-\bar f_{\l}\Big)\Bigg)=0.
\end{align}
Note that $g_{\l}(\phi,\psi)=g(\phi,\psi)$ for $\phi\in\left[2\e^{\alpha},\dfrac{\pi}{2}\right]$, so our modification is restricted to a small region near the grazing set and we can smoothen the normal derivative at the boundary.

This method can be easily generalized to treat other $g(\phi,\psi)$. In principle, for $g(\phi,\psi)\in C^1$, we can define a decomposition
\begin{align}
g(\phi,\psi)=\gb(\phi,\psi)+\gf(\phi,\psi),
\end{align}
such that $\gf(\phi,\psi)=0$ for $\sin\phi\geq2\e^{\alpha}$, and the solution to the equation \eqref{deco} with in-flow data $\gb(\phi)$ has $L^{\infty}$ normal derivative at $\eta$=0.
Such a decomposition comes with a price. Originally, we have $\lss{\dfrac{\p g}{\p\phi}}{\phi,\psi}\leq C$. However, now we only have $\lss{\dfrac{\p\gb}{\p\phi}}{\phi,\psi}\leq C\e^{-\alpha}$ and $\lss{\dfrac{\p\gf}{\p\phi}}{\phi,\psi}\leq C\e^{-\alpha}$ due to the short-ranged cut-off function.

For either $\gb$ and $\gf$, we may define the corresponding boundary layer $\ub$ and $\uf$. We call $\ub$ the regular boundary layer and expand it up to $O(\e)$, i.e.
\begin{align}
\ub(\eta,\iota_1,\iota_2,\phi,\psi)\sim\ub_0(\eta,\iota_1,\iota_2,\phi,\psi)+\e\ub_1(\eta,\iota_1,\iota_2,\phi,\psi).
\end{align}
Also, we call $\uf$ the singular boundary layer and only expand it to $O(1)$, i.e.
\begin{align}
\uf(\eta,\iota_1,\iota_2,\phi,\psi)\sim\uf_0(\eta,\iota_1,\iota_2,\phi,\psi).
\end{align}
They should both satisfy the $\e$-Milne problem with geometric correction.

\subsection{Matching Procedure}

The bridge between the interior solution and boundary layer
is the boundary condition of \eqref{transport}, so we
consider the boundary expansion:
\begin{align}
\u_0(\vx_0,\vw)+\ub_0(\vx_0,\vw)+\uf_0(\vx_0,\vw)=&g(\vx_0,\vw)\ \ \text{for}\ \ \vx_0\in\p\Omega,\\
\u_1(\vx_0,\vw)+\ub_1(\vx_0,\vw)=&0\ \ \text{for}\ \ \vx_0\in\p\Omega.
\end{align}
The construction and determination of asymptotic expansion are as follows:\\
\ \\
Step 0: Preliminaries.\\
Define the force
\begin{align}
F(\e;\eta,\iota_1,\iota_2,\psi)=-\e\bigg(\dfrac{\sin^2\psi}{R_1(\iota_1,\iota_2)-\e\eta}+\dfrac{\cos^2\psi}{R_2(\iota_1,\iota_2)-\e\eta}\bigg).
\end{align}
Define the length of boundary layer $L=\e^{-n}$ for $0<n<\dfrac{1}{2}$. For $\phi\in\left[-\dfrac{\pi}{2},\dfrac{\pi}{2}\right]$, denote $\rr[\phi]=-\phi$.\\
\ \\
Step 1: Construction of $\ub_0$, $\uf_0$ and $\u_0$.\\
Define the zeroth-order regular boundary layer as
\begin{align}\label{et 1}
\left\{
\begin{array}{l}
\ub_0(\eta,\iota_1,\iota_2,\phi,\psi)=\mathscr{F}_0 (\eta,\iota_1,\iota_2,\phi,\psi)-\mathscr{F}_{0,L}(\iota_1,\iota_2),\\\rule{0ex}{2em}
\sin\phi\dfrac{\p \mathscr{F}_0 }{\p\eta}+F(\e;\eta,\iota_1,\iota_2,\psi)\cos\phi\dfrac{\p
\mathscr{F}_0 }{\p\phi}+\mathscr{F}_0 -\bar{\mathscr{F}}_0 =0,\\\rule{0ex}{1.5em}
\mathscr{F}_0 (0,\iota_1,\iota_2,\phi,\psi)=\gb(\iota_1,\iota_2,\phi,\psi)\ \ \text{for}\ \
\sin\phi>0,\\\rule{0ex}{1.5em}
\mathscr{F}_0 (L,\iota_1,\iota_2,\phi,\psi)=\mathscr{F}_0 (L,\iota_1,\iota_2,\rr[\phi],\psi),
\end{array}
\right.
\end{align}
with $\mathscr{F}_{0,L}(\iota_1,\iota_2)$ is defined as in Theorem \ref{Milne theorem 1}.\\
\ \\
Define the zeroth-order singular boundary layer as
\begin{align}\label{et 2}
\left\{
\begin{array}{l}
\uf_0(\eta,\iota_1,\iota_2,\phi,\psi)=\mathfrak{F}_0 (\eta,\iota_1,\iota_2,\phi,\psi)-\mathfrak{F} _{0,L}(\iota_1,\iota_2),\\\rule{0ex}{2em}
\sin\phi\dfrac{\p \mathfrak{F}_0 }{\p\eta}+F(\e;\eta,\iota_1,\iota_2,\psi)\cos\phi\dfrac{\p
\mathfrak{F}_0 }{\p\phi}+\mathfrak{F}_0 -\bar{\mathfrak{F}}_0 =0,\\\rule{0ex}{1.5em}
\mathfrak{F}_0 (0,\iota_1,\iota_2,\phi,\psi)=\gf(\iota_1,\iota_2,\phi,\psi)\ \ \text{for}\ \
\sin\phi>0,\\\rule{0ex}{1.5em}
\mathfrak{F}_0 (L,\iota_1,\iota_2,\phi,\psi)=\mathfrak{F}_0 (L,\iota_1,\iota_2,\rr[\phi],\psi),
\end{array}
\right.
\end{align}
with $\mathfrak{F} _{0,L}(\iota_1,\iota_2)$ is defined as in Theorem \ref{Milne theorem 1}.\\
\ \\
Also, define the zeroth-order interior solution $\u_0(\vx,\vw)$ as
\begin{align}\label{et 3}
\left\{
\begin{array}{l}
\u_0(\vx,\vw)=\bu_0(\vx) ,\\\rule{0ex}{1.5em} \Delta_x\bu_0(\vx)=0\ \ \text{in}\
\ \Omega,\\\rule{0ex}{1.5em}
\bu_0(\vx_0)=\mathscr{F}_{0,L}(\iota_1,\iota_2)+\mathfrak{F}_{0,L}(\iota_1,\iota_2)\ \ \text{on}\ \
\p\Omega.
\end{array}
\right.
\end{align}
\ \\
Step 2: Construction of $\ub_1$ and $\u_1$.\\
Define the first-order regular boundary layer as
\begin{align}\label{et 4}
\left\{
\begin{array}{l}
\ub_1(\eta,\iota_1,\iota_2,\phi,\psi)=\mathscr{F}_1 (\eta,\iota_1,\iota_2,\phi,\psi)-\mathscr{F} _{1,L}(\iota_1,\iota_2),\\\rule{0ex}{2em}
\sin\phi\dfrac{\p \mathscr{F}_1 }{\p\eta}+F(\e;\eta,\iota_1,\iota_2,\psi)\cos\phi\dfrac{\p
\mathscr{F}_1 }{\p\phi}+\mathscr{F}_1 -\bar{\mathscr{F}}_1 =G[\ub_0],\\\rule{0ex}{1.5em}
\mathscr{F}_1 (0,\iota_1,\iota_2,\phi,\psi)=\vw\cdot\nx\u_0(0,\iota_1,\iota_2)\ \ \text{for}\ \
\sin\phi>0,\\\rule{0ex}{1.5em}
\mathscr{F}_1 (L,\iota_1,\iota_2,\phi,\psi)=\mathscr{F}_1 (L,\iota_1,\iota_2,\rr[\phi],\psi),
\end{array}
\right.
\end{align}
with $\mathscr{F}_{1,L}(\iota_1,\iota_2)$ is defined as in Theorem \ref{Milne theorem 1}, and $G_0$ is defined in \eqref{coordinate 23}.\\
\ \\
Then define the first-order interior solution $\u_1(\vx,\vw)$ as
\begin{align}\label{et 5}
\left\{
\begin{array}{l}
\u_1(\vx,\vw)=\bu_1(\vx)-\vw\cdot\nx\u_0(\vx,\vw),\\\rule{0ex}{1.5em}
\Delta_x\bu_1(\vx)=-\displaystyle\int_{\s^1}\Big(\vw\cdot\nx\u_{0}(\vx,\vw)\Big)\ud{\vw}\
\ \text{in}\ \ \Omega,\\\rule{0ex}{1em} \bu_1(\vx_0)=f _{1,L}(\iota_1,\iota_2)\ \ \text{on}\ \
\p\Omega.
\end{array}
\right.
\end{align}
Note that we do not define $\uf_1$ here.\\
\ \\
Step 3: Construction of $\u_2$.\\
Since we do not expand to $\ub_2$ and $\uf_2$, simply define the second-order interior solution as
\begin{align}
\left\{
\begin{array}{l}
\u_{2}(\vx,\vw)=\bu_{2}(\vx)-\vw\cdot\nx\u_{1}(\vx,\vw),\\\rule{0ex}{1.5em}
\Delta_x\bu_{2}(\vx)=-\displaystyle\int_{\s^1}\Big(\vw\cdot\nx\u_{1}(\vx,\vw)\Big)\ud{\vw}\
\ \text{in}\ \ \Omega,\\\rule{0ex}{1.5em} \bu_2(\vx_0)=0\ \ \text{on}\ \
\p\Omega.
\end{array}
\right.
\end{align}
Here, we might have $O(\e^3)$ error in this step due to the trivial boundary data. Thanks to the remainder estimate, it will not affect the diffusive limit.

\section{Remainder Estimate}

In this section, we consider the remainder equation for $u(\vx,\vw)$ as
\begin{align}\label{neutron}
\left\{
\begin{array}{l}\displaystyle
\e\vw\cdot\nx u+u-\bar
u=\ss(\vx,\vw)\ \ \text{in}\ \ \Omega\times\s^2,\\\rule{0ex}{1.0em}
u(\vx_0,\vw)=\g(\vx_0,\vw)\ \ \text{for}\ \
\vw\cdot\vn<0\ \ \text{and}\ \ \vx_0\in\p\Omega.
\end{array}
\right.
\end{align}
\ \\
Define the $L^p$ norms with $1\leq p<\infty$ and $L^{\infty}$ norm in $\Omega\times\s^2$ as
usual:
\begin{align}
\nm{f}_{L^p(\Omega\times\s^2)}=&\bigg(\int_{\Omega}\int_{\s^2}\abs{f(\vx,\vw)}^p\ud{\vw}\ud{\vx}\bigg)^{\frac{1}{p}},\\
\nm{f}_{L^{\infty}(\Omega\times\s^2)}=&\text{esssup}_{(\vx,\vw)\in\Omega\times\s^2}\abs{f(\vx,\vw)}.
\end{align}
Define the $L^p$ norm with $1\leq p<\infty$ and $L^{\infty}$ norm on the boundary $\Gamma=\p\Omega\times\s^2$ as follows:
\begin{align}
\nm{f}_{L^p(\Gamma)}=&\bigg(\iint_{\Gamma}\abs{f(\vx,\vw)}^p\abs{\vw\cdot\vn}\ud{\vw}\ud{\vx}\bigg)^{\frac{1}{p}},\\
\nm{f}_{L^p(\Gamma^{\pm})}=&\bigg(\iint_{\Gamma^{\pm}}\abs{f(\vx,\vw)}^p\abs{\vw\cdot\vn}\ud{\vw}\ud{\vx}\bigg)^{\frac{1}{p}},\\
\nm{f}_{L^{\infty}(\Gamma)}=&\text{esssup}_{(\vx,\vw)\in\Gamma}\abs{f(\vx,\vw)},\\
\nm{f}_{L^{\infty}(\Gamma^{\pm})}=&\text{esssup}_{(\vx,\vw)\in\Gamma^{\pm}}\abs{f(\vx,\vw)}.
\end{align}
In particular, we denote $\ud{\gamma}=(\vw\cdot\vn)\ud{\vw}\ud{\vx}$ on the boundary.

\subsection{$L^2$ Estimate}

\begin{lemma}[Green's Identity]\label{remainder lemma 1}
Assume $u(\vx,\vw),\ v(\vx,\vw)\in L^2(\Omega\times\s^2)$ and
$\vw\cdot\nx u,\ \vw\cdot\nx v\in L^2(\Omega\times\s^2)$ with $u,\
v\in L^2(\Gamma)$. Then
\begin{align}
\iint_{\Omega\times\s^2}\bigg((\vw\cdot\nx u)v+(\vw\cdot\nx
u)v\bigg)\ud{\vx}\ud{\vw}=\int_{\Gamma}uv\ud{\gamma}.
\end{align}
\end{lemma}
\begin{proof}
See \cite[Chapter 9]{Cercignani.Illner.Pulvirenti1994} and
\cite{Esposito.Guo.Kim.Marra2013}.
\end{proof}
\begin{theorem}\label{LT estimate}
There exists a unique solution $u(\vx,\vw)$ to the equation \eqref{neutron} that satisfies
\begin{align}
\frac{1}{\e^{\frac{1}{2}}}\nm{u}_{L^2(\Gamma^+)}+\nm{u}_{L^2(\Omega\times\s^2)}\leq
C \bigg(
\frac{1}{\e^2}\nm{\ss}_{L^2(\Omega\times\s^2)}+\frac{1}{\e^{\frac{1}{2}}}\nm{\g}_{L^2(\Gamma^-)}\bigg).
\end{align}
\end{theorem}
\begin{proof}
\ \\
Step 1: Kernel Estimate.\\
Applying Lemma \ref{remainder lemma 1} to the
equation \eqref{neutron}. Then for any
$\phi\in L^2(\Omega\times\s^2)$ satisfying $\vw\cdot\nx\phi\in
L^2(\Omega\times\s^2)$ and $\phi\in L^2(\Gamma)$, we have
\begin{align}\label{lt 1}
\e\int_{\Gamma}u\phi\ud{\gamma}
-\e\iint_{\Omega\times\s^2}(\vw\cdot\nx\phi)u+\iint_{\Omega\times\s^2}(u-\bar
u)\phi=\iint_{\Omega\times\s^2}\ss\phi.
\end{align}
Our goal is to choose a particular test function $\phi$. We first
construct an auxiliary function $\xi$. Clearly, $u\in
L^{2}(\Omega\times\s^2)$ implies that $\bar u\in
L^2(\Omega)$. We define $\xi(\vx)$ on $\Omega$ satisfying
\begin{align}\label{lt 2}
\left\{
\begin{array}{l}
\Delta \xi=\bar u\ \ \text{in}\ \
\Omega,\\\rule{0ex}{1.0em} \xi=0\ \ \text{on}\ \ \p\Omega.
\end{array}
\right.
\end{align}
In the bounded domain $\Omega$, based on the standard elliptic
estimates, there exists a unique $\xi\in H^2(\Omega)$ such that
\begin{align}\label{lt 6}
\nm{\xi}_{H^2(\Omega)}\leq C \nm{\bar
u}_{L^2(\Omega)}\leq
C \nm{\bar u}_{L^2(\Omega\times\s^2)}.
\end{align}
We plug the test function
\begin{align}\label{lt 3}
\phi=-\vw\cdot\nx\xi,
\end{align}
into the weak formulation \eqref{lt 1} and estimate
each term there. By definition, we have
\begin{align}\label{lt 12}
\nm{\phi}_{H^1(\Omega\times\s^2)}\leq C\nm{\xi}_{H^2(\Omega)}\leq
C \nm{\bar u}_{L^2(\Omega)}\leq
C \nm{\bar u}_{L^2(\Omega\times\s^2)}.
\end{align}
On the other hand, we decompose
\begin{align}\label{lt 4}
-\e\iint_{\Omega\times\s^2}(\vw\cdot\nx\phi)u=&-\e\iint_{\Omega\times\s^2}(\vw\cdot\nx\phi)\bar
u-\e\iint_{\Omega\times\s^2}(\vw\cdot\nx\phi)(u-\bar
u).
\end{align}
For the first term on the right-hand side of \eqref{lt 4}, by
\eqref{lt 2} and \eqref{lt 3}, we have
\begin{align}\label{lt 5}
&-\e\iint_{\Omega\times\s^2}(\vw\cdot\nx\phi)\bar
u\\
=&\e\iint_{\Omega\times\s^2}\bar
u\Big(w_1(w_1\p_{11}\xi+w_2\p_{12}\xi+w_3\p_{13}\xi)+w_2(w_1\p_{21}\xi+w_2\p_{22}\xi+w_3\p_{23}\xi)+w_3(w_1\p_{31}\xi+w_2\p_{32}\xi+w_3\p_{33}\xi)\Big)\no\\
=&\e\iint_{\Omega\times\s^2}\bar
u\Big(w_1^2\p_{11}\xi+w_2^2\p_{22}\xi+w_3^2\p_{33}\xi\Big)=\frac{4}{3}\e\pi\int_{\Omega}\bar u(\p_{11}\xi+\p_{22}\xi+\p_{33}\xi)=\frac{4}{3}\e\pi\nm{\bar u}_{L^2(\Omega)}^2=\frac{1}{3}\e\nm{\bar u}_{L^2(\Omega\times\s^2)}^2\no.
\end{align}
Here $\p_i$ denotes the derivative with respect to $x_i$. In the second equality, above cross terms vanish due to the symmetry
of the integral over $\s^2$.\\
For the second term
on the right-hand side of \eqref{lt 4}, H\"older's inequality and \eqref{lt 12} imply
\begin{align}\label{lt 7}
\abs{-\e\iint_{\Omega\times\s^2}(\vw\cdot\nx\phi)(u-\bar
u)}\leq& \e\nm{\vw\cdot\nx\phi}_{L^2(\Omega\times\s^2)}\nm{u-\bar u}_{L^2(\Omega\times\s^2)}\leq C \e\nm{\phi}_{H^1(\Omega\times\s^2)}\nm{u-\bar u}_{L^2(\Omega\times\s^2)}\\
\leq&C \e\nm{\bar u}_{L^2(\Omega\times\s^2)}\nm{u-\bar u}_{L^2(\Omega\times\s^2)}\no.
\end{align}
Using the trace theorem, H\"older's inequality and \eqref{lt 12}, we have
\begin{align}\label{lt 8}
\abs{\e\int_{\Gamma}u\phi\ud{\gamma}}\leq&\abs{\e\int_{\Gamma^+}u\phi\ud{\gamma}}+\abs{\e\int_{\Gamma^-}u\phi\ud{\gamma}}
\leq \e\nm{\phi}_{L^2(\Gamma)}\Big(\nm{u}_{L^2(\Gamma^+)}+\nm{\g}_{L^2(\Gamma^-)}\Big)\\
\leq& C\e\nm{\phi}_{H^1(\Omega\times\s^2)}\Big(\nm{u}_{L^2(\Gamma^+)}+\nm{\g}_{L^2(\Gamma^-)}\Big)\leq C\e\nm{\xi}_{H^2(\Omega)}\Big(\nm{u}_{L^2(\Gamma^+)}+\nm{\g}_{L^2(\Gamma^-)}\Big)\no\\
\leq& C\e\nm{\bar u}_{L^2(\Omega\times\s^2)}\Big(\nm{u}_{L^2(\Gamma^+)}+\nm{\g}_{L^2(\Gamma^-)}\Big).\no
\end{align}
Also, using H\"older's inequality and \eqref{lt 12}, we obtain
\begin{align}\label{lt 9}
\abs{\iint_{\Omega\times\s^2}(u-\bar u)\phi}\leq
 \nm{\phi}_{L^2(\Omega\times\s^2)}\nm{u-\bar
u}_{L^2(\Omega\times\s^2)}\leq
C \nm{\bar u}_{L^2(\Omega\times\s^2)}\nm{u-\bar
u}_{L^2(\Omega\times\s^2)},
\end{align}
and
\begin{align}\label{lt 10}
\abs{\iint_{\Omega\times\s^2}\ss\phi}\leq  \nm{\phi}_{L^2(\Omega\times\s^2)}\nm{\ss}_{L^2(\Omega\times\s^2)}\leq C \nm{\bar
u}_{L^2(\Omega\times\s^2)}\nm{\ss}_{L^2(\Omega\times\s^2)}.
\end{align}
Collecting estimates in \eqref{lt 5}, \eqref{lt 7}, \eqref{lt 8},
\eqref{lt 9} and \eqref{lt 10} for the weak formulation \eqref{lt 1}, we obtain
\begin{align}
\e\nm{\bar u}_{L^2(\Omega\times\s^2)}^2\leq&
C \nm{\bar u}_{L^2(\Omega\times\s^2)}\Big(\nm{u-\bar
u}_{L^2(\Omega\times\s^2)}+\e\nm{u}_{L^2(\Gamma^+)}+\nm{\ss}_{L^2(\Omega\times\s^2)}+\e\tm{\g}{\Gamma^-}\Big).
\end{align}
Then this implies that
\begin{align}\label{lt 13}
\e\nm{\bar u}_{L^2(\Omega\times\s^2)}\leq&
C \Big(\nm{u-\bar
u}_{L^2(\Omega\times\s^2)}+\e\nm{u}_{L^2(\Gamma^+)}+\nm{\ss}_{L^2(\Omega\times\s^2)}+\e\tm{\g}{\Gamma^-}\Big).
\end{align}
\ \\
Step 2: Energy Estimate.\\
In the weak formulation \eqref{lt 1}, we may take
the test function $\phi=u$ to get the energy estimate
\begin{align}
\half\e\int_{\Gamma}\abs{u}^2\ud{\gamma}+\nm{u-\bar
u}_{L^2(\Omega\times\s^2)}^2=\iint_{\Omega\times\s^2}\ss u,
\end{align}
where we use the fact that
\begin{align}
\iint_{\Omega\times\s^2}u(u-\bar u)=\iint_{\Omega\times\s^2}\bar u(u-\bar u)+\iint_{\Omega\times\s^2}(u-\bar u)^2=\nm{u-\bar
u}_{L^2(\Omega\times\s^2)}^2.
\end{align}
Then decomposing the boundary term, we have
\begin{align}\label{lt 14}
&&\half\e\nm{u}^2_{L^2(\Gamma^+)}+\nm{u-\bar
u}_{L^2(\Omega\times\s^2)}^2= \iint_{\Omega\times\s^2}\ss u+\half\e\nm{\g}_{L^2(\Gamma^-)}^2.
\end{align}
On the other hand, we can square on both sides of
\eqref{lt 13} to obtain
\begin{align}\label{lt 15}
\e^2\nm{\bar u}_{L^2(\Omega\times\s^2)}^2\leq&
C\Big(\nm{u-\bar
u}_{L^2(\Omega\times\s^2)}^2+\e^2\nm{u}_{L^2(\Gamma^+)}^2+\nm{\ss}_{L^2(\Omega\times\s^2)}^2+\e^2\tm{\g}{\Gamma^-}^2\Big).
\end{align}
Multiplying \eqref{lt 15} by a sufficiently small constant and adding it to \eqref{lt 14} to absorb $\nm{u}_{L^2(\Gamma^+)}^2$ and
$\nm{u-\bar u}_{L^2(\Omega\times\s^2)}^2$, we deduce
\begin{align}
&&\e\nm{u}_{L^2(\Gamma^+)}^2+\e^2\nm{\bar
u}_{L^2(\Omega\times\s^2)}^2+\nm{u-\bar
u}_{L^2(\Omega\times\s^2)}^2\leq
C \Big(\nm{\ss}_{L^2(\Omega\times\s^2)}^2+
\iint_{\Omega\times\s^2}\ss u+\e\nm{\g}_{L^2(\Gamma^-)}^2\Big).
\end{align}
Since
\begin{align}
\nm{u}_{L^2(\Omega\times\s^2)}^2=\nm{\bar u}_{L^2(\Omega\times\s^2)}^2+\nm{u-\bar u}_{L^2(\Omega\times\s^2)}^2,
\end{align}
we have
\begin{align}\label{lt 16}
\e\nm{u}_{L^2(\Gamma^+)}^2+\e^2\nm{u}_{L^2(\Omega\times\s^2)}^2\leq
C \Big(\nm{\ss}_{L^2(\Omega\times\s^2)}^2+
\iint_{\Omega\times\s^2}\ss u+\e\nm{\g}_{L^2(\Gamma^-)}^2\Big).
\end{align}
A direct application of Cauchy's inequality leads to
\begin{align}
\iint_{\Omega\times\s^2}\ss u\leq\frac{1}{4C_0\e^2}\nm{\ss}_{L^2(\Omega\times\s^2)}^2+C_0\e^2\nm{u}_{L^2(\Omega\times\s^2)}^2.
\end{align}
Taking $C_0$ sufficiently small to absorb $C_0\e^2\nm{u}_{L^2(\Omega\times\s^2)}^2$ in \eqref{lt 16}, we obtain
\begin{align}\label{lt 17}
\e\nm{u}_{L^2(\Gamma^+)}^2+\e^2\nm{u}_{L^2(\Omega\times\s^2)}^2\leq
C \bigg(\frac{1}{\e^2}\nm{\ss}_{L^2(\Omega\times\s^2)}^2+\e\nm{\g}_{L^2(\Gamma^-)}^2\bigg).
\end{align}
Then we can divide $\e^2$ on both sides of \eqref{lt 17} to obtain
\begin{align}\label{lt 18}
\frac{1}{\e}\nm{u}_{L^2(\Gamma^+)}^2+\nm{u}_{L^2(\Omega\times\s^2)}^2\leq
C \bigg(
\frac{1}{\e^4}\nm{\ss}_{L^2(\Omega\times\s^2)}^2+\frac{1}{\e}\nm{\g}_{L^2(\Gamma^-)}^2\bigg).
\end{align}
Hence, we have
\begin{align}
\frac{1}{\e^{\frac{1}{2}}}\nm{u}_{L^2(\Gamma^+)}+\nm{u}_{L^2(\Omega\times\s^2)}\leq
C \bigg(
\frac{1}{\e^2}\nm{\ss}_{L^2(\Omega\times\s^2)}+\frac{1}{\e^{\frac{1}{2}}}\nm{\g}_{L^2(\Gamma^-)}\bigg).
\end{align}
\end{proof}

\subsection{$L^{\infty}$ Estimate - First Round}

\begin{theorem}\label{LI estimate'}
The unique solution $u(\vx,\vw)$ to the equation \eqref{neutron} satisfies
\begin{align}
\nm{u}_{L^{\infty}(\Omega\times\s^2)}&\leq C\bigg(\frac{1}{\e^{\frac{7}{2}}}\nm{\ss}_{L^2(\Omega\times\s^2)}+\nm{\ss}_{L^{\infty}(\Omega\times\s^2)}
+\frac{1}{\e^{2}}\nm{\g}_{L^2(\Gamma^-)}+\nm{\g}_{L^{\infty}(\Gamma^-)}\bigg).
\end{align}
\end{theorem}
\begin{proof}
\ \\
Step 1: Double Duhamel iterations.\\
It is well-known that in kinetic equations, the mild formulation is stronger than the weak formulation. Let $\Big(\vec X(s),\vec W(s)\Big)$ for $s\in\r$ represent the characteristics of the equation \eqref{neutron}. We have
\begin{align}
\frac{\ud \vec{X}}{\ud s}=\e\vw,\quad\frac{\ud \vec{W}}{\ud s}=0,
\end{align}
which, combining with the initial data $\vec X(0)=\vx,\vec W(0)=\vw$, further implies
\begin{align}
\vec X(s)=\vx+\e\vw s,\quad \vec W(s)=\vw.
\end{align}
Along the characteristics, the equation \eqref{neutron} is
\begin{align}
\frac{\ud u}{\ud s}+u-\bar u=\ss.
\end{align}
Then we can rewrite the equation
\eqref{neutron} by tracking along the characteristics all the way back to the in-flow boundary $\Gamma^-$ as
\begin{align}\label{li 1}
u(\vx,\vw)=&\g(\vx-\e t_b\vw,\vw)\ue^{-t_b}+\int_{0}^{t_b}\ss(\vx-\e{s}\vw,\vw)\ue^{-{s}}\ud{s}+\int_{0}^{t_b}\bar u(\vx-\e{s}\vw)\ue^{-{s}}\ud{s}\\
=&\g(\vx-\e t_b\vw,\vw)\ue^{-t_b}+\int_{0}^{t_b}\ss(\vx-\e{s}\vw,\vw)\ue^{-{s}}\ud{s}
+\frac{1}{4\pi}\int_{0}^{t_b}\bigg(\int_{\s^2}u(\vx-\e{s}\vw,\vw_t)\ud{\vw_t}\bigg)\ue^{-{s}}\ud{s}\no,
\end{align}
where the backward exit time $t_b$ is defined as
\begin{align}
t_b(\vx,\vw)=\inf\{s\geq0: (\vx-\e s\vw,\vw)\in\Gamma^-\},
\end{align}
which represents the first time that the characteristics track back and hit the in-flow boundary. Note that we have replaced $\bar u$ by the integral of $u$ over the dummy velocity
variable $\vw_t$. For the last term in \eqref{li 1}, we rewrite $u(\vx-\e{s}\vw,\vw_t)$ by tracking back along the characteristics again to obtain
\begin{align}\label{li 2}
u(\vx,\vw)=&\g(\vx-\e t_b\vw,\vw)\ue^{-t_b}+\int_{0}^{t_b}\ss(\vx-\e{s}\vw,\vw)\ue^{-{s}}\ud{s}\\
&+\frac{1}{4\pi}\int_{0}^{t_b}\bigg(\int_{\s^2}\g(\vx-\e{s}\vw-\e
s_b\vw_t,\vw_t)\ue^{-s_b}\ud{\vw_t}\bigg)\ue^{-{s}}\ud{s}\no\\
&+\frac{1}{4\pi}\int_{0}^{t_b}\Bigg(\int_{\s^2}\bigg(\int_{0}^{s_b}\ss(\vx-\e{s}\vw-\e
{r}\vw_t,\vw_t)\ue^{-{r}}\ud{r}\bigg)\ud{\vw_t}\Bigg)\ue^{-{s}}\ud{s}\no\\
&+\frac{1}{4\pi}\int_{0}^{t_b}\Bigg(\int_{\s^2}\bigg(\int_{0}^{s_b}\bar u(\vx-\e{s}\vw-\e
{r}\vw_t)\ue^{-{r}}\ud{r}\bigg)\ud{\vw_t}\Bigg)\ue^{-{s}}\ud{s}\no,
\end{align}
where the exiting time from $(\vx-\e{s}\vw,\vw_t)$ is defined as
\begin{align}
s_b(\vx,\vw;s,\vw_t)=\inf\left\{r\geq0: (\vx-\e{s}\vw-\e
r\vw_t,\vw_t)\in\Gamma^-\right\}.
\end{align}
\ \\
Step 2: Estimates of all but the last term in \eqref{li 2}.\\
Note the fact that $0\leq s\leq t_b$ and $0\leq r\leq s_b$. We can directly estimate
\begin{align}\label{li 3}
\abs{\g(\vx-\e t_b\vw,\vw)\ue^{-t_b}}\leq\nm{\g}_{L^{\infty}(\Gamma^-)},
\end{align}
\begin{align}\label{li 4}
\abs{\frac{1}{4\pi}\int_{0}^{t_b}\bigg(\int_{\s^2}\g(\vx-\e{s}\vw-\e
s_b\vw_t,\vw_t)\ue^{-s_b}\ud{\vw_t}\bigg)\ue^{-{s}}\ud{s}} \leq
\nm{\g}_{L^{\infty}(\Gamma^-)},
\end{align}
\begin{align}\label{li 5}
\abs{\int_{0}^{t_b}\ss(\vx-\e{s}\vw,\vw)\ue^{-{s}}\ud{s}}\leq
\nm{\ss}_{L^{\infty}(\Omega\times\s^2)},
\end{align}
\begin{align}\label{li 6}
\abs{\frac{1}{4\pi}\int_{0}^{t_b}\Bigg(\int_{\s^2}\bigg(\int_{0}^{s_b}\ss(\vx-\e{s}\vw-\e
{r}\vw_t,\vw_t)\ue^{-{r}}\ud{r}\bigg)\ud{\vw_t}\Bigg)\ue^{-{s}}\ud{s}}
\leq \nm{\ss}_{L^{\infty}(\Omega\times\s^2)}.
\end{align}
\ \\
Step 3: Estimates of the last term in \eqref{li 2}.\\
Now we decompose the last term in \eqref{li 2} as
\begin{align}
\int_{0}^{t_b}\int_{\s^2}\int_0^{s_b}=\int_{0}^{t_b}\int_{\s^2}\int_{0\leq r\leq\delta}+
\int_{0}^{t_b}\int_{\s^2}\int_{r\geq\delta}=I_1+I_2,
\end{align}
for some $0<\delta<<1$ to be determined later. Since $I_1$ contains an integral in a very small region, we may directly estimate
\begin{align}\label{li 7}
\abs{I_1}
\leq&\frac{1}{4\pi}\int_{0}^{t_b}\Bigg(\int_{\s^2}\int_{0\leq r\leq\delta}\nm{u}_{L^{\infty}(\Omega\times\s^2)}\ue^{-{r}}
\ud{r}\ud{\vw_t}\Bigg)\ue^{-{s}}\ud{s}\\
\leq&\nm{u}_{L^{\infty}(\Omega\times\s^2)}\bigg(\int_{0\leq r\leq\delta}\ue^{-{r}}
\ud{r}\bigg)\leq C\delta\nm{u}_{L^{\infty}(\Omega\times\s^2)}.\no
\end{align}
Then we need to handle the more complicated term $I_2$,
\begin{align}
\abs{I_2}\leq&C\int_{0}^{t_b}\Bigg(\int_{\s^2}\int_{r\geq\delta}\abs{\bar u(\vx-\e{s}\vw-\e
{r}\vw_t)}\ue^{-{r}}\ud{r}\ud{\vw_t}\Bigg)\ue^{-{s}}\ud{s}.
\end{align}
By the definition of $t_b$ and $s_b$, we always have
\begin{align}
\vx-\e{s}\vw-\e {r}\vw_t\in \Omega.
\end{align}
Hence, we may introduce the indicator function ${\bf{1}}_{\Omega}$ and apply H\"older's inequality
to obtain
\begin{align}\label{li 9}
\abs{I_2}\leq&C\int_{0}^{t_b}\Bigg(\int_{\s^2}\int_{r\geq\delta}{\bf{1}}_{\Omega}(\vx-\e{s}\vw-\e
{r}\vw_t)\abs{\bar u(\vx-\e{s}\vw-\e
{r}\vw_t)}\ud{r}\ud{\vw_t}\Bigg)\ue^{-{s}}\ud{s}\\
\leq&C\int_{0}^{t_b}\Bigg(\bigg(\int_{\s^2}\int_{r\geq\delta}{\bf{1}}_{\Omega}(\vx-\e{s}\vw-\e
{r}\vw_t)\abs{\bar u(\vx-\e{s}\vw-\e
{r}\vw_t)}^2\ud{r}\ud{\vw_t}\bigg)^{\frac{1}{2}}\no\\
&\times\bigg(\int_{\s^2}\int_{r\geq\delta}{\bf{1}}_{\Omega}(\vx-\e{s}\vw-\e
{r}\vw_t)\ue^{-2{r}}\ud{r}\ud{\vw_t}\bigg)^{\frac{1}{2}}\Bigg)\ue^{-{s}}\ud{s}\no\\
\leq&C\int_{0}^{t_b}\Bigg(\bigg(\int_{\s^2}\int_{r\geq\delta}{\bf{1}}_{\Omega}(\vx-\e{s}\vw-\e
{r}\vw_t)\abs{\bar u(\vx-\e{s}\vw-\e
{r}\vw_t)}^2\ud{r}\ud{\vw_t}\bigg)^{\frac{1}{2}}\Bigg)\ue^{-{s}}\ud{s}.\no
\end{align}
Note $\vw_t\in\s^2$, which can be parameterized as
\begin{align}
\vw_t=(\sin\phi\cos\psi,\sin\phi\sin\psi,\cos\phi),
\end{align}
for $\phi\in[0,\pi]$ and $\psi\in[0,2\pi]$. Hence, we may write the integral
\begin{align}\label{li 12}
\int_{\s^2}\cdots\ud\vw_t=\int_0^{2\pi}\int_0^{\pi}\cdots\sin\phi\ud\phi\ud\psi,
\end{align}
where $\sin\phi$ is the Jacobian of spherical coordinates. Then we further define the change of variable
$[0,\pi]\times[0,2\pi]\times\r\rt \Omega: (\phi,\psi,r)\rt(y_1,y_2,y_3)=\vec
y=\vx-\e{s}\vw-\e {r}\vw_t$, i.e.
\begin{align}\label{li 11}
\left\{
\begin{array}{rcl}
y_1&=&x_1-\e{s}w_1-\e {r}\sin\phi\cos\psi,\\
y_2&=&x_2-\e{s}w_2-\e {r}\sin\phi\sin\psi,\\
y_3&=&x_3-\e{s}w_3-\e {r}\cos\phi.
\end{array}
\right.
\end{align}
The Jacobian is
\begin{align}
\abs{\frac{\p(y_1,y_2,y_3)}{\p(\phi,\psi,r)}}&=\abs{\abs{\begin{array}{ccc}
-\e{r}\cos\phi\cos\psi&\e{r}\sin\phi\sin\psi&\e\sin\phi\cos\psi\\
-\e{r}\cos\phi\sin\psi&-\e{r}\sin\phi\cos\psi&\e\sin\phi\sin\psi\\
\e{r}\sin\phi&0&\e\cos\phi
\end{array}}}=\e^3{r}^2\sin\phi.
\end{align}
Now we consider the restriction on $r$ and $\phi$. $r\geq\delta$ implies ${r}^2\geq\delta^2$. Therefore, we may bound the Jacobian
\begin{align}
\abs{\frac{\p(y_1,y_2,y_3)}{\p(\phi,\psi,r)}}=\e^3{r}^2\sin\phi\geq \e^3\delta^2\sin\phi.
\end{align}
The extra $\sin\phi$ can be cancelled out by \eqref{li 12}. Hence, we may simplify \eqref{li 9} as
\begin{align}\label{li 10}
\abs{I_2}\leq&C\int_{0}^{t_b}\bigg(\int_{\Omega}\frac{1}{\e^3\d^2}\abs{\bar u(\vec
y)}^2\ud{\vec y}\bigg)^{\frac{1}{2}}\ue^{-{s}}\ud{s}
\leq\frac{C}{\e^{\frac{3}{2}}\d}\int_{0}^{t_b}\bigg(\int_{\Omega}\abs{\bar u(\vec
y)}^2\ud{\vec y}\bigg)^{\frac{1}{2}}\ue^{-{s}}\ud{s}
\leq\frac{C}{\e^{\frac{3}{2}}\d}\nm{\bar u}_{L^2(\Omega)}.
\end{align}
\ \\
Step 4: Synthesis.\\
In summary, collecting \eqref{li 3}, \eqref{li 4}, \eqref{li 5}, \eqref{li 6}, \eqref{li 7}, and \eqref{li 10}, for fixed $0<\delta<<1$, we have
\begin{align}
\abs{u(\vx,\vw)}\leq C\bigg(\delta
\nm{u}_{L^{\infty}(\Omega\times\s^2)}+\frac{1}{\e^{\frac{3}{2}}}\nm{\bar u}_{L^2(\Omega)}+\nm{\ss}_{L^{\infty}(\Omega\times\s^2)}+\nm{\g}_{L^{\infty}(\Gamma^-)}\bigg).
\end{align}
Taking supremum over all $(\vx,\vw)\in\Omega\times\s^2$, we obtain
\begin{align}
\nm{u}_{L^{\infty}(\Omega\times\s^2)}\leq C\bigg(\delta
\nm{u}_{L^{\infty}(\Omega\times\s^2)}+\frac{1}{\e^{\frac{3}{2}}\d}\nm{\bar u}_{L^2(\Omega)}+\nm{\ss}_{L^{\infty}(\Omega\times\s^2)}+\nm{\g}_{L^{\infty}(\Gamma^-)}\bigg).
\end{align}
Then taking $\d$ sufficiently small to absorb $C\delta
\nm{u}_{L^{\infty}(\Omega\times\s^2)}$ into the left-hand side, we get
\begin{align}
\nm{u}_{L^{\infty}(\Omega\times\s^2)}\leq
C\bigg(\frac{1}{\e^{\frac{3}{2}}}\nm{\bar u}_{L^2(\Omega)}+\nm{\ss}_{L^{\infty}(\Omega\times\s^2)}+\nm{\g}_{L^{\infty}(\Gamma^-)}\bigg).
\end{align}
Using Theorem \ref{LT estimate}, we get
\begin{align}
\nm{u}_{L^{\infty}(\Omega\times\s^2)}\leq& C\bigg(\frac{1}{\e^{\frac{7}{2}}}\nm{\ss}_{L^2(\Omega\times\s^2)}+\nm{\ss}_{L^{\infty}(\Omega\times\s^2)}
+\frac{1}{\e^{2}}\nm{\g}_{L^2(\Gamma^-)}+\nm{\g}_{L^{\infty}(\Gamma^-)}\bigg).
\end{align}

\end{proof}

\subsection{$L^{2m}$ Estimate}

In the following, let $o(1)$ denote a sufficiently small constant.
\begin{theorem}\label{LN estimate}
The unique solution $u(\vx,\vw)$ to the equation \eqref{neutron} satisfies for integer $1\leq m< 3$,
\begin{align}
&\frac{1}{\e^{\frac{1}{2}}}\nm{u}_{L^2(\Gamma^+)}+\nm{
\bar u}_{L^{2m}(\Omega\times\s^2)}+\frac{1}{\e}\nm{u-\bar
u}_{L^2(\Omega\times\s^2)}\\
\leq&
C\bigg(o(1)\e^{\frac{3}{2m}}\Big(\nm{u}_{L^{\infty}(\Omega\times\s^2)}+\nm{u}_{L^{\infty}(\Gamma^+)}\Big)\no\\
&+\frac{1}{\e}\nm{\ss}_{L^2(\Omega\times\s^2)}+
\frac{1}{\e^2}\nm{\ss}_{L^{\frac{2m}{2m-1}}(\Omega\times\s^2)}+\frac{1}{\e^{\frac{1}{2}}}\nm{\g}_{L^2(\Gamma^-)}+\nm{\g}_{L^{\frac{4m}{3}}(\Gamma^-)}\bigg).\no
\end{align}
\end{theorem}
\begin{proof}
\ \\
Step 1: Kernel Estimate.\\
As in $L^2$ estimates, applying Green's identity in Lemma \ref{remainder lemma 1} to the
equation \eqref{neutron}, for any
$\phi\in L^2(\Omega\times\s^2)$ satisfying $\vw\cdot\nx\phi\in
L^2(\Omega\times\s^2)$ and $\phi\in L^2(\Gamma)$, we have the weak formulation
\begin{align}\label{ln 1}
\e\int_{\Gamma}u\phi\ud{\gamma}
-\e\iint_{\Omega\times\s^2}(\vw\cdot\nx\phi)u+\iint_{\Omega\times\s^2}(u-\bar
u)\phi=\iint_{\Omega\times\s^2}\ss\phi.
\end{align}
Now we choose a different test function $\phi$. We first
construct an auxiliary function $\xi$. Using Theorem \ref{LI estimate'}, $u\in
L^{\infty}(\Omega\times\s^2)$ implies that $\bar u\in
L^{2m}(\Omega)$ which further leads to $\bar u^{2m-1}\in
L^{\frac{2m}{2m-1}}(\Omega)$. Define $\xi(\vx)$ on $\Omega$ satisfying
\begin{align}\label{ln 2}
\left\{
\begin{array}{l}
\Delta \xi=\bar u^{2m-1}\ \ \text{in}\ \
\Omega,\\\rule{0ex}{1.0em} \xi=0\ \ \text{on}\ \ \p\Omega.
\end{array}
\right.
\end{align}
In the bounded domain $\Omega$, based on the standard elliptic
estimates, there exists a unique $\xi\in W^{2,\frac{2m}{2m-1}}(\Omega)$ satisfying
\begin{align}\label{ln 3}
\nm{\xi}_{W^{2,\frac{2m}{2m-1}}(\Omega)}\leq C\nm{\bar
u^{2m-1}}_{L^{\frac{2m}{2m-1}}(\Omega)}= C\nm{\bar
u}_{L^{2m}(\Omega)}^{2m-1}.
\end{align}
We plug the test function
\begin{align}\label{ln 4}
\phi=-\vw\cdot\nx\xi,
\end{align}
into the weak formulation \eqref{ln 1} and estimate
each term there. By Sobolev embedding theorem and \eqref{ln 3}, we have
\begin{align}
&\nm{\phi}_{L^2(\Omega)}\leq C\nm{\xi}_{H^1(\Omega)}\leq C\nm{\xi}_{W^{2,\frac{2m}{2m-1}}(\Omega)}\leq
C\nm{\bar
u}_{L^{2m}(\Omega)}^{2m-1},\label{ln 5}\\
&\nm{\phi}_{L^{\frac{2m}{2m-1}}(\Omega)}\leq C\nm{\xi}_{W^{1,\frac{2m}{2m-1}}(\Omega)}\leq C\nm{\xi}_{W^{2,\frac{2m}{2m-1}}(\Omega)}\leq
C\nm{\bar
u}_{L^{2m}(\Omega)}^{2m-1}.\label{ln 6}
\end{align}
Note that the embedding $W^{2,\frac{2m}{2m-1}}(\Omega)\hookrightarrow H^1(\Omega)$ requires $m\leq 3$.\\
On the other hand, we decompose
\begin{align}\label{ln 7}
-\e\iint_{\Omega\times\s^2}(\vw\cdot\nx\phi)u=&-\e\iint_{\Omega\times\s^2}(\vw\cdot\nx\phi)\bar
u-\e\iint_{\Omega\times\s^2}(\vw\cdot\nx\phi)(u-\bar
u).
\end{align}
For the first term on the right-hand side of \eqref{ln 7}, by
\eqref{ln 2} and \eqref{ln 4}, we have
\begin{align}\label{ln 8}
&-\e\iint_{\Omega\times\s^2}(\vw\cdot\nx\phi)\bar
u\\
=&\e\iint_{\Omega\times\s^2}\bar
u\Big(w_1(w_1\p_{11}\zeta+w_2\p_{12}\zeta+w_3\p_{13}\zeta)+w_2(w_1\p_{21}\zeta+w_2\p_{22}\zeta+w_3\p_{23}\zeta)+w_3(w_1\p_{31}\zeta+w_2\p_{32}\zeta+w_3\p_{33}\zeta)\Big)\no\\
=&\e\iint_{\Omega\times\s^2}\bar
u\Big(w_1^2\p_{11}\zeta+w_2^2\p_{22}\zeta+w_3^2\p_{33}\zeta\Big)=\frac{4}{3}\e\pi\int_{\Omega}\bar u(\p_{11}\zeta+\p_{22}\zeta+\p_{33}\zeta)=\frac{4}{3}\e\pi\nm{\bar u}_{L^{2m}(\Omega)}^{2m}=\frac{1}{3}\e\nm{\bar u}_{L^{2m}(\Omega)}^{2m}\no.
\end{align}
In the second equality, above cross terms vanish due to the symmetry
of the integral over $\s^2$.\\
For the second term
on the right-hand side of \eqref{ln 7}, H\"older's inequality and \eqref{ln 6} imply
\begin{align}\label{ln 9}
&\abs{-\e\iint_{\Omega\times\s^2}(\vw\cdot\nx\phi)(u-\bar
u)}\leq \e\nm{\vw\cdot\nx\phi}_{L^{\frac{2m}{2m-1}}(\Omega)}\nm{u-\bar u}_{L^{2m}(\Omega\times\s^2)}\\
\leq&C\e\nm{\xi}_{W^{2,\frac{2m}{2m-1}}(\Omega)}\nm{u-\bar u}_{L^{2m}(\Omega\times\s^2)}
\leq C\e\nm{\bar
u}_{L^{2m}(\Omega)}^{2m-1}\nm{u-\bar
u}_{L^{2m}(\Omega\times\s^2)}\no.
\end{align}
Based on Sobolev embedding theorem, trace theorem and \eqref{ln 5}, we have
\begin{align}\label{ln 10}
\nm{\nx\xi}_{L^{\frac{4m}{4m-3}}(\Gamma)}\leq C\nm{\nx\xi}_{W^{\frac{1}{2m},\frac{2m}{2m-1}}(\Gamma)}\leq C\nm{\nx\xi}_{W^{1,\frac{2m}{2m-1}}(\Omega)}\leq C\nm{\xi}_{W^{2,\frac{2m}{2m-1}}(\Omega)}\leq
C\nm{\bar
u}_{L^{2m}(\Omega)}^{2m-1}.
\end{align}
Using H\"older's inequality and \eqref{ln 10}, we obtain
\begin{align}\label{ln 11}
\abs{\e\int_{\Gamma}u\phi\ud{\gamma}}\leq &\abs{\e\int_{\Gamma^+}u\phi\ud{\gamma}}+\abs{\e\int_{\Gamma^-}u\phi\ud{\gamma}}\leq C\e\nm{\phi}_{L^{\frac{4m}{4m-3}}(\Gamma)}\bigg(\nm{u}_{L^{\frac{4m}{3}}(\Gamma^+)}+\nm{\g}_{L^{\frac{4m}{3}}(\Gamma^-)}\bigg)\\
\leq&C\e\nm{\nx\xi}_{L^{\frac{4m}{4m-3}}(\Gamma)}\bigg(\nm{u}_{L^{\frac{4m}{3}}(\Gamma^+)}+\nm{\g}_{L^{\frac{4m}{3}}(\Gamma^-)}\bigg)
\leq C\e\nm{\bar u}_{L^{2m}(\Omega)}^{2m-1}\bigg(\nm{u}_{L^{\frac{4m}{3}}(\Gamma^+)}+\nm{\g}_{L^{\frac{4m}{3}}(\Gamma^-)}\bigg).\no
\end{align}
Also, using H\"older's inequality and \eqref{ln 5}, we have
\begin{align}\label{ln 12}
\iint_{\Omega\times\s^2}(u-\bar u)\phi\leq
C\nm{\phi}_{L^2(\Omega\times\s^2)}\nm{u-\bar
u}_{L^2(\Omega\times\s^2)}\leq
C\nm{\bar
u}_{L^{2m}(\Omega)}^{2m-1}\nm{u-\bar
u}_{L^2(\Omega\times\s^2)},
\end{align}
and
\begin{align}\label{ln 13}
\iint_{\Omega\times\s^2}\ss\phi\leq C\nm{\phi}_{L^2(\Omega\times\s^2)}\nm{\ss}_{L^2(\Omega\times\s^2)}\leq C\nm{\bar
u}_{L^{2m}(\Omega)}^{2m-1}\nm{\ss}_{L^2(\Omega\times\s^2)}.
\end{align}
Collecting terms in \eqref{ln 8}, \eqref{ln 9}, \eqref{ln 11}, \eqref{ln 12} and \eqref{ln 13}, we obtain
\begin{align}
\e\nm{\bar u}_{L^{2m}(\Omega)}^{2m}\leq &C\Bigg(\e\nm{\bar
u}_{L^{2m}(\Omega)}^{2m-1}\nm{u-\bar
u}_{L^{2m}(\Omega\times\s^2)}+\e\nm{\bar u}_{L^{2m}(\Omega)}^{2m-1}\bigg(\nm{u}_{L^{\frac{4m}{3}}(\Gamma^+)}+\nm{\g}_{L^{\frac{4m}{3}}(\Gamma^-)}\bigg)\\
&+\nm{\bar
u}_{L^{2m}(\Omega)}^{2m-1}\nm{u-\bar
u}_{L^2(\Omega\times\s^2)}+\nm{\bar
u}_{L^{2m}(\Omega)}^{2m-1}\nm{\ss}_{L^2(\Omega\times\s^2)}\Bigg),\no
\end{align}
which further implies
\begin{align}\label{ln 14}
\\
\e\nm{\bar u}_{L^{2m}(\Omega)}&\leq
C\bigg(\e\nm{u-\bar
u}_{L^{2m}(\Omega\times\s^2)}+\nm{u-\bar
u}_{L^2(\Omega\times\s^2)}+\e\nm{u}_{L^{\frac{4m}{3}}(\Gamma^+)}+\nm{\ss}_{L^2(\Omega\times\s^2)}+\e\nm{\g}_{L^{\frac{4m}{3}}(\Gamma^-)}\bigg).\no
\end{align}
\ \\
Step 2: Energy Estimate.\\
Similar to the $L^2$ estimates, in the weak formulation \eqref{ln 1}, we may take
the test function $\phi=u$ to get the energy estimate
\begin{align}\label{ln 15}
&&\half\e\nm{u}^2_{L^2(\Gamma^+)}+\nm{u-\bar
u}_{L^2(\Omega\times\s^2)}^2= \iint_{\Omega\times\s^2}\ss u+\half\e\nm{\g}_{L^2(\Gamma^-)}^2.
\end{align}
On the other hand, we can square on both sides of
\eqref{ln 14} to obtain
\begin{align}\label{ln 16}
\\
\e^2\nm{\bar u}_{L^{2m}(\Omega)}^2\leq&
C\bigg(\e^2\nm{u-\bar
u}_{L^{2m}(\Omega\times\s^2)}^2+\nm{u-\bar
u}_{L^2(\Omega\times\s^2)}^2+\e^2\nm{u}_{L^{\frac{4m}{3}}(\Gamma^+)}+\nm{\ss}_{L^2(\Omega\times\s^2)}^2+\e^2\nm{\g}_{L^{\frac{4m}{3}}(\Gamma^-)}^2\bigg).\no
\end{align}
Multiplying \eqref{ln 16} by a sufficiently small constant and adding it to \eqref{ln 15} to absorb
$\nm{u-\bar u}_{L^2(\Omega\times\s^2)}^2$, we deduce
\begin{align}\label{ln 17}
&\e\nm{u}_{L^2(\Gamma^+)}^2+\e^2\nm{\bar
u}_{L^{2m}(\Omega)}^2+\nm{u-\bar
u}_{L^2(\Omega\times\s^2)}^2\\
\leq&
C\bigg(\e^2\nm{u-\bar
u}_{L^{2m}(\Omega\times\s^2)}^2+\e^2\nm{u}_{L^{\frac{4m}{3}}(\Gamma^+)}+\nm{\ss}_{L^2(\Omega\times\s^2)}^2+
\iint_{\Omega\times\s^2}\ss u+\e\nm{\g}_{L^2(\Gamma^-)}^2+\e^2\nm{\g}_{L^{\frac{4m}{3}}(\Gamma^-)}^2\bigg).\no
\end{align}
By interpolation estimates and Young's inequality, we have
\begin{align}\label{ln 18}
\nm{u}_{L^{\frac{4m}{3}}(\Gamma^+)}\leq&\nm{u}_{L^2(\Gamma^+)}^{\frac{3}{2m}}\nm{u}_{L^{\infty}(\Gamma^+)}^{\frac{2m-3}{2m}}
=\bigg(\frac{1}{\e^{\frac{6m-9}{4m^2}}}\nm{u}_{L^2(\Gamma^+)}^{\frac{3}{2m}}\bigg)
\bigg(\e^{\frac{6m-9}{4m^2}}\nm{u}_{L^{\infty}(\Gamma^+)}^{\frac{2m-3}{2m}}\bigg)\\
\leq&C\bigg(\frac{1}{\e^{\frac{6m-9}{4m^2}}}\nm{u}_{L^2(\Gamma^+)}^{\frac{3}{2m}}\bigg)^{\frac{2m}{3}}+o(1)
\bigg(\e^{\frac{6m-9}{4m^2}}\nm{u}_{L^{\infty}(\Gamma^+)}^{\frac{2m-3}{2m}}\bigg)^{\frac{2m}{2m-3}}\no\\
\leq&\frac{C}{\e^{\frac{2m-3}{2m}}}\nm{u}_{L^2(\Gamma^+)}+o(1)\e^{\frac{3}{2m}}\nm{u}_{L^{\infty}(\Gamma^+)}.\no
\end{align}
Similarly, we have
\begin{align}\label{ln 19}
\nm{u-\bar u}_{L^{2m}(\Omega\times\s^2)}\leq&\nm{u-\bar u}_{L^2(\Omega\times\s^2)}^{\frac{1}{m}}\nm{u-\bar u}_{L^{\infty}(\Omega\times\s^2)}^{\frac{m-1}{m}}\\
=&\bigg(\frac{1}{\e^{\frac{3m-3}{2m^2}}}\nm{u-\bar u}_{L^2(\Omega\times\s^2)}^{\frac{1}{m}}\bigg)\bigg(\e^{\frac{3m-3}{2m^2}}\nm{u-\bar u}_{L^{\infty}(\Omega\times\s^2)}^{\frac{m-1}{m}}\bigg)\no\\
\leq&C\bigg(\frac{1}{\e^{\frac{3m-3}{2m^2}}}\nm{u-\bar u}_{L^2(\Omega\times\s^2)}^{\frac{1}{m}}\bigg)^{m}+o(1)\bigg(\e^{\frac{3m-3}{2m^2}}\nm{u-\bar u}_{L^{\infty}(\Omega\times\s^2)}^{\frac{m-1}{m}}\bigg)^{\frac{m}{m-1}}\no\\
\leq&\frac{C}{\e^{\frac{3m-3}{2m}}}\nm{u-\bar u}_{L^2(\Omega\times\s^2)}+o(1)\e^{\frac{3}{2m}}\nm{u-\bar u}_{L^{\infty}(\Omega\times\s^2)}.\no
\end{align}
In \eqref{ln 18} and \eqref{ln 19}, we need this extra $\e^{\frac{3}{2m}}$ for the convenience of $L^{\infty}$ estimate.
Then we know for sufficiently small $\e$ and $1\leq m< 3$,
\begin{align}\label{ln 20}
\e^2\nm{u}_{L^{m}(\Gamma^+)}^2
\leq&C\e^{2-\frac{2m-3}{m}}\nm{u}_{L^2(\Gamma^+)}^2+o(1)\e^{2+\frac{3}{m}}\nm{u}_{L^{\infty}(\Gamma^+)}^2\\
\leq&o(1)\e^{\frac{3}{m}}\nm{u}_{L^2(\Gamma^+)}^2+o(1)\e^{2+\frac{3}{m}}\nm{u}_{L^{\infty}(\Gamma^+)}^2\no\\
\leq &o(1)\e\nm{u}_{L^2(\Gamma^+)}^2+o(1)\e^{2+\frac{3}{m}}\nm{u}_{L^{\infty}(\Gamma^+)}^2,\no
\end{align}
and
\begin{align}\label{ln 21}
\e^2\nm{u-\bar
u}_{L^{2m}(\Omega\times\s^2)}^2\leq&\e^{2-\frac{3m-3}{m}}\nm{u-\bar u}_{L^2(\Omega\times\s^2)}^2+o(1)\e^{2+\frac{3}{m}}\nm{u}_{L^{\infty}(\Omega\times\s^2)}^2\\
\leq& o(1)\e^{\frac{3}{m}-1}\nm{u-\bar u}_{L^2(\Omega\times\s^2)}^2+o(1)\e^{2+\frac{3}{m}}\nm{u}_{L^{\infty}(\Omega\times\s^2)}^2\no\\
\leq& o(1)\nm{u-\bar u}_{L^2(\Omega\times\s^2)}^2+o(1)\e^{2+\frac{3}{m}}\nm{u}_{L^{\infty}(\Omega\times\s^2)}^2.\no
\end{align}
Inserting \eqref{ln 20} and \eqref{ln 21} into \eqref{ln 17}, we can absorb $\nm{u-\bar u}_{L^2(\Omega\times\s^2)}$ and $\e\nm{u}_{L^2(\Gamma^+)}^2$ into the left-hand side to obtain
\begin{align}\label{ln 22}
&\e\nm{u}_{L^2(\Gamma^+)}^2+\e^2\nm{\bar
u}_{L^{2m}(\Omega\times\s^2)}^2+\nm{u-\bar
u}_{L^2(\Omega\times\s^2)}^2\\
\leq&
C\bigg(o(1)\e^{2+\frac{3}{m}}\Big(\nm{u}_{L^{\infty}(\Omega\times\s^2)}^2+\nm{u}_{L^{\infty}(\Gamma^+)}^2\Big)+\nm{\ss}_{L^2(\Omega\times\s^2)}^2+
\iint_{\Omega\times\s^2}\ss u+\e\nm{\g}_{L^2(\Gamma^-)}^2+\e^2\nm{\g}_{L^{\frac{4m}{3}}(\Gamma^-)}^2\bigg).\no
\end{align}
We decompose
\begin{align}
\iint_{\Omega\times\s^2}\ss u=\iint_{\Omega\times\s^2}\ss\bar u+\iint_{\Omega\times\s^2}\ss(u-\bar u).
\end{align}
H\"older's inequality and Cauchy's inequality imply
\begin{align}\label{ln 23}
\iint_{\Omega\times\s^2}\ss\bar u\leq\nm{\ss}_{L^{\frac{2m}{2m-1}}(\Omega\times\s^2)}\nm{\bar u}_{L^{2m}(\Omega\times\s^2)}
\leq\frac{C}{\e^{2}}\nm{\ss}_{L^{\frac{2m}{2m-1}}(\Omega\times\s^2)}^2+o(1)\e^2\nm{\bar u}_{L^{2m}(\Omega\times\s^2)}^2,
\end{align}
and
\begin{align}\label{ln 24}
\iint_{\Omega\times\s^2}\ss(u-\bar u)\leq C\nm{\ss}_{L^{2}(\Omega\times\s^2)}^2+o(1)\nm{u-\bar u}_{L^2(\Omega\times\s^2)}^2.
\end{align}
Hence, inserting \eqref{ln 23} and \eqref{ln 24} into \eqref{ln 22}, we can absorb $\e^2\nm{\bar u}_{L^{2m}(\Omega\times\s^2)}^2$ and $\nm{u-\bar u}_{L^2(\Omega\times\s^2)}^2$ into the left-hand side to get
\begin{align}
&\e\nm{u}_{L^2(\Gamma^+)}^2+\e^2\nm{\bar
u}_{L^{2m}(\Omega\times\s^2)}^2+\nm{u-\bar
u}_{L^2(\Omega\times\s^2)}^2\\
\leq&
C\bigg(o(1)\e^{2+\frac{3}{m}}\Big(\nm{u}_{L^{\infty}(\Omega\times\s^2)}^2+\nm{u}_{L^{\infty}(\Gamma^+)}^2\Big)+\nm{\ss}_{L^2(\Omega\times\s^2)}^2+
\frac{1}{\e^2}\nm{\ss}_{L^{\frac{2m}{2m-1}}(\Omega\times\s^2)}^2+\e\nm{\g}_{L^2(\Gamma^-)}^2+\e^2\nm{\g}_{L^{\frac{4m}{3}}(\Gamma^-)}^2\bigg),\no
\end{align}
which implies
\begin{align}
&\frac{1}{\e^{\frac{1}{2}}}\nm{u}_{L^2(\Gamma^+)}+\nm{
\bar u}_{L^{2m}(\Omega\times\s^2)}+\frac{1}{\e}\nm{u-\bar
u}_{L^2(\Omega\times\s^2)}\\
\leq&
C\bigg(o(1)\e^{\frac{3}{2m}}\Big(\nm{u}_{L^{\infty}(\Omega\times\s^2)}+\nm{u}_{L^{\infty}(\Gamma^+)}\Big)+\frac{1}{\e}\nm{\ss}_{L^2(\Omega\times\s^2)}+
\frac{1}{\e^2}\nm{\ss}_{L^{\frac{2m}{2m-1}}(\Omega\times\s^2)}+\frac{1}{\e^{\frac{1}{2}}}\nm{\g}_{L^2(\Gamma^-)}+\nm{\g}_{L^{\frac{4m}{3}}(\Gamma^-)}\bigg),\no
\end{align}

\end{proof}

\subsection{$L^{\infty}$ Estimate - Second Round}

\begin{theorem}\label{LI estimate}
The unique solution $u(\vx,\vw)$ to the equation \eqref{neutron} satisfies for integer $1\leq m< 3$,
\begin{align}
\nm{u}_{L^{\infty}(\Omega\times\s^2)}\leq& C\bigg(\frac{1}{\e^{1+\frac{3}{2m}}}\nm{\ss}_{L^{2}(\Gamma^-)}+
\frac{1}{\e^{2+\frac{3}{2m}}}\nm{\ss}_{L^{\frac{2m}{2m-1}}(\Omega\times\s^2)}+\nm{\ss}_{L^{\infty}(\Omega\times\s^2)}\\
&+\frac{1}{\e^{\frac{1}{2}+\frac{3}{2m}}}\nm{\g}_{L^2(\Gamma^-)}+\frac{1}{\e^{\frac{3}{2m}}}\nm{\g}_{L^{\frac{4m}{3}}(\Gamma^-)}+\nm{\g}_{L^{\infty}(\Gamma^-)}\bigg).\no
\end{align}
\end{theorem}
\begin{proof}
Following the argument in the proof of Theorem \eqref{LI estimate'}, Step 1, Step 2, and the estimate of $I_1$ in Step 3 are identical. We focus on the estimate of $I_2$. Similar to \eqref{li 9}, we apply H\"{o}lder's inequality to obtain
\begin{align}\label{li 1'}
\abs{I_2}\leq&C\int_{0}^{t_b}\Bigg(\bigg(\int_{\s^2}\int_{r\geq\delta}{\bf{1}}_{\Omega}\Big(\vx-\e s\vw-\e
r\vw_t\Big)\abs{\bar u\Big(\vx-\e s\vw-\e
r\vw_t\Big)}^{2m}\ud{r}\ud{\vw_t}\bigg)^{\frac{1}{2m}}\\
&\times\bigg(\int_{\s^2}\int_{r\geq\delta}{\bf{1}}_{\Omega}\Big(\vx-\e s\vw-\e
r\vw_t\Big)\ue^{-2r}\ud{r}\ud{\vw_t}\bigg)^{\frac{2m-1}{2m}}\Bigg)\ue^{-s}\ud{s}\no\\
\leq&C\int_{0}^{t_b}\Bigg(\bigg(\int_{\s^2}\int_{r\geq\delta}{\bf{1}}_{\Omega}\Big(\vx-\e s\vw-\e
r\vw_t\Big)\abs{\bar u\Big(\vx-\e s\vw-\e
r\vw_t\Big)}^{2m}\ud{r}\ud{\vw_t}\bigg)^{\frac{1}{2m}}\Bigg)\ue^{-s}\ud{s}.\no
\end{align}
Then, using the same substitution \eqref{li 11}, we may simplify \eqref{li 1'} as
\begin{align}
\abs{I_2}\leq&C\int_{0}^{t_b}\bigg(\int_{\Omega}\frac{1}{\e^3\d^2}\abs{\bar u(\vec
y)}^{2m}\ud{\vec y}\bigg)^{\frac{1}{2m}}\ue^{-s}\ud{s}\\
\leq&\frac{C}{\e^{\frac{3}{2m}}\d^{\frac{1}{m}}}\int_{0}^{t_b}\bigg(\int_{\Omega}\abs{\bar u(\vec
y)}^{2m}\ud{\vec y}\bigg)^{\frac{1}{2m}}\ue^{-s}\ud{s}
\leq\frac{C}{\e^{\frac{3}{2m}}\d^{\frac{1}{m}}}\nm{\bar u}_{L^{2m}(\Omega)}.\no
\end{align}
In summary, collecting \eqref{li 3}, \eqref{li 4}, \eqref{li 5}, \eqref{li 6}, \eqref{li 7}, and \eqref{li 1'}, for any $(\vx,\vw)\in\bar\Omega\times\s^2$, we have
\begin{align}\label{ctt 1}
\abs{u(\vx,\vw)}\leq C\bigg(\delta
\nm{u}_{L^{\infty}(\Omega\times\s^2)}+\frac{1}{\e^{\frac{3}{2m}}\d^{\frac{1}{m}}}\nm{
\bar u}_{L^{2m}(\Omega\times\s^2)}+
\nm{\ss}_{L^{\infty}(\Omega\times\s^2)}+\nm{\g}_{L^{\infty}(\Gamma^-)}\bigg).
\end{align}
Let $\d$ be sufficiently small such that $C\d\leq \dfrac{1}{2}$. Taking supremum over $(\vx,\vw)\in\Gamma^+$ and using Theorem \ref{LN estimate}, we have
\begin{align}
\nm{u}_{L^{\infty}(\Gamma^+)}\leq& \frac{1}{2}
\nm{u}_{L^{\infty}(\Omega\times\s^2)}+C\bigg(o(1)\Big(\nm{u}_{L^{\infty}(\Omega\times\s^2)}+\nm{u}_{L^{\infty}(\Gamma^+)}\Big)\\
&+\frac{1}{\e^{1+\frac{3}{2m}}}\nm{\ss}_{L^{2}(\Gamma^-)}+
\frac{1}{\e^{2+\frac{3}{2m}}}\nm{\ss}_{L^{\frac{2m}{2m-1}}(\Omega\times\s^2)}+\nm{\ss}_{L^{\infty}(\Omega\times\s^2)}\no\\
&+\frac{1}{\e^{\frac{1}{2}+\frac{3}{2m}}}\nm{\g}_{L^2(\Gamma^-)}+\frac{1}{\e^{\frac{3}{2m}}}\nm{\g}_{L^{\frac{4m}{3}}(\Gamma^-)}+\nm{\g}_{L^{\infty}(\Gamma^-)}\bigg)
.\no
\end{align}
Absorbing $o(1)\nm{u}_{L^{\infty}(\Gamma^+)}$ into the left-hand side, we obtain
\begin{align}\label{ctt 2}
\nm{u}_{L^{\infty}(\Gamma^+)}\leq& \frac{1}{2}
\nm{u}_{L^{\infty}(\Omega\times\s^2)}+C\bigg(o(1)\nm{u}_{L^{\infty}(\Omega\times\s^2)}\\
&+\frac{1}{\e^{1+\frac{3}{2m}}}\nm{\ss}_{L^{2}(\Gamma^-)}+
\frac{1}{\e^{2+\frac{3}{2m}}}\nm{\ss}_{L^{\frac{2m}{2m-1}}(\Omega\times\s^2)}+\nm{\ss}_{L^{\infty}(\Omega\times\s^2)}\no\\
&+\frac{1}{\e^{\frac{1}{2}+\frac{3}{2m}}}\nm{\g}_{L^2(\Gamma^-)}+\frac{1}{\e^{\frac{3}{2m}}}\nm{\g}_{L^{\frac{4m}{3}}(\Gamma^-)}+\nm{\g}_{L^{\infty}(\Gamma^-)}\bigg)
.\no
\end{align}
Taking supremum over $(\vx,\vw)\in\Omega\times\s^2$ and using Theorem \ref{LN estimate}, we have
\begin{align}\label{ctt 3}
\nm{u}_{L^{\infty}(\Omega\times\s^2)}\leq& \frac{1}{2}
\nm{u}_{L^{\infty}(\Omega\times\s^2)}+C\bigg(o(1)\Big(\nm{u}_{L^{\infty}(\Omega\times\s^2)}+\nm{u}_{L^{\infty}(\Gamma^+)}\Big)\\
&+\frac{1}{\e^{1+\frac{3}{2m}}}\nm{\ss}_{L^{2}(\Gamma^-)}+
\frac{1}{\e^{2+\frac{3}{2m}}}\nm{\ss}_{L^{\frac{2m}{2m-1}}(\Omega\times\s^2)}+\nm{\ss}_{L^{\infty}(\Omega\times\s^2)}\no\\
&+\frac{1}{\e^{\frac{1}{2}+\frac{3}{2m}}}\nm{\g}_{L^2(\Gamma^-)}+\frac{1}{\e^{\frac{3}{2m}}}\nm{\g}_{L^{\frac{4m}{3}}(\Gamma^-)}+\nm{\g}_{L^{\infty}(\Gamma^-)}\bigg)
.\no
\end{align}
Inserting (\ref{ctt 2}) into (\ref{ctt 3}), we obtain
\begin{align}
\nm{u}_{L^{\infty}(\Omega\times\s^2)}\leq& \frac{1}{2}
\nm{u}_{L^{\infty}(\Omega\times\s^2)}+C\bigg(o(1)\nm{u}_{L^{\infty}(\Omega\times\s^2)}\\
&+\frac{1}{\e^{1+\frac{3}{2m}}}\nm{\ss}_{L^{2}(\Gamma^-)}+
\frac{1}{\e^{2+\frac{3}{2m}}}\nm{\ss}_{L^{\frac{2m}{2m-1}}(\Omega\times\s^2)}+\nm{\ss}_{L^{\infty}(\Omega\times\s^2)}\no\\
&+\frac{1}{\e^{\frac{1}{2}+\frac{3}{2m}}}\nm{\g}_{L^2(\Gamma^-)}+\frac{1}{\e^{\frac{3}{2m}}}\nm{\g}_{L^{\frac{4m}{3}}(\Gamma^-)}+\nm{\g}_{L^{\infty}(\Gamma^-)}\bigg)
.\no
\end{align}
Then absorbing $\dfrac{1}{2}
\nm{u}_{L^{\infty}(\Omega\times\s^2)}$ and $o(1)\nm{u}_{L^{\infty}(\Omega\times\s^2)}$ into the left-hand side, we get
\begin{align}
\nm{u}_{L^{\infty}(\Omega\times\s^2)}\leq& C\bigg(\frac{1}{\e^{1+\frac{3}{2m}}}\nm{\ss}_{L^{2}(\Gamma^-)}+
\frac{1}{\e^{2+\frac{3}{2m}}}\nm{\ss}_{L^{\frac{2m}{2m-1}}(\Omega\times\s^2)}+\nm{\ss}_{L^{\infty}(\Omega\times\s^2)}\\
&+\frac{1}{\e^{\frac{1}{2}+\frac{3}{2m}}}\nm{\g}_{L^2(\Gamma^-)}+\frac{1}{\e^{\frac{3}{2m}}}\nm{\g}_{L^{\frac{4m}{3}}(\Gamma^-)}+\nm{\g}_{L^{\infty}(\Gamma^-)}\bigg).\no
\end{align}

\end{proof}

\section{Diffusive Limit}

\subsection{Analysis of Regular Boundary Layer}

In this subsection, we will justify that the regular boundary layers are all well-defined for $0<\alpha\leq 1$. In the following, we will not show $\e$, $\iota_i$ and $\psi$ dependence when there is no confusion.\\
\ \\
Step 1: Well-Posedness of $\ub_0$.\\
Based on \eqref{et 1}, $\ub_0$ satisfies the $\e$-Milne problem with geometric correction
\begin{align}\label{dte 1}
\left\{
\begin{array}{l}
\sin\phi\dfrac{\p\ub_0 }{\p\eta}+F(\eta)\cos\phi\dfrac{\p
\ub_0 }{\p\phi}+\ub_0 -\bub_0 =0,\\\rule{0ex}{2em}
\ub_0 (0,\phi)=\gb(\phi)-\mathscr{F}_{0,L}\ \ \text{for}\ \
\sin\phi>0,\\\rule{0ex}{2em}
\ub_0 (L,\phi)=\ub_0 (L,\rr[\phi]),
\end{array}
\right.
\end{align}
where
\begin{align}
F(\eta)=-\e\bigg(\dfrac{\sin^2\psi}{R_1-\e\eta}+\dfrac{\cos^2\psi}{R_2-\e\eta}\bigg).
\end{align}
Therefore, since $\lnmp{\gb}\leq C$, applying Theorem \ref{Milne theorem 2} to the equation \eqref{dte 1}, we know
\begin{align}
\lnnm{\ue^{K_0\eta}\ub_0}\leq C.
\end{align}
\ \\
Step 2: Tangential and Velocity Derivatives of $\ub_0$.\\
For $i=1,2$, the $\iota_i$ derivative $W_i=\dfrac{\p\ub_0}{\p\iota_i}$ satisfies
\begin{align}\label{dte 2}
\left\{
\begin{array}{l}
\sin\phi\dfrac{\p W_i}{\p\eta}+F(\eta)\cos\phi\dfrac{\p W_i}{\p\phi}+W_i-\bar W_i=-\e\left(\dfrac{\p_{\iota_i}R_1\sin^2\psi}{(R_1-\e\eta)^2}+\dfrac{\p_{\iota_i}R_2\cos^2\psi}{(R_2-\e\eta)^2}\right)
\cos\phi\dfrac{\p \ub_0}{\p\phi},\\\rule{0ex}{2em}
W_i (0,\phi)=\dfrac{\p\gb}{\p\iota_i}(\phi)-\dfrac{\p\mathscr{F}_{0,L}}{\p\iota_i}\ \ \text{for}\ \
\sin\phi>0,\\\rule{0ex}{2em}
W_i (L,\phi)=W_i (L,\rr[\phi]).
\end{array}
\right.
\end{align}
Since
\begin{align}\label{dt 4}
\abs{\dfrac{\p_{\iota_i}R_1\sin^2\psi}{(R_1-\e\eta)^2}+\dfrac{\p_{\iota_i}R_2\cos^2\psi}{(R_2-\e\eta)^2}}
\leq&\max\left\{\abs{\frac{\p_{\iota_i}R_1}{R_1-\e\eta}},\abs{\frac{\p_{\iota_i}R_2}{R_2-\e\eta}}\right\}
\abs{\dfrac{\sin^2\psi}{R_1-\e\eta}+\dfrac{\cos^2\psi}{R_2-\e\eta}}\\
\leq&C\abs{\dfrac{\sin^2\psi}{R_1-\e\eta}+\dfrac{\cos^2\psi}{R_2-\e\eta}},\no
\end{align}
we have
\begin{align}
\abs{-\e\left(\dfrac{\p_{\iota_i}R_1\sin^2\psi}{(R_1-\e\eta)^2}+\dfrac{\p_{\iota_i}R_2\cos^2\psi}{(R_2-\e\eta)^2}\right)
\cos\phi\dfrac{\p \ub_0}{\p\phi}}\leq C\abs{F(\eta)\cos\phi\frac{\p\ub_0}{\p\phi}}.
\end{align}
Applying Theorem \ref{pt theorem 2} to the equation \eqref{dte 1}, we know
\begin{align}\label{dt 1}
\\
\lnnm{\ue^{K_0\eta}F(\eta)\cos\phi\frac{\p\ub_0}{\p\phi}}
\leq&C\abs{\ln(\e)}^8\bigg(\lnmp{\gb}+\lnmp{\e\dfrac{\p\gb}{\p\phi}}+\lnnm{\ue^{K_0\eta}\ub_0}\bigg)\leq C\abs{\ln(\e)}^8.\no
\end{align}
Note that here although $\lnm{\dfrac{\p\gb}{\p\phi}}\leq C\e^{-\alpha}$, with the help of $\e$, we can get rid of this negative power. Therefore, applying Theorem \ref{Milne theorem 2} to the equation \eqref{dte 2}, we have
\begin{align}\label{dt 2}
\lnnm{\ue^{K_0\eta}W_i}\leq C\abs{\ln(\e)}^8.
\end{align}
\ \\
Step 3: Velocity Derivatives of $\ub_0$.\\
The $\psi$ derivative $H=\dfrac{\p\ub_0}{\p\psi}$ satisfies
\begin{align}\label{dte 3}
\left\{
\begin{array}{l}
\sin\phi\dfrac{\p H}{\p\eta}+F(\eta)\cos\phi\dfrac{\p H}{\p\phi}+H=\e\bigg(\dfrac{2\sin\psi\cos\psi}{R_1-\e\eta}-\dfrac{2\sin\psi\cos\psi}{R_2-\e\eta}\bigg)\cos\phi\dfrac{\p \ub_0}{\p\phi},\\\rule{0ex}{2em}
H(0,\phi)=\dfrac{\p\gb}{\p\psi}(\phi)\ \ \text{for}\ \
\sin\phi>0,\\\rule{0ex}{2em}
H(L,\phi)=H(L,\rr[\phi]).
\end{array}
\right.
\end{align}
Since
\begin{align}
\abs{\dfrac{2\sin\psi\cos\psi}{R_1-\e\eta}-\dfrac{2\sin\psi\cos\psi}{R_2-\e\eta}}
\leq&C\abs{\dfrac{1}{\max\{R_1,R_2\}-\e\eta}}
\leq C\abs{\dfrac{\sin^2\psi}{R_1-\e\eta}+\dfrac{\cos^2\psi}{R_2-\e\eta}},
\end{align}
we have
\begin{align}\label{dt 11}
\abs{\e\bigg(\dfrac{2\sin\psi\cos\psi}{R_1-\e\eta}-\dfrac{2\sin\psi\cos\psi}{R_2-\e\eta}\bigg)\cos\phi\dfrac{\p \ub_0}{\p\phi}}\leq C\abs{F(\eta)\cos\phi\frac{\p\ub_0}{\p\phi}}.
\end{align}
Therefore, using \eqref{dt 1}, applying Theorem \ref{Milne theorem 2'} to the equation \eqref{dte 3} (by the construction of $\ub_0$, $H$ must satisfy the requirement of Theorem \ref{Milne theorem 2'}), we have
\begin{align}\label{dt 3}
\lnnm{\ue^{K_0\eta}H}\leq C\abs{\ln(\e)}^8.
\end{align}
\ \\
Step 4: Well-Posedness of $\ub_1$.\\
$\ub_1$ satisfies the $\e$-Milne problem with geometric correction
\begin{align}\label{dte 4}
\left\{
\begin{array}{l}
\sin\phi\dfrac{\p \ub_1 }{\p\eta}+F(\eta)\cos\phi\dfrac{\p
\ub_1 }{\p\phi}+\ub_1 -\bub_1 =G[\ub_0],\\\rule{0ex}{2em}
\ub_1 (0,\phi)=\vw\cdot\nx\u_0(\vx_0,\vw)-\mathscr{F} _{1,L}\ \ \text{for}\ \
\sin\phi>0,\\\rule{0ex}{2em}
\ub_1 (L,\phi)=\ub_1 (L,\rr[\phi]),
\end{array}
\right.
\end{align}
where
\begin{align}
G[\ub_0]=&\bigg(\dfrac{\cos\phi\sin\psi}{P_1(1-\e\kk_1\eta)}W_1+\dfrac{\cos\phi\cos\psi}{P_2(1-\e\kk_2\eta)}W_2\bigg)\\
&+\Bigg(\dfrac{\sin\psi}{1-\e\kk_1\eta}\bigg(\cos\phi\Big(\vt_1\cdot\Big(\vt_2\times(\p_{12}\vr\times\vt_2)\Big)\Big)
-\kk_1P_1P_2\sin\phi\cos\psi\bigg)\no\\
&+\dfrac{\cos\psi}{1-\e\kk_2\eta}\bigg(-\cos\phi\Big(\vt_2\cdot\Big(\vt_1\times(\p_{12}\vr\times\vt_1)\Big)\Big)
+\kk_2P_1P_2\sin\phi\sin\psi\bigg)\Bigg)\dfrac{1}{P_1P_2}H\no\\
=&G_{W_1}W_1+G_{W_2}W_2+G_HH.\no
\end{align}
Here we use $G_{W_1}$, $G_{W_2}$ and $G_H$ to denote the quantities in front of $W_1$, $W_2$ and $H$. We may directly bound
\begin{align}
\abs{G_{W_1}}+\abs{G_{W_2}}+\abs{G_H}\leq C.
\end{align}
Using \eqref{dt 2} and \eqref{dt 3}, we have
\begin{align}
\lnnm{\ue^{K_0\eta}G[\ub_0]}\leq C\abs{\ln(\e)}^8.
\end{align}
Therefore, applying Theorem \ref{Milne theorem 2} to the equation \eqref{dte 4}, we know
\begin{align}
\lnnm{\ue^{K_0\eta}\ub_1}\leq C\abs{\ln(\e)}^8.
\end{align}
\ \\
Step 5: Tangential and Velocity Derivatives of $\ub_1$.\\
For $j=1,2$, the $\iota_j$ derivative $V_j=\dfrac{\p\ub_1}{\p\iota_j}$ satisfies
\begin{align}
\left\{
\begin{array}{l}
\sin\phi\dfrac{\p V_j}{\p\eta}+F(\eta)\cos\phi\dfrac{\p V_j}{\p\phi}+V_j-\bar V_j=S_1+S_2+S_3,\\\rule{0ex}{2em}
V_j(0,\phi)=\dfrac{\p}{\p\iota_j}\bigg(\vw\cdot\nx\u_0(\vx_0,\vw)-\mathscr{F} _{1,L}\bigg)\ \ \text{for}\ \
\sin\phi>0,\\\rule{0ex}{2em}
V_j(L,\phi)=V_j(L,\rr[\phi]),
\end{array}
\right.
\end{align}
where
\begin{align}
S_1=&-\e\left(\dfrac{\p_{\iota_j}R_1\sin^2\psi}{(R_1-\e\eta)^2}+\dfrac{\p_{\iota_j}R_2\cos^2\psi}{(R_2-\e\eta)^2}\right)
\cos\phi\dfrac{\p \ub_1}{\p\phi},\\
S_2=&\frac{\p G_{W_1}}{\p\iota_j}W_1+\frac{\p G_{W_2}}{\p\iota_j}W_2+\frac{\p G_{H}}{\p\iota_j}H,\\
S_3=&G_{W_1}\frac{\p W_1}{\p\iota_j}+G_{W_2}\frac{\p W_2}{\p\iota_j}+G_H\frac{\p H}{\p\iota_j}.
\end{align}
Using \eqref{dt 4}, we know
\begin{align}\label{dt 5}
\lnnm{\ue^{K_0\eta}S_1}\leq&C\lnnm{\ue^{K_0\eta}F(\eta)\cos\phi\frac{\p \ub_1}{\p\phi}}.
\end{align}
Using \eqref{dt 2} and \eqref{dt 3}, applying Theorem \ref{pt theorem 2} to the equation \eqref{dte 4}, we have
\begin{align}\label{dt 6}
&\lnnm{\ue^{K_0\eta}F(\eta)\cos\phi\frac{\p \ub_1}{\p\phi}}\\
\leq& C\bigg(\lnnm{\ue^{K_0\eta}\Big(G_{W_1}W_1+G_{W_2}W_2+G_HH\Big)}+\lnnm{\ue^{K_0\eta}\zeta\frac{\p}{\p\eta}\Big(G_{W_1}W_1+G_{W_2}W_2+G_HH\Big)}\bigg)\no\\
\leq& C\bigg(\lnnm{\ue^{K_0\eta}W_1}+\lnnm{\ue^{K_0\eta}W_2}+\lnnm{\ue^{K_0\eta}H}\no\\
&+\lnnm{\ue^{K_0\eta}\zeta\frac{\p W_1}{\p\eta}}+\lnnm{\ue^{K_0\eta}\zeta\frac{\p W_2}{\p\eta}}+\lnnm{\ue^{K_0\eta}\zeta\frac{\p H}{\p\eta}}\bigg)\no\\
\leq& C\bigg(1+\lnnm{\ue^{K_0\eta}\zeta\frac{\p W_1}{\p\eta}}+\lnnm{\ue^{K_0\eta}\zeta\frac{\p W_2}{\p\eta}}+\lnnm{\ue^{K_0\eta}\zeta\frac{\p H}{\p\eta}}\bigg).\no
\end{align}
Combining \eqref{dt 5} and \eqref{dt 6}, we know
\begin{align}\label{dt 7}
\lnnm{\ue^{K_0\eta}S_1}\leq&C\bigg(1+\lnnm{\ue^{K_0\eta}\zeta\frac{\p W_1}{\p\eta}}+\lnnm{\ue^{K_0\eta}\zeta\frac{\p W_2}{\p\eta}}+\lnnm{\ue^{K_0\eta}\zeta\frac{\p H}{\p\eta}}\bigg).
\end{align}
On the other hand, using \eqref{dt 2} and \eqref{dt 3}, we directly estimate
\begin{align}\label{dt 8}
\lnnm{\ue^{K_0\eta}S_2}\leq&C\bigg(\lnnm{\ue^{K_0\eta}W_1}+\lnnm{\ue^{K_0\eta}W_2}+\lnnm{\ue^{K_0\eta}H}\bigg)\leq C.
\end{align}
Also, we know
\begin{align}\label{dt 9}
\lnnm{\ue^{K_0\eta}S_3}\leq&C\bigg(\lnnm{\ue^{K_0\eta}\frac{\p W_1}{\p\iota_j}}+\lnnm{\ue^{K_0\eta}\frac{\p W_2}{\p\iota_j}}+\lnnm{\ue^{K_0\eta}\frac{\p H}{\p\iota_j}}\bigg).
\end{align}
Summarizing \eqref{dt 7}, \eqref{dt 8} and \eqref{dt 9} and applying Theorem \ref{Milne theorem 2}, we know
\begin{align}\label{dt 10}
\lnnm{\ue^{K_0\eta}V_j}\leq&\lnnm{\ue^{K_0\eta}S_1}+\lnnm{\ue^{K_0\eta}S_2}+\lnnm{\ue^{K_0\eta}S_3}\\
\leq&C\bigg(1+\lnnm{\ue^{K_0\eta}\zeta\frac{\p W_1}{\p\eta}}+\lnnm{\ue^{K_0\eta}\zeta\frac{\p W_2}{\p\eta}}+\lnnm{\ue^{K_0\eta}\zeta\frac{\p H}{\p\eta}}\no\\
&+\lnnm{\ue^{K_0\eta}\frac{\p W_1}{\p\iota_j}}+\lnnm{\ue^{K_0\eta}\frac{\p W_2}{\p\iota_j}}+\lnnm{\ue^{K_0\eta}\frac{\p H}{\p\iota_j}}\bigg).\no
\end{align}
Therefore, in order to bound $V_j$, we need the $\eta$ and $\iota_i$ derivative estimates of $W_i$ and $H$.\\
\ \\
Using a similar argument, we obtain the estimate of $\psi$ derivative $M=\dfrac{\p\ub_1}{\p\psi}$:
\begin{align}\label{dt 15}
\lnnm{\ue^{K_0\eta}M}\leq&C\bigg(1+\lnnm{\ue^{K_0\eta}\zeta\frac{\p W_1}{\p\eta}}+\lnnm{\ue^{K_0\eta}\zeta\frac{\p W_2}{\p\eta}}+\lnnm{\ue^{K_0\eta}\zeta\frac{\p H}{\p\eta}}\\
&+\lnnm{\ue^{K_0\eta}\frac{\p W_1}{\p\psi}}+\lnnm{\ue^{K_0\eta}\frac{\p W_2}{\p\psi}}+\lnnm{\ue^{K_0\eta}\frac{\p H}{\p\psi}}\bigg).\no
\end{align}
Therefore, in order to bound $M$, we need the $\eta$ and $\psi$ derivative estimates of $W_i$ and $H$.\\
\ \\
Step 6: Tangential and Velocity Derivatives of $W_i$ and $H$.\\
The $\iota_j$ derivative $\Theta_{ij}=\dfrac{\p W_i}{\p\iota_j}$ satisfies
\begin{align}\label{dte 6}
\left\{
\begin{array}{l}
\sin\phi\dfrac{\p \Theta_{ij}}{\p\eta}+F(\eta)\cos\phi\dfrac{\p \Theta_{ij}}{\p\phi}+\Theta_{ij}-\bar \Theta_{ij}=T_1+T_2,\\\rule{0ex}{2em}
\Theta_{ij} (0,\phi)=\dfrac{\p^2\gb}{\p\iota_i\p\iota_j}(\phi)-\dfrac{\p^2\mathscr{F}_0}{\p\iota_i\p\iota_j}\ \ \text{for}\ \
\sin\phi>0,\\\rule{0ex}{2em}
\Theta_{ij} (L,\phi)=\Theta_{ij} (L,\rr[\phi]),
\end{array}
\right.
\end{align}
where
\begin{align}
T_1=&-\e\left(\dfrac{\p_{\iota_j}R_1\sin^2\psi}{(R_1-\e\eta)^2}+\dfrac{\p_{\iota_j}R_2\cos^2\psi}{(R_2-\e\eta)^2}\right)
\cos\phi\bigg(\dfrac{\p W_i}{\p\phi}+\dfrac{\p W_j}{\p\phi}\bigg),\\
T_2=&-\e\dfrac{\p}{\p\iota_j}\left(\dfrac{\p_{\iota_i}R_1\sin^2\psi}{(R_1-\e\eta)^2}+\dfrac{\p_{\iota_i}R_2\cos^2\psi}{(R_2-\e\eta)^2}\right)\cos\phi\dfrac{\p\ub_0}{\p\phi}.
\end{align}
Using \eqref{dt 4}, we have
\begin{align}\label{dt 16}
\lnnm{\ue^{K_0\eta}T_1}\leq&C\lnnm{\ue^{K_0\eta}F(\eta)\cos\phi\dfrac{\p W_i}{\p\phi}}.
\end{align}
By a similar argument as \eqref{dt 4}, using \eqref{dt 1}, we know
\begin{align}\label{dt 17}
\lnnm{\ue^{K_0\eta}T_2}\leq&C\lnnm{F(\eta)\cos\phi\dfrac{\p\ub_0}{\p\phi}}\leq C\abs{\ln(\e)}^8.
\end{align}
Applying Theorem \ref{Milne theorem 2} to the equation \eqref{dte 6}, using \eqref{dt 16} and \eqref{dt 17}, we have
\begin{align}\label{dt 18}
\lnnm{\ue^{K_0\eta}\Theta_{ij}}\leq& C\bigg(\lnnm{\ue^{K_0\eta}T_1}+\lnnm{\ue^{K_0\eta}T_2}\bigg)\\
\leq& C\abs{\ln(\e)}^8+C\lnnm{\ue^{K_0\eta}F(\eta)\cos\phi\dfrac{\p W_i}{\p\phi}}.\no
\end{align}
\ \\
Using a similar argument, we obtain the estimate of $\psi$ derivative $\Lambda_i=\dfrac{\p W_i}{\p\psi}$:
\begin{align}\label{dt 19}
\lnnm{\ue^{K_0\eta}\Lambda_i}\leq  C\abs{\ln(\e)}^8+C\bigg(\lnnm{\ue^{K_0\eta}F(\eta)\cos\phi\dfrac{\p W_i}{\p\phi}}+\lnnm{\ue^{K_0\eta}F(\eta)\cos\phi\dfrac{\p H}{\p\phi}}\bigg).
\end{align}
Using a similar argument with Theorem \ref{Milne theorem 2'}, we obtain the estimate of $\iota_j$ derivative $\Upsilon_j=\dfrac{\p H}{\p\iota_j}$:
\begin{align}\label{dt 20}
\lnnm{\ue^{K_0\eta}\Upsilon_j}\leq  C\abs{\ln(\e)}^8+C\bigg(\lnnm{\ue^{K_0\eta}F(\eta)\cos\phi\dfrac{\p W_j}{\p\phi}}+\lnnm{\ue^{K_0\eta}F(\eta)\cos\phi\dfrac{\p H}{\p\phi}}\bigg).
\end{align}
Using a similar argument with Theorem \ref{Milne theorem 2'}, we obtain the estimate of $\psi$ derivative $\Xi=\dfrac{\p H}{\p\psi}$:
\begin{align}\label{dt 21}
\lnnm{\ue^{K_0\eta}\Xi}\leq  C\abs{\ln(\e)}^8+C\lnnm{\ue^{K_0\eta}F(\eta)\cos\phi\dfrac{\p H}{\p\phi}}.
\end{align}
\ \\
Inserting \eqref{dt 18}, \eqref{dt 19}, \eqref{dt 20}, \eqref{dt 21} into \eqref{dt 10} and \eqref{dt 15}, we know
\begin{align}\label{dt 29}
&\lnnm{\ue^{K_0\eta}V_j}+\lnnm{\ue^{K_0\eta}M}\\
\leq&C\bigg(1+\lnnm{\ue^{K_0\eta}\zeta\frac{\p W_1}{\p\eta}}+\lnnm{\ue^{K_0\eta}\zeta\frac{\p W_2}{\p\eta}}+\lnnm{\ue^{K_0\eta}\zeta\frac{\p H}{\p\eta}}\no\\
&+\lnnm{\ue^{K_0\eta}F(\eta)\cos\phi\dfrac{\p W_1}{\p\phi}}+\lnnm{\ue^{K_0\eta}F(\eta)\cos\phi\dfrac{\p W_2}{\p\phi}}+\lnnm{\ue^{K_0\eta}F(\eta)\cos\phi\dfrac{\p H}{\p\phi}}\bigg).\no
\end{align}
Hence, we need the regularity estimate of $W_i$ and $H$. However, this cannot be done directly. We will first study the normal derivative of $\ub_0$.\\
\ \\
Step 7: Regularity of Normal Derivative.\\
The normal derivative $A=\dfrac{\p\ub_0}{\p\eta}$ satisfies
\begin{align}\label{dte 7}
\left\{
\begin{array}{l}
\sin\phi\dfrac{\p A}{\p\eta}+F(\eta)\cos\phi\dfrac{\p A}{\p\phi}+A-\bar A=\dfrac{\p F}{\p\eta}\cos\phi\dfrac{\p\ub_0}{\p\phi},\\\rule{0ex}{2em}
A (0,\phi)=\dfrac{1}{\sin\phi}\bigg(F(\eta)\cos\phi\dfrac{\p\gb}{\p\phi}(\phi)-\gb(0,\phi)+\bar\ub_0(0,\phi)\bigg)\ \ \text{for}\ \
\sin\phi>0,\\\rule{0ex}{2em}
A (L,\phi)=A (L,\rr[\phi]).
\end{array}
\right.
\end{align}
This is where the cut-off in $\gb$ plays a role. For $0<\phi<\e^{\alpha}$, we know $A(0,\phi)=0$. Here note the fact that $\abs{\dfrac{\p F}{\p\eta}}\leq C\e \abs{F}$. \\
\ \\
Based on the construction of $\gb$, we know $\lnmp{\dfrac{\p\gb}{\p\phi}}\leq C\e^{-\alpha}$ and $\abs{F(\eta)}\leq C\e$. Hence, we have $\lnmp{A(0,\phi)}\leq C\e^{-\alpha}$ and $\lnmp{\e\dfrac{\p A}{\p\phi}(0,\phi)}\leq C\e^{-\alpha}$. Therefore, applying Theorem \ref{Milne theorem 2} to the equation \eqref{dte 7} and using \eqref{dt 1}, we have
\begin{align}\label{dt 22}
\lnnm{\ue^{K_0\eta}A}\leq&C\bigg(\lnmp{A(0,\phi)}+\lnnm{\ue^{K_0\eta}F(\eta)\cos\phi\dfrac{\p\ub_0}{\p\phi}}\bigg)\leq C\e^{-\alpha}.
\end{align}
Applying Theorem \ref{pt theorem 2} to the equation \eqref{dte 7} and using \eqref{dt 1}, we know
\begin{align}
&\lnnm{\ue^{K_0\eta}\zeta\frac{\p A}{\p\eta}}+\lnnm{\ue^{K_0\eta}F(\eta)\cos\phi\frac{\p A}{\p\phi}}\\
\leq& C\abs{\ln(\e)}^8\Bigg(\lnmp{\e\dfrac{\p A}{\p\phi}(0,\phi)}+\lnnm{\ue^{K_0\eta}\dfrac{\p F}{\p\eta}\cos\phi\frac{\p \ub_0}{\p\phi}}+\lnnm{\ue^{K_0\eta}\zeta\frac{\p}{\p{\eta}}\bigg(\dfrac{\p F}{\p\eta}\cos\phi\frac{\p\ub_0}{\p\phi}\bigg)}\Bigg)\no\\
\leq& C\abs{\ln(\e)}^8\Bigg(\e^{-\alpha}+\e\lnnm{\ue^{K_0\eta}F(\eta)\cos\phi\frac{\p \ub_0}{\p\phi}}+\e\lnnm{\ue^{K_0\eta}F(\eta)\cos\phi\frac{\p A}{\p\phi}}\Bigg)\no\\
\leq& C\abs{\ln(\e)}^8\Bigg(\e^{-\alpha}+\e\lnnm{\ue^{K_0\eta}F(\eta)\cos\phi\frac{\p A}{\p\phi}}\Bigg).\no
\end{align}
Then we may absorb $\e\lnnm{\ue^{K_0\eta}F(\eta)\cos\phi\dfrac{\p A}{\p\phi}}$ into the left-hand side to obtain
\begin{align}\label{dt 23}
&\lnnm{\ue^{K_0\eta}\zeta\frac{\p A}{\p\eta}}+\lnnm{\ue^{K_0\eta}F(\eta)\cos\phi\frac{\p A}{\p\phi}}
\leq C\e^{-\alpha}\abs{\ln(\e)}^8.
\end{align}
\ \\
Step 8: Regularity of Velocity Derivative.\\
The velocity derivative $B=\dfrac{\p\ub_0}{\p\phi}$ satisfies
\begin{align}\label{dte 8}
\left\{
\begin{array}{l}
\sin\phi\dfrac{\p B}{\p\eta}+F(\eta)\cos\phi\dfrac{\p B}{\p\phi}+B=-\cos\phi\dfrac{\p \ub_0}{\p\eta}+F(\eta)\sin\phi\dfrac{\p \ub_0}{\p\phi},\\\rule{0ex}{2em}
B(0,\phi)=\dfrac{\p\gb}{\p\phi}(\phi)\ \ \text{for}\ \
\sin\phi>0,\\\rule{0ex}{2em}
B (L,\phi)=B (L,\rr[\phi]).
\end{array}
\right.
\end{align}
Applying Theorem \ref{Milne theorem 2'} to the equation \eqref{dte 8}, using \eqref{dt 22}, we obtain
\begin{align}
\lnnm{\ue^{K_0\eta}B}\leq&C\bigg(\lnmp{\dfrac{\p\gb}{\p\phi}}+\lnnm{\ue^{K_0\eta}\cos\phi\dfrac{\p \ub_0}{\p\eta}}+\lnnm{\ue^{K_0\eta}F(\eta)\sin\phi\dfrac{\p \ub_0}{\p\phi}}\bigg)\\
\leq&C\bigg(\e^{-\alpha}+\lnnm{\ue^{K_0\eta}A}+\e\lnnm{\ue^{K_0\eta}B}\bigg)\leq C\bigg(\e^{-\alpha}\abs{\ln(\e)}^8+\e\lnnm{\ue^{K_0\eta}B}\bigg).\no
\end{align}
Then we may absorb $\e\lnnm{\ue^{K_0\eta}B}$ into the left-hand side to obtain
\begin{align}\label{dt 24}
\lnnm{\ue^{K_0\eta}B}\leq C\e^{-\alpha}\abs{\ln(\e)}^8.
\end{align}
Then applying Theorem \ref{pt theorem 2'} to the equation \eqref{dte 8}, using \eqref{dt 22}, \eqref{dt 23} and \eqref{dt 24}, we have
\begin{align}
&\lnnm{\ue^{K_0\eta}\zeta\frac{\p B}{\p\eta}}+\lnnm{\ue^{K_0\eta}F(\eta)\cos\phi\frac{\p B}{\p\phi}}\\
\leq&C\abs{\ln(\e)}^8\Bigg(\lnmp{\e\dfrac{\p^2\gb}{\p\phi^2}}+\lnnm{\ue^{K_0\eta}B}+\lnnm{\ue^{K_0\eta}\cos\phi\dfrac{\p \ub_0}{\p\eta}}+\lnnm{\ue^{K_0\eta}F(\eta)\sin\phi\dfrac{\p \ub_0}{\p\phi}}\no\\
&+\lnnm{\ue^{K_0\eta}\zeta\dfrac{\p}{\p\eta}\bigg(\cos\phi\dfrac{\p \ub_0}{\p\eta}\bigg)}+\lnnm{\ue^{K_0\eta}\zeta\dfrac{\p}{\p\eta}\bigg(F(\eta)\sin\phi\dfrac{\p \ub_0}{\p\phi}\bigg)}\Bigg)\no\\
\leq&C\abs{\ln(\e)}^8\Bigg(\e^{-\alpha}+\e^{-\alpha}+\lnnm{\ue^{K_0\eta}A}+\e\lnnm{\ue^{K_0\eta}B}+\lnnm{\ue^{K_0\eta}\zeta\dfrac{\p A}{\p\eta}}+\e\lnnm{\ue^{K_0\eta}\zeta\dfrac{\p B}{\p\eta}}\Bigg)\no\\
\leq&C\abs{\ln(\e)}^8\Bigg(\e^{-\alpha}+\e\lnnm{\ue^{K_0\eta}\zeta\dfrac{\p B}{\p\eta}}\Bigg).\no
\end{align}
Then we may absorb $\e\lnnm{\ue^{K_0\eta}\zeta\dfrac{\p B}{\p\eta}}$ into the left-hand side to obtain
\begin{align}\label{dt 25}
&\lnnm{\ue^{K_0\eta}\zeta\frac{\p B}{\p\eta}}+\lnnm{\ue^{K_0\eta}F(\eta)\cos\phi\frac{\p B}{\p\phi}}
\leq C\e^{-\alpha}\abs{\ln(\e)}^8.
\end{align}
\ \\
Step 9: Regularity of Mixed Derivative.\\
The weighted mixed derivative $C=\zeta\dfrac{\p A}{\p\phi}$ satisfies
\begin{align}\label{dte 8'}
\left\{
\begin{array}{l}
\sin\phi\dfrac{\p C}{\p\eta}+F(\eta)\cos\phi\dfrac{\p C}{\p\phi}+C=\zeta\dfrac{\p F}{\p\eta}\cos\phi\dfrac{\p B}{\p\phi}-\zeta\dfrac{\p F}{\p\eta}\sin\phi B
-\zeta\cos\phi\dfrac{\p A}{\p\eta}+\zeta F(\eta)\sin\phi\dfrac{\p A}{\p\phi},\\\rule{0ex}{2em}
C(0,\phi)=\zeta\dfrac{\p}{\p\phi}\Bigg(\dfrac{1}{\sin\phi}\bigg(F(\eta)\cos\phi\dfrac{\p\gb}{\p\phi}(\phi)-\gb(0,\phi)+\bar\ub_0(0,\phi)\bigg)\Bigg)\ \ \text{for}\ \
\sin\phi>0,\\\rule{0ex}{2em}
C (L,\phi)=C (L,\rr[\phi]).
\end{array}
\right.
\end{align}
Applying Theorem \ref{Milne lemma 2'} to the equation \eqref{dte 8'}, using \eqref{dt 23}, \eqref{dt 24} and \eqref{dt 25}, we have
\begin{align}
\lnnm{\ue^{K_0\eta}C}\leq&C\bigg(\e^{-\alpha}+\e\lnnm{\ue^{K_0\eta} F(\eta)\cos\phi\dfrac{\p B}{\p\phi}}+\e^2\lnnm{\ue^{K_0\eta} B}\\
&+\lnnm{\ue^{K_0\eta}\zeta\dfrac{\p A}{\p\eta}}+\e\lnnm{\ue^{K_0\eta}C}\bigg)\no\\
\leq&C\bigg(\e^{-\alpha}\abs{\ln(\e)}^8+\e\lnnm{\ue^{K_0\eta}C}\bigg)\no.
\end{align}
Then we may absorb $\e\lnnm{\ue^{K_0\eta}C}$ into the left-hand side to obtain
\begin{align}\label{dt 26}
\lnnm{\ue^{K_0\eta}C}\leq&\e^{-\alpha}\abs{\ln(\e)}^8.
\end{align}
\ \\
Step 10: Regularity of Tangential and Velocity Derivative.\\
We turn to the regularity of $W_i$. Applying Theorem \ref{pt theorem 2} to the equation \eqref{dte 2}, using \eqref{dt 1} and \eqref{dt 26}, noting $\abs{F}\leq C\e$, we have
\begin{align}\label{dt 27}
&\lnnm{\ue^{K_0\eta}\zeta\frac{\p W_i}{\p\eta}}+\lnnm{\ue^{K_0\eta}F(\eta)\cos\phi\frac{\p W_i}{\p\phi}}\\
\leq& C\abs{\ln(\e)}^8\Bigg(\lnmp{\e\dfrac{\p^2\gb}{\p\iota_i\p\phi}}+\lnnm{\ue^{K_0\eta}\dfrac{\p F}{\p\iota_i}\cos\phi\dfrac{\p \ub_0}{\p\phi}}+\lnnm{\ue^{K_0\eta}\zeta\frac{\p}{\p{\eta}}\bigg(\dfrac{\p F}{\p\iota_i}\cos\phi\dfrac{\p \ub_0}{\p\phi}\bigg)}\Bigg)\no\\
\leq& C\abs{\ln(\e)}^8\Bigg(1+\lnnm{\ue^{K_0\eta}F(\eta)\cos\phi\dfrac{\p \ub_0}{\p\phi}}+\lnnm{\ue^{K_0\eta}F(\eta)C}\Bigg)\no\\
\leq&C\abs{\ln(\e)}^{16}\bigg(1+\e^{1-\alpha}\bigg)\leq C\abs{\ln(\e)}^{16}.\no
\end{align}
A similar argument holds for the regularity of $H$. Applying Theorem \ref{pt theorem 2'} to the equation \eqref{dte 3}, we have
\begin{align}\label{dt 28}
&\lnnm{\ue^{K_0\eta}\zeta\frac{\p H}{\p\eta}}+\lnnm{\ue^{K_0\eta}F(\eta)\cos\phi\frac{\p H}{\p\phi}}\leq C\abs{\ln(\e)}^{16}.
\end{align}
\ \\
Summary.\\
Inserting \eqref{dt 27} and \eqref{dt 28} into \eqref{dt 29}, we obtain
\begin{align}
&\lnnm{\ue^{K_0\eta}V_j}+\lnnm{\ue^{K_0\eta}M}\leq C\abs{\ln(\e)}^{16}.
\end{align}
This is actually
\begin{align}\label{dt 30}
&\lnnm{\ue^{K_0\eta}\frac{\p\ub_1}{\p\iota_1}}+\lnnm{\ue^{K_0\eta}\frac{\p\ub_1}{\p\iota_2}}+\lnnm{\ue^{K_0\eta}\frac{\p\ub_1}{\p\psi}}+\leq C\abs{\ln(\e)}^{16}.
\end{align}
\begin{theorem}\label{dt theorem 1}
For $K_0>0$ sufficiently small, the regular boundary layer satisfies
\begin{align}
\begin{array}{ll}
\lnnm{\ue^{K_0\eta}\ub_0}\leq C,& \lnnm{\ue^{K_0\eta}\ub_1}\leq C\abs{\ln(\e)}^8,\\\rule{0ex}{2em}
\lnnm{\ue^{K_0\eta}\dfrac{\p\ub_0}{\p\iota_i}}\leq C\abs{\ln(\e)}^8,&\lnnm{\ue^{K_0\eta}\dfrac{\p\ub_1}{\p\iota_i}}\leq C\abs{\ln(\e)}^{16},\\\rule{0ex}{2em}
\lnnm{\ue^{K_0\eta}\dfrac{\p\ub_0}{\p\psi}}\leq C\abs{\ln(\e)}^8,&\lnnm{\ue^{K_0\eta}\dfrac{\p\ub_1}{\p\psi}}\leq C\abs{\ln(\e)}^{16}.
\end{array}
\end{align}
\end{theorem}

\subsection{Analysis of Singular Boundary Layer}

In this subsection, we will justify that the singular boundary layers are all well-defined.\\
\ \\
Step 1: Well-Posedness of $\uf_0$.\\
$\uf_0$ satisfies the $\e$-Milne problem with geometric correction
\begin{align}
\left\{
\begin{array}{l}
\sin\phi\dfrac{\p \uf_0 }{\p\eta}+F(\eta)\cos\phi\dfrac{\p
\uf_0 }{\p\phi}+\uf_0 -\buf_0 =0,\\\rule{0ex}{1.5em}
\uf_0 (0,\phi)=\gf(\phi)-\mathfrak{F} _{0,L}\ \ \text{for}\ \
\sin\phi>0,\\\rule{0ex}{1.5em}
\uf_0 (L,\phi)=\uf_0 (L,\rr[\phi]).
\end{array}
\right.
\end{align}
Therefore, by Theorem \ref{Milne theorem 2}, we know
\begin{align}
\lnnm{\ue^{K_0\eta}\uf_0}\leq C.
\end{align}
However, this is not sufficient for future use and we need more detailed analysis. We will divide the domain $(\eta,\phi)\in[0,L]\times\left[-\dfrac{\pi}{2},\dfrac{\pi}{2}\right]$ into two regions:
\begin{itemize}
\item
Region I $\chi_1$: $0\leq\zeta<2\e^{\alpha}$.
\item
Region II $\chi_2$: $2\e^{\alpha}\leq\zeta\leq1$.
\end{itemize}
Here we use $\chi_i$ to represent either the corresponding region or the indicator function. It is easy to see that $\gf=0$ in Region II. Similarly we decompose the solution $\uf_0=\chi_1\uf_0+\chi_2\uf_0=f^{(1)}+f^{(2)}$ in these two regions. In the following, the estimates for $f_i$ will be restricted to the region $\chi_i$ for $i=1,2$. Using Theorem \ref{Milne theorem 2}, we can easily show that
\begin{align}
\tnnm{\ue^{K_0\eta}\uf_0}\leq&C\e^{\alpha}.
\end{align}
The key to $L^{\infty}$ estimates in Theorem \ref{Milne theorem 2} is Lemma \ref{Milne lemma 3}. Its proof is basically tracking along the characteristics. Hence, we know
\begin{align}
\lnnm{\ue^{K_0\eta}\buf_0}\leq& C\bigg(\e^{\alpha}\ltnm{\ue^{K_0\eta}f^{(1)}}+\ltnm{\ue^{K_0\eta}f^{(2)}}\bigg)\\
\leq&C\bigg(\tnnm{\ue^{K_0\eta}\uf_0}+\d\e^{\alpha}\lnnm{\ue^{K_0\eta}f^{(1)}}+\d\lnnm{\ue^{K_0\eta}f^{(2)}}\bigg).\no
\end{align}
Thus, considering $\chi_1\gf=\gf$ and $\chi_2\gf=0$, we may directly obtain
\begin{align}
\lnnm{\ue^{K_0\eta}f^{(1)}}\leq&C\bigg(\lnm{\chi_1\gf}+\lnnm{\ue^{K_0\eta}\buf_0}\bigg)\\
\leq&C\bigg(\lnm{\chi_1\gf}+\tnnm{\ue^{K_0\eta}\uf_0}+\d\e^{\alpha}\lnnm{\ue^{K_0\eta}f^{(1)}}+\d\lnnm{\ue^{K_0\eta}f^{(2)}}\bigg)\no\\
\leq&C\bigg(1+\d\e^{\alpha}\lnnm{\ue^{K_0\eta}f^{(1)}}+\d\lnnm{\ue^{K_0\eta}f^{(2)}}\bigg),\no\\
\lnnm{\ue^{K_0\eta}f^{(2)}}\leq&C\bigg(\lnm{\chi_2\gf}+\lnnm{\ue^{K_0\eta}\buf_0}\bigg)\\
\leq&C\bigg(\lnm{\chi_2\gf}+\tnnm{\ue^{K_0\eta}\uf_0}+\d\e^{\alpha}\lnnm{\ue^{K_0\eta}f^{(1)}}+\d\lnnm{\ue^{K_0\eta}f^{(2)}}\bigg)\no\\
\leq&C\bigg(\e^{\alpha}+\d\e^{\alpha}\lnnm{\ue^{K_0\eta}f^{(1)}}+\d\lnnm{\ue^{K_0\eta}f^{(2)}}\bigg).\no
\end{align}
Letting $\d$ small, absorbing $\lnnm{\ue^{K_0\eta}f^{(1)}}$ and $\lnnm{\ue^{K_0\eta}f^{(2)}}$, we know
\begin{align}
\lnnm{\ue^{K_0\eta}f^{(1)}}\leq&C\bigg(1+\d\lnnm{\ue^{K_0\eta}f^{(2)}}\bigg),\\
\lnnm{\ue^{K_0\eta}f^{(2)}}\leq&C\bigg(\e^{\alpha}+\d\e^{\alpha}\lnnm{\ue^{K_0\eta}f^{(1)}}\bigg).
\end{align}
Combining them together, we can easily see that
\begin{align}
\lnnm{\ue^{K_0\eta}f^{(1)}}\leq&C,\\
\lnnm{\ue^{K_0\eta}f^{(2)}}\leq&C\e^{\alpha}.
\end{align}
In total, we can derive
\begin{align}
\lnnm{\ue^{K_0\eta}\buf_0}\leq C\e^{\alpha}.
\end{align}
\ \\
Step 2: Regularity of $\uf_0$.\\
This is very similar to the well-posedness proof, we will also consider the regularity of $\uf_0$ in two regions. Note that in the proof of Theorem \ref{pt theorem 2}, the $L^{\infty}$ estimates relies on two kinds of quantities:
\begin{itemize}
\item
$\abs{\zeta\dfrac{\p\uf_0}{\p\eta}}$ on the same characteristics.
\item
$\ds\int_{-\pi}^{\pi}\zeta\dfrac{\p\uf_0}{\p\eta}\ud{\phi}$ for some $\eta>0$.
\end{itemize}
Correspondingly, we may handle them separately: for the first case, since $\zeta$ is preserved along the characteristics, we can directly separate the estimate of $f^{(1)}$ and $f^{(2)}$; for the second case, we may use the simple domain decomposition
\begin{align}
\int_{-\pi}^{\pi}\zeta\dfrac{\p\uf_0}{\p\eta}(\eta,\phi)\ud{\phi}=&\int_{\chi_1}\zeta\dfrac{\p f^{(1)}}{\p\eta}\ud{\phi}+\int_{\chi_2}\zeta\dfrac{\p f^{(2)}}{\p\eta}\ud{\phi}\leq C\bigg(\e^{\alpha}\ltnm{\zeta\dfrac{\p f^{(1)}}{\p\eta}}+\ltnm{\zeta\dfrac{\p f^{(2)}}{\p\eta}}\bigg).
\end{align}
Then following a similar absorbing argument as in above well-posedness proof, we have
\begin{align}
&\lnnm{\ue^{K_0\eta}\zeta\frac{\p f^{(1)}}{\p\eta}}+\lnnm{\ue^{K_0\eta}F(\eta)\cos\phi\frac{\p f^{(1)}}{\p\phi}}\\
\leq& C\abs{\ln(\e)}^8\bigg(\lnm{\gf}+\lnm{\e\frac{\p\gf}{\p\phi}}+\lnnm{\ue^{K_0\eta}\uf_0}\bigg)\leq C\abs{\ln(\e)}^8,\no\\
&\lnnm{\ue^{K_0\eta}\zeta\frac{\p f^{(2)}}{\p\eta}}+\lnnm{\ue^{K_0\eta}F(\eta)\cos\phi\frac{\p f^{(2)}}{\p\phi}}\\
\leq& C\abs{\ln(\e)}^8\bigg(\lnnm{\ue^{K_0\eta}f^{(2)}}+\e^{\alpha}\lnnm{\ue^{K_0\eta}f^{(1)}}\bigg) \leq C\e^{\alpha}\abs{\ln(\e)}^8.\no
\end{align}
Note that although $\lnm{\dfrac{\p\gf}{\p\phi}}\leq C\e^{-\alpha}$, with the help of $\e$, we can get rid of this negative power. \\
\ \\
Step 3: Tangential and Velocity Derivatives of $\uf_0$.\\
The $\iota_i$ derivative $P_i=\dfrac{\p\uf_0}{\p\iota_i}$ satisfies
\begin{align}
\left\{
\begin{array}{l}
\sin\phi\dfrac{\p P_i}{\p\eta}+F(\eta)\cos\phi\dfrac{\p P_i}{\p\phi}+P_i-\bar P_i=-\dfrac{\p F}{\p\iota_i}\cos\phi\dfrac{\p \uf_0}{\p\phi},\\\rule{0ex}{2.0em}
P_i (0,\phi)=\dfrac{\p\gf}{\p\iota_i}(\phi)-\dfrac{\p\mathfrak{F}_{0,L}}{\p\iota_i}\ \ \text{for}\ \
\sin\phi>0,\\\rule{0ex}{2.0em}
P_i (L,\phi)=P_i (L,\rr[\phi]).
\end{array}
\right.
\end{align}
It is easy to check that
\begin{align}
\int_{-\pi}^{\pi}\cos\phi\dfrac{\p \uf_0}{\p\phi}\ud{\phi}=\int_{-\pi}^{\pi}\uf_0\sin\phi\ud{\phi}=0,
\end{align}
due to the orthogonal property. Hence, using Theorem \ref{Milne theorem 2} with $S_Q=0$, we have
\begin{align}
\tnnm{\ue^{K_0\eta}P_i}\leq&C\e^{\alpha}\abs{\ln(\e)}^8,
\end{align}
which further implies
\begin{align}
\lnnm{\ue^{K_0\eta}P_i^{(1)}}\leq&C\bigg(\lnnm{\frac{\p\gf}{\p\iota_i}}+\tnnm{\ue^{K_0\eta}P_i}+\lnnm{\ue^{K_0\eta}F(\eta)\cos\phi\frac{\p \uf_0}{\p\phi}}\bigg)\\
\leq& C\abs{\ln(\e)}^8,\no\\
\\
\lnnm{\ue^{K_0\eta}P_i^{(2)}}\leq&C\bigg(\ue^{K_0\eta}\tnnm{P_i}+\e^{\alpha}\lnnm{\ue^{K_0\eta}F(\eta)\cos\phi\frac{\p f^{(1)}}{\p\phi}}+\lnnm{\ue^{K_0\eta}F(\eta)\cos\phi\frac{\p f^{(2)}}{\p\phi}}\bigg)\no\\
\leq& C\e^{\alpha}\abs{\ln(\e)}^8,\no
\end{align}
where $P_i^{(1)}=\dfrac{\p f^{(1)}}{\p\iota_i}$ and $P_i^{(2)}=\dfrac{\p f^{(2)}}{\p\iota_i}$.\\
\ \\
A similar argument justifies similar results for velocity derivative $Q^{(1)}=\dfrac{\p f^{(1)}}{\p\psi}$ and $Q^{(2)}=\dfrac{\p f^{(2)}}{\p\psi}$:
\begin{align}
\lnnm{\ue^{K_0\eta}Q^{(1)}}\leq&C\abs{\ln(\e)}^8,\\
\lnnm{\ue^{K_0\eta}Q^{(2)}}\leq& C\e^{\alpha}\abs{\ln(\e)}^8.
\end{align}
\ \\
\begin{theorem}\label{dt theorem 2}
Let
\begin{align}
\left\{
\begin{array}{ll}
\chi_1: &0\leq\zeta<2\e^{\alpha},\\
\chi_2: &2\e^{\alpha}\leq\zeta\leq1.
\end{array}
\right.
\end{align}
For $K_0>0$ sufficiently small, the singular boundary layer satisfies
\begin{align}
\begin{array}{ll}
\lnnm{\ue^{K_0\eta}(\chi_1\uf_0)}\leq C,& \lnnm{\ue^{K_0\eta}(\chi_2\uf_0)}\leq C\e^{\alpha},\\\rule{0ex}{2em}
\lnnm{\ue^{K_0\eta}\dfrac{\p(\chi_1\uf_0)}{\p\iota_i}}\leq C\abs{\ln(\e)}^8,&\lnnm{\ue^{K_0\eta}\dfrac{\p(\chi_2\uf_0)}{\p\iota_i}}\leq C\e^{\alpha}\abs{\ln(\e)}^{8},\\\rule{0ex}{2em}
\lnnm{\ue^{K_0\eta}\dfrac{\p(\chi_1\uf_0)}{\p\psi}}\leq C\abs{\ln(\e)}^8,&\lnnm{\ue^{K_0\eta}\dfrac{\p(\chi_2\uf_0)}{\p\psi}}\leq C\e^{\alpha}\abs{\ln(\e)}^{8}.
\end{array}
\end{align}
\end{theorem}

\subsection{Analysis of Interior Solution}

In this subsection, we will justify that the interior solutions are all well-defined. \\
\ \\
Step 1: Well-Posedness of $\u_0$.\\
$\u_0$ satisfies an elliptic equation
\begin{align}
\left\{
\begin{array}{l}
\u_0(\vx,\vw)=\bu_0(\vx) ,\\\rule{0ex}{1.5em} \Delta_x\bu_0(\vx)=0\ \ \text{in}\
\ \Omega,\\\rule{0ex}{1.5em}
\bu_0(\vx_0)=\mathscr{F}_{0,L}(\iota_1,\iota_2)+\mathfrak{F}_{0,L}(\iota_1,\iota_2)\ \ \text{on}\ \
\p\Omega.
\end{array}
\right.
\end{align}
Based on standard elliptic theory, we have
\begin{align}
\nm{\u_0}_{H^3(\Omega)}\leq C\bigg(\nm{\mathscr{F}_{0,L}}_{H^{\frac{5}{2}}(\p\Omega)}+\nm{\mathfrak{F}_{0,L}}_{H^{\frac{5}{2}}(\p\Omega)}\bigg)\leq C.
\end{align}
\ \\
Step 2: Well-Posedness of $\u_1$.\\
$\u_1$ satisfies an elliptic equation
\begin{align}
\left\{
\begin{array}{rcl}
\u_1(\vx,\vw)&=&\bu_1(\vx)-\vw\cdot\nx\u_0(\vx,\vw),\\\rule{0ex}{1.5em}
\Delta_x\bu_1(\vx)&=&-\displaystyle\int_{\s^2}\Big(\vw\cdot\nx\u_{0}(\vx,\vw)\Big)\ud{\vw}\
\ \text{in}\ \ \Omega,\\\rule{0ex}{1.5em} \bu_1(\vx_0)&=&f _{1,L}(\iota_1,\iota_2)\ \ \text{on}\ \
\p\Omega.
\end{array}
\right.
\end{align}
Based on standard elliptic theory, we have
\begin{align}
\nm{\u_1}_{H^3(\Omega)}\leq C\bigg(\nm{\mathscr{F}_{1,L}}_{H^{\frac{5}{2}}(\p\Omega)}+\nm{\u_0}_{H^{2}(\Omega)}\bigg)\leq C\abs{\ln(\e)}^8.
\end{align}
\ \\
Step 3: Well-Posedness of $\u_2$.\\
$\u_2$ satisfies an elliptic equation
\begin{align}
\left\{
\begin{array}{rcl}
\u_{2}(\vx,\vw)&=&\bu_{2}(\vx)-\vw\cdot\nx\u_{1}(\vx,\vw),\\\rule{0ex}{1.5em}
\Delta_x\bu_{2}(\vx)&=&-\displaystyle\int_{\s^2}\Big(\vw\cdot\nx\u_{1}(\vx,\vw)\Big)\ud{\vw}\
\ \text{in}\ \ \Omega,\\\rule{0ex}{1.5em} \bu_2(\vx_0)&=&0\ \ \text{on}\ \
\p\Omega.
\end{array}
\right.
\end{align}
Based on standard elliptic theory, we have
\begin{align}
\nm{\u_2}_{H^3(\Omega)}\leq C\bigg(\nm{\bu_0}_{H^{3}(\Omega)}+\nm{\bu_1}_{H^{2}(\Omega)}\bigg)\leq C\abs{\ln(\e)}^8.
\end{align}

\begin{theorem}\label{dt theorem 3}
The interior solution satisfies
\begin{align}
\nm{\u_0}_{H^3(\Omega)}\leq C,\quad \nm{\u_1}_{H^3(\Omega)}\leq C\abs{\ln(\e)}^8,\quad \nm{\u_2}_{H^3(\Omega)}\leq C\abs{\ln(\e)}^8.
\end{align}
\end{theorem}

\subsection{Proof of Main Theorem}

\begin{theorem}\label{diffusive limit}
Assume $g(\vx_0,\vw)\in C^3(\Gamma^-)$. Then for the steady neutron
transport equation \eqref{transport}, there exists a unique solution
$u^{\e}(\vx,\vw)\in L^{\infty}(\Omega\times\s^2)$. Moreover, for any $0<\d<<1$, the solution obeys the estimate
\begin{align}
\im{u^{\e}-\u-\uu}{\Omega\times\s^2}\leq C(\d)\e^{\frac{1}{3}-\d},
\end{align}
where $\u(\vx)$ satisfies the Laplace equation with Dirichlet boundary condition
\begin{align}
\left\{
\begin{array}{l}
\Delta_x\u(\vx)=0\ \ \text{in}\
\ \Omega,\\\rule{0ex}{1.5em}
\u(\vx_0)=D(\vx_0)\ \ \text{on}\ \
\p\Omega,
\end{array}
\right.
\end{align}
and $\uu(\eta,\iota_1,\iota_2,\phi,\psi)$ satisfies the $\e$-Milne problem with geometric correction
\begin{align}
\left\{
\begin{array}{l}
\sin\phi\dfrac{\p \uu }{\p\eta}+F(\e;\eta,\psi,\iota_1,\iota_2)\cos\phi\dfrac{\p
\uu }{\p\phi}+\uu -\buu =0,\\\rule{0ex}{1.5em}
\uu (0,\iota_1,\iota_2,\phi,\psi)=g(\iota_1,\iota_2,\phi,\psi)-D(\iota_1,\iota_2)\ \ \text{for}\ \
\sin\phi>0,\\\rule{0ex}{1.5em}
\uu (L,\iota_1,\iota_2,\phi,\psi)=\uu (L,\iota_1,\iota_2,\rr[\phi],\psi),
\end{array}
\right.
\end{align}
for $L=\e^{-n}$ with $0<n<\dfrac{1}{2}$, $\rr[\phi]=-\phi$, $\eta$ the rescaled normal variable, $(\iota_1,\iota_2)$ the tangential variables, and $(\phi,\psi)$ the velocity variables.
\end{theorem}
\begin{proof}
Based on Theorem \ref{LI estimate}, we know there exists a unique $u^{\e}(\vx,\vw)\in L^{\infty}(\Omega\times\s^2)$, so we focus on the diffusive limit. \\
\ \\
Step 1: Remainder definitions.\\
We define the remainder as
\begin{align}\label{pf 1_}
R=&u^{\e}-\sum_{k=0}^{2}\e^k\u_k-\sum_{k=0}^{1}\e^k\ub_k-\uf_0=u^{\e}-\q-\qb-\qf,
\end{align}
where
\begin{align}
\q=&\u_0+\e\u_1+\e^2\u_2,\\
\qb=&\ub_0+\e\ub_1,\\
\qf=&\uf_0.
\end{align}
Noting the equation \eqref{coordinate 22} is equivalent to the
equation \eqref{transport}, we write $\ll$ to denote the neutron
transport operator as follows:
\begin{align}
\ll[u]=&\e\vw\cdot\nx u+ u-\bar u\\
=&\sin\phi\dfrac{\p u}{\p\eta}-\e\bigg(\dfrac{\sin^2\psi}{R_1-\e\eta}+\dfrac{\cos^2\psi}{R_2-\e\eta}\bigg)\cos\phi\dfrac{\p u}{\p\phi}\no\\\rule{0ex}{2.0em}
&+\e\bigg(\dfrac{\cos\phi\sin\psi}{P_1(1-\e\kk_1\eta)}\dfrac{\p u}{\p\iota_1}+\dfrac{\cos\phi\cos\psi}{P_2(1-\e\kk_2\eta)}\dfrac{\p u}{\p\iota_2}\bigg)\no\\\rule{0ex}{2.0em}
&+\e\Bigg(\dfrac{\sin\psi}{1-\e\kk_1\eta}\bigg(\cos\phi\Big(\vt_1\cdot\Big(\vt_2\times(\p_{12}\vr\times\vt_2)\Big)\Big)
-\kk_1P_1P_2\sin\phi\cos\psi\bigg)\no\\\rule{0ex}{2.0em}
&+\dfrac{\cos\psi}{1-\e\kk_2\eta}\bigg(-\cos\phi\Big(\vt_2\cdot\Big(\vt_1\times(\p_{12}\vr\times\vt_1)\Big)\Big)
+\kk_2P_1P_2\sin\phi\sin\psi\bigg)\Bigg)\dfrac{1}{P_1P_2}\dfrac{\p u}{\p\psi}+u-\bar u.\no
\end{align}
\ \\
Step 2: Estimates of $\ll[\q]$.\\
The interior contribution can be estimated as
\begin{align}
\ll[\q]=\e\vw\cdot\nx \q+ \q-\bar
\q=&\e^{3}\vw\cdot\nx \u_2.
\end{align}
Based on Theorem \ref{dt theorem 3}, we have
\begin{align}
\im{\ll[\q]}{\Omega\times\s^2}\leq&\im{\e^{3}\vw\cdot\nx \u_2}{\Omega\times\s^2}\leq C\e^{3}\im{\nx\u_2}{\Omega\times\s^2}\leq
C\e^{3}\abs{\ln(\e)}^8.
\end{align}
This implies
\begin{align}
\tm{\ll[\q]}{\Omega\times\s^2}\leq& C\e^{3}\abs{\ln(\e)}^8,\\
\nm{\ll[\q]}_{L^{\frac{2m}{2m-1}}(\Omega\times\s^2)}\leq& C\e^{3}\abs{\ln(\e)}^8,\\
\im{\ll[\q]}{\Omega\times\s^2}\leq& C\e^{3}\abs{\ln(\e)}^8.
\end{align}
\ \\
Step 3: Estimates of $\ll \qb$.\\
We need to estimate $\ub_0+\e\ub_1$. The boundary layer contribution can be
estimated as
\begin{align}
\ll[\ub_0+\e\ub_1]=&\sin\phi\dfrac{\p(\ub_0+\e\ub_1)}{\p\eta}-\e\bigg(\dfrac{\sin^2\psi}{R_1-\e\eta}+\dfrac{\cos^2\psi}{R_2-\e\eta}\bigg)\cos\phi\dfrac{\p (\ub_0+\e\ub_1)}{\p\phi}\no\\\rule{0ex}{2.0em}
&+\e\bigg(\dfrac{\cos\phi\sin\psi}{P_1(1-\e\kk_1\eta)}\dfrac{\p (\ub_0+\e\ub_1)}{\p\iota_1}+\dfrac{\cos\phi\cos\psi}{P_2(1-\e\kk_2\eta)}\dfrac{\p (\ub_0+\e\ub_1)}{\p\iota_2}\bigg)\no\\\rule{0ex}{2.0em}
&+\e\Bigg(\dfrac{\sin\psi}{1-\e\kk_1\eta}\bigg(\cos\phi\Big(\vt_1\cdot\Big(\vt_2\times(\p_{12}\vr\times\vt_2)\Big)\Big)
-\kk_1P_1P_2\sin\phi\cos\psi\bigg)\no\\\rule{0ex}{2.0em}
&+\dfrac{\cos\psi}{1-\e\kk_2\eta}\bigg(-\cos\phi\Big(\vt_2\cdot\Big(\vt_1\times(\p_{12}\vr\times\vt_1)\Big)\Big)
+\kk_2P_1P_2\sin\phi\sin\psi\bigg)\Bigg)\dfrac{1}{P_1P_2}\dfrac{\p (\ub_0+\e\ub_1)}{\p\psi}\no\\
&+(\ub_0+\e\ub_1)-(\bub_0+\e\bub_1)\no\\
=&\e^2\bigg(\dfrac{\cos\phi\sin\psi}{P_1(1-\e\kk_1\eta)}\dfrac{\p \ub_1}{\p\iota_1}+\dfrac{\cos\phi\cos\psi}{P_2(1-\e\kk_2\eta)}\dfrac{\p \ub_1}{\p\iota_2}\bigg)\no\\\rule{0ex}{2.0em}
&+\e^2\Bigg(\dfrac{\sin\psi}{1-\e\kk_1\eta}\bigg(\cos\phi\Big(\vt_1\cdot\Big(\vt_2\times(\p_{12}\vr\times\vt_2)\Big)\Big)
-\kk_1P_1P_2\sin\phi\cos\psi\bigg)\no\\\rule{0ex}{2.0em}
&+\dfrac{\cos\psi}{1-\e\kk_2\eta}\bigg(-\cos\phi\Big(\vt_2\cdot\Big(\vt_1\times(\p_{12}\vr\times\vt_1)\Big)\Big)
+\kk_2P_1P_2\sin\phi\sin\psi\bigg)\Bigg)\dfrac{1}{P_1P_2}\dfrac{\p\ub_1}{\p\psi}\no.
\end{align}
Based on Theorem \ref{dt theorem 1}, we have
\begin{align}
&\im{\e^2\bigg(\dfrac{\cos\phi\sin\psi}{P_1(1-\e\kk_1\eta)}\dfrac{\p \ub_1}{\p\iota_1}+\dfrac{\cos\phi\cos\psi}{P_2(1-\e\kk_2\eta)}\dfrac{\p \ub_1}{\p\iota_2}\bigg)}{\Omega\times\s^2}\\
\leq&C\e^2\bigg(\im{\frac{\p
\ub_1}{\p\iota_1}}{\Omega\times\s^2}+\im{\frac{\p
\ub_1}{\p\iota_2}}{\Omega\times\s^2}\bigg)\leq C\e^{2}\abs{\ln(\e)}^8.\no
\end{align}
Similarly, we can show that
\begin{align}
&\Bigg\Vert\e^2\Bigg(\dfrac{\sin\psi}{1-\e\kk_1\eta}\bigg(\cos\phi\Big(\vt_1\cdot\Big(\vt_2\times(\p_{12}\vr\times\vt_2)\Big)\Big)
-\kk_1P_1P_2\sin\phi\cos\psi\bigg)\\\rule{0ex}{2.0em}
&+\dfrac{\cos\psi}{1-\e\kk_2\eta}\bigg(-\cos\phi\Big(\vt_2\cdot\Big(\vt_1\times(\p_{12}\vr\times\vt_1)\Big)\Big)
+\kk_2P_1P_2\sin\phi\sin\psi\bigg)\Bigg)\dfrac{1}{P_1P_2}\dfrac{\p\ub_1}{\p\psi}\Bigg\Vert_{L^{\infty}(\Omega\times\s^2)}\no\\
\leq&C\e^2\im{\frac{\p
\ub_1}{\p\psi}}{\Omega\times\s^2}\leq C\e^{2}\abs{\ln(\e)}^8.\no
\end{align}
Also, the exponential decay of $\dfrac{\p\ub_1}{\p\iota_i}$ in Theorem \ref{dt theorem 1} and the rescaling $\eta=\dfrac{\mu}{\e}$ implies
\begin{align}
\\
&\tm{\e^2\bigg(\dfrac{\cos\phi\sin\psi}{P_1(1-\e\kk_1\eta)}\dfrac{\p \ub_1}{\p\iota_1}+\dfrac{\cos\phi\cos\psi}{P_2(1-\e\kk_2\eta)}\dfrac{\p \ub_1}{\p\iota_2}\bigg)}{\Omega\times\s^2}
\leq C\e^2\bigg(\tm{\frac{\p
\ub_1}{\p\iota_1}}{\Omega\times\s^2}+\tm{\frac{\p
\ub_1}{\p\iota_2}}{\Omega\times\s^2}\bigg)\no\\
\leq&\e^2\Bigg(\int_0^{R_{\min}}(R_{\min}-\mu)\bigg(\lnm{\frac{\p\ub_1}{\p\iota_1}(\mu)}^2
+\lnm{\frac{\p\ub_1}{\p\iota_2}(\mu)}^2\bigg)\ud{\mu}\Bigg)^{\frac{1}{2}}\no\\
\leq&\e^{\frac{5}{2}}\Bigg(\int_0^{\frac{R_{\min}}{\e}}(R_{\min}-\e\eta)
\bigg(\lnm{\frac{\p\ub_1}{\p\iota_1}(\eta)}^2
+\lnm{\frac{\p\ub_1}{\p\iota_2}(\eta)}^2\bigg)\ud{\eta}\Bigg)^{\frac{1}{2}}\no\\
\leq&C\e^{\frac{5}{2}}\abs{\ln(\e)}^8\Bigg(\int_0^{\infty}\ue^{-2K_0\eta}\ud{\eta}\Bigg)^{\frac{1}{2}}\no\\
\leq& C\e^{\frac{5}{2}}\abs{\ln(\e)}^8.\no
\end{align}
Similarly, we have
\begin{align}
\nm{\e^2\bigg(\dfrac{\cos\phi\sin\psi}{P_1(1-\e\kk_1\eta)}\dfrac{\p \ub_1}{\p\iota_1}+\dfrac{\cos\phi\cos\psi}{P_2(1-\e\kk_2\eta)}\dfrac{\p \ub_1}{\p\iota_2}\bigg)}_{L^{\frac{2m}{2m-1}}(\Omega\times\s^2)}\leq&C\e^{3-\frac{1}{2m}}\abs{\ln(\e)}^8.
\end{align}
Using similar arguments, we may justify
\begin{align}
&\Bigg\Vert\e^2\Bigg(\dfrac{\sin\psi}{1-\e\kk_1\eta}\bigg(\cos\phi\Big(\vt_1\cdot\Big(\vt_2\times(\p_{12}\vr\times\vt_2)\Big)\Big)
-\kk_1P_1P_2\sin\phi\cos\psi\bigg)\\\rule{0ex}{2.0em}
&+\dfrac{\cos\psi}{1-\e\kk_2\eta}\bigg(-\cos\phi\Big(\vt_2\cdot\Big(\vt_1\times(\p_{12}\vr\times\vt_1)\Big)\Big)
+\kk_2P_1P_2\sin\phi\sin\psi\bigg)\Bigg)\dfrac{1}{P_1P_2}\dfrac{\p\ub_1}{\p\psi}\Bigg\Vert_{L^{2}(\Omega\times\s^2)}
\leq C\e^{\frac{5}{2}}\abs{\ln(\e)}^8,\no
\end{align}
and
\begin{align}
&\Bigg\Vert\e^2\Bigg(\dfrac{\sin\psi}{1-\e\kk_1\eta}\bigg(\cos\phi\Big(\vt_1\cdot\Big(\vt_2\times(\p_{12}\vr\times\vt_2)\Big)\Big)
-\kk_1P_1P_2\sin\phi\cos\psi\bigg)\\\rule{0ex}{2.0em}
&+\dfrac{\cos\psi}{1-\e\kk_2\eta}\bigg(-\cos\phi\Big(\vt_2\cdot\Big(\vt_1\times(\p_{12}\vr\times\vt_1)\Big)\Big)
+\kk_2P_1P_2\sin\phi\sin\psi\bigg)\Bigg)\dfrac{1}{P_1P_2}\dfrac{\p\ub_1}{\p\psi}\Bigg\Vert_{L^{\frac{2m}{2m-1}}(\Omega\times\s^2)}
\leq C\e^{3-\frac{1}{2m}}\abs{\ln(\e)}^8.\no
\end{align}
In total, we have
\begin{align}
\tm{\ll[\qb]}{\Omega\times\s^2}\leq& C\e^{\frac{5}{2}}\abs{\ln(\e)}^8,\\
\nm{\ll[\qb]}_{L^{\frac{2m}{2m-1}}(\Omega\times\s^2)}\leq& C\e^{3-\frac{1}{2m}}\abs{\ln(\e)}^8,\\
\im{\ll[\qb]}{\Omega\times\s^2}\leq& C\e^{2}\abs{\ln(\e)}^8.
\end{align}
\ \\
Step 4: Estimates of $\ll \qf$.\\
We need to estimate $\uf_0$. The boundary layer contribution can be
estimated as
\begin{align}
\ll[\uf_0]=&\sin\phi\dfrac{\p \uf_0}{\p\eta}-\e\bigg(\dfrac{\sin^2\psi}{R_1-\e\eta}+\dfrac{\cos^2\psi}{R_2-\e\eta}\bigg)\cos\phi\dfrac{\p \uf_0}{\p\phi}\\\rule{0ex}{2.0em}
&+\e\bigg(\dfrac{\cos\phi\sin\psi}{P_1(1-\e\kk_1\eta)}\dfrac{\p \uf_0}{\p\iota_1}+\dfrac{\cos\phi\cos\psi}{P_2(1-\e\kk_2\eta)}\dfrac{\p \uf_0}{\p\iota_2}\bigg)\no\\\rule{0ex}{2.0em}
&+\e\Bigg(\dfrac{\sin\psi}{1-\e\kk_1\eta}\bigg(\cos\phi\Big(\vt_1\cdot\Big(\vt_2\times(\p_{12}\vr\times\vt_2)\Big)\Big)
-\kk_1P_1P_2\sin\phi\cos\psi\bigg)\no\\\rule{0ex}{2.0em}
&+\dfrac{\cos\psi}{1-\e\kk_2\eta}\bigg(-\cos\phi\Big(\vt_2\cdot\Big(\vt_1\times(\p_{12}\vr\times\vt_1)\Big)\Big)
+\kk_2P_1P_2\sin\phi\sin\psi\bigg)\Bigg)\dfrac{1}{P_1P_2}\dfrac{\p \uf_0}{\p\psi}+\uf_0-
\buf_0\no\\
=&\e\bigg(\dfrac{\cos\phi\sin\psi}{P_1(1-\e\kk_1\eta)}\dfrac{\p \uf_0}{\p\iota_1}+\dfrac{\cos\phi\cos\psi}{P_2(1-\e\kk_2\eta)}\dfrac{\p \uf_0}{\p\iota_2}\bigg)\no\\\rule{0ex}{2.0em}
&+\e\Bigg(\dfrac{\sin\psi}{1-\e\kk_1\eta}\bigg(\cos\phi\Big(\vt_1\cdot\Big(\vt_2\times(\p_{12}\vr\times\vt_2)\Big)\Big)
-\kk_1P_1P_2\sin\phi\cos\psi\bigg)\no\\\rule{0ex}{2.0em}
&+\dfrac{\cos\psi}{1-\e\kk_2\eta}\bigg(-\cos\phi\Big(\vt_2\cdot\Big(\vt_1\times(\p_{12}\vr\times\vt_1)\Big)\Big)
+\kk_2P_1P_2\sin\phi\sin\psi\bigg)\Bigg)\dfrac{1}{P_1P_2}\dfrac{\p \uf_0}{\p\psi}.\no
\end{align}
Based on Theorem \ref{dt theorem 2}, we have
\begin{align}
&\im{\e\bigg(\dfrac{\cos\phi\sin\psi}{P_1(1-\e\kk_1\eta)}\dfrac{\p \uf_0}{\p\iota_1}+\dfrac{\cos\phi\cos\psi}{P_2(1-\e\kk_2\eta)}\dfrac{\p \uf_0}{\p\iota_2}\bigg)}{\Omega\times\s^2}\\
\leq&C\e\bigg(\im{\frac{\p
\uf_0}{\p\iota_1}}{\Omega\times\s^2}+\im{\frac{\p
\uf_0}{\p\iota_2}}{\Omega\times\s^2}\bigg)\leq C\e\abs{\ln(\e)}^8.\no
\end{align}
Similarly, we can show that
\begin{align}
&\Bigg\Vert\e\Bigg(\dfrac{\sin\psi}{1-\e\kk_1\eta}\bigg(\cos\phi\Big(\vt_1\cdot\Big(\vt_2\times(\p_{12}\vr\times\vt_2)\Big)\Big)
-\kk_1P_1P_2\sin\phi\cos\psi\bigg)\\\rule{0ex}{2.0em}
&+\dfrac{\cos\psi}{1-\e\kk_2\eta}\bigg(-\cos\phi\Big(\vt_2\cdot\Big(\vt_1\times(\p_{12}\vr\times\vt_1)\Big)\Big)
+\kk_2P_1P_2\sin\phi\sin\psi\bigg)\Bigg)\dfrac{1}{P_1P_2}\dfrac{\p\uf_0}{\p\psi}\Bigg\Vert_{L^{\infty}(\Omega\times\s^2)}\no\\
\leq&C\e\im{\frac{\p
\uf_0}{\p\psi}}{\Omega\times\s^2}\leq C\e\abs{\ln(\e)}^8.\no
\end{align}
Also, the exponential decay of $\dfrac{\p\uf_0}{\p\iota_1}$ in Theorem \ref{dt theorem 2} and the rescaling $\eta=\dfrac{\mu}{\e}$ implies
\begin{align}
&\tm{\e\dfrac{\cos\phi\sin\psi}{P_1(1-\e\kk_1\eta)}\dfrac{\p \uf_0}{\p\iota_1}}{\Omega\times\s^2}
\leq \e\tm{\frac{\p
\uf_0}{\p\iota_1}}{\Omega\times\s^2}\\
\leq&\e\Bigg(\int_0^{R_{\min}}\int_{-\frac{\pi}{2}}^{\frac{\pi}{2}}\chi_1(R_{\min}-\mu)\lnm{\frac{\p \uf_0}{\p\iota_1}(\mu)}^2\ud{\phi}\ud{\mu}\Bigg)^{\frac{1}{2}}\no\\
&+\e\Bigg(\int_0^{R_{\min}}\int_{-\frac{\pi}{2}}^{\frac{\pi}{2}}\chi_2(R_{\min}-\mu)\lnm{\frac{\p \uf_0}{\p\iota_1}(\mu)}^2\ud{\phi}\ud{\mu}\Bigg)^{\frac{1}{2}}\no\\
\leq&\e^{\frac{3}{2}}\Bigg(\int_0^{\frac{R_{\min}}{\e}}\int_{-\frac{\pi}{2}}^{\frac{\pi}{2}}\chi_1(R_{\min}-\e\eta)\lnm{\frac{\p \uf_0}{\p\iota_1}(\eta)}^2\ud{\phi}\ud{\eta}\Bigg)^{\frac{1}{2}}\no\\
&+\e^{\frac{3}{2}}\Bigg(\int_0^{\frac{R_{\min}}{\e}}\int_{-\frac{\pi}{2}}^{\frac{\pi}{2}}\chi_2(R_{\min}-\e\eta)\lnm{\frac{\p \uf_0}{\p\iota_1}(\eta)}^2\ud{\phi}\ud{\eta}\Bigg)^{\frac{1}{2}}\no\\
\leq&C\Big(\e^{1+\frac{3}{2}\alpha}+\e^{\frac{3}{2}+\alpha}\Big)
\abs{\ln(\e)}^{8}\Bigg(\int_{-\pi}^{\pi}\int_0^{\frac{R_{\min}}{\e}}\ue^{-2K_0\eta}\ud{\eta}\Bigg)^{\frac{1}{2}}\no\\
\leq& C\e^{1+\frac{3}{2}\alpha}\abs{\ln(\e)}^8.\no
\end{align}
Here the smallness of $\chi_1$ quantity comes from the small domain $\abs{\phi}\leq\e^{\alpha}$ and $\abs{\eta}\leq \e^{2\alpha-1}$. The smallness of $\chi_2$ quantity comes from the extra $\e^{\alpha}$ for $0<\alpha<1$. This can naturally be extended to treat $\iota_2$ and $\psi$ derivatives. Hence, we have
\begin{align}
\tm{\e\bigg(\dfrac{\cos\phi\sin\psi}{P_1(1-\e\kk_1\eta)}\dfrac{\p \uf_0}{\p\iota_1}+\dfrac{\cos\phi\cos\psi}{P_2(1-\e\kk_2\eta)}\dfrac{\p \uf_0}{\p\iota_2}\bigg)}{\Omega\times\s^2}
\leq& C\e^{1+\frac{3}{2}\alpha}\abs{\ln(\e)}^8,
\end{align}
\begin{align}
\nm{\e\bigg(\dfrac{\cos\phi\sin\psi}{P_1(1-\e\kk_1\eta)}\dfrac{\p \uf_0}{\p\iota_1}+\dfrac{\cos\phi\cos\psi}{P_2(1-\e\kk_2\eta)}\dfrac{\p \uf_0}{\p\iota_2}\bigg)}_{L^{\frac{2m}{2m-1}}(\Omega\times\s^2)}\leq&C\e^{2-\frac{1}{2m}+\alpha}\abs{\ln(\e)}^8,
\end{align}
Using similar arguments, we may justify
\begin{align}
&\Bigg\Vert\e\Bigg(\dfrac{\sin\psi}{1-\e\kk_1\eta}\bigg(\cos\phi\Big(\vt_1\cdot\Big(\vt_2\times(\p_{12}\vr\times\vt_2)\Big)\Big)
-\kk_1P_1P_2\sin\phi\cos\psi\bigg)\\\rule{0ex}{2.0em}
&+\dfrac{\cos\psi}{1-\e\kk_2\eta}\bigg(-\cos\phi\Big(\vt_2\cdot\Big(\vt_1\times(\p_{12}\vr\times\vt_1)\Big)\Big)
+\kk_2P_1P_2\sin\phi\sin\psi\bigg)\Bigg)\dfrac{1}{P_1P_2}\dfrac{\p\uf_0}{\p\psi}\Bigg\Vert_{L^{2}(\Omega\times\s^2)}
\leq C\e^{1+\frac{3}{2}\alpha}\abs{\ln(\e)}^8,\no
\end{align}
and
\begin{align}
&\Bigg\Vert\e\Bigg(\dfrac{\sin\psi}{1-\e\kk_1\eta}\bigg(\cos\phi\Big(\vt_1\cdot\Big(\vt_2\times(\p_{12}\vr\times\vt_2)\Big)\Big)
-\kk_1P_1P_2\sin\phi\cos\psi\bigg)\\\rule{0ex}{2.0em}
&+\dfrac{\cos\psi}{1-\e\kk_2\eta}\bigg(-\cos\phi\Big(\vt_2\cdot\Big(\vt_1\times(\p_{12}\vr\times\vt_1)\Big)\Big)
+\kk_2P_1P_2\sin\phi\sin\psi\bigg)\Bigg)\dfrac{1}{P_1P_2}\dfrac{\p\uf_0}{\p\psi}\Bigg\Vert_{L^{\frac{2m}{2m-1}}(\Omega\times\s^2)}
\leq C\e^{2-\frac{1}{2m}+\alpha}\abs{\ln(\e)}^8.\no
\end{align}
In total, we have
\begin{align}
\tm{\ll[\qb]}{\Omega\times\s^2}\leq& C\e^{1+\frac{3}{2}\alpha}\abs{\ln(\e)}^8,\\
\nm{\ll[\qb]}_{L^{\frac{2m}{2m-1}}(\Omega\times\s^2)}\leq& C\e^{2-\frac{1}{2m}+\alpha}\abs{\ln(\e)}^8,\\
\im{\ll[\qb]}{\Omega\times\s^2}\leq& C\e\abs{\ln(\e)}^8.
\end{align}
\ \\
Step 5: Source Term and Boundary Condition.\\
In summary, since $\ll[u^{\e}]=0$, collecting estimates in Step 2 to Step 4 with $\alpha=1$, we can prove
\begin{align}
\tm{\ll[R]}{\Omega\times\s^2}\leq& C\e^{\frac{5}{2}}\abs{\ln(\e)}^8,\\
\nm{\ll[R]}_{L^{\frac{2m}{2m-1}}(\Omega\times\s^2)}\leq& C\e^{3-\frac{1}{2m}}\abs{\ln(\e)}^8,\\
\im{\ll[R]}{\Omega\times\s^2}\leq& C\e\abs{\ln(\e)}^8.
\end{align}
We can directly obtain that the boundary data is satisfied up to $O(\e)$, so we know for the boundary data $R^B$ of $R$,
\begin{align}
\tm{R^B}{\Gamma^-}\leq& C\e^2,\\
\nm{R^B}_{L^{m}(\Gamma^-)}\leq&C\e^2,\\
\im{R^B}{\Gamma^-}\leq& C\e^2
\end{align}
\ \\
Step 6: Diffusive Limit.\\
Hence, the remainder $R$ satisfies the equation
\begin{align}
\left\{
\begin{array}{l}
\e \vw\cdot\nabla_x R+R-\bar R=\ll[R]\ \ \text{in}\ \ \Omega\times\s^2,\\\rule{0ex}{2.0em}
R=R^B\ \ \text{for}\ \ \vw\cdot\vn<0\ \ \text{and}\ \
\vx_0\in\p\Omega.
\end{array}
\right.
\end{align}
By Theorem \ref{LI estimate}, we have for $1\leq m< 3$,
\begin{align}
\im{R}{\Omega\times\s^2}
\leq& C\bigg(\frac{1}{\e^{1+\frac{3}{2m}}}\nm{\ll[R]}_{L^{2}(\Omega\times\s^2)}+
\frac{1}{\e^{2+\frac{3}{2m}}}\nm{\ll[R]}_{L^{\frac{2m}{2m-1}}(\Omega\times\s^2)}+\nm{\ll[R]}_{L^{\infty}(\Omega\times\s^2)}\\
&+\frac{1}{\e^{\frac{1}{2}+\frac{3}{2m}}}\nm{R^B}_{L^2(\Gamma^-)}+\frac{1}{\e^{\frac{3}{2m}}}\nm{R^B}_{L^{\frac{4m}{3}}(\Gamma^-)}+\nm{R^B}_{L^{\infty}(\Gamma^-)}\bigg)\no\\
\leq& C\Bigg(\frac{1}{\e^{1+\frac{3}{2m}}}\e^{\frac{5}{2}}\abs{\ln(\e)}^8+
\frac{1}{\e^{2+\frac{3}{2m}}}\e^{3-\frac{1}{2m}}\abs{\ln(\e)}^8+(\e)\abs{\ln(\e)}^8\no\\
&+\frac{1}{\e^{\frac{1}{2}+\frac{3}{2m}}}(\e^2)+\frac{1}{\e^{\frac{3}{2m}}}(\e^2)+(\e^2)\Bigg)\no\\
\leq&C\e^{1-\frac{2}{m}}\abs{\ln(\e)}^8.\no
\end{align}
Here taking $m=3-$, we get the estimates
\begin{align}
\im{R}{\Omega\times\s^2}
\leq& C\e^{\frac{1}{3}-}\abs{\ln(\e)}^8.\no
\end{align}
Since it is easy to see
\begin{align}
\im{\sum_{k=1}^{2}\e^k\u_k+\sum_{k=1}^{1}\e^k\ub_k}{\Omega\times\s^2}\leq C\e,
\end{align}
our result naturally follows. We simply take $\u=\u_0$ and $\uu=\ub_0+\uf_0$. It is obvious that $\uu$ satisfies the $\e$-Milne problem with geometric correction with the full boundary data $g(\phi,\psi,\iota_1,\iota_2)-\mathscr{F}_{0,L}(\iota_1,\iota_2)-\mathfrak{F}_{0,L}(\iota_1,\iota_2)$. This completes the proof of main theorem.
\end{proof}

\chapter{Unsteady Neutron Transport Equation}

In this chapter, we prove the diffusive limit of the unsteady neutron transport equation \eqref{transport.}

\section{Asymptotic Expansions}

\subsection{Interior Expansion}

We define the interior expansion as follows:
\begin{align}\label{interior expansion.}
\u(t,\vx,\vw)\sim\sum_{k=0}^{2}\e^k\u_k(t,\vx,\vw),
\end{align}
where $\u_k$ can be defined by comparing the order of $\e$ via
plugging \eqref{interior expansion.} into the equation
\eqref{transport.}. Thus, we have
\begin{align}
\u_0-\bu_0=&0,\label{expansion temp 1.}\\
\u_1-\bu_1=&-\vw\cdot\nx\u_0,\label{expansion temp 2.}\\
\u_2-\bu_2=&-\dt\u_0-\vw\cdot\nx\u_1.\label{expansion temp 3.}
\end{align}
Plugging \eqref{expansion temp 1.} into \eqref{expansion temp 2.}, we
obtain
\begin{align}
\u_1=\bu_1-\vw\cdot\nx\bu_0.\label{expansion temp 4.}
\end{align}
Plugging \eqref{expansion temp 4.} into \eqref{expansion temp 3.}, we
get
\begin{align}\label{expansion temp 13.}
\u_2-\bu_2=&-\dt\u_0-\vw\cdot\nx(\bu_1-\vw\cdot\nx\bu_0)\\
=&-\dt\u_0-\vw\cdot\nx\bu_1+\Big(w_1^2\p_{x_1x_1}\bu_0+w_2^2\p_{x_2x_2}\bu_0+w_3^2\p_{x_3x_3}\bu_0\Big)\no\\
&+2\Big(w_1w_2\p_{x_1x_2}\bu_0+w_1w_3\p_{x_1x_3}\bu_0+w_2w_3\p_{x_2x_3}\bu_0\Big).\no
\end{align}
Integrating \eqref{expansion temp 13.} over $\vw\in\s^2$, we achieve
the final form
\begin{align}
\dt\bu_0-\frac{1}{3}\Delta_x\bu_0=0,
\end{align}
where all cross terms vanish due to the symmetry of $\s^2$. Hence, $\u_0(t,\vx,\vw)$ satisfies the equation
\begin{align}\label{interior 1.}
\left\{
\begin{array}{l}
\u_0=\bu_0,\\\rule{0ex}{2em}
\dt\bu_0-\dfrac{1}{3}\Delta_x\bu_0=0.
\end{array}
\right.
\end{align}
Similarly, we can derive that $\u_k(t,\vx,\vw)$ for $k=1,2$ satisfies
\begin{align}\label{interior 2.}
\left\{
\begin{array}{l}
\u_k=\bu_k-\vw\cdot\nx\u_{k-1},\\\rule{0ex}{1em}
\dt\bu_k-\dfrac{1}{3}\Delta_x\bu_k=0,
\end{array}
\right.
\end{align}
It is easy to see that $\bu_k$ satisfies an parabolic equation. However, the initial and boundary conditions of $\bu_k$ is unknown at this stage, since generally $\u_k$ does not necessarily satisfy the initial and boundary condition of \eqref{transport.}. Therefore, we have to resort to initial and boundary layer analysis.

\subsection{Initial Layer Expansion}\label{substitution..}

In order to determine the initial condition for $\u_k$, we need to define the initial layer expansion. Hence, we need a substitution:\\
\ \\
Temporal Substitution:\\
We define the rescaled variable $\tau$ by making the
scaling transform for $\tau=\dfrac{t}{\e^2}$,
which implies $\dfrac{\p u^{\e}}{\p t}=\dfrac{1}{\e^2}\dfrac{\p u^{\e}}{\p\tau}$.
Then, under the substitution $t\rt\tau$, the equation \eqref{transport.} is transformed into
\begin{align}\label{initial.}
\left\{ \begin{array}{l}\displaystyle \p_{\tau}u^{\e}+\e\vw\cdot\nabla_xu^{\e}+u^{\e}-\bar u^{\e}=0\ \ \ \text{for}\ \
(\tau,\vx,\vw)\in\rp\times\Omega\times\s^2,\\\rule{0ex}{2.0em}
u^{\e}(0,\vx,\vw)=h(\vx,\vw)\ \ \ \text{for}\ \
(\vx,\vw)\in\Omega\times\s^2,\\\rule{0ex}{2.0em}
u^{\e}(\tau,\vx_0,\vw)=g(\tau,\vx_0,\vw)\ \ \text{for}\ \ \tau\in\rp,\ \ \vx_0\in\p\Omega,\ \ \text{and}\ \
\vw\cdot\vn<0.
\end{array}
\right.
\end{align}
We define the initial layer expansion as follows:
\begin{align}\label{initial layer expansion.}
\ui(\tau,\vx,\vw)\sim\ui_{0}(\tau,\vx,\vw)+\e\ui_{1}(\tau,\vx,\vw),
\end{align}
where $\ui_{k}$ can be determined by comparing the order of $\e$ via
plugging \eqref{initial layer expansion.} into the equation
\eqref{initial.}. Thus, we
have
\begin{align}
\p_{\tau}\ui_{0}+\ui_{0}-\bui_{0}=&0,\label{initial expansion 1.}\\
\p_{\tau}\ui_{1}+\ui_{1}-\bui_{1}=&-\vw\cdot\nabla_x\ui_{0}.\label{initial expansion 2.}
\end{align}
Integrate \eqref{initial expansion 1.} over $\vw\in\s^2$, we have
\begin{align}
\p_{\tau}\bui_{0}=0,
\end{align}
which further implies
\begin{align}
\bui_{0}(\tau,\vx)=\bui_{0}(0,\vx)\ \ \text{for}\ \ \tau\in\rp.
\end{align}
Therefore, from \eqref{initial expansion 1.}, we can deduce
\begin{align}
\ui_{0}(\tau,\vx,\vw)=&\ue^{-\tau}\ui_{0}(0,\vx,\vw)+\int_0^{\tau}\bui_{0}(s,\vx)\ue^{s-\tau}\ud{s}=\ue^{-\tau}\ui_{0}(0,\vx,\vw)+(1-\ue^{-\tau})\bui_{0}(0,\vx).
\end{align}
This means that we have
\begin{align}
\left\{
\begin{array}{l}
\p_{\tau}\bui_{0}=0,\\\rule{0ex}{2.0em}
\ui_{0}(\tau,\vx,\vw)=\ue^{-\tau}\ui_{0}(0,\vx,\vw)+(1-\ue^{-\tau})\bui_{0}(0,\vx).
\end{array}
\right.
\end{align}
Similarly, we can derive that $\ub_{1}^I(\tau,\vx,\vw)$ satisfies
\begin{align}
\left\{
\begin{array}{l}
\p_{\tau}\bui_{1}=-\displaystyle\int_{\s^2}\bigg(\vw\cdot\nabla_x\ui_{0}\bigg)\ud{\vw},\\\rule{0ex}{2em}
\ui_{1}(\tau,\vx,\vw)=\ue^{-\tau}\ui_{1}(0,\vx,\vw)+\displaystyle\int_0^{\tau}\bigg(\bui_{1}-\vw\cdot\nabla_x\ui_{0}\bigg)(s,\vx,\vw)\ue^{s-\tau}\ud{s}.
\end{array}
\right.
\end{align}

\subsection{Boundary Layers Expansion}\label{substitution.}

Here, we implement the same geometric substitution as in steady problems. \\
\ \\
\textbf{Substitution 1: Spacial Substitution:}\\
In a neighborhood of $\vx_0\in\p\Omega$, define an orthogonal curvilinear coordinates system $(\iota_1,\iota_2)$ such that at $\vx_0$ the coordinate lines coincide with the principal directions. The boundary surface is $\vr=\vr(\iota_1,\iota_2)$. Let
\begin{align}
P=\abs{\p_1\vr\times\p_2\vr}=\abs{\p_1\vr}\abs{\p_2\vr}=P_1P_2,
\end{align}
for $P_i=\abs{\p_i\vr}$ with the unit tangential vectors
\begin{align}
\vt_1=\frac{\p_1\vr}{P_1},\ \ \vt_2=\frac{\p_2\vr}{P_2}.
\end{align}
Then consider the new coordinate system $(\mu,\iota_1,\iota_2)$, where $\mu$ denotes the normal distance to boundary surface $\p\Omega$, i.e.
\begin{align}
\vx=\vr-\mu\vn.
\end{align}
Let $\kk_1$ and $\kk_2$ be principal curvatures and
\begin{align}
R_{\min}=\min_{\iota_1,\iota_2}\{R_1(\iota_1,\iota_2),R_2(\iota_1,\iota_2)\},
\end{align}
where $R_1(\iota_1,\iota_2)=\dfrac{1}{\kk_1(\iota_1,\iota_2)}$ and $R_2(\iota_1,\iota_2)=\dfrac{1}{\kk_2(\iota_1,\iota_2)}$. Therefore, under the substitution $(x_1,x_2,x_3)\rt(\mu,\iota_1,\iota_2)$ for $0\leq\mu<R_{\min}$, the equation \eqref{transport.} is transformed into
\begin{align}\label{coordinate 9.}
\left\{
\begin{array}{l}\displaystyle
\e^2\dt u^{\e}+\e\bigg(-(\vw\cdot\vn)\frac{\p u^{\e}}{\p\mu}-\frac{\vw\cdot\vt_1}{P_1(\kk_1\mu-1)}\frac{\p u^{\e}}{\p\iota_1}-\frac{\vw\cdot\vt_2}{P_2(\kk_2\mu-1)}\frac{\p u^{\e}}{\p\iota_2}\bigg)+u^{\e}-\bar u^{\e}=0\ \ \text{in}\ \ \rp\times(0,R_{\min})\times\Sigma\times\s^2,\\\rule{0ex}{2.0em}
u^{\e}(0,\mu,\iota_1,\iota_2,\vw)=h(\mu,\iota_1,\iota_2,\vw)\ \ \text{in}\ \ (0,R_{\min})\times\Sigma\times\s^2,\\\rule{0ex}{2.0em}
u^{\e}(t,0,\iota_1,\iota_2,\vw)=g(t,\iota_1,\iota_2,\vw)\ \ \text{for}\
\ t\in\rp\ \ \text{and}\ \ \vw\cdot\vn<0.
\end{array}
\right.
\end{align}
\ \\
\textbf{Substitution 2: Velocity Substitution:}\\
Define the orthogonal velocity substitution
\begin{align}\label{coordinate 11.}
\left\{
\begin{aligned}
-\vw\cdot\vn=&\sin\phi,\\
\vw\cdot\vt_1=&\cos\phi\sin\psi,\\
\vw\cdot\vt_2=&\cos\phi\cos\psi,
\end{aligned}
\right.
\end{align}
for $\phi\in\left[-\dfrac{\pi}{2},\dfrac{\pi}{2}\right]$ and $\psi\in[-\pi,\pi]$.
Hence, under substitution $(w_1,w_2,w_3)\rt(\phi,\psi)$ for $\phi\in\left[-\dfrac{\pi}{2},\dfrac{\pi}{2}\right]$ and $\psi\in[-\pi,\pi]$,
the equation \eqref{transport.} is transformed into
\begin{align}\label{coordinate 21.}
\left\{
\begin{array}{l}\displaystyle
\e^2\dt u^{\e}+\e\sin\phi\dfrac{\p u^{\e}}{\p\mu}-\e\bigg(\dfrac{\sin^2\psi}{R_1-\mu}+\dfrac{\cos^2\psi}{R_2-\mu}\bigg)\cos\phi\dfrac{\p u^{\e}}{\p\phi}\\\rule{0ex}{2.0em}
+\e\bigg(\dfrac{\cos\phi\sin\psi}{P_1(1-\kk_1\mu)}\dfrac{\p u^{\e}}{\p\iota_1}+\dfrac{\cos\phi\cos\psi}{P_2(1-\kk_2\mu)}\dfrac{\p u^{\e}}{\p\iota_2}\bigg)\\\rule{0ex}{2.0em}
+\e\Bigg(\dfrac{\sin\psi}{1-\kk_1\mu}\bigg(\cos\phi\Big(\vt_1\cdot\Big(\vt_2\times(\p_{12}\vr\times\vt_2)\Big)\Big)
-\kk_1P_1P_2\sin\phi\cos\psi\bigg)\\\rule{0ex}{2.0em}
+\dfrac{\cos\psi}{1-\kk_2\mu}\bigg(-\cos\phi\Big(\vt_2\cdot\Big(\vt_1\times(\p_{12}\vr\times\vt_1)\Big)\Big)
+\kk_2P_1P_2\sin\phi\sin\psi\bigg)\Bigg)\dfrac{1}{P_1P_2}\dfrac{\p u^{\e}}{\p\psi}\\
+ u^{\e}-\bar u^{\e}=0\ \ \text{in}\ \ \rp\times(0,R_{\min})\times\Sigma\times\left[-\dfrac{\pi}{2},\dfrac{\pi}{2}\right]\times[-\pi,\pi],\\\rule{0ex}{2.0em}
u^{\e}(0,\mu,\iota_1,\iota_2,\phi,\psi)=h(\mu,\iota_1,\iota_2,\phi,\psi)\ \ \text{in}\ \ (0,R_{\min})\times\Sigma\times\left[-\dfrac{\pi}{2},\dfrac{\pi}{2}\right]\times[-\pi,\pi],\\\rule{0ex}{2.0em}
u^{\e}(t,0,\iota_1,\iota_2,\phi,\psi)=g(t,\iota_1,\iota_2,\phi,\psi)\ \ \text{for}\
\ t\in\rp\ \ \text{and}\ \ \sin\phi>0.
\end{array}
\right.
\end{align}
\ \\
\textbf{Substitution 3: Scaling Substitution:}\\
Define the scaled variable $\eta=\dfrac{\mu}{\e}$, which implies $\dfrac{\p}{\p\mu}=\dfrac{1}{\e}\dfrac{\p}{\p\eta}$. Then, under the substitution $\mu\rt\eta$, the equation \eqref{transport.} is transformed into
\begin{align}\label{coordinate 22.}
\left\{
\begin{array}{l}\displaystyle
\e^2\dt u^{\e}+\sin\phi\dfrac{\p u^{\e}}{\p\eta}-\e\bigg(\dfrac{\sin^2\psi}{R_1-\e\eta}+\dfrac{\cos^2\psi}{R_2-\e\eta}\bigg)\cos\phi\dfrac{\p u^{\e}}{\p\phi}\\\rule{0ex}{2.0em}
+\e\bigg(\dfrac{\cos\phi\sin\psi}{P_1(1-\e\kk_1\eta)}\dfrac{\p u^{\e}}{\p\iota_1}+\dfrac{\cos\phi\cos\psi}{P_2(1-\e\kk_2\eta)}\dfrac{\p u^{\e}}{\p\iota_2}\bigg)\\\rule{0ex}{2.0em}
+\e\Bigg(\dfrac{\sin\psi}{1-\e\kk_1\eta}\bigg(\cos\phi\Big(\vt_1\cdot\Big(\vt_2\times(\p_{12}\vr\times\vt_2)\Big)\Big)
-\kk_1P_1P_2\sin\phi\cos\psi\bigg)\\\rule{0ex}{2.0em}
+\dfrac{\cos\psi}{1-\e\kk_2\eta}\bigg(-\cos\phi\Big(\vt_2\cdot\Big(\vt_1\times(\p_{12}\vr\times\vt_1)\Big)\Big)
+\kk_2P_1P_2\sin\phi\sin\psi\bigg)\Bigg)\dfrac{1}{P_1P_2}\dfrac{\p u^{\e}}{\p\psi}+u^{\e}-\bar u^{\e}=0\\
+ u^{\e}-\bar u^{\e}=0\ \ \text{in}\ \ \rp\times\left(0,\dfrac{R_{\min}}{\e}\right)\times\Sigma\times\left[-\dfrac{\pi}{2},\dfrac{\pi}{2}\right]\times[-\pi,\pi],\\\rule{0ex}{2.0em}
u^{\e}(0,\eta,\iota_1,\iota_2,\phi,\psi)=h(\eta,\iota_1,\iota_2,\phi,\psi)\ \ \text{in}\ \ \left(0,\dfrac{R_{\min}}{\e}\right)\times\Sigma\times\left[-\dfrac{\pi}{2},\dfrac{\pi}{2}\right]\times[-\pi,\pi],\\\rule{0ex}{2.0em}
u^{\e}(t,0,\iota_1,\iota_2,\phi,\psi)=g(t,\iota_1,\iota_2,\phi,\psi)\ \ \text{for}\
\ t\in\rp\ \ \text{and}\ \ \sin\phi>0.
\end{array}
\right.
\end{align}
We define the boundary layer expansion as follows:
\begin{align}\label{boundary layer expansion.}
\uu(t,\eta,\iota_1,\iota_2,\phi,\psi)\sim\uu_0(t,\eta,\iota_1,\iota_2,\phi,\psi)+\e\uu_1(t,\eta,\iota_1,\iota_2,\phi,\psi),
\end{align}
where $\uu_k$ can be defined by comparing the order of $\e$ via
plugging (\eqref{boundary layer expansion.}) into the equation
(\eqref{coordinate 22.}). Thus, in a neighborhood of the boundary, we have
\begin{align}
\sin\phi\frac{\p\uu_0}{\p\eta}-\e\bigg(\dfrac{\sin^2\psi}{R_1-\e\eta}+\dfrac{\cos^2\psi}{R_2-\e\eta}\bigg)\cos\phi\dfrac{\p \uu_0}{\p\phi}+\uu_0-\buu_0=&0,\label{expansion temp 6.}\\
\sin\phi\frac{\p\uu_1}{\p\eta}-\e\bigg(\dfrac{\sin^2\psi}{R_1-\e\eta}+\dfrac{\cos^2\psi}{R_2-\e\eta}\bigg)\cos\phi\dfrac{\p \uu_1}{\p\phi}+\uu_1-\buu_1=&-G[\uu_0],\label{expansion temp 7.}
\end{align}
where
\begin{align}\label{coordinate 23.}
G[\uu_0]=&\bigg(\dfrac{\cos\phi\sin\psi}{P_1(1-\e\kk_1\eta)}\dfrac{\p \uu_0}{\p\iota_1}+\dfrac{\cos\phi\cos\psi}{P_2(1-\e\kk_2\eta)}\dfrac{\p \uu_0}{\p\iota_2}\bigg)\\
&+\Bigg(\dfrac{\sin\psi}{1-\e\kk_1\eta}\bigg(\cos\phi\Big(\vt_1\cdot\Big(\vt_2\times(\p_{12}\vr\times\vt_2)\Big)\Big)
-\kk_1P_1P_2\sin\phi\cos\psi\bigg)\no\\
&+\dfrac{\cos\psi}{1-\e\kk_2\eta}\bigg(-\cos\phi\Big(\vt_2\cdot\Big(\vt_1\times(\p_{12}\vr\times\vt_1)\Big)\Big)
+\kk_2P_1P_2\sin\phi\sin\psi\bigg)\Bigg)\dfrac{1}{P_1P_2}\dfrac{\p \uu_0}{\p\psi},\no
\end{align}
and
\begin{align}
\buu_k(\eta,\iota_1,\iota_2)=\frac{1}{4\pi}\int_{-\pi}^{\pi}\int_{-\frac{\pi}{2}}^{\frac{\pi}{2}}\uu_k(\eta,\iota_1,\iota_2,\phi,\psi)\cos\phi\ud{\phi}\ud{\psi}.
\end{align}

\subsection{Matching Procedure}

Here we still define the boundary data decomposition
\begin{align}
g(\phi,\psi)=\gb(\phi,\psi)+\gf(\phi,\psi).
\end{align}
For either $\gb$ and $\gf$, we may define the corresponding boundary layer $\ub$ and $\uf$. We call $\ub$ the regular boundary layer and expand it up to $O(\e)$, i.e.
\begin{align}
\ub(\eta,\iota_1,\iota_2,\phi,\psi)\sim\ub_0(\eta,\iota_1,\iota_2,\phi,\psi)+\e\ub_1(\eta,\iota_1,\iota_2,\phi,\psi).
\end{align}
Also, we call $\uf$ the singular boundary layer and only expand it to $O(1)$, i.e.
\begin{align}
\uf(\eta,\iota_1,\iota_2,\phi,\psi)\sim\uf_0(\eta,\iota_1,\iota_2,\phi,\psi).
\end{align}
They should both satisfy the $\e$-Milne problem with geometric correction.\\
\ \\
The bridge between the interior solution, initial layer and boundary layer
is the initial and boundary conditions of \eqref{transport.}, so we
consider the initial and boundary expansion:
\begin{align}
\u_0(0,\vx,\vw)+\ui_{0}(0,\vx,\vw)=&h(\vx,\vw),\\
\u_0(t,\vx_0,\vw)+\ub_0(t,\vx_0,\vw)+\uf_0(t,\vx_0,\vw)=&g(t,\vx_0,\vw)\ \ \text{for}\ \ \vx_0\in\p\Omega,\\
\u_1(t,\vx_0,\vw)+\ub_1(t,\vx_0,\vw)=&0\ \ \text{for}\ \ \vx_0\in\p\Omega.
\end{align}
The construction and determination of asymptotic expansion are as follows:\\
\ \\
Step 0: Preliminaries.\\
Define the force
\begin{align}
F(\e;\eta,\iota_1,\iota_2,\psi)=-\e\bigg(\dfrac{\sin^2\psi}{R_1(\iota_1,\iota_2)-\e\eta}+\dfrac{\cos^2\psi}{R_2(\iota_1,\iota_2)-\e\eta}\bigg).
\end{align}
Define the length of boundary layer $L=\e^{-n}$ for $0<n<\dfrac{1}{2}$. For $\phi\in\left[-\dfrac{\pi}{2},\dfrac{\pi}{2}\right]$, denote $\rr[\phi]=-\phi$.\\
\ \\
Step 1: Construction of $\ub_0$, $\uf_0$, $\ui_0$ and $\u_0$.\\
Define the zeroth-order regular boundary layer as
\begin{align}\label{et 1.}
\left\{
\begin{array}{l}
\ub_0(t,\eta,\iota_1,\iota_2,\phi,\psi)=\mathscr{F}_0 (t,\eta,\iota_1,\iota_2,\phi,\psi)-\mathscr{F}_{0,L}(t,\iota_1,\iota_2),\\\rule{0ex}{2em}
\sin\phi\dfrac{\p \mathscr{F}_0 }{\p\eta}+F(\e;\eta,\iota_1,\iota_2,\psi)\cos\phi\dfrac{\p
\mathscr{F}_0 }{\p\phi}+\mathscr{F}_0 -\bar{\mathscr{F}}_0 =0,\\\rule{0ex}{1.5em}
\mathscr{F}_0 (t,0,\iota_1,\iota_2,\phi,\psi)=\gb(t,\iota_1,\iota_2,\phi,\psi)\ \ \text{for}\ \
\sin\phi>0,\\\rule{0ex}{1.5em}
\mathscr{F}_0 (t,L,\iota_1,\iota_2,\phi,\psi)=\mathscr{F}_0 (t,L,\iota_1,\iota_2,\rr[\phi],\psi),
\end{array}
\right.
\end{align}
with $\mathscr{F}_{0,L}(t,\iota_1,\iota_2)$ is defined as in Theorem \ref{Milne theorem 1}.\\
\ \\
Define the zeroth-order singular boundary layer as
\begin{align}\label{et 2.}
\left\{
\begin{array}{l}
\uf_0(t,\eta,\iota_1,\iota_2,\phi,\psi)=\mathfrak{F}_0 (t,\eta,\iota_1,\iota_2,\phi,\psi)-\mathfrak{F} _{0,L}(\iota_1,\iota_2),\\\rule{0ex}{2em}
\sin\phi\dfrac{\p \mathfrak{F}_0 }{\p\eta}+F(\e;\eta,\iota_1,\iota_2,\psi)\cos\phi\dfrac{\p
\mathfrak{F}_0 }{\p\phi}+\mathfrak{F}_0 -\bar{\mathfrak{F}}_0 =0,\\\rule{0ex}{1.5em}
\mathfrak{F}_0 (t,0,\iota_1,\iota_2,\phi,\psi)=\gf(t,\iota_1,\iota_2,\phi,\psi)\ \ \text{for}\ \
\sin\phi>0,\\\rule{0ex}{1.5em}
\mathfrak{F}_0 (t,L,\iota_1,\iota_2,\phi,\psi)=\mathfrak{F}_0 (t,L,\iota_1,\iota_2,\rr[\phi],\psi),
\end{array}
\right.
\end{align}
with $\mathfrak{F} _{0,L}(t,\iota_1,\iota_2)$ is defined as in Theorem \ref{Milne theorem 1}.\\
\ \\
Define the zeroth-order initial layer as
\begin{align}\label{et 6.}
\left\{
\begin{array}{l}
\ui_{0}(\tau,\vx,\vw)=\mathfrak{f}_0(\tau,\vx,\vw)-\mathfrak{f}_0(\infty,\vx)\\\rule{0ex}{2.0em}
\p_{\tau}\bar{\mathfrak{f}}_0=0,\\\rule{0ex}{2.0em}
\mathfrak{f}_0(\tau,\vx,\vw)=\ue^{-\tau}\mathfrak{f}_0(0,\vx,\vw)+(1-\ue^{-\tau})\bar{\mathfrak{f}}_0(0,\vx),\\\rule{0ex}{2.0em}
\mathfrak{f}_0(0,\vx,\vw)=h(\vx,\vw),\\\rule{0ex}{2.0em}
\lim_{\tau\rt\infty}\mathfrak{f}_0(\tau,\vx,\vw)=\mathfrak{f}_0(\infty,\vx).
\end{array}
\right.
\end{align}
\ \\
Also, define the zeroth-order interior solution $\u_0(t,\vx,\vw)$ as
\begin{align}\label{et 3.}
\left\{
\begin{array}{l}
\u_0(t,\vx,\vw)=\bu_0(t,\vx) ,\\\rule{0ex}{1.5em} \dt\bu_0-\dfrac{1}{3}\Delta_x\bu_0=0\ \ \text{in}\
\ \Omega,\\\rule{0ex}{1.5em}
\bu_0(0,\vx)=\mathfrak{f}_0(\infty,\vx)\ \ \text{in}\ \
\Omega,\\\rule{0ex}{1.5em}
\bu_0(t,\vx_0)=\mathscr{F}_{0,L}(t,\iota_1,\iota_2)+\mathfrak{F}_{0,L}(t,\iota_1,\iota_2)\ \ \text{on}\ \
\p\Omega.
\end{array}
\right.
\end{align}
\ \\
Step 2: Construction of $\ub_1$, $\ui_1$ and $\u_1$.\\
Define the first-order regular boundary layer as
\begin{align}\label{et 4.}
\left\{
\begin{array}{l}
\ub_1(t,\eta,\iota_1,\iota_2,\phi,\psi)=\mathscr{F}_1 (t,\eta,\iota_1,\iota_2,\phi,\psi)-\mathscr{F} _{1,L}(t,\iota_1,\iota_2),\\\rule{0ex}{2em}
\sin\phi\dfrac{\p \mathscr{F}_1 }{\p\eta}+F(\e;\eta,\iota_1,\iota_2,\psi)\cos\phi\dfrac{\p
\mathscr{F}_1 }{\p\phi}+\mathscr{F}_1 -\bar{\mathscr{F}}_1 =G[\ub_0],\\\rule{0ex}{1.5em}
\mathscr{F}_1 (t,0,\iota_1,\iota_2,\phi,\psi)=\vw\cdot\nx\u_0(t,\vx_0)\ \ \text{for}\ \
\sin\phi>0,\\\rule{0ex}{1.5em}
\mathscr{F}_1 (t,L,\iota_1,\iota_2,\phi,\psi)=\mathscr{F}_1 (t,L,\iota_1,\iota_2,\rr[\phi],\psi),
\end{array}
\right.
\end{align}
with $\mathscr{F}_{1,L}(t,\iota_1,\iota_2)$ is defined as in Theorem \ref{Milne theorem 1}, and $G_0$ is defined in \eqref{coordinate 23.}.\\
\ \\
Define the first-order initial layer as
\begin{eqnarray}\label{et 7.}
\left\{
\begin{array}{l}
\ui_{1}(\tau,\vx,\vw)=\mathfrak{f}_1(\tau,\vx,\vw)-\mathfrak{f}_1(\infty,\vx)\\\rule{0ex}{2em}
\p_{\tau}\bar{\mathfrak{f}}_1=-\displaystyle\int_{\s^2}\bigg(\vw\cdot\nabla_x\ui_{0}\bigg)\ud{\vw},\\\rule{0ex}{2em}
\mathfrak{f}_1(\tau,\vx,\vw)=\ue^{-\tau}\mathfrak{f}_1(0,\vx,\vw)+\displaystyle\int_0^{\tau}
\bigg(\bar{\mathfrak{f}}_1-\vw\cdot\nabla_x\ui_{0}\bigg)(s,\vx,\vw)\ue^{s-\tau}\ud{s},\\\rule{0ex}{2em}
\mathfrak{f}_1(0,\vx,\vw)=\vw\cdot\nx\u_0(0,\vx),\\\rule{0ex}{2em}
\lim_{\tau\rt\infty}\mathfrak{f}_1(\tau,\vx,\vw)=\mathfrak{f}_1(\infty,\vx).
\end{array}
\right.
\end{eqnarray}
\ \\
Then define the first-order interior solution $\u_1(\vx,\vw)$ as
\begin{align}\label{et 5}
\left\{
\begin{array}{l}
\u_1(\vx,\vw)=\bu_1(\vx)-\vw\cdot\nx\u_0(\vx,\vw),\\\rule{0ex}{1.5em}
\dt\bu_1-\dfrac{1}{3}\Delta_x\bu_1=-\displaystyle\int_{\s^1}\Big(\vw\cdot\nx\u_{0}(\vx,\vw)\Big)\ud{\vw}\
\ \text{in}\ \ \Omega,\\\rule{0ex}{1em}
\bu_1(0,\vx)=\mathfrak{f}_1(\infty,\vx)\ \ \text{in}\ \
\Omega,\\\rule{0ex}{1.5em}\bu_1(t,\vx_0)=f _{1,L}(t,\iota_1,\iota_2)\ \ \text{on}\ \
\p\Omega.
\end{array}
\right.
\end{align}
Note that we do not define $\uf_1$ here.\\
\ \\
Step 3: Construction of $\u_2$.\\
Since we do not expand to $\ub_2$ and $\uf_2$, simply define the second-order interior solution as
\begin{align}
\left\{
\begin{array}{l}
\u_{2}(\vx,\vw)=\bu_{2}(\vx)-\vw\cdot\nx\u_{1}(\vx,\vw),\\\rule{0ex}{1.5em}
\dt\bu_2-\dfrac{1}{3}\Delta_x\bu_{2}=-\displaystyle\int_{\s^1}\Big(\vw\cdot\nx\u_{1}(\vx,\vw)\Big)\ud{\vw}\
\ \text{in}\ \ \Omega,\\\rule{0ex}{1.5em}
\bu_2(0,\vx)=0\ \ \text{in}\ \
\Omega,\\\rule{0ex}{1.5em}\bu_2(t,\vx_0)=0\ \ \text{on}\ \
\p\Omega.
\end{array}
\right.
\end{align}
Here, we might have $O(\e^3)$ error in this step due to the trivial boundary data. Thanks to the remainder estimate, it will not affect the diffusive limit.

\section{Remainder Estimate}

In this section, we consider the remainder equation for $u(t,\vx,\vw)$ as
\begin{align}\label{neutron.}
\left\{
\begin{array}{l}
\e^2\dt u+\e \vw\cdot\nabla_x u+u-\bar
u=\ss\ \ \ \text{for}\ \
(t,\vx,\vw)\in\rp\times\Omega\times\s^2,\\\rule{0ex}{2.0em}
u(0,\vx,\vw)=\h(\vx,\vw)\ \ \text{for}\ \ (\vx,\vw)\in\Omega\times\s^2\\\rule{0ex}{2.0em}
u(t,\vx_0,\vw)=\g(t,\vx_0,\vw)\ \ \text{for}\ \ t\in\rp,\ \ \vx_0\in\p\Omega\ \ \text{and}\ \ \vw\cdot\vn<0.
\end{array}
\right.
\end{align}
The initial and boundary data satisfy the
compatibility condition
\begin{align}
\h(\vx_0,\vw)=\g(0,\vx_0,\vw)\ \ \text{for}\ \ \vx_0\in\p\Omega\ \ \text{and}\ \ \vw\cdot\vn<0.
\end{align}
\ \\
Define the $L^p$ norms with $1\leq p<\infty$ and $L^{\infty}$ norm in $\rp\times\Omega\times\s^2$ as
usual:
\begin{align}
\nm{f}_{L^p(\rp\times\Omega\times\s^2)}=&\bigg(\int_0^{\infty}\int_{\Omega}\int_{\s^2}\abs{f(t,\vx,\vw)}^p\ud{\vw}\ud{\vx}\ud t\bigg)^{\frac{1}{p}},\\
\nm{f}_{L^{\infty}(\rp\times\Omega\times\s^2)}=&\text{esssup}_{(t,\vx,\vw)\in\rp\times\Omega\times\s^2}\abs{f(t,\vx,\vw)}.
\end{align}
Define the $L^p$ norm with $1\leq p<\infty$ and $L^{\infty}$ norm on the boundary $\Gamma=\p\Omega\times\s^2$ as follows:
\begin{align}
\nm{f}_{L^p(\rp\times\Gamma)}=&\bigg(\int_0^{\infty}\iint_{\Gamma}\abs{f(t,\vx,\vw)}^p\abs{\vw\cdot\vn}\ud{\vw}\ud{\vx}\ud t\bigg)^{\frac{1}{p}},\\
\nm{f}_{L^p(\rp\times\Gamma^{\pm})}=&\bigg(\int_0^{\infty}\iint_{\Gamma^{\pm}}\abs{f(t,\vx,\vw)}^p\abs{\vw\cdot\vn}\ud{\vw}\ud{\vx}\ud t\bigg)^{\frac{1}{p}},\\
\nm{f}_{L^{\infty}(\rp\times\Gamma)}=&\text{esssup}_{(t,\vx,\vw)\in\rp\times\Gamma}\abs{f(t,\vx,\vw)},\\
\nm{f}_{L^{\infty}(\rp\times\Gamma^{\pm})}=&\text{esssup}_{(t,\vx,\vw)\in\rp\times\Gamma^{\pm}}\abs{f(t,\vx,\vw)}.
\end{align}
In particular, we denote $\ud{\gamma}=(\vw\cdot\vn)\ud{\vw}\ud{\vx}$ on the boundary.\\
\ \\
Similar notation also applies to the space
$[0,t]\times\Omega\times\s^2$, $[0,t]\times\Gamma$, and
$[0,t]\times\Gamma^{\pm}$.

\subsection{$L^2$ Estimate}

\begin{lemma}[Green's Identity]\label{wt lemma 1}
Assume $f(t,\vx,\vw),\ g(t,\vx,\vw)\in
L^{\infty}(\rp\times\Omega\times\s^2)$ and $\dt f+\vw\cdot\nx
f,\ \dt g+\vw\cdot\nx g\in L^2(\rp\times\Omega\times\s^2)$
with $f,\ g\in L^2(\rp\times\Gamma)$. Then for almost all
$s,t\in\rp$,
\begin{align}
&\int_s^t\iint_{\Omega\times\s^2}\bigg((\dt f+\vw\cdot\nx f)g+(\dt
g+\vw\cdot\nx
g)f\bigg)\ud{\vx}\ud{\vw}\ud{r}\\
=&\int_s^t\int_{\Gamma}fg\ud{\gamma}\ud{r}+\iint_{\Omega\times\s^2}f(t)g(t)\ud{\vx}\ud{\vw}-\iint_{\Omega\times\s^2}f(s)g(s)\ud{\vx}\ud{\vw}.\no
\end{align}
\end{lemma}
\begin{proof}
See \cite[Chapter 9]{Cercignani.Illner.Pulvirenti1994} and
\cite{Esposito.Guo.Kim.Marra2013}.
\end{proof}
\begin{theorem}\label{LT estimate.}
Assume $\ss(t,\vx,\vw)\in
L^{\infty}(\rp\times\Omega\times\s^2)$, $\h(\vx,\vw)\in
L^{\infty}(\Omega\times\s^2)$ and $\g(t,x_0,\vw)\in
L^{\infty}(\rp\times\Gamma^-)$. Then the neutron
transport equation \eqref{neutron.} has a unique solution
$u(t,\vx,\vw)\in L^2(\rp\times\Omega\times\s^2)$ satisfying
\begin{align}
&\nm{u (t)}_{L^2(\Omega\times\s^2)}+\frac{1}{\e^{\frac{1}{2}}}\nm{u}_{L^2([0,t]\times\Gamma^+)}+\nm{u }_{L^2([0,t]\times\Omega\times\s^2)}\\
\leq&
C \bigg(\frac{1}{\e^2}\tm{\ss}{[0,t]\times\Omega\times\s^2}+
\nm{\h}_{L^2(\Omega\times\s^2)}+\frac{1}{\e^{\frac{1}{2}}}\nm{\g}_{L^2([0,t]\times\Gamma^-)}\bigg).\no
\end{align}
\end{theorem}
\begin{proof}
\ \\
Step 1: Kernel Estimate.\\
Applying Lemma \ref{wt lemma 1} to the
equation \eqref{neutron.}. Then for any
$\phi\in L^{2}([0,t]\times\Omega\times\s^2)$ satisfying
$\e\dt\phi+\vw\cdot\nx\phi\in L^2([0,t]\times\Omega\times\s^2)$
and $\phi\in L^{2}([0,t]\times\Gamma)$, we have
\begin{align}\label{wt 1}
&-\e^2\int_0^t\iint_{\Omega\times\s^2}\dt\phi
u -\e\int_0^t\iint_{\Omega\times\s^2}(\vw\cdot\nx\phi)u +\int_0^t\iint_{\Omega\times\s^2}(u -\bar
u )\phi\\
=&-\e^2\iint_{\Omega\times\s^2}u (t)\phi(t)+
\e^2\iint_{\Omega\times\s^2}u (0)\phi(0)-\e\int_0^t\int_{\Gamma}u \phi\ud{\gamma}+\int_0^t\iint_{\Omega\times\s^2}\ss\phi.\no
\end{align}
Our goal is to choose a particular test function $\phi$. We first
construct an auxiliary function $\zeta(t)$. Clearly, $u(t)\in
L^{2}(\Omega\times\s^2)$ implies that $\bar
u (t)\in L^{2}(\Omega)$. Define $\zeta(t,\vx)$ on $\Omega$
satisfying
\begin{align}\label{wt 2}
\left\{
\begin{array}{l}
\Delta_x \zeta(t)=\bar u (t)\ \ \text{in}\ \
\Omega,\\\rule{0ex}{1.0em} \zeta(t)=0\ \ \text{on}\ \ \p\Omega.
\end{array}
\right.
\end{align}
In the bounded domain $\Omega$, based on the standard elliptic
estimates, there exists a unique $\xi(t)\in H^2(\Omega)$ such that
\begin{align}\label{wt 3}
\nm{\zeta(t)}_{H^2(\Omega)}\leq C\nm{\bar
u (t)}_{L^2(\Omega)}.
\end{align}
We plug the test function
\begin{align}\label{wt 4}
\phi(t)=-\vw\cdot\nx\zeta(t)
\end{align}
into the weak formulation \eqref{wt 1} and estimate
each term there. By definition, we have
\begin{align}\label{wt 5}
\nm{\phi(t)}_{L^2(\Omega)}\leq C\nm{\zeta(t)}_{H^1(\Omega)}\leq
C \nm{\bar u (t)}_{L^2(\Omega)}.
\end{align}
On the other hand, we decompose
\begin{align}\label{wt 6}
-\e\int_0^t\iint_{\Omega\times\s^2}(\vw\cdot\nx\phi)u =&-\e\int_0^t\iint_{\Omega\times\s^2}(\vw\cdot\nx\phi)\bar
u -\e\int_0^t\iint_{\Omega\times\s^2}(\vw\cdot\nx\phi)(u -\bar
u ).
\end{align}
For the first term on the right-hand side of \eqref{wt 6}, by
\eqref{wt 2} and \eqref{wt 4}, we have
\begin{align}\label{wt 7}
&-\e\int_0^t\iint_{\Omega\times\s^2}(\vw\cdot\nx\phi)\bar
u \\
=&\e\int_0^t\iint_{\Omega\times\s^2}\bar
u \Big(w_1(w_1\p_{11}\xi+w_2\p_{12}\xi+w_3\p_{13}\xi)+w_2(w_1\p_{21}\xi+w_2\p_{22}\xi+w_3\p_{23}\xi)+w_3(w_1\p_{31}\xi+w_2\p_{32}\xi+w_3\p_{33}\xi)\Big)\no\\
=&\e\int_0^t\iint_{\Omega\times\s^2}\bar
u \Big(w_1^2\p_{11}\xi+w_2^2\p_{22}\xi+w_3^2\p_{33}\xi\Big)
=\frac{4}{3}\e\pi\int_0^t\int_{\Omega}\bar u (\p_{11}\zeta+\p_{22}\zeta+\p_{33}\zeta)
=\frac{4}{3}\e\pi\nm{\bar u }_{L^2([0,t]\times\Omega)}^2\no\\
=&\frac{1}{3}\e\nm{\bar
u }_{L^2([0,t]\times\Omega\times\s^2)}^2\no.
\end{align}
Here $\p_i$ denotes the derivative with respect to $x_i$. In the second equality, above cross terms vanish due to the symmetry
of the integral over $\s^2$.\\
For the second term
on the right-hand side of \eqref{wt 6}, H\"older's inequality and \eqref{wt 5} imply
\begin{align}\label{wt 8}
&\abs{-\e\int_0^t\iint_{\Omega\times\s^2}(\vw\cdot\nx\phi)(u -\bar
u )}\leq C \e\nm{\vw\cdot\nx\phi}_{L^2([0,t]\times\Omega\times\s^2)}\nm{u -\bar u }_{L^2([0,t]\times\Omega\times\s^2)}\\
\leq&C \e\nm{u -\bar u }_{L^2([0,t]\times\Omega\times\s^2)}\bigg(\int_0^t\nm{\zeta(s)}^2_{H^2(\Omega)}\ud{s}\bigg)^{\frac{1}{2}}
\leq C \e\nm{u -\bar
u }_{L^2([0,t]\times\Omega\times\s^2)}\nm{\bar
u }_{L^2([0,t]\times\Omega\times\s^2)}\no.
\end{align}
Using the trace theorem, H\"older's inequality and \eqref{wt 5}, we have
\begin{align}\label{wt 9}
\abs{\e\int_0^t\int_{\Gamma}u \phi\ud{\gamma}}
=&\e\int_0^t\int_{\Gamma^-}\g\phi\ud{\gamma}+\e\int_0^t\int_{\Gamma^+}u\phi\ud{\gamma}\\
\leq&\e\nm{\phi}_{L^2([0,t]\times\Gamma)}\bigg(\nm{u}_{L^2([0,t]\times\Gamma^+)}+\nm{\g}_{L^2([0,t]\times\Gamma^-)}\bigg)\no\\
\leq&\e\nm{\phi}_{H^1([0,t]\times\Omega\times\s^2)}\bigg(\nm{u}_{L^2([0,t]\times\Gamma^+)}+\nm{\g}_{L^2([0,t]\times\Gamma^-)}\bigg)\no\\
\leq&\e\nm{\bar u }_{L^2([0,t]\times\Omega\times\s^2)}\bigg(\nm{u}_{L^2([0,t]\times\Gamma^+)}+\nm{\g}_{L^2([0,t]\times\Gamma^-)}\bigg).\no
\end{align}
Also, using H\"older's inequality and \eqref{wt 5}, we obtain
\begin{align}\label{wt 10}
\abs{\int_0^t\iint_{\Omega\times\s^2}(u -\bar u )\phi}\leq
C \nm{\bar
u }_{L^2([0,t]\times\Omega\times\s^2)}\nm{u -\bar
u }_{L^2([0,t]\times\Omega\times\s^2)},
\end{align}
and
\begin{align}\label{wt 11}
\abs{\int_0^t\iint_{\Omega\times\s^2}\ss\phi}\leq C \nm{\bar
u }_{L^2([0,t]\times\Omega\times\s^2)}\nm{\ss}_{L^2([0,t]\times\Omega\times\s^2)}.
\end{align}
On the other hand, using H\"older's inequality and \eqref{wt 5}, we may directly estimate
\begin{align}\label{wt 12}
\abs{\e^2\iint_{\Omega\times\s^2}u (t)\phi(t)}\leq& C\e^2\nm{\phi(t)}_{L^2(\Omega\times\s^2)}\nm{u (t)}_{L^2(\Omega\times\s^2)}\\
\leq& C\e^2\nm{\bar
u (t)}_{L^2(\Omega\times\s^2)}\nm{u (t)}_{L^2(\Omega\times\s^2)}\leq C\e^2\nm{u (t)}^2_{L^2(\Omega\times\s^2)}.\no
\end{align}
Similarly, we know
\begin{align}\label{wt 13}
\abs{\e^2\iint_{\Omega\times\s^2}u (0)\phi(0)}\leq  C\e^2\nm{h}^2_{L^2(\Omega\times\s^2)}.
\end{align}
Then the only remaining term in \eqref{wt 1} is
\begin{align}\label{wt 14}
&\abs{-\e^2\int_0^t\iint_{\Omega\times\s^2}\dt\phi u} =\abs{\e^2\int_0^t\iint_{\Omega\times\s^2}\dt\phi (u -\bar u )}\\
&\leq\e^2\nm{\dt\phi}_{L^2([0,t]\times\Omega\times\s^2)}\nm{u -\bar
u }_{L^2([0,t]\times\Omega\times\s^2)}\leq\e^2\nm{\dt\nx\zeta}_{L^2([0,t]\times\Omega\times\s^2)}\nm{u -\bar
u }_{L^2([0,t]\times\Omega\times\s^2)}.\no
\end{align}
Now we have to tackle $\nm{\dt\nx\zeta}_{L^2([0,t]\times\Omega\times\s^2)}$. We will implement difference quotient.\\
\ \\
For test function $\phi(\vx,\vw)$ which is independent of time $t$,
in time interval $[t-\delta,t]$ the weak formulation in
\eqref{wt 1} can be simplified as
\begin{align}\label{wt 15}
&\e^2\iint_{\Omega\times\s^2}u (t)\phi-\e^2\iint_{\Omega\times\s^2}u ({t-\delta})\phi
-\e\int_{t-\delta}^t\iint_{\Omega\times\s^2}(\vw\cdot\nx\phi)u +\int_{t-\delta}^t\iint_{\Omega\times\s^2}(u -\bar
u )\phi\\
=&-\e\int_{t-\delta}^t\int_{\Gamma}u \phi\ud{\gamma}
+\int_{t-\delta}^t\iint_{\Omega\times\s^2}\ss\phi.\no
\end{align}
Taking difference quotient as $\delta\rt0$, we know
\begin{align}
\frac{\e^2\displaystyle\iint_{\Omega\times\s^2}u (t)\phi-\e^2\displaystyle\iint_{\Omega\times\s^2}u ({t-\delta})\phi}{\delta}\rt
\e^2\iint_{\Omega\times\s^2}\dt
u (t)\phi.
\end{align}
Then \eqref{wt 15} can be simplified into
\begin{align}\label{wt 16}
\e^2\iint_{\Omega\times\s^2}\dt u (t)\phi
=\e\iint_{\Omega\times\s^2}(\vw\cdot\nx\phi)u (t)-\iint_{\Omega\times\s^2}\Big(u(t) -\bar
u(t)\Big)\phi-\e\int_{\Gamma}u (t)\phi\ud{\gamma}+\iint_{\Omega\times\s^2}\ss(t)\phi.
\end{align}
For fixed $t$, taking $\phi=-\Phi(\vx)$ which satisfies
\begin{align}
\left\{
\begin{array}{l}
\Delta_x \Phi=\dt\bar u(t)\ \ \text{in}\ \
\Omega,\\\rule{0ex}{1.0em} \Phi(t)=0\ \ \text{on}\ \ \p\Omega,
\end{array}
\right.
\end{align}
which further implies $\Phi=\dt\zeta$.
Then the left-hand side of \eqref{wt 16} is actually
\begin{align}\label{wt 17}
LHS=&-\e^2\iint_{\Omega\times\s^2}\Phi\dt u (t)=-\e^2\iint_{\Omega\times\s^2}\Phi\dt\bar u
=-\e^2\iint_{\Omega\times\s^2}\Phi\Delta_x\Phi=\e^2\iint_{\Omega\times\s^2}\abs{\nx\Phi}^2\\
=&\e^2\nm{\dt\nx\zeta(t)}_{L^2(\Omega\times\s^2)}^2.\no
\end{align}
By a similar argument as above and Poincar\'e's inequality, the right-hand side of
\eqref{wt 16} can be bounded as
\begin{align}\label{wt 18}
RHS\leq& \nm{\dt\nx\zeta(t)}_{L^2(\Omega\times\s^2)}\bigg(\nm{u (t)-\bar
u (t)}_{L^2(\Omega\times\s^2)}
+\nm{\ss(t)}_{L^2(\Omega\times\s^2)}\bigg).
\end{align}
Note that the boundary terms vanish due to the construction of $\Phi$. Therefore, combining \eqref{wt 17} and \eqref{wt 18}, we have
\begin{align}\label{wt 19}
\e^2\nm{\dt\nx\zeta(t)}_{L^2(\Omega\times\s^2)}\leq&
\nm{u (t)-\bar u (t)}_{L^2(\Omega\times\s^2)}
+\nm{\ss(t)}_{L^2(\Omega\times\s^2)}.
\end{align}
\eqref{wt 19} is true for all $t$. Then we can further integrate it over $[0,t]$ to obtain
\begin{align}\label{wt 20}
\e^2\nm{\dt\nx\zeta}_{L^2([0,t]\times\Omega\times\s^2)}
\leq& \nm{u -\bar
u }_{L^2([0,t]\times\Omega\times\s^2)}
+\nm{\ss}_{L^2([0,t]\times\Omega\times\s^2)}.
\end{align}
Collecting terms in \eqref{wt 7}, \eqref{wt 8}, \eqref{wt 9}, \eqref{wt 10}, \eqref{wt 11}, \eqref{wt 12}, \eqref{wt 13},\eqref{wt 14} and \eqref{wt 20}, we have
\begin{align}\label{wt 21}
&\e\nm{\bar u }_{L^2([0,t]\times\Omega\times\s^2)}^2\\
\leq&C\Bigg(\e\nm{u -\bar
u }_{L^2([0,t]\times\Omega\times\s^2)}\nm{\bar
u }_{L^2([0,t]\times\Omega\times\s^2)}+\e\nm{\bar u }_{L^2([0,t]\times\Omega\times\s^2)}\bigg(\nm{u}_{L^2([0,t]\times\Gamma^+)}+\nm{\g}_{L^2([0,t]\times\Gamma^-)}\bigg)\no\\
&+\nm{\bar
u }_{L^2([0,t]\times\Omega\times\s^2)}\nm{u -\bar
u }_{L^2([0,t]\times\Omega\times\s^2)}+\nm{\bar
u }_{L^2([0,t]\times\Omega\times\s^2)}\nm{\ss}_{L^2([0,t]\times\Omega\times\s^2)}\no\\
&+\e^2\nm{u (t)}^2_{L^2(\Omega\times\s^2)}+\e^2\nm{h}^2_{L^2(\Omega\times\s^2)}+\nm{u -\bar
u }_{L^2([0,t]\times\Omega\times\s^2)}\bigg(\nm{u -\bar
u }_{L^2([0,t]\times\Omega\times\s^2)}
+\nm{\ss}_{L^2([0,t]\times\Omega\times\s^2)}\bigg)\Bigg)
.\no
\end{align}
Applying Cauchy's inequality to each term on the right-hand side of \eqref{wt 21}, we
obtain
\begin{align}\label{wt 22}
\e\nm{\bar u }_{L^2([0,t]\times\Omega\times\s^2)}
\leq& C \bigg(\nm{u -\bar
u }_{L^2([0,t]\times\Omega\times\s^2)}+\e^{\frac{3}{2}}\nm{u (t)}_{L^2(\Omega\times\s^2)}+\e\nm{u}_{L^2([0,t]\times\Gamma^+)}\\
&+\nm{\ss}_{L^2([0,t]\times\Omega\times\s^2)}+\e^{\frac{3}{2}}\nm{h}_{L^2(\Omega\times\s^2)}+\e\tm{\g}{[0,t]\times\Gamma^-}\bigg).\no
\end{align}
\ \\
Step 2: Energy Estimates.\\
In the weak formulation \eqref{wt 1}, we may take
the test function $\phi=u $ to get the energy estimate
\begin{align}\label{wt 23}
\\
\frac{\e^2}{2}\nm{u (t)}_{L^2(\Omega\times\s^2)}^2+\frac{\e}{2}\nm{u}^2_{L^2([0,t]\times\Gamma^+)}+\nm{u -\bar
u }_{L^2([0,t]\times\Omega\times\s^2)}^2
=\int_0^t\iint_{\Omega\times\s^2}\ss u +\frac{\e^2}{2}\nm{\h}_{L^2(\Omega\times\s^2)}^2+\frac{\e}{2}\nm{\g}_{L^2([0,t]\times\Gamma^-)}^2.\no
\end{align}
On the other hand, we can square on both sides of
\eqref{wt 22} to obtain
\begin{align}\label{wt 24}
\e^2\nm{\bar u }_{L^2([0,t]\times\Omega\times\s^2)}^2
\leq& C \bigg(\nm{u -\bar
u }_{L^2([0,t]\times\Omega\times\s^2)}^2+\e^{3}\nm{u (t)}_{L^2(\Omega\times\s^2)}^2+\e^2\nm{u}_{L^2([0,t]\times\Gamma^+)}^2\\
&+\nm{\ss}_{L^2([0,t]\times\Omega\times\s^2)}^2+\e^{3}\nm{h}_{L^2(\Omega\times\s^2)}^2+\e^2\tm{\g}{[0,t]\times\Gamma^-}^2\bigg).\no
\end{align}
Multiplying \eqref{wt 24} by a sufficiently small constant and adding it to \eqref{wt 23} to absorb $\nm{u -\bar
u }_{L^2([0,t]\times\Omega\times\s^2)}^2$, $\e^{3}\nm{u (t)}_{L^2(\Omega\times\s^2)}^2$,
and $\e^2\nm{u}_{L^2([0,t]\times\Gamma^+)}^2$, we deduce
\begin{align}
&\e^2\nm{u (t)}_{L^2(\Omega\times\s^2)}^2+\e\nm{u}_{L^2([0,t]\times\Gamma^+)}^2+\e^2\nm{\bar
u }_{L^2([0,t]\times\Omega\times\s^2)}^2+\nm{u -\bar
u }_{L^2([0,t]\times\Omega\times\s^2)}^2\\
\leq&
C \bigg(\tm{\ss}{[0,t]\times\Omega\times\s^2}^2+
\int_0^t\iint_{\Omega\times\s^2}\ss u +\e^2\nm{\h}_{L^2(\Omega\times\s^2)}^2+\e\nm{\g}_{L^2([0,t]\times\Gamma^-)}^2\bigg).\no
\end{align}
Hence, we have
\begin{align}\label{wt 25}
&\e^2\nm{u (t)}_{L^2(\Omega\times\s^2)}^2+\e\nm{u}_{L^2([0,t]\times\Gamma^+)}^2+\e^2\nm{u }_{L^2([0,t]\times\Omega\times\s^2)}^2\\
\leq&
C \bigg(\tm{\ss}{[0,t]\times\Omega\times\s^2}^2+
\int_0^t\iint_{\Omega\times\s^2}\ss u +\e^2\nm{\h}_{L^2(\Omega\times\s^2)}^2+\e\nm{\g}_{L^2([0,t]\times\Gamma^-)}^2\bigg).\no
\end{align}
A direct application of Cauchy's inequality leads to
\begin{align}\label{wt 26}
\int_0^t\iint_{\Omega\times\s^2}\ss u \leq\frac{1}{4C_0\e^2}\tm{\ss}{[0,t]\times\Omega\times\s^2}^2+C_0\e^2\tm{u }{[0,t]\times\Omega\times\s^2}^2.
\end{align}
Taking $C_0$ sufficiently small and inserting \eqref{wt 26} into \eqref{wt 25}, we obtain
\begin{align}\label{wt 27}
&\e^2\nm{u (t)}_{L^2(\Omega\times\s^2)}^2+\e\nm{u}_{L^2([0,t]\times\Gamma^+)}^2+\e^2\nm{u }_{L^2([0,t]\times\Omega\times\s^2)}^2\\
\leq&
C \bigg(\frac{1}{\e^2}\tm{\ss}{[0,t]\times\Omega\times\s^2}^2+
\e^2\nm{\h}_{L^2(\Omega\times\s^2)}^2+\e\nm{\g}_{L^2([0,t]\times\Gamma^-)}^2\bigg).\no
\end{align}
Then we have
\begin{align}\label{wt 28}
&\nm{u (t)}_{L^2(\Omega\times\s^2)}+\frac{1}{\e^{\frac{1}{2}}}\nm{u}_{L^2([0,t]\times\Gamma^+)}+\nm{u }_{L^2([0,t]\times\Omega\times\s^2)}\\
\leq&
C \bigg(\frac{1}{\e^2}\tm{\ss}{[0,t]\times\Omega\times\s^2}+
\nm{\h}_{L^2(\Omega\times\s^2)}+\frac{1}{\e^{\frac{1}{2}}}\nm{\g}_{L^2([0,t]\times\Gamma^-)}\bigg).\no
\end{align}
\end{proof}

\subsection{$L^{\infty}$ Estimate - First Round}

\begin{theorem}\label{LI estimate.'}
Assume $\ss(t,\vx,\vw)\in
L^{\infty}([0,t]\times\Omega\times\s^2)$, $\h(\vx,\vw)\in
L^{\infty}(\Omega\times\s^2)$ and $\g(t,x_0,\vw)\in
L^{\infty}([0,t]\times\Gamma^-)$. Then the solution $u(t,\vx,\vw)$ to the neutron transport
equation \eqref{neutron.} satisfies
\begin{align}
\im{u}{[0,t]\times\Omega\times\s^2}\leq& C \bigg(\frac{1}{\e^3}\nm{\ss}_{L^2([0,t]\times\Omega\times\s^2)}
+\im{\ss}{[0,t]\times\Omega\times\s^2}\\
&+\frac{1}{\e}\nm{\h}_{L^2(\Omega\times\s^2)}+\im{\h}{[0,t]\times\Omega\times\s^2}\bigg)\no\\
&+\frac{1}{\e^2}\nm{\g}_{L^2([0,t]\times\Gamma^-)}+\im{g}{[0,t]\times\Gamma^-}.\no
\end{align}
\end{theorem}
\begin{proof}
\ \\
Step 1: Mild formulation.\\
The characteristics $\Big(T(s),X(s),W(s)\Big)$ for $s\in\r$ of the equation \eqref{neutron.} which goes through $(t,\vx,\vw)$ is defined by
\begin{eqnarray}\label{character}
\Big(T(0),X(0),W(0)\Big)=(t,\vx,\vw),\quad
\dfrac{\ud{T(s)}}{\ud{s}}=\e^2,\quad
\dfrac{\ud{X(s)}}{\ud{s}}=\e W(s),\quad
\dfrac{\ud{W(s)}}{\ud{s}}=0,
\end{eqnarray}
which implies
\begin{eqnarray}
T(s)=t+\e^2s,\quad
X(s)=\vx+(\e\vw)s,\quad
W(s)=\vw.
\end{eqnarray}
We rewrite the equation \eqref{neutron.} along the characteristics as
\begin{align}\label{wt 31}
u(t,\vx,\vw)
=&{\bf 1}_{\{t> \e^2t_b\}}\g(t-\e^2t_b,\vx-{\e\vw}
t_b,\vw)\ue^{-t_b}+{\bf 1}_{\{t=\e^2t_b\}}\h(\vx-{\e\vw}
t_b,\vw)\ue^{- t_b}\\
&+\int_{0}^{t_b}\ss(t-\e^2{s},\vx-\e{s}\vw,\vw)\ue^{-{s}}\ud{s}+\int_{0}^{t_b}\bar u(t-\e^2{s},\vx-\e{s}\vw)\ue^{-{s}}\ud{s}\no\\
=&{\bf 1}_{\{t> \e^2t_b\}}\g(t-\e^2t_b,\vx-{\e\vw}
t_b,\vw)\ue^{-t_b}+{\bf 1}_{\{t=\e^2t_b\}}\h(\vx-{\e\vw}
t_b,\vw)\ue^{- t_b}\no\\
&+\int_{0}^{t_b}\ss(t-\e^2{s},\vx-\e{s}\vw,\vw)\ue^{-{s}}\ud{s}+\frac{1}{4\pi}\int_{0}^{t_b}\bigg(\int_{\s^2}u(t-\e^2{s},\vx-\e{s}\vw,\vw_t)\ud{\vw_t}\bigg)\ue^{-{s}}\ud{s},\no
\end{align}
where the backward exit time $t_b\in\left[0,t\e^{-2}\right]$ is defined as
\begin{equation}\label{exit time}
t_b(t,\vx,\vw)=\inf\left\{s\geq0: (t-\e^2s,\vx-\e s\vw,\vw)\in([0,t]\times\Gamma^-)\cup(\{0\}\times\Omega\times\s^2)\right\},
\end{equation}
and $\vw_t$ is a dummy variable for velocity.\\
\ \\
For the last term in \eqref{wt 31}, we rewrite $u(t-\e^2{s},\vx-\e{s}\vw,\vw_t)$ by tracking back along the characteristics again to obtain
\begin{align}\label{wt 32}
u(t,\vx,\vw)
=&{\bf 1}_{\{t> \e^2t_b\}}\g(t-\e^2t_b,\vx-{\e\vw}
t_b,\vw)\ue^{-t_b}+{\bf 1}_{\{t=\e^2t_b\}}\h(\vx-{\e\vw}
t_b,\vw)\ue^{- t_b}+\int_{0}^{t_b}\ss(t',\vx_t,\vw)\ue^{-{s}}\ud{s}\\
&+\frac{1}{4\pi}\int_{0}^{t_b}\bigg(\int_{\s^2}{\bf 1}_{\{t'> \e^2s_b\}}\g(t'-\e^2s_b,\vx_t-{\e\vw}
s_b,\vw_t)\ue^{-s_b}\ud{\vw_t}\bigg)\ue^{-{s}}\ud{s},\no\\
&+\frac{1}{4\pi}\int_{0}^{t_b}\bigg(\int_{\s^2}{\bf 1}_{\{t'=\e^2s_b\}}\h(\vx_t-{\e\vw}
s_b,\vw_t)\ue^{- s_b}\ud{\vw_t}\bigg)\ue^{-{s}}\ud{s},\no\\
&+\frac{1}{4\pi}\int_{0}^{t_b}\Bigg(\int_{\s^2}\bigg(\int_0^{s_b}\ss(t'-\e^2{r},\vx_t-\e{r}\vw_t,\vw_t)
\ue^{-{r}}\ud r\bigg)\ud{\vw_t}\Bigg)\ue^{-{s}}\ud{s}\no\\
&+\frac{1}{4\pi}\int_{0}^{t_b}\Bigg(\int_{\s^2}\bigg(\int_0^{s_b}\bar u(t'-\e^2{r},\vx_t-\e{r}\vw_t)
\ue^{-{r}}\ud r\bigg)\ud{\vw_t}\Bigg)\ue^{-{s}}\ud{s},\no
\end{align}
where the exiting time from $(t',\vx_t,\vw_t)=(t-\e^2s,\vx-\e{s}\vw,\vw_t)$ is defined as
\begin{align}
s_b(t,\vx,\vw;s,\vw_t)=\inf\left\{r\geq0: (t'-\e^2r,\vx_t-\e
r\vw_t,\vw_t)\in([0,t]\times\Gamma^-)\cup(\{0\}\times\Omega\times\s^2)\right\}.
\end{align}
We may reiterate \eqref{wt 32} again along the characteristics to obtain
\begin{align}\label{wt 32'}
&u(t,\vx,\vw)\\
=&{\bf 1}_{\{t> \e^2t_b\}}\g(t-\e^2t_b,\vx-{\e\vw}
t_b,\vw)\ue^{-t_b}+{\bf 1}_{\{t=\e^2t_b\}}\h(\vx-{\e\vw}
t_b,\vw)\ue^{- t_b}+\int_{0}^{t_b}\ss(t',\vx_t,\vw)\ue^{-{s}}\ud{s}\no\\
&+\frac{1}{4\pi}\int_{0}^{t_b}\bigg(\int_{\s^2}{\bf 1}_{\{t'> \e^2s_b\}}\g(t'-\e^2s_b,\vx_t-{\e\vw}
s_b,\vw_t)\ue^{-s_b}\ud{\vw_t}\bigg)\ue^{-{s}}\ud{s},\no\\
&+\frac{1}{4\pi}\int_{0}^{t_b}\bigg(\int_{\s^2}{\bf 1}_{\{t'=\e^2s_b\}}\h(\vx_t-{\e\vw}
s_b,\vw_t)\ue^{- s_b}\ud{\vw_t}\bigg)\ue^{-{s}}\ud{s},\no\\
&+\frac{1}{4\pi}\int_{0}^{t_b}\Bigg(\int_{\s^2}\bigg(\int_0^{s_b}\ss(t'',\vx_s,\vw_t)
\ue^{-{r}}\ud r\bigg)\ud{\vw_t}\Bigg)\ue^{-{s}}\ud{s}\no\\
&+\left(\frac{1}{4\pi}\right)^2\int_{0}^{t_b}\Bigg(\int_{\s^2}\bigg(\int_0^{s_b}\bigg(\int_{\s^2}{\bf 1}_{\{t''> \e^2r_b\}}\g(t''-\e^2r_b,\vx_s-\e\vw_sr_b,\vw_s)\ue^{-r_b}\ud{\vw_s}\bigg)
\ue^{-{r}}\ud r\bigg)\ud{\vw_t}\Bigg)\ue^{-{s}}\ud{s}\no\\
&+\left(\frac{1}{4\pi}\right)^2\int_{0}^{t_b}\Bigg(\int_{\s^2}\bigg(\int_0^{s_b}\bigg(\int_{\s^2}{\bf 1}_{\{t''=\e^2r_b\}}\h(\vx_s-{\e\vw_s}
r_b,\vw_s)\ue^{- r_b}\ud{\vw_s}\bigg)
\ue^{-{r}}\ud r\bigg)\ud{\vw_t}\Bigg)\ue^{-{s}}\ud{s}\no\\
&+\left(\frac{1}{4\pi}\right)^2\int_{0}^{t_b}\Bigg(\int_{\s^2}\bigg(\int_0^{s_b}\bigg(\int_{\s^2}\int_0^{r_b}\ss(t''-\e^2q,\vx_s-\e q\vw_s,\vw_s)
\ue^{-{q}}\ud q\ud\vw_s\bigg)
\ue^{-{r}}\ud r\bigg)\ud{\vw_t}\Bigg)\ue^{-{s}}\ud{s}\no\\
&+\left(\frac{1}{4\pi}\right)^2\int_{0}^{t_b}\Bigg(\int_{\s^2}\bigg(\int_0^{s_b}\bigg(\int_{\s^2}\int_0^{r_b}\bar u(t''-\e^2q,\vx_s-\e q\vw_s)
\ue^{-{q}}\ud q\ud\vw_s\bigg)
\ue^{-{r}}\ud r\bigg)\ud{\vw_t}\Bigg)\ue^{-{s}}\ud{s},\no
\end{align}
where the dummy variable $\vw_s$ and the exiting time from $(t'',\vx_s,\vw_s)=(t'-\e^2r,\vx_t-\e r\vw_t,\vw_s)$ is defined as
\begin{align}
r_b(t,\vx,\vw;s,\vw_t;r,\vw_s)=\inf\left\{q\geq0: (t''-\e^2q,\vx_s-\e q\vw_s,\vw_s)\in([0,t]\times\Gamma^-)\cup(\{0\}\times\Omega\times\s^2)\right\}.
\end{align}
\ \\
Step 2: Estimates of all but the last term in \eqref{wt 32'}.\\
Note the fact that $0\leq s\leq t_b$ and $0\leq r\leq s_b$. We can directly estimate
\begin{align}\label{wt 33}
\abs{{\bf 1}_{\{t> \e^2t_b\}}\g(t-\e^2t_b,\vx-{\e\vw}
t_b,\vw)\ue^{-t_b}}\leq \im{\g}{[0,t]\times\Gamma^-},
\end{align}
\begin{align}\label{wt 34}
\abs{{\bf 1}_{\{t=\e^2t_b\}}\h(\vx-{\e\vw}
t_b,\vw)\ue^{- t_b}}\leq \im{\h}{\Omega\times\s^2},
\end{align}
\begin{align}\label{wt 35}
\abs{\int_{0}^{t_b}\ss(t-\e^2{s},\vx-\e{s}\vw,\vw)\ue^{-{s}}\ud{s}}\leq\im{\ss}{[0,t]\times\Omega\times\s^2},
\end{align}
\begin{align}\label{wt 36}
\abs{\frac{1}{4\pi}\int_{0}^{t_b}\bigg(\int_{\s^2}{\bf 1}_{\{t'> \e^2s_b\}}\g(t'-\e^2s_b,\vx_t-{\e\vw}
s_b,\vw_t)\ue^{-s_b}\ud{\vw_t}\bigg)\ue^{-{s}}\ud{s}}
\leq\im{\g}{[0,t]\times\Gamma^-},
\end{align}
\begin{align}\label{wt 37}
\abs{\frac{1}{4\pi}\int_{0}^{t_b}\bigg(\int_{\s^2}{\bf 1}_{\{t'=\e^2s_b\}}\h(\vx_t-{\e\vw}
s_b,\vw_t)\ue^{- s_b}\ud{\vw_t}\bigg)\ue^{-{s}}\ud{s}}
\leq\im{\h}{\Omega\times\s^2},
\end{align}
\begin{align}\label{wt 38}
\abs{\frac{1}{4\pi}\int_{0}^{t_b}\Bigg(\int_{\s^2}\bigg(\int_0^{s_b}\ss(t'',\vx_s,\vw_t)
\ue^{-{r}}\ud r\bigg)\ud{\vw_t}\Bigg)\ue^{-{s}}\ud{s}}
\leq\im{\ss}{[0,t]\times\Omega\times\s^2}.
\end{align}
\begin{align}\label{wt 39}
\\
\abs{\left(\frac{1}{4\pi}\right)^2\int_{0}^{t_b}\Bigg(\int_{\s^2}\bigg(\int_0^{s_b}\bigg(\int_{\s^2}{\bf 1}_{\{t''> \e^2r_b\}}\g(t''-\e^2r_b,\vx_s-\e\vw_sr_b,\vw_s)\ue^{-r_b}\ud{\vw_s}\bigg)
\ue^{-{r}}\ud r\bigg)\ud{\vw_t}\Bigg)\ue^{-{s}}\ud{s}}
\leq\im{\g}{[0,t]\times\Gamma^-},\no
\end{align}
\begin{align}\label{wt 40}
\\
\abs{\left(\frac{1}{4\pi}\right)^2\int_{0}^{t_b}\Bigg(\int_{\s^2}\bigg(\int_0^{s_b}\bigg(\int_{\s^2}{\bf 1}_{\{t''=\e^2r_b\}}\h(\vx_s-{\e\vw_s}
r_b,\vw_s)\ue^{- r_b}\ud{\vw_s}\bigg)
\ue^{-{r}}\ud r\bigg)\ud{\vw_t}\Bigg)\ue^{-{s}}\ud{s}}
\leq\im{\h}{\Omega\times\s^2},\no
\end{align}
\begin{align}\label{wt 41}
\\
\abs{\left(\frac{1}{4\pi}\right)^2\int_{0}^{t_b}\Bigg(\int_{\s^2}\bigg(\int_0^{s_b}\bigg(\int_{\s^2}\int_0^{r_b}\ss(t''-\e^2q,\vx_s-\e q\vw_s,\vw_s)
\ue^{-{q}}\ud q\ud\vw_s\bigg)
\ue^{-{r}}\ud r\bigg)\ud{\vw_t}\Bigg)\ue^{-{s}}\ud{s}}
\leq\im{\ss}{[0,t]\times\Omega\times\s^2}.\no
\end{align}
\ \\
Step 3: Estimates of the last term in \eqref{wt 32'}.\\
Now we decompose the last term in \eqref{wt 32} as
\begin{align}
\int_{0}^{t_b}\int_{\s^2}\int_0^{s_b}\int_{\s^2}\int_0^{r_b}=&\int_{0}^{t_b}\int_{\s^2}\int_{0\leq r\leq\d}\int_{\s^2}\int_0^{r_b}+\int_{0}^{t_b}\int_{\s^2}\int_{ r\geq\d}\int_{\s^2}\int_0^{r_b}=I_1+I_1^{\ast},
\end{align}
for some $0<\delta<<1$ to be determined later. Since $I_1$ contains an integral in a very small region, we may directly estimate
\begin{align}\label{wt 42}
\abs{I_1}
\leq&\nm{u}_{L^{\infty}([0,t]\times\Omega\times\s^2)}\bigg(\int_{r\leq\delta}\ue^{-{r}}
\ud{r}\bigg)\leq C\delta\nm{u}_{L^{\infty}([0,t]\times\Omega\times\s^2)}.
\end{align}
We may further decompose
\begin{align}
I_1^{\ast}=\int_{0}^{t_b}\int_{\s^2}\int_{ r\geq\d}\int_{\s^2}\int_{0\leq q\leq\d}+\int_{0}^{t_b}\int_{\s^2}\int_{ r\geq\d}\int_{\s^2}\int_{q\geq\d}=I_2+I_2^{\ast}.
\end{align}
Similarly, we may bound
\begin{align}\label{wt 43}
\abs{I_2}
\leq&\nm{u}_{L^{\infty}([0,t]\times\Omega\times\s^2)}\bigg(\int_{q\leq\delta}\ue^{-{q}}
\ud{r}\bigg)\leq C\delta\nm{u}_{L^{\infty}([0,t]\times\Omega\times\s^2)}.
\end{align}
Then we need to handle the most complicated term $I_2^{\ast}$,
\begin{align}
\abs{I_2^{\ast}}\leq&C\int_{0}^{t_b}\int_{\s^2}\int_{ r\geq\d}\int_{\s^2}\int_{q\geq\d}\abs{\bar u(t''-\e^2q,\vx_t-\e r\vw_t-\e q\vw_s)}\ue^{-{q}}\ue^{-{r}}\ue^{-{s}}\ud q\ud\vw_s\ud{r}\ud{\vw_t}\Bigg)\ud{s}.
\end{align}
By the definition of $t_b$ and $s_b$, we always have
\begin{align}
\vx_t-\e r\vw_t-\e q\vw_s\in \Omega.
\end{align}
Hence, we may introduce the indicator function ${\bf{1}}_{\Omega}$ and apply H\"older's inequality
to obtain
\begin{align}\label{wt 44}
\\
\abs{I_2^{\ast}}\leq&C\int_{0}^{t_b}\Bigg(\int_{\s^2}\int_{ r\geq\d}\int_{\s^2}\int_{q\geq\d}{\bf{1}}_{\Omega}(\vx_t-\e r\vw_t-\e q\vw_s)\abs{\bar u(t''-\e^2q,\vx_t-\e r\vw_t-\e q\vw_s)}\ue^{-{q}}\ue^{-{r}}\ud q\ud\vw_s\ud{r}\ud{\vw_t}\Bigg)\ue^{-{s}}\ud{s}\no\\
\leq&C\int_{0}^{t_b}\Bigg(\bigg(\int_{\s^2}\int_{ r\geq\d}\int_{\s^2}\int_{q\geq\d}{\bf{1}}_{\Omega}(\vx_t-\e r\vw_t-\e q\vw_s)\abs{\bar u(t''-\e^2q,\vx_t-\e r\vw_t-\e q\vw_s)}^2\ue^{-{q}}\ue^{-{r}}\ud q\ud\vw_s\ud{r}\ud{\vw_t}\bigg)^{\frac{1}{2}}\no\\
&\times\bigg(\int_{\s^2}\int_{ r\geq\d}\int_{\s^2}\int_{q\geq\d}{\bf{1}}_{\Omega}(\vx_t-\e r\vw_t-\e q\vw_s)\ue^{-{q}}\ue^{-{r}}\ud q\ud\vw_s\ud{r}\ud{\vw_t}\bigg)^{\frac{1}{2}}\Bigg)\ue^{-{s}}\ud{s}\no\\
\leq&C\int_{0}^{t_b}\Bigg(\bigg(\int_{\s^2}\int_{ r\geq\d}\int_{\s^2}\int_{q\geq\d}{\bf{1}}_{\Omega}(\vx_t-\e r\vw_t-\e q\vw_s)\abs{\bar u(t''-\e^2q,\vx_t-\e r\vw_t-\e q\vw_s)}^2\ue^{-{q}}\ue^{-{r}}\ud q\ud\vw_s\ud{r}\ud{\vw_t}\bigg)^{\frac{1}{2}}\Bigg)\ue^{-{s}}\ud{s}.\no
\end{align}
Note $\vw_t,\vw_s\in\s^2$, which can be parameterized as
\begin{align}
\vw_t=&(\sin\phi\cos\psi,\sin\phi\sin\psi,\cos\phi),\\
\vw_s=&(\sin\phi'\cos\psi',\sin\phi'\sin\psi',\cos\phi'),
\end{align}
for $\phi,\phi'\in[0,\pi]$ and $\psi,\psi'\in[0,2\pi]$. Hence, we may write the integral
\begin{align}\label{wt 45}
\int_{\s^2}\cdots\ud\vw_t=\int_0^{2\pi}\int_0^{\pi}\cdots\sin\phi\ud\phi\ud\psi,\quad \int_{\s^2}\cdots\ud\vw_s=\int_0^{2\pi}\int_0^{\pi}\cdots\sin\phi'\ud\phi'\ud\psi'
\end{align}
where $\sin\phi$ and $\sin\phi'$ are the Jacobian of spherical coordinates. Then we define the change of variable
$[0,\pi]\times[0,2\pi]\times[0,2\pi]\rt \Omega: (\phi,\psi,\psi')\rt(y_1,y_2,y_3)=\vec
y=\vx_t-\e r\vw_t-\e q\vw_s$, i.e.
\begin{align}\label{wt 46}
\left\{
\begin{array}{rcl}
y_1&=&(x_t)_1-\e {r}\sin\phi\cos\psi-\e {q}\sin\phi'\cos\psi',\\
y_2&=&(x_t)_2-\e {r}\sin\phi\sin\psi-\e {q}\sin\phi'\sin\psi',\\
y_3&=&(x_t)_3-\e {r}\cos\phi-\e {q}\cos\phi'.
\end{array}
\right.
\end{align}
The Jacobian is
\begin{align}
\abs{\frac{\p(y_1,y_2,y_3)}{\p(\phi,\psi,\psi')}}=&\abs{\abs{\begin{array}{ccc}
-\e{r}\cos\phi\cos\psi&\e{r}\sin\phi\sin\psi&-\e q\sin\phi'\sin\psi'\\
-\e{r}\cos\phi\sin\psi&-\e{r}\sin\phi\cos\psi&\e q\sin\phi'\cos\psi'\\
\e{r}\sin\phi&0&0
\end{array}}}=\e^3{r}^2q\sin^2\phi\sin\phi'\sin(\phi-\phi').
\end{align}
Now we consider the restriction on $r$ and $q$. Naturally, $r,q\geq\delta$ implies ${r}^2q\geq\delta^3$. Therefore, we may bound the Jacobian
\begin{align}\label{wt 48}
\abs{\frac{\p(y_1,y_2,y_3)}{\p(\phi,\psi,\psi')}}\geq \e^3\delta^3\sin^2\phi\sin\phi'\sin(\phi-\phi').
\end{align}
$\sin\phi\sin\phi'$ can be cancelled out by \eqref{wt 45}. Now, we need to further restrict $\sin\phi\sin(\phi-\phi')$. Decompose
\begin{align}
I_2^{\ast}=&\int_{0}^{t_b}\int_{0\leq\sin\phi\leq\d}\int_{ r\geq\d}\int_{\s^2}\int_{q\geq\d}
+\int_{0}^{t_b}\int_{\sin\phi\geq\d}\int_{ r\geq\d}\int_{\abs{\sin(\phi-\phi')}\leq\d}\int_{q\geq\d}\\
&+\int_{0}^{t_b}\int_{\sin\phi\geq\d}\int_{ r\geq\d}\int_{\abs{\sin(\phi-\phi')}\geq\d}\int_{q\geq\d}=I_3+I_4+I_4^{\ast}.\no
\end{align}
Similar to the estimate of $I_1$ and $I_2$ in \eqref{wt 42} and \eqref{wt 43}, we may bound
\begin{align}\label{wt 47}
\abs{I_3}
\leq&\nm{u}_{L^{\infty}([0,t]\times\Omega\times\s^2)}\bigg(\int_{0\leq\sin\phi\leq\d}\sin\phi
\ud{\phi}\bigg)\leq C\delta\nm{u}_{L^{\infty}([0,t]\times\Omega\times\s^2)},
\end{align}
and
\begin{align}\label{wt 47'}
\abs{I_4}
\leq&\nm{u}_{L^{\infty}([0,t]\times\Omega\times\s^2)}\bigg(\int_{\abs{\sin(\phi-\phi')}\leq\d}\sin\phi'
\ud{\phi'}\bigg)\leq C\delta\nm{u}_{L^{\infty}([0,t]\times\Omega\times\s^2)}.
\end{align}
Then in the estimate of $I_4^{\ast}$, using \eqref{wt 48}, we know
\begin{align}\label{wt 49}
\abs{\frac{\p(y_1,y_2,y_3)}{\p(\phi,\psi,\psi')}}\geq \e^3\delta^5\sin\phi\sin\phi'.
\end{align}
Hence, we may bound
\begin{align}\label{wt 50}
\\
\abs{I_4^{\ast}}\leq&C\int_{0}^{t_b}\Bigg(\int_{ r\geq\d}\int_{q\geq\d}\bigg(\int_{\Omega}\frac{1}{\e^3\d^5}\abs{\bar u(t''-\e^2q,\vec
y)}^2\ud{\vec y}\bigg)\ue^{-{q}}\ue^{-{r}}\ud q\ud{r}\Bigg)^{\frac{1}{2}}\ue^{-{s}}\ud{s}
\leq\frac{C}{\e^{\frac{3}{2}}\d^{\frac{5}{2}}}\nm{\bar u(t)}_{L^2([0,t]\times\Omega)}.\no
\end{align}
\ \\
Step 4: Synthesis.\\
Summarizing \eqref{wt 33}, \eqref{wt 34}, \eqref{wt 35}, \eqref{wt 36}, \eqref{wt 37}, \eqref{wt 38}, \eqref{wt 39}, \eqref{wt 40}, \eqref{wt 41}, \eqref{wt 42}, \eqref{wt 43}, \eqref{wt 47}, \eqref{wt 47'}, \eqref{wt 50}, we have shown
\begin{align}
\\
\abs{u}\leq& C\bigg(\delta\im{u}{[0,t]\times\Omega\times\s^2}+\frac{1}{\d^{\frac{5}{2}}\e^{\frac{3}{2}}}\nm{\bar u(t)}_{L^2(\Omega\times\s^2)}+\im{\ss}{[0,t]\times\Omega\times\s^2}+\im{\h}{\Omega\times\s^2}+\im{\g}{[0,t]\times\Gamma^-}\bigg).\no
\end{align}
Taking supremum over all $(t,\vx,\vw)\in[0,t]\times\Omega\times\s^2$, we obtain
\begin{align}
\nm{u}_{L^{\infty}([0,t]\times\Omega\times\s^2)}\leq& C\bigg(\delta\im{u}{[0,t]\times\Omega\times\s^2}+\frac{1}{\d^{\frac{5}{2}}\e^{\frac{3}{2}}}\nm{\bar u(t)}_{L^2(\Omega\times\s^2)}\\
&+\im{\ss}{[0,t]\times\Omega\times\s^2}+\im{\h}{\Omega\times\s^2}+\im{\g}{[0,t]\times\Gamma^-}\bigg).\no
\end{align}
Then taking $\d$ sufficiently small to absorb $C\delta
\im{u}{[0,t]\times\Omega\times\s^2}$ into the left-hand side, we get
\begin{align}
\\
\nm{u}_{L^{\infty}([0,t]\times\Omega\times\s^2)}\leq C\bigg(\frac{1}{\d^{\frac{5}{2}}\e^{\frac{3}{2}}}\nm{\bar u(t)}_{L^2(\Omega\times\s^2)}+\im{\ss}{[0,t]\times\Omega\times\s^2}+\im{\h}{\Omega\times\s^2}+\im{\g}{[0,t]\times\Gamma^-}\bigg).\no
\end{align}
Using Theorem \ref{LT estimate.}, we get
\begin{align}
\nm{u}_{L^{\infty}(\Omega\times\s^2)}\leq& C\bigg(
\frac{1}{\e^{\frac{7}{2}}}\tm{\ss}{[0,t]\times\Omega\times\s^2}+\nm{\ss}_{L^{\infty}(\Omega\times\s^2)}\\
&+\frac{1}{\e^{\frac{3}{2}}}\nm{\h}_{L^2(\Omega\times\s^2)}+\im{\h}{\Omega\times\s^2}
+\frac{1}{\e^{2}}\nm{\g}_{L^2([0,t]\times\Gamma^-)}+\nm{\g}_{L^{\infty}(\Gamma^-)}\bigg).\no
\end{align}

\end{proof}

\subsection{$L^{2m}$ Estimate}

In this subsection, we try to improve previous estimates. In the following, let $o(1)$ denote a sufficiently small constant.
\begin{lemma}\label{LN estimate.}
Assume $\ss(t,\vx,\vw)\in
L^{\infty}(\rp\times\Omega\times\s^2)$, $\h(\vx,\vw)\in
L^{\infty}(\Omega\times\s^2)$ and $\g(t,x_0,\vw)\in
L^{\infty}(\rp\times\Gamma^-)$. Then the solution $u(t,\vx,\vw)$ to the neutron transport
equation \eqref{neutron.} satisfies
\begin{align}
&\nm{u (t)}_{L^{2m}(\Omega\times\s^2)}+\frac{1}{\e^{\frac{1}{2}+\frac{1}{2m}}}\nm{u}_{L^2([0,t]\times\Gamma^+)}
+\frac{1}{\e^{\frac{1}{2m}}}\nm{u}_{L^{2m}([0,t]\times\Gamma^+)}\\
&+\frac{1}{\e^{\frac{1}{2m}}}\nm{\bar
u }_{L^{2m}([0,t]\times\Omega\times\s^2)}+\dfrac{1}{\e^{1+\frac{1}{2m}}}\nm{u -\bar
u }_{L^2([0,t]\times\Omega\times\s^2)}+\frac{1}{\e^{\frac{1}{m}}}\nm{u -\bar
u }_{L^{2m}([0,t]\times\Omega\times\s^2)}\no\\
\leq& C \bigg(\frac{1}{\e^{2+\frac{1}{2m}}}\nm{S}_{L^{\frac{2m}{2m-1}}([0,t]\times\Omega\times\s^2)}
+\frac{1}{\e^{1+\frac{1}{2m}}}\nm{\ss}_{L^2([0,t]\times\Omega\times\s^2)}+\frac{1}{\e^{1+\frac{1}{2m}}}\nm{S}_{L^{2m}([0,t]\times\Omega\times\s^2)}\no\\
&+\frac{1}{\e^{\frac{1}{2m}}}\tm{\dt\ss}{[0,t]\times\Omega\times\s^2}+
\frac{1}{\e^{\frac{1}{2m}}}\nm{\ss(0)}_{L^2(\Omega\times\s^2)}\no\\
&+
\frac{1}{\e^{\frac{1}{2m}}}\nm{\h}_{L^2(\Omega\times\s^2)}+\nm{\h}_{L^{2m}(\Omega\times\s^2)}+
\e^{1-\frac{1}{2m}}\nm{\vw\cdot\nx\h}_{L^2(\Omega\times\s^2)}\no\\
&+\frac{1}{\e^{\frac{1}{2m}}}\nm{\g}_{L^{\frac{4m}{3}}([0,t]\times\Gamma^-)}+\frac{1}{\e^{\frac{1}{2m}}}\nm{\g}_{L^{2m}([0,t]\times\Gamma^-)}
+\frac{1}{\e^{\frac{1}{2}+\frac{1}{2m}}}\nm{\g}_{L^2([0,t]\times\Gamma^-)}+\e^{\frac{3}{2}-\frac{1}{2m}}\nm{\dt\g}_{L^2([0,t]\times\Gamma^-)}\bigg).\no
\end{align}
\end{lemma}
\begin{proof}
\ \\
Step 1: Kernel Estimate.\\
Applying Lemma \eqref{wt lemma 1} to the
equation \eqref{neutron.}. Then for any
$\phi\in L^{2}([0,t]\times\Omega\times\s^2)$ satisfying
$\e\dt\phi+\vw\cdot\nx\phi\in L^2([0,t]\times\Omega\times\s^2)$
and $\phi\in L^{2}([0,t]\times\Gamma)$, we have
\begin{align}\label{wt 51}
\e^2\int_0^t\iint_{\Omega\times\s^2}\phi
\dt u-\e\int_0^t\iint_{\Omega\times\s^2}(\vw\cdot\nx\phi)u+\int_0^t\iint_{\Omega\times\s^2}(u-\bar
u)\phi
=&-\e\int_0^t\int_{\Gamma}u\phi\ud{\gamma}+\int_0^t\iint_{\Omega\times\s^2}\ss\phi.
\end{align}
Our goal is to choose a particular test function $\phi$. We first
construct an auxiliary function $\zeta(t)$. Since $u(t)\in
L^{\infty}(\Omega\times\s^2)$, it implies $\bar
u(t)\in L^{\infty}(\Omega)$. Define $\zeta(t,\vx)$ on $\Omega$
satisfying
\begin{align}\label{wt 52}
\left\{
\begin{array}{l}
\Delta_x \zeta(t)=\bar u^{2m-1}(t)\ud{\vx}\ \ \text{in}\ \
\Omega,\\\rule{0ex}{1.0em} \zeta(t)=0\ \ \text{on}\ \ \p\Omega.
\end{array}
\right.
\end{align}
In the bounded domain $\Omega$, based on the standard elliptic
estimates, there exists a unique $\xi(t)\in W^{2,\frac{2m}{2m-1}}(\Omega)$ satisfying
\begin{align}\label{wt 53}
\nm{\zeta(t)}_{W^{2,\frac{2m}{2m-1}}(\Omega)}\leq C\nm{\bar
u^{2m-1}(t)}_{L^{\frac{2m}{2m-1}}(\Omega)}= C\nm{\bar
u(t)}_{L^{2m}(\Omega)}^{2m-1}.
\end{align}
We plug the test function
\begin{align}\label{wt 54}
\phi=-\vw\cdot\nx\zeta
\end{align}
into the weak formulation \eqref{wt 51} and estimate
each term there. By Sobolev embedding theorem, we have
\begin{align}
\nm{\phi(t)}_{L^2(\Omega)}\leq& C\nm{\zeta(t)}_{H^1(\Omega)}\leq C\nm{\zeta(t)}_{W^{2,\frac{2m}{2m-1}}(\Omega)}\leq
C\nm{\bar
u(t)}_{L^{2m}(\Omega)}^{2m-1},\label{wt 55}\\
\nm{\phi(t)}_{L^{\frac{2m}{2m-1}}(\Omega)}\leq&C\nm{\zeta(t)}_{W^{1,\frac{2m}{2m-1}}(\Omega)}\leq
C\nm{\bar
u(t)}_{L^{2m}(\Omega)}^{2m-1}.\label{wt 56}
\end{align}
Note that the embedding $W^{2,\frac{2m}{2m-1}}(\Omega)\hookrightarrow H^1(\Omega)$ requires $m\leq 3$.\\
On the other hand, we decompose
\begin{align}\label{wt 57}
-\e\int_0^t\iint_{\Omega\times\s^2}(\vw\cdot\nx\phi)u=&-\e\int_0^t\iint_{\Omega\times\s^2}(\vw\cdot\nx\phi)\bar
u-\e\int_0^t\iint_{\Omega\times\s^2}(\vw\cdot\nx\phi)(u-\bar
u).
\end{align}
For the first term on the right-hand side of \eqref{wt 57}, by
\eqref{wt 52} and \eqref{wt 54}, we have
\begin{align}\label{wt 58}
&-\e\int_0^t\iint_{\Omega\times\s^2}(\vw\cdot\nx\phi)\bar
u\\
=&\e\int_0^t\iint_{\Omega\times\s^2}\bar
u\Big(w_1(w_1\p_{11}\zeta+w_2\p_{12}\zeta+w_3\p_{13}\zeta)+w_2(w_1\p_{21}\zeta+w_2\p_{22}\zeta+w_3\p_{23}\zeta)+w_3(w_1\p_{31}\zeta+w_2\p_{32}\zeta+w_3\p_{33}\zeta)\Big)\no\\
=&\e\int_0^t\iint_{\Omega\times\s^2}\bar
u\Big(w_1^2\p_{11}\zeta+w_2^2\p_{22}\zeta+w_3^2\p_{33}\zeta\Big)=\frac{4}{3}\e\pi\int_0^t\int_{\Omega}\bar u(\p_{11}\zeta+\p_{22}\zeta+\p_{33}\zeta)\no\\
=&\frac{4}{3}\e\pi\nm{\bar u}_{L^{2m}([0,t]\times\Omega)}^{2m}=\frac{1}{3}\e\nm{\bar u}_{L^{2m}([0,t]\times\Omega\times\s^2)}^{2m}\no.
\end{align}
In the second equality, above cross terms vanish due to the symmetry
of the integral over $\s^2$.\\
For the second term
on the right-hand side of \eqref{wt 57}, H\"older's inequality and \eqref{wt 56} imply
\begin{align}\label{wt 59}
&\abs{-\e\int_0^t\iint_{\Omega\times\s^2}(\vw\cdot\nx\phi)(u-\bar
u)}\leq C \e\nm{u-\bar u}_{L^{2m}([0,t]\times\Omega\times\s^2)}\bigg(\int_0^t\nm{\nx\phi}^{\frac{2m}{2m-1}}_{W^{2,\frac{2m}{2m-1}}(\Omega)}\bigg)^{\frac{2m-1}{2m}}\\
\leq&C \e\nm{u-\bar u}_{L^{2m}([0,t]\times\Omega\times\s^2)}\bigg(\int_0^t\nm{\zeta}^{\frac{2m}{2m-1}}_{W^{2,\frac{2m}{2m-1}}(\Omega)}\bigg)^{\frac{2m-1}{2m}}
\leq C \e\nm{u-\bar
u}_{L^{2m}([0,t]\times\Omega\times\s^2)}\nm{\bar
u}_{L^{2m}([0,t]\times\Omega\times\s^2)}^{2m-1}\no.
\end{align}
Based on Sobolev embedding theorem, trace theorem and \eqref{wt 55}, we have
\begin{align}\label{wt 60}
\nm{\nx\xi}_{L^{\frac{4m}{4m-3}}(\Gamma)}\leq C\nm{\nx\xi}_{W^{\frac{1}{2m},\frac{2m}{2m-1}}(\Gamma)}\leq C\nm{\nx\xi}_{W^{1,\frac{2m}{2m-1}}(\Omega)}\leq C\nm{\xi}_{W^{2,\frac{2m}{2m-1}}(\Omega)}\leq
C\nm{\bar
u}_{L^{2m}(\Omega)}^{2m-1}.
\end{align}
Using H\"older's inequality and \eqref{wt 60}, we obtain
\begin{align}\label{wt 61}
\abs{\e\int_0^t\int_{\Gamma}u\phi\ud{\gamma}}\leq &\abs{\e\int_0^t\int_{\Gamma^+}u\phi\ud{\gamma}}
+\abs{\e\int_0^t\int_{\Gamma^-}\g\phi\ud{\gamma}}\\
\leq&\e\nm{\phi}_{L^{\frac{4m}{4m-3}}(\Gamma)}\bigg(\nm{u}_{L^{\frac{4m}{3}}([0,t]\times\Gamma^+)}+\nm{\g}_{L^{\frac{4m}{3}}([0,t]\times\Gamma^-)}\bigg)\no\\
\leq&\e\nm{\bar u}_{L^{2m}([0,t]\times\Omega\times\s^2)}^{2m-1}\bigg(\nm{u}_{L^{\frac{4m}{3}}([0,t]\times\Gamma^+)}+\nm{\g}_{L^{\frac{4m}{3}}([0,t]\times\Gamma^-)}\bigg).\no
\end{align}
Also, using H\"older's inequality and \eqref{wt 55}, we have
\begin{align}\label{wt 62}
\\
\int_0^t\iint_{\Omega\times\s^2}(u-\bar u)\phi\leq&\nm{\phi}_{L^{2}([0,t]\times\Omega\times\s^2)}\nm{u-\bar
u}_{L^2([0,t]\times\Omega\times\s^2)}
\leq
C \nm{\bar
u}_{L^{2m}([0,t]\times\Omega\times\s^2)}^{2m-1}\nm{u-\bar
u}_{L^2([0,t]\times\Omega\times\s^2)},\no
\end{align}
and
\begin{align}\label{wt 63}
\int_0^t\iint_{\Omega\times\s^2}\ss\phi\leq C \leq&\nm{\phi}_{L^{2}([0,t]\times\Omega\times\s^2)}\nm{\ss}_{L^2([0,t]\times\Omega\times\s^2)}
\leq
C \nm{\bar
u}_{L^{2m}([0,t]\times\Omega\times\s^2)}^{2m-1}\nm{\ss}_{L^2([0,t]\times\Omega\times\s^2)}.
\end{align}
Then the only remaining term is
\begin{align}\label{wt 64}
\abs{\e^2\int_0^t\iint_{\Omega\times\s^2}\phi\dt u}=&\abs{\e^2\int_0^t\iint_{\Omega\times\s^2}\phi\dt(u-\bar u)}
\leq\e^2\nm{\phi}_{L^{2}([0,t]\times\Omega\times\s^2)}\nm{\dt(u-\bar u)}_{L^2([0,t]\times\Omega\times\s^2)}\\ \leq& \e^2\nm{\bar u}_{L^{2m}([0,t]\times\Omega\times\s^2)}^{2m-1}\nm{\dt(u-\bar u)}_{L^2([0,t]\times\Omega\times\s^2)}.\no
\end{align}
Now we have to tackle $\nm{\dt u}_{L^2([0,t]\times\Omega\times\s^2)}$.\\
\ \\
Taking $\dt$ on both sides of the equation \eqref{neutron.}, we obtain
\begin{align}\label{wt 65}
\left\{
\begin{array}{l}
\e^2\dt (\dt u)+\e \vw\cdot\nabla_x (\dt u)+\dt u-\dt\bar
u=\dt\ss\ \ \ \text{for}\ \
(t,\vx,\vw)\in\rp\times\Omega\times\s^2,\\\rule{0ex}{2.0em}
\dt u(0,\vx,\vw)=\tilde h=\dfrac{1}{\e^2}(-\e\vw\cdot\nx h-h+\bar h+\ss)(0,\vx,\vw)\ \ \text{for}\ \ (\vx,\vw)\in\Omega\times\s^2\\\rule{0ex}{2.0em}
\dt u(t,\vx_0,\vw)=\dt\g(t,\vx_0,\vw)\ \ \text{for}\ \ t\in\rp,\ \ \vx_0\in\p\Omega\ \ \text{and}\ \ \vw\cdot\vn<0,
\end{array}
\right.
\end{align}
Based on the proof of Theorem \ref{LT estimate.}, we have
\begin{align}\label{wt 66}
&\e\nm{\dt u (t)}_{L^2(\Omega\times\s^2)}+\e^{\frac{1}{2}}\nm{\dt u}_{L^2([0,t]\times\Gamma^+)}+\nm{\dt( u-\bar u) }_{L^2([0,t]\times\Omega\times\s^2)}\\
\leq&
C \bigg(\frac{1}{\e}\tm{\dt\ss}{[0,t]\times\Omega\times\s^2}+
\e\nm{\tilde\h}_{L^2(\Omega\times\s^2)}+\e^{\frac{1}{2}}\nm{\dt\g}_{L^2([0,t]\times\Gamma^-)}\bigg)\no\\
\leq&
C \bigg(\frac{1}{\e}\tm{\dt\ss}{[0,t]\times\Omega\times\s^2}+
\frac{1}{\e}\nm{\ss(0)}_{L^2(\Omega\times\s^2)}+
\nm{\vw\cdot\nx\h}_{L^2(\Omega\times\s^2)}+
\frac{1}{\e}\nm{\h}_{L^2(\Omega\times\s^2)}+\e^{\frac{1}{2}}\nm{\dt\g}_{L^2([0,t]\times\Gamma^-)}\bigg).\no
\end{align}
Inserting \eqref{wt 66} into \eqref{wt 64}, we obtain
\begin{align}\label{wt 67}
\abs{\e^2\int_0^t\iint_{\Omega\times\s^2}\phi\dt u}
\leq&C\nm{\bar u}_{L^{2m}([0,t]\times\Omega\times\s^2)}^{2m-1}\bigg(\e\tm{\dt\ss}{[0,t]\times\Omega\times\s^2}+
\e\nm{\ss(0)}_{L^2(\Omega\times\s^2)}\\&+
\e^2\nm{\vw\cdot\nx\h}_{L^2(\Omega\times\s^2)}+
\e\nm{\h}_{L^2(\Omega\times\s^2)}+\e^{\frac{5}{2}}\nm{\dt\g}_{L^2([0,t]\times\Gamma^-)}\bigg).\no
\end{align}
Collecting all terms in \eqref{wt 58}, \eqref{wt 59}, \eqref{wt 61}, \eqref{wt 62}, \eqref{wt 63} and \eqref{wt 67}, we
obtain
\begin{align}\label{wt 68}
&\e\nm{\bar u}_{L^{2m}([0,t]\times\Omega\times\s^2)}\\
\leq& C \bigg(\e\nm{u-\bar
u}_{L^{2m}([0,t]\times\Omega\times\s^2)}+\nm{u-\bar
u}_{L^{2}([0,t]\times\Omega\times\s^2)}+\e\nm{u}_{L^{\frac{4m}{3}}([0,t]\times\Gamma^+)}
+\nm{\ss}_{L^2([0,t]\times\Omega\times\s^2)}+\e\tm{\dt\ss}{[0,t]\times\Omega\times\s^2}\no\\
&+
\e\nm{\ss(0)}_{L^2(\Omega\times\s^2)}+
\e^2\nm{\vw\cdot\nx\h}_{L^2(\Omega\times\s^2)}+
\e\nm{\h}_{L^2(\Omega\times\s^2)}+\e\nm{\g}_{L^{\frac{4m}{3}}([0,t]\times\Gamma^-)}+\e^{\frac{5}{2}}\nm{\dt\g}_{L^2([0,t]\times\Gamma^-)}\bigg).\no
\end{align}
\ \\
Step 2: $L^{2m}$ Energy Estimate.\\
Similar to the $L^2$ estimates, in the weak formulation \eqref{wt 1}, we may take
the test function $\phi=u^{2m-1}$ to get the energy estimate
\begin{align}\label{wt 11'}
&\frac{\e^2}{2m}\nm{u (t)}_{L^{2m}(\Omega\times\s^2)}^{2m}+\frac{\e}{2m}\nm{u}^{2m}_{L^{2m}([0,t]\times\Gamma^+)}+\int_0^t\iint_{\Omega\times\s^2}(u-\bar u)u^{2m-1}\\
=&\int_0^t\iint_{\Omega\times\s^2}\ss u^{2m-1} +\frac{\e^2}{2m}\nm{\h}_{L^{2m}(\Omega\times\s^2)}^{2m}+\frac{\e}{2m}\nm{\g}_{L^{2m}([0,t]\times\Gamma^-)}^{2m}.\no
\end{align}
Here, direct computation reveals that for $1\leq m\leq 3$, 
\begin{align}\label{wt 12'}
\int_0^t\iint_{\Omega\times\s^2}(u-\bar u)u^{2m-1}\geq C\nm{u -\bar
u }_{L^{2m}([0,t]\times\Omega\times\s^2)}^{2m}.
\end{align}
Hence, inserting \eqref{wt 12'} into \eqref{wt 11'}, we obtain
\begin{align}\label{wt 3'}
&\e^2\nm{u (t)}_{L^{2m}(\Omega\times\s^2)}^{2m}+\e\nm{u}^{2m}_{L^{2m}([0,t]\times\Gamma^+)}+\nm{u -\bar
u }_{L^{2m}([0,t]\times\Omega\times\s^2)}^{2m}\\
\leq&C\bigg(\int_0^t\iint_{\Omega\times\s^2}\ss u^{2m-1} +\e^2\nm{\h}_{L^{2m}(\Omega\times\s^2)}^{2m}+\e\nm{\g}_{L^{2m}([0,t]\times\Gamma^-)}^{2m}\bigg).\no
\end{align}
Since
\begin{align}
\abs{u}^{2m-1}\leq \abs{u-\bar u}^{2m-1}+\abs{\bar u}^{2m-1},
\end{align}
we have
\begin{align}\label{wt 13'}
\abs{\int_0^t\iint_{\Omega\times\s^2}\ss u^{2m-1}}\leq C\bigg(\int_0^t\iint_{\Omega\times\s^2}\abs{\ss}\abs{u-\bar u}^{2m-1} +\int_0^t\iint_{\Omega\times\s^2}\abs{\ss} \abs{\bar u}^{2m-1}\bigg).
\end{align}
Using H\"{o}lder's inequality and Young's inequality, we obtain
\begin{align}\label{wt 1'}
\int_0^t\iint_{\Omega\times\s^2}\abs{\ss}\abs{u-\bar u}^{2m-1}\leq&\nm{\abs{u -\bar
u}^{2m-1} }_{L^{\frac{2m}{2m-1}}([0,t]\times\Omega\times\s^2)}^{\frac{2m-1}{2m}}\nm{S}_{L^{2m}([0,t]\times\Omega\times\s^2)}^{2m}\\
=&\nm{u -\bar
u }_{L^{2m}([0,t]\times\Omega\times\s^2)}^{2m-1}\nm{S}_{L^{2m}([0,t]\times\Omega\times\s^2)}^{2m}\no\\
\leq&C\bigg(\Big(o(1)\nm{u -\bar
u }_{L^{2m}([0,t]\times\Omega\times\s^2)}^{2m-1}\Big)^{\frac{2m}{2m-1}}+\Big(\nm{S}_{L^{2m}([0,t]\times\Omega\times\s^2)}^{2m}\Big)^{2m}\bigg)\no\\
\leq&o(1)\nm{u -\bar
u }_{L^{2m}([0,t]\times\Omega\times\s^2)}^{2m}+C\nm{S}_{L^{2m}([0,t]\times\Omega\times\s^2)}^{2m},\no
\end{align}
and
\begin{align}\label{wt 2'}
\int_0^t\iint_{\Omega\times\s^2}\abs{\ss} \abs{\bar u}^{2m-1} \leq&\nm{\abs{\bar
u}^{2m-1} }_{L^{\frac{2m}{2m-1}}([0,t]\times\Omega\times\s^2)}^{\frac{2m-1}{2m}}\nm{S}_{L^{2m}([0,t]\times\Omega\times\s^2)}^{2m}\\
=&\nm{\bar
u }_{L^{2m}([0,t]\times\Omega\times\s^2)}^{2m-1}\nm{S}_{L^{2m}([0,t]\times\Omega\times\s^2)}^{2m}\no\\
=&\left(\e^{\frac{2m-1}{2m}}\nm{\bar
u }_{L^{2m}([0,t]\times\Omega\times\s^2)}^{2m-1}\right)\left(\e^{-\frac{2m-1}{2m}}\nm{S}_{L^{2m}([0,t]\times\Omega\times\s^2)}^{2m}\right)\no\\
\leq&C\bigg(\left(o(1)\e^{\frac{2m-1}{2m}}\nm{\bar
u }_{L^{2m}([0,t]\times\Omega\times\s^2)}^{2m-1}\right)^{\frac{2m}{2m-1}}+\left(\e^{-\frac{2m-1}{2m}}\nm{S}_{L^{2m}([0,t]\times\Omega\times\s^2)}^{2m}\right)^{2m}\bigg)\no\\
\leq&o(1)\e\nm{\bar
u }_{L^{2m}([0,t]\times\Omega\times\s^2)}^{2m}+\frac{C}{\e^{2m-1}}\nm{S}_{L^{2m}([0,t]\times\Omega\times\s^2)}^{2m}.\no
\end{align}
Inserting \eqref{wt 1'} and \eqref{wt 2'} into \eqref{wt 13'}, we obtain
\begin{align}\label{wt 14'}
\\
\abs{\int_0^t\iint_{\Omega\times\s^2}\ss u^{2m-1}}\leq o(1)\nm{u -\bar
u }_{L^{2m}([0,t]\times\Omega\times\s^2)}^{2m}+o(1)\e\nm{\bar
u }_{L^{2m}([0,t]\times\Omega\times\s^2)}^{2m}+\frac{C}{\e^{2m-1}}\nm{S}_{L^{2m}([0,t]\times\Omega\times\s^2)}^{2m}.\no
\end{align}
Inserting \eqref{wt 14'} into \eqref{wt 3'} to absorb $o(1)\nm{u -\bar
u }_{L^{2m}([0,t]\times\Omega\times\s^2)}^{2m}$ into the left-hand side, we have
\begin{align}
&\e^2\nm{u (t)}_{L^{2m}(\Omega\times\s^2)}^{2m}+\e\nm{u}^{2m}_{L^{2m}([0,t]\times\Gamma^+)}+\nm{u -\bar
u }_{L^{2m}([0,t]\times\Omega\times\s^2)}^{2m}\\
\leq&C\bigg(o(1)\e\nm{\bar
u }_{L^{2m}([0,t]\times\Omega\times\s^2)}^{2m}+\frac{1}{\e^{2m-1}}\nm{S}_{L^{2m}([0,t]\times\Omega\times\s^2)}^{2m} +\e^2\nm{\h}_{L^{2m}(\Omega\times\s^2)}^{2m}+\e\nm{\g}_{L^{2m}([0,t]\times\Gamma^-)}^{2m}\bigg).\no
\end{align}
Taking $(2m)^{th}$ root, we actually have
\begin{align}\label{wt 4'}
&\e^{\frac{1}{m}}\nm{u (t)}_{L^{2m}(\Omega\times\s^2)}+\e^{\frac{1}{2m}}\nm{u}_{L^{2m}([0,t]\times\Gamma^+)}+\nm{u -\bar
u }_{L^{2m}([0,t]\times\Omega\times\s^2)}\\
\leq&C\bigg(o(1)\e^{\frac{1}{2m}}\nm{\bar
u }_{L^{2m}([0,t]\times\Omega\times\s^2)}+\frac{1}{\e^{1-\frac{1}{2m}}}\nm{S}_{L^{2m}([0,t]\times\Omega\times\s^2)} +\e^{\frac{1}{m}}\nm{\h}_{L^{2m}(\Omega\times\s^2)}+\e^{\frac{1}{2m}}\nm{\g}_{L^{2m}([0,t]\times\Gamma^-)}\bigg).\no
\end{align}
Multiplying \eqref{wt 68} by $\e^{-\frac{2m-1}{2m}}$, we have
\begin{align}\label{wt 15'}
&\e^{\frac{1}{2m}}\nm{\bar u}_{L^{2m}([0,t]\times\Omega\times\s^2)}\\
\leq& C \bigg(\e^{\frac{1}{2m}}\nm{u-\bar
u}_{L^{2m}([0,t]\times\Omega\times\s^2)}+\frac{1}{\e^{1-\frac{1}{2m}}}\nm{u-\bar
u}_{L^{2}([0,t]\times\Omega\times\s^2)}+\e^{\frac{1}{2m}}\nm{u}_{L^{\frac{4m}{3}}([0,t]\times\Gamma^+)}\no\\
&+\frac{1}{\e^{1-\frac{1}{2m}}}\nm{\ss}_{L^2([0,t]\times\Omega\times\s^2)}+\e^{\frac{1}{2m}}\tm{\dt\ss}{[0,t]\times\Omega\times\s^2}+
\e^{\frac{1}{2m}}\nm{\ss(0)}_{L^2(\Omega\times\s^2)}\no\\
&+
\e^{\frac{1}{2m}}\nm{\h}_{L^2(\Omega\times\s^2)}+
\e^{1+\frac{1}{2m}}\nm{\vw\cdot\nx\h}_{L^2(\Omega\times\s^2)}+\e^{\frac{1}{2m}}\nm{\g}_{L^{\frac{4m}{3}}([0,t]\times\Gamma^-)}
+\e^{\frac{3}{2}+\frac{1}{2m}}\nm{\dt\g}_{L^2([0,t]\times\Gamma^-)}\bigg).\no
\end{align}
Adding \eqref{wt 15'} to \eqref{wt 4'} to absorb $o(1)\e^{\frac{1}{2m}}\nm{\bar
u }_{L^{2m}([0,t]\times\Omega\times\s^2)}$ and $\e^{\frac{1}{2m}}\nm{u-\bar
u}_{L^{2m}([0,t]\times\Omega\times\s^2)}$ into the left-hand side, we have
\begin{align}\label{wt 5'}
&\e^{\frac{1}{m}}\nm{u (t)}_{L^{2m}(\Omega\times\s^2)}+\e^{\frac{1}{2m}}\nm{u}_{L^{2m}([0,t]\times\Gamma^+)}+\e^{\frac{1}{2m}}\nm{\bar
u }_{L^{2m}([0,t]\times\Omega\times\s^2)}+\nm{u -\bar
u }_{L^{2m}([0,t]\times\Omega\times\s^2)}\\
\leq& C \bigg(\frac{1}{\e^{1-\frac{1}{2m}}}\nm{u-\bar
u}_{L^{2}([0,t]\times\Omega\times\s^2)}+\e^{\frac{1}{2m}}\nm{u}_{L^{\frac{4m}{3}}([0,t]\times\Gamma^+)}\no\\
&+\frac{1}{\e^{1-\frac{1}{2m}}}\nm{\ss}_{L^2([0,t]\times\Omega\times\s^2)}+\frac{1}{\e^{1-\frac{1}{2m}}}\nm{S}_{L^{2m}([0,t]\times\Omega\times\s^2)}
+\e^{\frac{1}{2m}}\tm{\dt\ss}{[0,t]\times\Omega\times\s^2}+
\e^{\frac{1}{2m}}\nm{\ss(0)}_{L^2(\Omega\times\s^2)}\no\\
&+
\e^{\frac{1}{2m}}\nm{\h}_{L^2(\Omega\times\s^2)}+\e^{\frac{1}{m}}\nm{\h}_{L^{2m}(\Omega\times\s^2)}+
\e^{1+\frac{1}{2m}}\nm{\vw\cdot\nx\h}_{L^2(\Omega\times\s^2)}\no\\
&+\e^{\frac{1}{2m}}\nm{\g}_{L^{\frac{4m}{3}}([0,t]\times\Gamma^-)}+\e^{\frac{1}{2m}}\nm{\g}_{L^{2m}([0,t]\times\Gamma^-)}
+\e^{\frac{3}{2}+\frac{1}{2m}}\nm{\dt\g}_{L^2([0,t]\times\Gamma^-)}\bigg).\no
\end{align}
\ \\
Step 3: $L^2$ Energy Estimate.\\
Recall the standard $L^2$ estimates \eqref{wt 23}, where, in the weak formulation \eqref{wt 1}, we take
the test function $\phi=u$ to get
\begin{align}
&\frac{\e^2}{2}\nm{u (t)}_{L^2(\Omega\times\s^2)}^2+\frac{\e}{2}\nm{u}^2_{L^2([0,t]\times\Gamma^+)}+\nm{u -\bar
u }_{L^2([0,t]\times\Omega\times\s^2)}^2\\
=&\int_0^t\iint_{\Omega\times\s^2}\ss u +\frac{\e^2}{2}\nm{\h}_{L^2(\Omega\times\s^2)}^2+\frac{\e}{2}\nm{\g}_{L^2([0,t]\times\Gamma^-)}^2.\no
\end{align}
Naturally, this implies
\begin{align}\label{wt 6'}
&\e^2\nm{u (t)}_{L^2(\Omega\times\s^2)}^2+\e\nm{u}_{L^2([0,t]\times\Gamma^+)}^2+\nm{u -\bar
u }_{L^2([0,t]\times\Omega\times\s^2)}^2\\
\leq&\int_0^t\iint_{\Omega\times\s^2}\ss u +\e^2\nm{\h}_{L^2(\Omega\times\s^2)}^2+\e\nm{\g}_{L^2([0,t]\times\Gamma^-)}^2.\no
\end{align}
Multiplying \eqref{wt 6'} by $\e^{-2+\frac{1}{m}}$, we have
\begin{align}\label{wt 16'}
&\e^{\frac{1}{m}}\nm{u (t)}_{L^2(\Omega\times\s^2)}^2+\dfrac{1}{\e^{1-\frac{1}{m}}}\nm{u}_{L^2([0,t]\times\Gamma^+)}^2+\dfrac{1}{\e^{2-\frac{1}{m}}}\nm{u -\bar
u }_{L^2([0,t]\times\Omega\times\s^2)}^2\\
\leq&\dfrac{1}{\e^{2-\frac{1}{m}}}\int_0^t\iint_{\Omega\times\s^2}\ss u +\e^{\frac{1}{m}}\nm{\h}_{L^2(\Omega\times\s^2)}^2+\dfrac{1}{\e^{1-\frac{1}{m}}}\nm{\g}_{L^2([0,t]\times\Gamma^-)}^2.\no
\end{align}
Note that
\begin{align}\label{wt 17'}
\abs{\int_0^t\iint_{\Omega\times\s^2}\ss u}\leq\abs{\int_0^t\iint_{\Omega\times\s^2}\ss(u-\bar u)}+\abs{\int_0^t\iint_{\Omega\times\s^2}\ss \bar u}.
\end{align}
Using H\"{o}lder's inequality and Cauchy's inequality, we have
\begin{align}\label{wt 18'}
\abs{\int_0^t\iint_{\Omega\times\s^2}\ss(u-\bar u)}\leq&\nm{u -\bar
u }_{L^{2m}([0,t]\times\Omega\times\s^2)}\nm{S}_{L^{\frac{2m}{2m-1}}([0,t]\times\Omega\times\s^2)}\\
\leq& o(1)\e^{2-\frac{1}{m}}\nm{u -\bar
u }_{L^{2m}([0,t]\times\Omega\times\s^2)}^2+\frac{C}{\e^{2-\frac{1}{m}}}\nm{S}_{L^{\frac{2m}{2m-1}}([0,t]\times\Omega\times\s^2)}^2,\no
\end{align}
and
\begin{align}\label{wt 19'}
\abs{\int_0^t\iint_{\Omega\times\s^2}\ss \bar u}\leq&\nm{\bar
u }_{L^{2m}([0,t]\times\Omega\times\s^2)}\nm{S}_{L^{\frac{2m}{2m-1}}([0,t]\times\Omega\times\s^2)}\\
\leq&o(1)\e^2\nm{\bar
u }_{L^{2m}([0,t]\times\Omega\times\s^2)}^2+\frac{C}{\e^2}\nm{S}_{L^{\frac{2m}{2m-1}}([0,t]\times\Omega\times\s^2)}^2.\no
\end{align}
Inserting \eqref{wt 18'} and \eqref{wt 19'} into \eqref{wt 17'}, we obtain
\begin{align}\label{wt 20'}
\\
\abs{\int_0^t\iint_{\Omega\times\s^2}\ss u}\leq o(1)\e^{2-\frac{1}{m}}\nm{u -\bar
u }_{L^{2m}([0,t]\times\Omega\times\s^2)}^2+o(1)\e^2\nm{\bar
u }_{L^{2m}([0,t]\times\Omega\times\s^2)}^2+\frac{C}{\e^2}\nm{S}_{L^{\frac{2m}{2m-1}}([0,t]\times\Omega\times\s^2)}^2.\no
\end{align}
Inserting \eqref{wt 20'} into \eqref{wt 16'}, we have
\begin{align}\label{wt 21'}
&\e^{\frac{1}{m}}\nm{u (t)}_{L^2(\Omega\times\s^2)}^2+\dfrac{1}{\e^{1-\frac{1}{m}}}\nm{u}_{L^2([0,t]\times\Gamma^+)}^2+\dfrac{1}{\e^{2-\frac{1}{m}}}\nm{u -\bar
u }_{L^2([0,t]\times\Omega\times\s^2)}^2\\
\leq&o(1)\nm{u -\bar
u }_{L^{2m}([0,t]\times\Omega\times\s^2)}^2+o(1)\e^{\frac{1}{m}}\nm{\bar
u }_{L^{2m}([0,t]\times\Omega\times\s^2)}^2+\frac{C}{\e^{4-\frac{1}{m}}}\nm{S}_{L^{\frac{2m}{2m-1}}([0,t]\times\Omega\times\s^2)}^2\no\\ &+\e^{\frac{1}{m}}\nm{\h}_{L^2(\Omega\times\s^2)}^2+\dfrac{1}{\e^{1-\frac{1}{m}}}\nm{\g}_{L^2([0,t]\times\Gamma^-)}^2.\no
\end{align}
Taking square root in \eqref{wt 21'}, we obtain
\begin{align}\label{wt 22'}
&\e^{\frac{1}{2m}}\nm{u (t)}_{L^2(\Omega\times\s^2)}+\dfrac{1}{\e^{\frac{1}{2}-\frac{1}{2m}}}\nm{u}_{L^2([0,t]\times\Gamma^+)}+\dfrac{1}{\e^{1-\frac{1}{2m}}}\nm{u -\bar
u }_{L^2([0,t]\times\Omega\times\s^2)}\\
\leq&o(1)\nm{u -\bar
u }_{L^{2m}([0,t]\times\Omega\times\s^2)}+o(1)\e^{\frac{1}{2m}}\nm{\bar
u }_{L^{2m}([0,t]\times\Omega\times\s^2)}+\frac{C}{\e^{2-\frac{1}{2m}}}\nm{S}_{L^{\frac{2m}{2m-1}}([0,t]\times\Omega\times\s^2)}\no\\ &+\e^{\frac{1}{2m}}\nm{\h}_{L^2(\Omega\times\s^2)}+\dfrac{1}{\e^{\frac{1}{2}-\frac{1}{2m}}}\nm{\g}_{L^2([0,t]\times\Gamma^-)}.\no
\end{align}
Multiplying \eqref{wt 5'} by a small constant and adding it to \eqref{wt 22'} to absorb $\frac{1}{\e^{1-\frac{1}{2m}}}\nm{u-\bar
u}_{L^{2}([0,t]\times\Omega\times\s^2)}$, $o(1)\nm{u -\bar
u }_{L^{2m}([0,t]\times\Omega\times\s^2)}$ and $o(1)\e^{\frac{1}{2m}}\nm{\bar
u }_{L^{2m}([0,t]\times\Omega\times\s^2)}$ into the left-hand side, we obtain
\begin{align}\label{wt 7'}
&\e^{\frac{1}{m}}\nm{u (t)}_{L^{2m}(\Omega\times\s^2)}+\frac{1}{\e^{\frac{1}{2}-\frac{1}{2m}}}\nm{u}_{L^2([0,t]\times\Gamma^+)}
+\e^{\frac{1}{2m}}\nm{u}_{L^{2m}([0,t]\times\Gamma^+)}\\
&+\e^{\frac{1}{2m}}\nm{\bar
u }_{L^{2m}([0,t]\times\Omega\times\s^2)}+\dfrac{1}{\e^{1-\frac{1}{2m}}}\nm{u -\bar
u }_{L^2([0,t]\times\Omega\times\s^2)}+\nm{u -\bar
u }_{L^{2m}([0,t]\times\Omega\times\s^2)}\no\\
\leq& C \bigg(\e^{\frac{1}{2m}}\nm{u}_{L^{\frac{4m}{3}}([0,t]\times\Gamma^+)}+\frac{1}{\e^{2-\frac{1}{2m}}}\nm{S}_{L^{\frac{2m}{2m-1}}([0,t]\times\Omega\times\s^2)}
+\frac{1}{\e^{1-\frac{1}{2m}}}\nm{\ss}_{L^2([0,t]\times\Omega\times\s^2)}+\frac{1}{\e^{1-\frac{1}{2m}}}\nm{S}_{L^{2m}([0,t]\times\Omega\times\s^2)}\no\\
&+\e^{\frac{1}{2m}}\tm{\dt\ss}{[0,t]\times\Omega\times\s^2}+
\e^{\frac{1}{2m}}\nm{\ss(0)}_{L^2(\Omega\times\s^2)}\no\\
&+
\e^{\frac{1}{2m}}\nm{\h}_{L^2(\Omega\times\s^2)}+\e^{\frac{1}{m}}\nm{\h}_{L^{2m}(\Omega\times\s^2)}+
\e^{1+\frac{1}{2m}}\nm{\vw\cdot\nx\h}_{L^2(\Omega\times\s^2)}\no\\
&+\e^{\frac{1}{2m}}\nm{\g}_{L^{\frac{4m}{3}}([0,t]\times\Gamma^-)}+\e^{\frac{1}{2m}}\nm{\g}_{L^{2m}([0,t]\times\Gamma^-)}
+\frac{1}{\e^{\frac{1}{2}-\frac{1}{2m}}}\nm{\g}_{L^2([0,t]\times\Gamma^-)}+\e^{\frac{3}{2}+\frac{1}{2m}}\nm{\dt\g}_{L^2([0,t]\times\Gamma^-)}\bigg).\no
\end{align}
Using the interpolation estimate and Young's inequality, we know
\begin{align}
\nm{u}_{L^{\frac{4m}{3}}([0,t]\times\Gamma^+)}\leq&\nm{u}_{L^{2m}([0,t]\times\Gamma^+)}^{\frac{2m-3}{2m-2}}\nm{u}_{L^{2}([0,t]\times\Gamma^+)}^{\frac{1}{2m-2}}\\
\leq&o(1)\Big(\nm{u}_{L^{2m}([0,t]\times\Gamma^+)}^{\frac{2m-3}{2m-2}}\Big)^{\frac{2m-2}{2m-3}}+C\Big(\nm{u}_{L^{2}([0,t]\times\Gamma^+)}^{\frac{1}{2m-2}}\Big)^{2m-2}\no\\
\leq&o(1)\nm{u}_{L^{2m}([0,t]\times\Gamma^+)}+C\nm{u}_{L^{2}([0,t]\times\Gamma^+)}^{\frac{1}{2m-2}}.\no
\end{align}
Hence, we know
\begin{align}\label{wt 8'}
\e^{\frac{1}{2m}}\nm{u}_{L^{\frac{4m}{3}}([0,t]\times\Gamma^+)}\leq& o(1)\e^{\frac{1}{2m}}\nm{u}_{L^{2m}([0,t]\times\Gamma^+)}+C\e^{\frac{1}{2m}}\nm{u}_{L^{2}([0,t]\times\Gamma^+)}\\
\leq&o(1)\e^{\frac{1}{2m}}\nm{u}_{L^{2m}([0,t]\times\Gamma^+)}+o(1)\frac{1}{\e^{\frac{1}{2}-\frac{1}{2m}}}\nm{u}_{L^{2}([0,t]\times\Gamma^+)}.\no
\end{align}
Inserting \eqref{wt 8'} into \eqref{wt 7'} to absorb $\e^{\frac{1}{2m}}\nm{u}_{L^{\frac{4m}{3}}([0,t]\times\Gamma^+)}$ into the left-hand side, we have
\begin{align}
&\e^{\frac{1}{m}}\nm{u (t)}_{L^{2m}(\Omega\times\s^2)}+\frac{1}{\e^{\frac{1}{2}-\frac{1}{2m}}}\nm{u}_{L^2([0,t]\times\Gamma^+)}
+\e^{\frac{1}{2m}}\nm{u}_{L^{2m}([0,t]\times\Gamma^+)}\\
&+\e^{\frac{1}{2m}}\nm{\bar
u }_{L^{2m}([0,t]\times\Omega\times\s^2)}+\dfrac{1}{\e^{1-\frac{1}{2m}}}\nm{u -\bar
u }_{L^2([0,t]\times\Omega\times\s^2)}+\nm{u -\bar
u }_{L^{2m}([0,t]\times\Omega\times\s^2)}\no\\
\leq& C \bigg(\frac{1}{\e^{2-\frac{1}{2m}}}\nm{S}_{L^{\frac{2m}{2m-1}}([0,t]\times\Omega\times\s^2)}
+\frac{1}{\e^{1-\frac{1}{2m}}}\nm{\ss}_{L^2([0,t]\times\Omega\times\s^2)}+\frac{1}{\e^{1-\frac{1}{2m}}}\nm{S}_{L^{2m}([0,t]\times\Omega\times\s^2)}\no\\
&+\e^{\frac{1}{2m}}\tm{\dt\ss}{[0,t]\times\Omega\times\s^2}+
\e^{\frac{1}{2m}}\nm{\ss(0)}_{L^2(\Omega\times\s^2)}\no\\
&+
\e^{\frac{1}{2m}}\nm{\h}_{L^2(\Omega\times\s^2)}+\e^{\frac{1}{m}}\nm{\h}_{L^{2m}(\Omega\times\s^2)}+
\e^{1+\frac{1}{2m}}\nm{\vw\cdot\nx\h}_{L^2(\Omega\times\s^2)}\no\\
&+\e^{\frac{1}{2m}}\nm{\g}_{L^{\frac{4m}{3}}([0,t]\times\Gamma^-)}+\e^{\frac{1}{2m}}\nm{\g}_{L^{2m}([0,t]\times\Gamma^-)}
+\frac{1}{\e^{\frac{1}{2}-\frac{1}{2m}}}\nm{\g}_{L^2([0,t]\times\Gamma^-)}+\e^{\frac{3}{2}+\frac{1}{2m}}\nm{\dt\g}_{L^2([0,t]\times\Gamma^-)}\bigg).\no
\end{align}
Hence, we get the desired estimate
\begin{align}\label{wt 10'}
&\nm{u (t)}_{L^{2m}(\Omega\times\s^2)}+\frac{1}{\e^{\frac{1}{2}+\frac{1}{2m}}}\nm{u}_{L^2([0,t]\times\Gamma^+)}
+\frac{1}{\e^{\frac{1}{2m}}}\nm{u}_{L^{2m}([0,t]\times\Gamma^+)}\\
&+\frac{1}{\e^{\frac{1}{2m}}}\nm{\bar
u }_{L^{2m}([0,t]\times\Omega\times\s^2)}+\dfrac{1}{\e^{1+\frac{1}{2m}}}\nm{u -\bar
u }_{L^2([0,t]\times\Omega\times\s^2)}+\frac{1}{\e^{\frac{1}{m}}}\nm{u -\bar
u }_{L^{2m}([0,t]\times\Omega\times\s^2)}\no\\
\leq& C \bigg(\frac{1}{\e^{2+\frac{1}{2m}}}\nm{S}_{L^{\frac{2m}{2m-1}}([0,t]\times\Omega\times\s^2)}
+\frac{1}{\e^{1+\frac{1}{2m}}}\nm{\ss}_{L^2([0,t]\times\Omega\times\s^2)}+\frac{1}{\e^{1+\frac{1}{2m}}}\nm{S}_{L^{2m}([0,t]\times\Omega\times\s^2)}\no\\
&+\frac{1}{\e^{\frac{1}{2m}}}\tm{\dt\ss}{[0,t]\times\Omega\times\s^2}+
\frac{1}{\e^{\frac{1}{2m}}}\nm{\ss(0)}_{L^2(\Omega\times\s^2)}\no\\
&+
\frac{1}{\e^{\frac{1}{2m}}}\nm{\h}_{L^2(\Omega\times\s^2)}+\nm{\h}_{L^{2m}(\Omega\times\s^2)}+
\e^{1-\frac{1}{2m}}\nm{\vw\cdot\nx\h}_{L^2(\Omega\times\s^2)}\no\\
&+\frac{1}{\e^{\frac{1}{2m}}}\nm{\g}_{L^{\frac{4m}{3}}([0,t]\times\Gamma^-)}+\frac{1}{\e^{\frac{1}{2m}}}\nm{\g}_{L^{2m}([0,t]\times\Gamma^-)}
+\frac{1}{\e^{\frac{1}{2}+\frac{1}{2m}}}\nm{\g}_{L^2([0,t]\times\Gamma^-)}+\e^{\frac{3}{2}-\frac{1}{2m}}\nm{\dt\g}_{L^2([0,t]\times\Gamma^-)}\bigg).\no
\end{align}

\end{proof}

\subsection{$L^{\infty}$ Estimate - Second Round}

\begin{theorem}\label{LI estimate..}
Assume $\ss(t,\vx,\vw)\in
L^{\infty}(\rp\times\Omega\times\s^2)$, $\h(\vx,\vw)\in
L^{\infty}(\Omega\times\s^2)$ and $\g(t,x_0,\vw)\in
L^{\infty}(\rp\times\Gamma^-)$. Then the solution $u(t,\vx,\vw)$ to the neutron transport
equation \eqref{neutron.} satisfies
\begin{align}
&\nm{u}_{L^{\infty}([0,t]\times\Omega\times\s^2)}\\
\leq& C \bigg(\frac{1}{\e^{2+\frac{2}{m}}}\nm{S}_{L^{\frac{2m}{2m-1}}([0,t]\times\Omega\times\s^2)}
+\frac{1}{\e^{1+\frac{2}{m}}}\nm{\ss}_{L^2([0,t]\times\Omega\times\s^2)}+\frac{1}{\e^{1+\frac{2}{m}}}\nm{S}_{L^{2m}([0,t]\times\Omega\times\s^2)}
+\im{\ss}{[0,t]\times\Omega\times\s^2}\no\\
&+\frac{1}{\e^{\frac{2}{m}}}\tm{\dt\ss}{[0,t]\times\Omega\times\s^2}+
\frac{1}{\e^{\frac{2}{m}}}\nm{\ss(0)}_{L^2(\Omega\times\s^2)}\no\\
&+
\e^{1-\frac{2}{m}}\nm{\vw\cdot\nx\h}_{L^2(\Omega\times\s^2)}+
\frac{1}{\e^{\frac{2}{m}}}\nm{\h}_{L^2(\Omega\times\s^2)}+\frac{1}{\e^{\frac{3}{2m}}}\nm{\h}_{L^{2m}(\Omega\times\s^2)}+\im{\h}{\Omega\times\s^2}\no\\
&+\frac{1}{\e^{\frac{2}{m}}}\nm{\g}_{L^{\frac{4m}{3}}([0,t]\times\Gamma^-)}+\frac{1}{\e^{\frac{2}{m}}}\nm{\g}_{L^{2m}([0,t]\times\Gamma^-)}
+\frac{1}{\e^{\frac{1}{2}+\frac{2}{m}}}\nm{\g}_{L^2([0,t]\times\Gamma^-)}+\im{\g}{[0,t]\times\Gamma^-}+\e^{\frac{3}{2}-\frac{2}{m}}\nm{\dt\g}_{L^2([0,t]\times\Gamma^-)}\bigg).\no
\end{align}
\end{theorem}
\begin{proof}
Based on the analysis in proving Theorem \eqref{LI estimate.'}, the key step is the estimate of $I_4^{\ast}$. We utilize the same substitution in \eqref{wt 46} and apply H\"{o}lder's inequality with a different exponent to obtain
\begin{align}\label{wt 81}
\abs{I_4^{\ast}}\leq&C\int_{0}^{t_b}\Bigg(\int_{ r\geq\d}\int_{q\geq\d}\bigg(\int_{\Omega}\frac{1}{\e^3\d^5}\abs{\bar u(t''-\e^2q,\vec
y)}^{2m}\ud{\vec y}\bigg)\ue^{-{q}}\ue^{-{r}}\ud q\ud{r}\Bigg)^{\frac{1}{2m}}\ue^{-{s}}\ud{s}\\
\leq&\frac{C}{\e^{\frac{3}{2m}}\d^{\frac{5}{2m}}}\nm{\bar u(t)}_{L^{2m}([0,t]\times\Omega)}.\no
\end{align}
Summarizing \eqref{wt 33}, \eqref{wt 34}, \eqref{wt 35}, \eqref{wt 36}, \eqref{wt 37}, \eqref{wt 38}, \eqref{wt 39}, \eqref{wt 40}, \eqref{wt 41}, \eqref{wt 42}, \eqref{wt 43}, \eqref{wt 47}, \eqref{wt 81}, we have shown that for any $(t,\vx,\vw)\in[0,t]\times\bar\Omega\times\s^2$,
\begin{align}\label{ctt 1.}
\abs{u(t,\vx,\vw)}\leq& C\bigg(\delta\im{u}{[0,t]\times\Omega\times\s^2}+\frac{1}{\e^{\frac{3}{2m}}\d^{\frac{5}{2m}}}\nm{\bar u(t)}_{L^{2m}(\Omega\times\s^2)}\\
&+\im{\ss}{[0,t]\times\Omega\times\s^2}+\im{\h}{\Omega\times\s^2}+\im{\g}{[0,t]\times\Gamma^-}\bigg).\no
\end{align}
Let $\d$ be sufficiently small such that $C\d\leq\dfrac{1}{2}$. Taking supremum over $(t,\vx,\vw)\in[0,t]\times\Omega\times\s^2$ in (\ref{ctt 1.}) and using Theorem \ref{LN estimate.}, we have
\begin{align}
&\nm{u}_{L^{\infty}([0,t]\times\Omega\times\s^2)}\\
\leq& C \bigg(\frac{1}{\e^{2+\frac{2}{m}}}\nm{S}_{L^{\frac{2m}{2m-1}}([0,t]\times\Omega\times\s^2)}
+\frac{1}{\e^{1+\frac{2}{m}}}\nm{\ss}_{L^2([0,t]\times\Omega\times\s^2)}+\frac{1}{\e^{1+\frac{2}{m}}}\nm{S}_{L^{2m}([0,t]\times\Omega\times\s^2)}
+\im{\ss}{[0,t]\times\Omega\times\s^2}\no\\
&+\frac{1}{\e^{\frac{2}{m}}}\tm{\dt\ss}{[0,t]\times\Omega\times\s^2}+
\frac{1}{\e^{\frac{2}{m}}}\nm{\ss(0)}_{L^2(\Omega\times\s^2)}\no\\
&+
\e^{1-\frac{2}{m}}\nm{\vw\cdot\nx\h}_{L^2(\Omega\times\s^2)}+
\frac{1}{\e^{\frac{2}{m}}}\nm{\h}_{L^2(\Omega\times\s^2)}+\frac{1}{\e^{\frac{3}{2m}}}\nm{\h}_{L^{2m}(\Omega\times\s^2)}+\im{\h}{\Omega\times\s^2}\no\\
&+\frac{1}{\e^{\frac{2}{m}}}\nm{\g}_{L^{\frac{4m}{3}}([0,t]\times\Gamma^-)}+\frac{1}{\e^{\frac{2}{m}}}\nm{\g}_{L^{2m}([0,t]\times\Gamma^-)}
+\frac{1}{\e^{\frac{1}{2}+\frac{2}{m}}}\nm{\g}_{L^2([0,t]\times\Gamma^-)}+\im{\g}{[0,t]\times\Gamma^-}+\e^{\frac{3}{2}-\frac{2}{m}}\nm{\dt\g}_{L^2([0,t]\times\Gamma^-)}\bigg).\no
\end{align}

\end{proof}

\subsection{Superposition Argument}

In Theorem \ref{LI estimate..}, the estimates depend on the derivative bounds of initial and boundary data. Here we can use superposition property to get rid of such dependence and get a cleaner form.
\begin{theorem}\label{LI estimate...}
Assume $\ss(t,\vx,\vw)\in
L^{\infty}(\rp\times\Omega\times\s^2)$, $\h(\vx,\vw)\in
L^{\infty}(\Omega\times\s^2)$ and $\g(t,x_0,\vw)\in
L^{\infty}(\rp\times\Gamma^-)$. Then the solution $u(t,\vx,\vw)$ to the neutron transport
equation \eqref{neutron.} satisfies
\begin{align}
&\nm{u}_{L^{\infty}([0,t]\times\Omega\times\s^2)}\\
\leq& C \bigg(\frac{1}{\e^{2+\frac{2}{m}}}\nm{S}_{L^{\frac{2m}{2m-1}}([0,t]\times\Omega\times\s^2)}
+\frac{1}{\e^{1+\frac{2}{m}}}\nm{\ss}_{L^2([0,t]\times\Omega\times\s^2)}+\frac{1}{\e^{1+\frac{2}{m}}}\nm{S}_{L^{2m}([0,t]\times\Omega\times\s^2)}
\no\\
&+\im{\ss}{[0,t]\times\Omega\times\s^2}+\frac{1}{\e^{\frac{2}{m}}}\tm{\dt\ss}{[0,t]\times\Omega\times\s^2}+
\frac{1}{\e^{\frac{2}{m}}}\nm{\ss(0)}_{L^2(\Omega\times\s^2)}\no\\
&+\frac{1}{\e^{\frac{3}{2m}}}\nm{\h}_{L^{2m}(\Omega\times\s^2)}+\im{\h}{\Omega\times\s^2}
+\frac{1}{\e^{\frac{2}{m}}}\nm{\g}_{L^{2m}([0,t]\times\Gamma^-)}+\im{\g}{[0,t]\times\Gamma^-}\bigg).\no
\end{align}
\end{theorem}
\begin{proof}
Since equation \eqref{neutron.} is linear, we may decompose the solution $u=u_1+u_2$ satisfying
\begin{align}\label{neutron..}
\left\{
\begin{array}{l}
\e^2\dt u_1+\e \vw\cdot\nabla_x u_1+u_1-\bar
u_1=0\ \ \ \text{for}\ \
(t,\vx,\vw)\in\rp\times\Omega\times\s^2,\\\rule{0ex}{2.0em}
u_1(0,\vx,\vw)=\h(\vx,\vw)\ \ \text{for}\ \ (\vx,\vw)\in\Omega\times\s^2\\\rule{0ex}{2.0em}
u_1(t,\vx_0,\vw)=\g(t,\vx_0,\vw)\ \ \text{for}\ \ t\in\rp,\ \ \vx_0\in\p\Omega\ \ \text{and}\ \ \vw\cdot\vn<0,
\end{array}
\right.
\end{align}
and
\begin{align}\label{neutron...}
\left\{
\begin{array}{l}
\e^2\dt u_2+\e \vw\cdot\nabla_x u_2+u_2-\bar
u_2=\ss\ \ \ \text{for}\ \
(t,\vx,\vw)\in\rp\times\Omega\times\s^2,\\\rule{0ex}{2.0em}
u_2(0,\vx,\vw)=0\ \ \text{for}\ \ (\vx,\vw)\in\Omega\times\s^2\\\rule{0ex}{2.0em}
u_2(t,\vx_0,\vw)=0\ \ \text{for}\ \ t\in\rp,\ \ \vx_0\in\p\Omega\ \ \text{and}\ \ \vw\cdot\vn<0.
\end{array}
\right.
\end{align}
\ \\
Step 1: Estimate of $u_1$.\\
In Step 2 of proof of Theorem \ref{LN estimate.}, we have shown
\begin{align}
&\e^2\nm{u_1 (t)}_{L^{2m}(\Omega\times\s^2)}^{2m}+\e\nm{u_1}^{2m}_{L^{2m}([0,t]\times\Gamma^+)}+\nm{u_1 -\bar
u_1 }_{L^{2m}([0,t]\times\Omega\times\s^2)}^{2m}\\
\leq&C\bigg(\e^2\nm{\h}_{L^{2m}(\Omega\times\s^2)}^{2m}+\e\nm{\g}_{L^{2m}([0,t]\times\Gamma^-)}^{2m}\bigg).\no
\end{align}
Take $(2m)^{th}$ root on both sides, we have
\begin{align}
&\e^{\frac{1}{m}}\nm{u_1 (t)}_{L^{2m}(\Omega\times\s^2)}+\e^{\frac{1}{2m}}\nm{u_1}_{L^{2m}([0,t]\times\Gamma^+)}+\nm{u_1 -\bar
u_1 }_{L^{2m}([0,t]\times\Omega\times\s^2)}\\
\leq&C\bigg(\e^{\frac{1}{m}}\nm{\h}_{L^{2m}(\Omega\times\s^2)}+\e^{\frac{1}{2m}}\nm{\g}_{L^{2m}([0,t]\times\Gamma^-)}\bigg).\no
\end{align}
Dividing $\e^{\frac{1}{m}}$ on both sides, we arrive at
\begin{align}
&\nm{u_1 (t)}_{L^{2m}(\Omega\times\s^2)}+\frac{1}{\e^{\frac{1}{2m}}}\nm{u_1}_{L^{2m}([0,t]\times\Gamma^+)}+\frac{1}{\e^{\frac{1}{m}}}\nm{u_1 -\bar
u_1 }_{L^{2m}([0,t]\times\Omega\times\s^2)}\\
\leq&C\bigg(\nm{\h}_{L^{2m}(\Omega\times\s^2)}+\frac{1}{\e^{\frac{1}{2m}}}\nm{\g}_{L^{2m}([0,t]\times\Gamma^-)}\bigg).\no
\end{align}
Then using the argument in the proof of Theorem \ref{LI estimate..}, we have
\begin{align}\label{ltt 101}
\nm{u_1}_{L^{\infty}([0,t]\times\Omega\times\s^2)}\leq& C\bigg(\frac{1}{\e^{\frac{3}{2m}}}\nm{\bar u_1(t)}_{L^{2m}(\Omega\times\s^2)}+\im{\h}{\Omega\times\s^2}+\im{\g}{[0,t]\times\Gamma^-}\bigg)\\
\leq&C\bigg(\frac{1}{\e^{\frac{3}{2m}}}\nm{\h}_{L^{2m}(\Omega\times\s^2)}+\im{\h}{\Omega\times\s^2}
+\frac{1}{\e^{\frac{2}{m}}}\nm{\g}_{L^{2m}([0,t]\times\Gamma^-)}+\im{\g}{[0,t]\times\Gamma^-}\bigg).\no
\end{align}
\ \\
Step 2: Estimate of $u_2$.\\
Here we refer to Theorem \ref{LI estimate..}. Since initial and boundary data are all zero, the estimate is much simpler
\begin{align}\label{ltt 102}
&\nm{u_2}_{L^{\infty}([0,t]\times\Omega\times\s^2)}\\
\leq& C \bigg(\frac{1}{\e^{2+\frac{2}{m}}}\nm{S}_{L^{\frac{2m}{2m-1}}([0,t]\times\Omega\times\s^2)}
+\frac{1}{\e^{1+\frac{2}{m}}}\nm{\ss}_{L^2([0,t]\times\Omega\times\s^2)}+\frac{1}{\e^{1+\frac{2}{m}}}\nm{S}_{L^{2m}([0,t]\times\Omega\times\s^2)}
\no\\
&+\im{\ss}{[0,t]\times\Omega\times\s^2}+\frac{1}{\e^{\frac{2}{m}}}\tm{\dt\ss}{[0,t]\times\Omega\times\s^2}+
\frac{1}{\e^{\frac{2}{m}}}\nm{\ss(0)}_{L^2(\Omega\times\s^2)}\bigg).\no
\end{align}
\ \\
Step 3: Synthesis.\\
Combining \eqref{ltt 101} and \eqref{ltt 102}, we obtain
\begin{align}
&\nm{u}_{L^{\infty}([0,t]\times\Omega\times\s^2)}\leq \nm{u_1}_{L^{\infty}([0,t]\times\Omega\times\s^2)}+\nm{u_2}_{L^{\infty}([0,t]\times\Omega\times\s^2)}\\
\leq& C \bigg(\frac{1}{\e^{2+\frac{2}{m}}}\nm{S}_{L^{\frac{2m}{2m-1}}([0,t]\times\Omega\times\s^2)}
+\frac{1}{\e^{1+\frac{2}{m}}}\nm{\ss}_{L^2([0,t]\times\Omega\times\s^2)}+\frac{1}{\e^{1+\frac{2}{m}}}\nm{S}_{L^{2m}([0,t]\times\Omega\times\s^2)}
\no\\
&+\im{\ss}{[0,t]\times\Omega\times\s^2}+\frac{1}{\e^{\frac{2}{m}}}\tm{\dt\ss}{[0,t]\times\Omega\times\s^2}+
\frac{1}{\e^{\frac{2}{m}}}\nm{\ss(0)}_{L^2(\Omega\times\s^2)}\no\\
&+\frac{1}{\e^{\frac{3}{2m}}}\nm{\h}_{L^{2m}(\Omega\times\s^2)}+\im{\h}{\Omega\times\s^2}
+\frac{1}{\e^{\frac{2}{m}}}\nm{\g}_{L^{2m}([0,t]\times\Gamma^-)}+\im{\g}{[0,t]\times\Gamma^-}\bigg).\no
\end{align}
\end{proof}

\begin{theorem}\label{LI estimate.}
Assume $\ue^{Kt}\ss(t,\vx,\vw)\in
L^{\infty}(\rp\times\Omega\times\s^2)$, $\h(\vx,\vw)\in
L^{\infty}(\Omega\times\s^2)$ and $\ue^{K t}\g(t,x_0,\vw)\in
L^{\infty}(\rp\times\Gamma^-)$ for some $K>0$. Then there exists $0<K_0\leq K$ such that the solution
$u(t,\vx,\vw)$ to the neutron transport equation \eqref{neutron.} satisfies
\begin{align}
&\nm{\ue^{K_0t}u}_{L^{\infty}([0,t]\times\Omega\times\s^2)}\\
\leq& C \bigg(\frac{1}{\e^{2+\frac{2}{m}}}\nm{\ue^{K_0t}S}_{L^{\frac{2m}{2m-1}}([0,t]\times\Omega\times\s^2)}
+\frac{1}{\e^{1+\frac{2}{m}}}\nm{\ue^{K_0t}\ss}_{L^2([0,t]\times\Omega\times\s^2)}+\frac{1}{\e^{1+\frac{2}{m}}}\nm{\ue^{K_0t}S}_{L^{2m}([0,t]\times\Omega\times\s^2)}\no\\
&
+\im{\ue^{K_0t}\ss}{[0,t]\times\Omega\times\s^2}+\frac{1}{\e^{\frac{2}{m}}}\tm{\ue^{K_0t}\dt\ss}{[0,t]\times\Omega\times\s^2}+
\frac{1}{\e^{\frac{2}{m}}}\nm{\ss(0)}_{L^2(\Omega\times\s^2)}\no\\
&+\frac{1}{\e^{\frac{3}{2m}}}\nm{\h}_{L^{2m}(\Omega\times\s^2)}+\im{\h}{\Omega\times\s^2}+\frac{1}{\e^{\frac{2}{m}}}\nm{\ue^{K_0t}\g}_{L^{2m}([0,t]\times\Gamma^-)}
+\im{\ue^{K_0t}\g}{[0,t]\times\Gamma^-}\bigg).\no
\end{align}
\end{theorem}
\begin{proof}
Let $v=\ue^{K_0 t}u$. Then $v$ satisfies the equation
\begin{align}
\left\{
\begin{array}{l}
\e^2\dt v+\e\vw\cdot\nabla_xv+v-\bar v=\ue^{K_0t}\ss+K_0\e^2v\ \ \text{for}\ \
(t,\vx,\vw)\in\rp\times\Omega\times\s^2,\\\rule{0ex}{2.0em} v(0,\vx,\vw)=\h(\vx,\vw)\ \ \text{for}\ \
(\vx,\vw)\in\Omega\times\s^2,\\\rule{0ex}{2.0em} v(t,\vx_0,\vw)=\ue^{K_0t}\g(t,\vx_0,\vw)\
\ \text{for}\ \ t\in\rp\ \ \vx_0\in\p\Omega\ \ \text{and}\ \vw\cdot\vn<0.
\end{array}
\right.
\end{align}
Note that we have an extra term $K_0\e^2v$. However, $\e^2$ helps to recover all the estimates in previous theorems and we can obtain exactly the same results.
\end{proof}

\section{Asymptotic Analysis}

\subsection{Analysis of Regular Boundary Layer}

Based on a similar analysis as in steady problems, we have the following:
\begin{theorem}\label{dt theorem 1.}
For $K_0>0$ sufficiently small, the regular boundary layer satisfies
\begin{align}
\begin{array}{ll}
\lnnm{\ue^{K_0t}\ue^{K_0\eta}\ub_0(t)}\leq C,& \lnnm{\ue^{K_0t}\ue^{K_0\eta}\ub_1(t)}\leq C\abs{\ln(\e)}^8,\\\rule{0ex}{2em}
\lnnm{\ue^{K_0t}\ue^{K_0\eta}\dfrac{\p\ub_0}{\p t}(t)}\leq C\abs{\ln(\e)}^8,&\lnnm{\ue^{K_0t}\ue^{K_0\eta}\dfrac{\p\ub_1}{\p t}(t)}\leq C\abs{\ln(\e)}^{16},\\\rule{0ex}{2em}
\lnnm{\ue^{K_0t}\ue^{K_0\eta}\dfrac{\p\ub_0}{\p\iota_i}(t)}\leq C\abs{\ln(\e)}^8,&\lnnm{\ue^{K_0t}\ue^{K_0\eta}\dfrac{\p\ub_1}{\p\iota_i}(t)}\leq C\abs{\ln(\e)}^{16},\\\rule{0ex}{2em}
\lnnm{\ue^{K_0t}\ue^{K_0\eta}\dfrac{\p\ub_0}{\p\psi}(t)}\leq C\abs{\ln(\e)}^8,&\lnnm{\ue^{K_0t}\ue^{K_0\eta}\dfrac{\p\ub_1}{\p\psi}(t)}\leq C\abs{\ln(\e)}^{16}.
\end{array}
\end{align}
\end{theorem}

\subsection{Analysis of Singular Boundary Layer}

Based on a similar analysis as in steady problems, we have the following:
\begin{theorem}\label{dt theorem 2.}
Let
\begin{align}
\left\{
\begin{array}{ll}
\chi_1: &0\leq\zeta<2\e^{\alpha},\\
\chi_2: &2\e^{\alpha}\leq\zeta\leq1.
\end{array}
\right.
\end{align}
For $K_0>0$ sufficiently small, the singular boundary layer satisfies
\begin{align}
\begin{array}{ll}
\lnnm{\ue^{K_0t}\ue^{K_0\eta}(\chi_1\uf_0)(t)}\leq C,& \lnnm{\ue^{K_0t}\ue^{K_0\eta}(\chi_2\uf_0)(t)}\leq C\e^{\alpha},\\\rule{0ex}{2em}
\lnnm{\ue^{K_0t}\ue^{K_0\eta}\dfrac{\p(\chi_1\uf_0)(t)}{\p t}}\leq C\abs{\ln(\e)}^8,&\lnnm{\ue^{K_0t}\ue^{K_0\eta}\dfrac{\p(\chi_2\uf_0)(t)}{\p t}}\leq C\e^{\alpha}\abs{\ln(\e)}^{8},\\\rule{0ex}{2em}
\lnnm{\ue^{K_0t}\ue^{K_0\eta}\dfrac{\p(\chi_1\uf_0)(t)}{\p\iota_i}}\leq C\abs{\ln(\e)}^8,&\lnnm{\ue^{K_0t}\ue^{K_0\eta}\dfrac{\p(\chi_2\uf_0)(t)}{\p\iota_i}}\leq C\e^{\alpha}\abs{\ln(\e)}^{8},\\\rule{0ex}{2em}
\lnnm{\ue^{K_0t}\ue^{K_0\eta}\dfrac{\p(\chi_1\uf_0)(t)}{\p\psi}}\leq C\abs{\ln(\e)}^8,&\lnnm{\ue^{K_0t}\ue^{K_0\eta}\dfrac{\p(\chi_2\uf_0)}{\p\psi}(t)}\leq C\e^{\alpha}\abs{\ln(\e)}^{8}.
\end{array}
\end{align}
\end{theorem}

\subsection{Analysis of Initial Layer}

We divide the analysis into several steps:\\
\ \\
Step 1: Analysis of $\ui_0$.\\
Using \eqref{et 6.}, $\ui_0$ satisfies
\begin{align}
\ui_0(\tau,\vx,\vw)=\ue^{-\tau}(\ui_0-\bui_0)=\ue^{-\tau}(\h-\bar h).
\end{align}
Naturally, it decays exponentially in $\tau$.\\
\ \\
Step 1: Analysis of $\ui_1$.\\
Using \eqref{et 7.}, $\ui_1$ satisfies
\begin{align}
\p_{\tau}\bui_1=-\displaystyle\int_{\s^1}\Big(\vw\cdot\nabla_x\ui_{0}\Big)\ud{\vw}=-\ue^{-\tau}\int_{\s^2}\Big(\vw\cdot\nx(\h-\bar\h)\Big)\ud{\vw},
\end{align}
which, combined with $\bui_1(\infty,\vx)=0$, implies that
\begin{align}
\bui_1(\tau,\vx)=\ue^{-\tau}\int_{\s^2}\Big(\vw\cdot\nx(\h-\bar\h)\Big)\ud{\vw}.
\end{align}
Then using \eqref{et 7.}, we have
\begin{align}
\ui_1(\tau,\vx,\vw)=\ue^{-\tau}\bigg(\vw\cdot\nx\u_0(0)+\int_{\s^2}\Big(\vw\cdot\nx(\h-\bar\h)\Big)\ud{\vw}-\vw\cdot\nx(\h-\bar\h)\bigg).
\end{align}
Naturally, it decays exponentially in $\tau$.
\begin{theorem}\label{dt theorem 4.}
For $K_0>0$ sufficiently small, the initial layer satisfies
\begin{align}
\begin{array}{ll}
\im{\ue^{K_0\tau}\ui_0(\vx)}{\rp\times\s^2}\leq C,& \im{\ue^{K_0\tau}\ui_1(\vx)}{\rp\times\s^2}\leq C,\\\rule{0ex}{2em}
\im{\ue^{K_0\tau}\nx\ui_0(\vx)}{\rp\times\s^2}\leq C,&\im{\ue^{K_0\tau}\nx\ui_1(\vx)}{\rp\times\s^2}\leq C.
\end{array}
\end{align}
\end{theorem}

\subsection{Analysis of Interior Solution}

In this subsection, we will justify that the interior solutions are all well-defined. We divide it into several steps:\\
\ \\
Step 1: Well-Posedness of $\u_0$.\\
$\u_0$ satisfies a parabolic equation
\begin{align}
\left\{
\begin{array}{l}
\u_0(t,\vx,\vw)=\bu_0(t,\vx) ,\\\rule{0ex}{1.5em} \dt\bu_0-\dfrac{1}{3}\Delta_x\bu_0=0\ \ \text{in}\
\ \Omega,\\\rule{0ex}{1.5em}
\bu_0(0,\vx)=\mathfrak{f}_0(\infty,\vx)\ \ \text{in}\ \
\Omega,\\\rule{0ex}{1.5em}
\bu_0(t,\vx_0)=\mathscr{F}_{0,L}(t,\iota_1,\iota_2)+\mathfrak{F}_{0,L}(t,\iota_1,\iota_2)\ \ \text{on}\ \
\p\Omega.
\end{array}
\right.
\end{align}
Based on standard parabolic theory, we have
\begin{align}
\nm{\ue^{K_0t}\u_0}_{L^{\infty}_tH^3_xL^{\infty}_w([0,\infty)\times\Omega\times\s^2)}\leq C.
\end{align}
\ \\
Step 2: Well-Posedness of $\u_1$.\\
$\u_1$ satisfies a parabolic equation
\begin{align}
\left\{
\begin{array}{l}
\u_1(\vx,\vw)=\bu_1(\vx)-\vw\cdot\nx\u_0(\vx,\vw),\\\rule{0ex}{1.5em}
\dt\bu_1-\dfrac{1}{3}\Delta_x\bu_1=-\displaystyle\int_{\s^1}\Big(\vw\cdot\nx\u_{0}(\vx,\vw)\Big)\ud{\vw}\
\ \text{in}\ \ \Omega,\\\rule{0ex}{1em}
\bu_1(0,\vx)=\mathfrak{f}_1(\infty,\vx)\ \ \text{in}\ \
\Omega,\\\rule{0ex}{1.5em}\bu_1(t,\vx_0)=f _{1,L}(t,\iota_1,\iota_2)\ \ \text{on}\ \
\p\Omega.
\end{array}
\right.
\end{align}
Based on standard parabolic theory, we have
\begin{align}
\nm{\ue^{K_0t}\u_1}_{L^{\infty}_tH^3_xL^{\infty}_w([0,\infty)\times\Omega\times\s^2)}\leq C\abs{\ln(\e)}^8.
\end{align}
\ \\
Step 3: Well-Posedness of $\u_2$.\\
$\u_2$ satisfies a parabolic equation
\begin{align}
\left\{
\begin{array}{l}
\u_{2}(\vx,\vw)=\bu_{2}(\vx)-\vw\cdot\nx\u_{1}(\vx,\vw),\\\rule{0ex}{1.5em}
\dt\bu_2-\dfrac{1}{3}\Delta_x\bu_{2}=-\displaystyle\int_{\s^1}\Big(\vw\cdot\nx\u_{1}(\vx,\vw)\Big)\ud{\vw}\
\ \text{in}\ \ \Omega,\\\rule{0ex}{1.5em}
\bu_2(0,\vx)=0\ \ \text{in}\ \
\Omega,\\\rule{0ex}{1.5em}\bu_2(t,\vx_0)=0\ \ \text{on}\ \
\p\Omega.
\end{array}
\right.
\end{align}
Based on standard parabolic theory, we have
\begin{align}
\nm{\ue^{K_0t}\u_2}_{L^{\infty}_tH^3_xL^{\infty}_w([0,\infty)\times\Omega\times\s^2)}\leq C\abs{\ln(\e)}^8.
\end{align}
\ \\
\begin{theorem}\label{dt theorem 3.}
The interior solution satisfies
\begin{align}
\\
&\nm{\ue^{K_0t}\u_0}_{L^{\infty}_tH^3_xL^{\infty}_w([0,\infty)\times\Omega\times\s^2)}+\nm{\ue^{K_0t}\dt\u_0}_{L^{\infty}_tH^3_xL^{\infty}_w([0,\infty)\times\Omega\times\s^2)}
+\nm{\ue^{K_0t}\nx\u_0}_{L^{\infty}_tH^3_xL^{\infty}_w([0,\infty)\times\Omega\times\s^2)}\leq C,\no\\
\\
&\nm{\ue^{K_0t}\u_1}_{L^{\infty}_tH^3_xL^{\infty}_w([0,\infty)\times\Omega\times\s^2)}+\nm{\ue^{K_0t}\dt\u_1}_{L^{\infty}_tH^3_xL^{\infty}_w([0,\infty)\times\Omega\times\s^2)}
+\nm{\ue^{K_0t}\nx\u_1}_{L^{\infty}_tH^3_xL^{\infty}_w([0,\infty)\times\Omega\times\s^2)}\leq C\abs{\ln(\e)}^8,\no\\
\\
&\nm{\ue^{K_0t}\u_2}_{L^{\infty}_tH^3_xL^{\infty}_w([0,\infty)\times\Omega\times\s^2)}+\nm{\ue^{K_0t}\dt\u_2}_{L^{\infty}_tH^3_xL^{\infty}_w([0,\infty)\times\Omega\times\s^2)}
+\nm{\ue^{K_0t}\nx\u_2}_{L^{\infty}_tH^3_xL^{\infty}_w([0,\infty)\times\Omega\times\s^2)}\leq C\abs{\ln(\e)}^8.\no
\end{align}
\end{theorem}

\subsection{Analysis of Initial-Boundary Layer}

The initial and boundary data satisfy the
compatibility condition
\begin{align}
\h(\vx_0,\vw)=\g(0,\vx_0,\vw)\ \ \text{for}\ \ \vx_0\in\p\Omega\ \ \text{and}\ \ \vw\cdot\vn<0.
\end{align}
Then in the half-space $\vw\cdot\vn<0$ at $(0,\vx_0)$, the equation
\begin{align}
\e^2\dt u^{\e}+\e \vw\cdot\nabla_x u^{\e}+u^{\e}-\bar
u^{\e}=&0,
\end{align}
is valid, which implies that for arbitrary $\e$,
\begin{align}
\e^2\dt g(0,\vx_0,\vw)+\e \vw\cdot\nabla_xh(\vx_0,\vw)+h(\vx_0,\vw)-\bar
h(\vx_0)=&0.
\end{align}
Since $g$ and $h$ are both independent of $\e$, we must have that for $\vw\cdot\vn<0$,
\begin{align}
\dt g(0,\vx_0,\vw)=\vw\cdot\nabla_xh(\vx_0,\vw)=h(\vx_0,\vw)-\bar
h(\vx_0)=&0.
\end{align}
Then we obtain the improved compatibility condition
\begin{align}\label{improved compatibility condition}
\left\{
\begin{array}{l}
h(\vx_0,\vw)=g(0,\vx_0,\vw)=C_0,\\\rule{0ex}{2.0em}
\dt g(0,\vx_0,\vw)=\vw\cdot\nabla_xh(\vx_0,\vw)=h(\vx_0,\vw)-\bar
h(\vx_0)=0,
\end{array}
\right.\ \ \text{for}\ \ \vx_0\in\p\Omega\ \ \text{and}\ \ \vw\cdot\vn<0,
\end{align}
for some constant $C_0$.

\begin{theorem}\label{dt theorem 5.}
The initial and boundary layers at $\{0\}\times\p\Omega$ satisfy
\begin{align}
&\ub_0(0,\eta,\iota_1,\iota_2,\phi,\psi)=\uf_0(0,\eta,\iota_1,\iota_2,\phi,\psi)=0,\\
&\ui_0(t,\vx_0,\vw)=0\ \ \text{for}\ \ \vw\cdot\vn<0.
\end{align}
\end{theorem}

\subsection{Proof of Main Theorem}

\begin{theorem}\label{diffusive limit.}
Assume $h(\vx,\vw)\in C(\Omega\times\s^2)$ and $\ue^{Kt}g(t,\vx_0,\vw)\in C^3(\rp\times\Gamma^-)$ with $\ue^{Kt}\dt g\in C^1(\rp\times\Gamma^-)$. Then for the unsteady neutron
transport equation \eqref{transport.}, there exists a unique solution
$u^{\e}(t,\vx,\vw)\in L^{\infty}(\rp\times\Omega\times\s^2)$. Moreover, for some $0<K_0\leq K$ and any $0<\d<<1$, the solution obeys the estimate
\begin{align}
\nm{\ue^{K_0t}\Big(u^{\e}-\u-\ui-\uu\Big)}_{L^{\infty}(\rp\times\Omega\times\s^2)}\leq C(\d)\e^{\frac{1}{6}-\d},
\end{align}
where $\u(t,\vx)$ satisfies the heat equation with Dirichlet boundary condition
\begin{align}
\left\{
\begin{array}{l}
\dt\u-\Delta_x\u(\vx)=0\ \ \text{in}\
\ \rp\times\Omega,\\\rule{0ex}{1.5em}
\u(0,\vx)=\ds\frac{1}{4\pi}\int_{\s^2}h(\vx,\vw)\ud{\vw}\ \ \text{in}\ \ \Omega,\\\rule{0ex}{1.5em}
\u(t,\vx_0)=D(t,\vx_0)\ \ \text{on}\ \
\p\Omega,
\end{array}
\right.
\end{align}
$\ui(\iota,\vx,\vw)$ satisfies
\begin{align}
\ui(\iota,\vx,\vw)=\ue^{-\iota}\left(h(\vx,\vw)-\frac{1}{4\pi}\int_{\s^2}h(\vx,\vw)\ud{\vw}\right),
\end{align}
for $\iota$ the rescaled time variable, and $\uu(t,\eta,\iota_1,\iota_2,\phi)$ satisfies the $\e$-Milne problem with geometric correction
\begin{align}
\left\{
\begin{array}{l}
\sin\phi\dfrac{\p \uu }{\p\eta}-\bigg(\dfrac{\sin^2\psi}{R_1(\iota_1,\iota_2)-\e\eta}+\dfrac{\cos^2\psi}{R_2(\iota_1,\iota_2)-\e\eta}\bigg)\cos\phi\dfrac{\p
\uu}{\p\phi}+\uu -\buu =0,\\\rule{0ex}{1.5em}
\uu (t,0,\iota_1,\iota_2,\phi,\psi)=g(t,\iota_1,\iota_2,\phi,\psi)-D(t,\iota_1,\iota_2,)\ \ \text{for}\ \
\sin\phi>0,\\\rule{0ex}{1.5em}
\uu (t,L,\iota_1,\iota_2,\phi,\psi)=\uu (t,L,\iota_1,\iota_2,\rr[\phi],\psi),
\end{array}
\right.
\end{align}
for $L=\e^{-n}$ with $0<n<\dfrac{1}{2}$, $\rr[\phi]=-\phi$, $\eta$ the rescaled normal variable, $(\iota_1,\iota_2)$ the tangential variables, and $(\phi,\psi)$ the velocity variables.
\end{theorem}
\begin{proof}
Based on Theorem \ref{LI estimate.}, we know there exists a unique $u^{\e}(t,\vx,\vw)\in L^{\infty}(\rp\times\Omega\times\s^2)$, so we focus on the diffusive limit. \\
\ \\
Step 1: Remainder definitions.\\
We may rewrite the asymptotic expansion as follows:
\begin{align}
u^{\e}\sim&\sum_{k=0}^{2}\e^k\u_k+\sum_{k=0}^{1}\e^k\ui_{k}+\sum_{k=0}^{1}\e^k\ub_{k}+\uf_0.
\end{align}
The remainder can be defined as
\begin{align}
R=&u^{\e}-\sum_{k=0}^{2}\e^k\u_k-\sum_{k=0}^{1}\e^k\ui_{k}-\sum_{k=0}^{1}\e^k\ub_{k}-\uf_0=u^{\e}-\q-\qi-\qb-\qf,
\end{align}
where
\begin{align}
\q=\sum_{k=0}^{2}\e^k\u_k,\quad
\qi=\sum_{k=0}^{1}\e^k\ui_{k},\quad
\qb=\sum_{k=0}^{1}\e^k\ub_{k},\quad
\qf=\uf_0.
\end{align}
Noting the equation \eqref{transport.} is equivalent to the
equations \eqref{initial.} and \eqref{coordinate 23.}, we write $\ll$ to denote the neutron
transport operator as follows:
\begin{align}
\ll u=&\e^2\dt u+\e\vw\cdot\nx u+u-\bar u
=\p_{\tau}u+\e\vw\cdot\nabla_xu+u-\bar u\\
=&\e^2\frac{\p u}{\p t}+\sin\phi\dfrac{\p u}{\p\eta}-\e\bigg(\dfrac{\sin^2\psi}{R_1-\e\eta}+\dfrac{\cos^2\psi}{R_2-\e\eta}\bigg)\cos\phi\dfrac{\p u}{\p\phi}\no\\\rule{0ex}{2.0em}
&+\e\bigg(\dfrac{\cos\phi\sin\psi}{P_1(1-\e\kk_1\eta)}\dfrac{\p u}{\p\iota_1}+\dfrac{\cos\phi\cos\psi}{P_2(1-\e\kk_2\eta)}\dfrac{\p u}{\p\iota_2}\bigg)\no\\\rule{0ex}{2.0em}
&+\e\Bigg(\dfrac{\sin\psi}{1-\e\kk_1\eta}\bigg(\cos\phi\Big(\vt_1\cdot\Big(\vt_2\times(\p_{12}\vr\times\vt_2)\Big)\Big)
-\kk_1P_1P_2\sin\phi\cos\psi\bigg)\no\\\rule{0ex}{2.0em}
&+\dfrac{\cos\psi}{1-\e\kk_2\eta}\bigg(-\cos\phi\Big(\vt_2\cdot\Big(\vt_1\times(\p_{12}\vr\times\vt_1)\Big)\Big)
+\kk_2P_1P_2\sin\phi\sin\psi\bigg)\Bigg)\dfrac{1}{P_1P_2}\dfrac{\p u}{\p\psi}+u-\bar u.\no
\end{align}
\ \\
Step 2: Estimates of $\ll[\q]$.\\
The interior contribution can be estimated as
\begin{align}
\ll[\q]=&\e^3\dt\u_1+\e^4\dt\u_2+\e^{3}\vw\cdot\nx \u_2.
\end{align}
Using Theorem \ref{dt theorem 3.}, we have
\begin{align}
\tm{\ll[\q]}{\rp\times\Omega\times\s^2}\leq& C\e^{3},\\
\nm{\ll[\q]}_{L^{\frac{2m}{2m-1}}(\rp\times\Omega\times\s^2)}\leq& C\e^{3},\\
\nm{\ll[\q]}_{L^{2m}(\rp\times\Omega\times\s^2)}\leq& C\e^{3},\\
\im{\ll[\q]}{\rp\times\Omega\times\s^2}\leq& C\e^{3},\\
\tm{\dt\ll[\q]}{\rp\times\Omega\times\s^2}\leq& C\e^{3},\\
\tm{\ll[\q](0)}{\Omega\times\s^2}\leq& C\e^{3}.
\end{align}
\ \\
Step 3: Estimates of $\ll[\qi]$.\\
The initial layer contribution can be estimated as
\begin{align}
\ll[\qi]=\e^{2}\vw\cdot\nabla_x\ui_{1}.
\end{align}
Based on Theorem \ref{dt theorem 4.}, we know
\begin{align}
\im{\e^{2}\vw\cdot\nabla_x\ui_{1}}{\rp\times\Omega\times\s^2}\leq& C\e^{2}.
\end{align}
Note that $\ub_{1}^I$ decays exponentially in $\tau$ and the scaling $\tau=\dfrac{t}{\e^2}$, we have
\begin{align}
\tm{\e^{2}\vw\cdot\nabla_x\ui_{1}}{\rp\times\Omega\times\s^2}\leq&
\e^2\bigg(\int_0^{\infty}\int_{\Omega\times\s^2}\abs{\nabla_x\ui_1}^2(\tau,\vx,\vw)\ud{\vx}\ud{\vw}\ud{t}\bigg)^{\frac{1}{2}}\\
\leq&C\e^2\bigg(\int_0^{\infty}\ue^{-\tau}\ud{t}\bigg)^{\frac{1}{2}}
=C\e^3\bigg(\int_0^{\infty}\ue^{-\tau}\ud{\tau}\bigg)^{\frac{1}{2}}\leq C\e^3.\no
\end{align}
Similarly, we can derive that
\begin{align}
\tm{\ll[\qi]}{\rp\times\Omega\times\s^2}\leq& C\e^{3},\\
\nm{\ll[\qi]}_{L^{\frac{2m}{2m-1}}(\rp\times\Omega\times\s^2)}\leq& C\e^{4-\frac{1}{m}},\\
\nm{\ll[\qi]}_{L^{2m}(\rp\times\Omega\times\s^2)}\leq& C\e^{2+\frac{1}{m}},\\
\im{\ll[\qi]}{\rp\times\Omega\times\s^2}\leq& C\e^{2},\\
\tm{\dt\ll[\qi]}{\rp\times\Omega\times\s^2}\leq& C\e^{3},\\
\tm{\ll[\qi](0)}{\Omega\times\s^2}\leq& C\e^{2}.
\end{align}
\ \\
Step 4: Estimates of $\ll[\qb]$.\\
The regular boundary layer contribution can be
estimated as
\begin{align}
\ll[\qb]=
&\e^2\dt\ub_0+\e^3\dt\ub_1+\e^2\bigg(\dfrac{\cos\phi\sin\psi}{P_1(1-\e\kk_1\eta)}\dfrac{\p \ub_1}{\p\iota_1}+\dfrac{\cos\phi\cos\psi}{P_2(1-\e\kk_2\eta)}\dfrac{\p \ub_1}{\p\iota_2}\bigg)\\\rule{0ex}{2.0em}
&+\e^2\Bigg(\dfrac{\sin\psi}{1-\e\kk_1\eta}\bigg(\cos\phi\Big(\vt_1\cdot\Big(\vt_2\times(\p_{12}\vr\times\vt_2)\Big)\Big)
-\kk_1P_1P_2\sin\phi\cos\psi\bigg)\no\\\rule{0ex}{2.0em}
&+\dfrac{\cos\psi}{1-\e\kk_2\eta}\bigg(-\cos\phi\Big(\vt_2\cdot\Big(\vt_1\times(\p_{12}\vr\times\vt_1)\Big)\Big)
+\kk_2P_1P_2\sin\phi\sin\psi\bigg)\Bigg)\dfrac{1}{P_1P_2}\dfrac{\p\ub_1}{\p\psi}\no.
\end{align}
Similarly to steady problems, based on Theorem \ref{dt theorem 1.} and Theorem \ref{dt theorem 5.}, we have
\begin{align}
\tm{\ll[\qb]}{\rp\times\Omega\times\s^2}\leq& C\e^{\frac{5}{2}}\abs{\ln(\e)}^8,\\
\nm{\ll[\qb]}_{L^{\frac{2m}{2m-1}}(\rp\times\Omega\times\s^2)}\leq& C\e^{3-\frac{1}{2m}}\abs{\ln{\e}}^8,\\
\nm{\ll[\qb]}_{L^{2m}(\rp\times\Omega\times\s^2)}\leq& C\e^{2+\frac{1}{2m}}\abs{\ln{\e}}^8,\\
\im{\ll[\qb]}{\rp\times\Omega\times\s^2}\leq& C\e^2\abs{\ln{\e}}^8,\\
\tm{\dt\ll[\qb]}{\rp\times\Omega\times\s^2}\leq& C\e^{\frac{5}{2}}\abs{\ln(\e)}^8,\\
\tm{\ll[\qb](0)}{\Omega\times\s^2}=& 0.
\end{align}
\ \\
Step 5: Estimates of $\ll[\qf]$.\\
The singular boundary layer contribution can be
estimated as
\begin{align}
\ll[\qf]=
&\e^2\dt\uf_0+\e\bigg(\dfrac{\cos\phi\sin\psi}{P_1(1-\e\kk_1\eta)}\dfrac{\p \uf_0}{\p\iota_1}+\dfrac{\cos\phi\cos\psi}{P_2(1-\e\kk_2\eta)}\dfrac{\p \uf_0}{\p\iota_2}\bigg)\no\\\rule{0ex}{2.0em}
&+\e\Bigg(\dfrac{\sin\psi}{1-\e\kk_1\eta}\bigg(\cos\phi\Big(\vt_1\cdot\Big(\vt_2\times(\p_{12}\vr\times\vt_2)\Big)\Big)
-\kk_1P_1P_2\sin\phi\cos\psi\bigg)\no\\\rule{0ex}{2.0em}
&+\dfrac{\cos\psi}{1-\e\kk_2\eta}\bigg(-\cos\phi\Big(\vt_2\cdot\Big(\vt_1\times(\p_{12}\vr\times\vt_1)\Big)\Big)
+\kk_2P_1P_2\sin\phi\sin\psi\bigg)\Bigg)\dfrac{1}{P_1P_2}\dfrac{\p \uf_0}{\p\psi}.\no
\end{align}
Similarly to steady problems, based on Theorem \ref{dt theorem 2.} and Theorem \ref{dt theorem 5.}, we have
\begin{align}
\tm{\ll[\qf]}{\rp\times\Omega\times\s^2}\leq& C\e^{1+\frac{3}{2}\alpha}\abs{\ln(\e)}^8,\\
\nm{\ll[\qf]}_{L^{\frac{2m}{2m-1}}(\rp\times\Omega\times\s^2)}\leq& C\e^{2-\frac{1}{2m}+\alpha}\abs{\ln(\e)}^8,\\
\nm{\ll[\qf]}_{L^{2m}(\rp\times\Omega\times\s^2)}\leq& C\e^{1+\frac{1}{2m}+\alpha}\abs{\ln(\e)}^8,\\
\im{\ll[\qf]}{\rp\times\Omega\times\s^2}\leq& C\e\abs{\ln(\e)}^8,\\
\tm{\dt\ll[\qf]}{\rp\times\Omega\times\s^2}\leq& C\e^{1+\frac{3}{2}\alpha}\abs{\ln(\e)}^8,\\
\tm{\ll[\qf](0)}{\Omega\times\s^2}=& 0.
\end{align}
\ \\
Step 6: Estimate of Source Term.\\
Summarizing all above, for $\alpha= 1$, we have
\begin{align}
\tm{\ll[R]}{\rp\times\Omega\times\s^2}\leq& C\e^{\frac{5}{2}}\abs{\ln(\e)}^8,\\
\nm{\ll[R]}_{L^{\frac{2m}{2m-1}}(\rp\times\Omega\times\s^2)}\leq& C\e^{3-\frac{1}{2m}}\abs{\ln(\e)}^8,\\
\nm{\ll[R]}_{L^{2m}(\rp\times\Omega\times\s^2)}\leq& C\e^{2+\frac{1}{2m}}\abs{\ln(\e)}^8,\\
\im{\ll[R]}{\rp\times\Omega\times\s^2}\leq& C\e\abs{\ln(\e)}^8,\\
\tm{\dt\ll[R]}{\rp\times\Omega\times\s^2}\leq& C\e^{\frac{5}{2}}\abs{\ln(\e)}^8,\\
\tm{\ll[R](0)}{\Omega\times\s^2}\leq&C\e^2.
\end{align}
\ \\
Step 7: Estimate of Initial Data.\\
At $t=0$, the initial data of $R$ is
\begin{align}
R^I=\e\ub_1(0,\eta,\iota_1,\iota_2,\phi,\psi)+\e^2\vw\cdot\nx\u_1(0,\vx,\vw).
\end{align}
Using Theorem \ref{dt theorem 1.} and Theorem \ref{dt theorem 3.}, considering the spacial rescaling, we may derive
\begin{align}
\nm{R^I}_{L^{2m}({\Omega\times\s^2})}\leq&C\e^{1+\frac{1}{2m}},\\
\im{R^I}{\Omega\times\s^2}\leq&C\e.
\end{align}
\ \\
Step 8: Estimate of Boundary Data.\\
At $\Gamma^-$, the boundary data or $R$ is
\begin{align}
R^B=\e\ui_1(\tau,\vx_0,\vw)+\e^2\nx\u_1(t,\vx_0,\vw).
\end{align}
Using Theorem \ref{dt theorem 4.} and Theorem \ref{dt theorem 3.}, considering the temporal rescaling, we may derive
\begin{align}
\nm{R^B}_{L^{2m}(\rp\times\Gamma^-)}\leq&C\e^{1+\frac{1}{m}},\\
\im{R^B}{(\rp\times\Gamma^-)}\leq&C\e.
\end{align}
\ \\
Step 9: Diffusive Limit.\\
Therefore, the remainder $R$ satisfies the equation
\begin{align}
\left\{
\begin{array}{l}
\e^2\dt R+\e \vw\cdot\nabla_x R+R-\bar
R=\ll[R]\ \ \ \text{for}\ \
(t,\vx,\vw)\in\rp\times\Omega\times\s^2,\\\rule{0ex}{2.0em}
R(0,\vx,\vw)=R^I(\vx,\vw)\ \ \text{for}\ \ (\vx,\vw)\in\Omega\times\s^2\\\rule{0ex}{2.0em}
R(t,\vx_0,\vw)=R^B(t,\vx_0,\vw)\ \ \text{for}\ \ t\in\rp,\ \ \vx_0\in\p\Omega,\ \ \text{and}\ \ \vw\cdot\vn<0.
\end{array}
\right.
\end{align}
By Theorem \ref{LI estimate..}, for integer $1\leq m\leq 3$,
\begin{align}
&\nm{R}_{L^{\infty}(\rp\times\Omega\times\s^2)}\\
\leq& C \bigg(\frac{1}{\e^{2+\frac{2}{m}}}\nm{\ll[R]}_{L^{\frac{2m}{2m-1}}(\rp\times\Omega\times\s^2)}
+\frac{1}{\e^{1+\frac{2}{m}}}\nm{\ll[R]}_{L^2(\rp\times\Omega\times\s^2)}+\frac{1}{\e^{1+\frac{2}{m}}}\nm{\ll[R]}_{L^{2m}(\rp\times\Omega\times\s^2)}
+\im{\ll[R]}{\rp\times\Omega\times\s^2}\no\\
&+\frac{1}{\e^{\frac{2}{m}}}\tm{\dt\ll[R]}{\rp\times\Omega\times\s^2}+
\frac{1}{\e^{\frac{2}{m}}}\nm{\ll[R](0)}_{L^2(\Omega\times\s^2)}\no\\
&+\frac{1}{\e^{\frac{3}{2m}}}\nm{R^I}_{L^{2m}(\Omega\times\s^2)}+\im{R^I}{\Omega\times\s^2}+\frac{1}{\e^{\frac{2}{m}}}\nm{R^B}_{L^{2m}(\rp\times\Gamma^-)}
+\im{R^B}{\rp\times\Gamma^-}\bigg)\no\\
\leq&C\e^{1-\frac{5}{2m}}\abs{\ln(\e)}^8\leq \e^{\frac{1}{6}}\abs{\ln(\e)}^8.\no
\end{align}
Since it is obvious that
\begin{align}
\im{\e\u_1+\e^2\u_2+\e\ui_{1}+\e\ub_{1}}{\rp\times\Omega\times\s^2}\leq C\e,
\end{align}
we naturally have for any $0<\d<<1$,
\begin{align}
\nm{u^{\e}-\u_0-\ui_0-\ub_0-\uf_0}_{L^{\infty}(\rp\times\Omega\times\s^2)}\leq C(\d)\e^{\frac{1}{6}-\d}.
\end{align}
\ \\
Step 10: Exponential Decay.\\
The exponential decay can be easily derived following a similar argument using Theorem \ref{LI estimate.}. In particular, all the estimates remain the same with extra exponential terms. Then we simply take $\u=\u_0$, $\ui=\ui_0$ and $\uu=\ub_0+\uf_0$.
\end{proof}

\chapter{$\e$-Milne Problem with Geometric Correction}

\section{Well-Posedness and Decay}

In this section, we temporarily ignore the superscript $\e$ and hide the dependence on $(\iota_1,\iota_2)$, i.e. we consider the $\e$-Milne problem with geometric correction for $f(\eta,\phi,\psi)$ in
the domain $(\eta,\phi,\psi)\in[0,L]\times\left[-\dfrac{\pi}{2},\dfrac{\pi}{2}\right]\times[-\pi,\pi]$:
\begin{align}\label{Milne problem}
\left\{ \begin{array}{l}\displaystyle \sin\phi\frac{\p
f}{\p\eta}+F(\eta,\psi)\cos\phi\frac{\p
f}{\p\phi}+f-\bar f=S(\eta,\phi,\psi),\\\rule{0ex}{2.0em}
f(0,\phi,\psi)= h(\phi,\psi)\ \ \text{for}\
\ \sin\phi>0,\\\rule{0ex}{2.0em}
f(L,\phi,\psi)=f(L,\rr[\phi],\psi),
\end{array}
\right.
\end{align}
where $L=\e^{-n}$ for some $0<n<\dfrac{1}{2}$, $\rr[\phi]=-\phi$,
\begin{align}\label{force}
F(\eta,\psi)=-\e\bigg(\dfrac{\sin^2\psi}{R_1-\e\eta}+\dfrac{\cos^2\psi}{R_2-\e\eta}\bigg),
\end{align}
for $R_1$ and $R_2$ radium of two principal curvatures, and
\begin{align}
\bar f(\eta)=\frac{1}{4\pi}\int_{-\pi}^{\pi}\int_{-\frac{\pi}{2}}^{\frac{\pi}{2}}f(\eta,\phi,\psi)\cos\phi\ud{\phi}\ud{\psi},
\end{align}
in which $\cos\phi$ shows up as the Jacobian of spherical coordinates in integration. Note that for $\phi\in\left[-\dfrac{\pi}{2},\dfrac{\pi}{2}\right]$, we always have $\cos\phi\geq0$, which means this will not destroy the positivity of the integral.
Here, all estimates we get will be uniform in $\e$, $\iota_1$ and $\iota_2$.\\
\ \\
In this section, we need to introduce special notation to highlight the different contribution of space and velocity. Define the norms in the
space $(\eta,\phi,\psi)\in[0,L]\times\left[-\dfrac{\pi}{2},\dfrac{\pi}{2}\right]\times[-\pi,\pi]$ as follows:
\begin{align}
\tnnm{f}=&\bigg(\int_0^{L}\int_{-\pi}^{\pi}\int_{-\frac{\pi}{2}}^{\frac{\pi}{2}}
\abs{f(\eta,\phi,\psi)}^2\cos\phi\ud{\phi}\ud{\psi}\ud{\eta}\bigg)^{\frac{1}{2}},\\
\lnnm{f}=&\text{esssup}_{(\eta,\phi,\psi)\in[0,L]\times[-\frac{\pi}{2},\frac{\pi}{2}]\times[-\pi,\pi]}\abs{f(\eta,\phi,\psi)},\\
\ltnm{f}=&\text{esssup}_{\eta\in[0,L]}\bigg(\int_{-\pi}^{\pi}\int_{-\frac{\pi}{2}}^{\frac{\pi}{2}}
\abs{f(\eta,\phi,\psi)}^2\cos\phi\ud{\phi}\ud{\psi}\bigg)^{\frac{1}{2}}.
\end{align}
Define the norms in the
space $(\phi,\psi)\in\left[-\dfrac{\pi}{2},\dfrac{\pi}{2}\right]\times[-\pi,\pi]$ as follows:
\begin{align}
\tnm{f(\eta)}=&\bigg(\int_{-\pi}^{\pi}\int_{-\frac{\pi}{2}}^{\frac{\pi}{2}}
\abs{f(\eta,\phi,\psi)}^2\cos\phi\ud{\phi}\ud{\psi}\bigg)^{\frac{1}{2}},\\
\lnm{f(\eta)}=&\text{esssup}_{(\phi,\psi)\in[-\frac{\pi}{2},\frac{\pi}{2}]\times[-\pi,\pi]}\abs{f(\eta,\phi,\psi)}.
\end{align}
Also, we define the inner product in $(\phi,\psi)$ as
\begin{align}
\br{f,g}_{\phi,\psi}(\eta)=&\int_{-\pi}^{\pi}\int_{-\frac{\pi}{2}}^{\frac{\pi}{2}}f(\eta,\phi,\psi)g(\eta,\phi,\psi)\cos\phi
\ud{\phi}\ud{\psi}.
\end{align}
In particular, define the norms at the in-flow boundary as
\begin{align}
\tnmp{h}=&\bigg(\iint_{\sin\phi>0,\psi\in[-\pi,\pi]}\abs{h(\phi,\psi)}^2\sin\phi\cos\phi\ud{\phi}\ud{\psi}\bigg)^{\frac{1}{2}},\\
\lnmp{h}=&\text{esssup}_{\sin\phi>0,\psi\in[-\pi,\pi]}\abs{h(\phi,\psi)}.
\end{align}
Here we purposely use $+$ to indicate in-flow boundary to simplify the analysis, which is different from the convention in remainder estimates.\\
\ \\
Assume the source term $S$ and boundary data $h$ satisfy
\begin{align}\label{Milne bounded}
\lnmp{h}+\lnmp{\frac{\p h}{\p\phi}}+\lnmp{\frac{\p h}{\p\psi}}\leq C,
\end{align}
and
\begin{align}\label{Milne decay}
\lnnm{\ue^{K\eta}S}+\lnnm{\ue^{K\eta}\frac{\p S}{\p\eta}}
+\lnnm{\ue^{K\eta}\frac{\p S}{\p\phi}}+\lnnm{\ue^{K\eta}\frac{\p S}{\p\psi}}\leq C,
\end{align}
for some constants $C>0$ and $K>0$ uniform in $\e$ and $\iota_i$ for $i=1,2$.

\subsection{$L^2$ Estimates}

\subsubsection{$\bar S=0$ Case}

In \eqref{force}, we may decompose the force
\begin{align}
F(\eta,\psi)=\tf(\eta)+G(\eta)\cos^2\psi,
\end{align}
for
\begin{align}\label{mt 17}
\tf(\eta)=-\frac{\e}{R_1-\e\eta},\quad
G(\eta)=-\frac{\e(R_1-R_2)}{(R_1-\e\eta)(R_2-\e\eta)}.
\end{align}
We may decompose the solution
\begin{align}
f(\eta,\phi,\psi)=q(\eta)+r(\eta,\phi,\psi),
\end{align}
where the hydrodynamical part $q$ is in the null space of the
operator $f-\bar f$, and the microscopic part $r$ is
the orthogonal complement, i.e.
\begin{align}\label{hydro}
q(\eta)=\frac{1}{4\pi}\int_{-\pi}^{\pi}\int_{-\frac{\pi}{2}}^{\frac{\pi}{2}}f(\eta,\phi,\psi)\cos\phi\ud{\phi}\ud{\psi},\quad
r(\eta,\phi,\psi)=f(\eta,\phi,\psi)-q(\eta).
\end{align}
Define a potential function $\tv(\eta)$ satisfying $\tv(0)=0$ and $\dfrac{\p \tv}{\p\eta}=-\tf(\eta)$. We may directly obtain
\begin{align}
\tv(\eta)=\ln\left(\frac{R_1}{R_1-\e\eta}\right).
\end{align}

\begin{lemma}\label{Milne finite LT}
Assume $\bar S=0$ and $h$ satisfy \eqref{Milne bounded} and \eqref{Milne decay}. Then there exists a solution $f(\eta,\phi,\psi)$ to the equation
\eqref{Milne problem}, satisfying that for $f_L=\dfrac{\br{\sin^2\phi,f}_{\phi,\psi}(L)}{\tnm{\sin\phi}^2}$,
\begin{align}\label{LT estimate}
\abs{f_L}\leq C,\quad \tnnm{f-f_L}\leq C.
\end{align}
The solution is unique among functions such that \eqref{LT estimate} holds.
\end{lemma}
\begin{proof}
\ \\
Step 1: Estimates of $r$.\\
Multiplying $f\cos\phi$
on both sides of \eqref{Milne problem} and
integrating over $(\phi,\psi)\in\left[-\dfrac{\pi}{2},\dfrac{\pi}{2}\right]\times[-\pi,\pi]$, we get the energy estimate
\begin{align}\label{mt 01}
\half\frac{\ud{}}{\ud{\eta}}\br{f
,f\sin\phi}_{\phi,\psi}(\eta)+\tf(\eta)\br{\frac{\p
f}{\p\phi},f\cos\phi}_{\phi,\psi}(\eta)=&-\tnm{r(\eta)}^2\\
&-G(\eta)\br{\frac{\p
f}{\p\phi}\cos^2\psi,f\cos\phi}_{\phi,\psi}(\eta)+\br{S,f}_{\phi,\psi}(\eta).\no
\end{align}
Integration by parts in $\phi$ reveals
\begin{align}
\tf(\eta)\br{\frac{\p
f}{\p\phi},f\cos\phi}_{\phi,\psi}(\eta)=&\tf(\eta)
\br{f,f\sin\phi}_{\phi,\psi}(\eta),\\
-G(\eta)\br{\frac{\p
f}{\p\phi}\cos^2\psi,f\cos\phi}_{\phi,\psi}(\eta)=&-G(\eta)\br{f\cos^2\psi,f\sin\phi}_{\phi,\psi}(\eta).
\end{align}
Hence, we can simplify \eqref{mt 01}
\begin{align}\label{mt 02}
\half\frac{\ud{}}{\ud{\eta}}\br{
f,f\sin\phi}_{\phi,\psi}(\eta)+\tf(\eta)\br{
f,f\sin\phi}_{\phi,\psi}(\eta)=&-\tnm{r(\eta)}^2\\
&-G(\eta)\br{f\cos^2\psi,f\sin\phi}_{\phi,\psi}(\eta)+\br{S,f}_{\phi,\psi}(\eta).\no
\end{align}
Let
\begin{align}\label{mt 27}
\alpha(\eta)=\half\br{f,f\sin\phi}_{\phi,\psi}(\eta).
\end{align}
Then \eqref{mt 02} can be rewritten as
\begin{align}
\frac{\ud{\alpha}}{\ud{\eta}}+2\tf(\eta)\alpha(\eta)=-\tnm{r(\eta)}^2-G(\eta)\br{f\cos^2\psi,f\sin\phi}_{\phi,\psi}(\eta)+\br{S,f}_{\phi,\psi}(\eta).
\end{align}
This is a first-order linear equation for $\alpha$. We can solve it in $[\eta,L]$ and $[0,\eta]$ respectively to obtain
\begin{align}
\label{mt 03}\\
\alpha(\eta)=&\ue^{2\tv(\eta)-2\tv(L)}\alpha(L)+\int_{\eta}^L\ue^{2\tv(\eta)-2\tv(y)}
\Big(\tnm{r(y)}^2+G(y)\br{f\cos^2\psi,f\sin\phi}_{\phi,\psi}(y)-\br{S,f}_{\phi,\psi}(y)\Big)\ud{y},\no\\
\label{mt 04}\\
\alpha(\eta)=&\ue^{2\tv(\eta)}\alpha(0)+\int_{0}^{\eta}\ue^{2\tv(\eta)-2\tv(y)}
\Big(-\tnm{r(y)}^2-G(y)\br{f\cos^2\psi,f\sin\phi}_{\phi,\psi}(y)+\br{S,f}_{\phi,\psi}(y)\Big)\ud{y}.\no
\end{align}
The specular reflexive boundary $f(L,\phi,\psi)=f(L,\rr[\phi],\psi)$
ensures that $\alpha(L)=0$. Then we may simplify \eqref{mt 03} as
\begin{align}
\label{mt 11}
\alpha(\eta)=&\int_{\eta}^L\ue^{2\tv(\eta)-2\tv(y)}
\Big(\tnm{r(y)}^2+G(y)\br{f\cos^2\psi,f\sin\phi}_{\phi,\psi}(y)-\br{S,f}_{\phi,\psi}(y)\Big)\ud{y}.
\end{align}
Hence, this yields
\begin{align}\label{mt 28}
\alpha(\eta)\geq\int_{\eta}^L\ue^{2\tv(\eta)-2\tv(y)}\Big(G(y)\br{f\cos^2\psi,f\sin\phi}_{\phi,\psi}(y)-\br{S,f}_{\phi,\psi}(y)\Big)\ud{y}.
\end{align}
On the other hand, 
\begin{align}\label{mt 12}
\alpha(0)=\half\br{f\sin\phi
,f}_{\phi,\psi}(0)\leq\half\iint_{\sin\phi>0}\abs{h(\phi,\psi)}^2\sin\phi\cos\phi
\ud{\phi}\ud\psi= \half\tnmp{h}^2.
\end{align}
Then in \eqref{mt 04}, taking $\eta=L$ and inserting \eqref{mt 12} with $\alpha(L)=0$, we have
\begin{align}\label{mt 13}
\int_{0}^L\ue^{-2\tv(y)}\tnm{r(y)}^2\ud{y}
=&\alpha(0)+\int_{0}^L\ue^{-2\tv(y)}\Big(-G(y)\br{f\cos^2\psi,f\sin\phi}_{\phi,\psi}(y)+\br{S,f}_{\phi,\psi}(y)\Big)\ud{y}\\
\leq&
C\tnmp{h}^2+\int_{0}^L\ue^{-2\tv(y)}\Big(-G(y)\br{f\cos^2\psi,f\sin\phi}_{\phi,\psi}(y)+\br{S,f}_{\phi,\psi}(y)\Big)\ud{y}\nonumber.
\end{align}
Also, we can directly
estimate
\begin{align}\label{mt 14}
\int_{0}^L\ue^{-2\tv(y)}\tnm{r(y)}^2\ud{y}\geq C\tnnm{r}^2.
\end{align}
Combining \eqref{mt 13} and \eqref{mt 14} yields
\begin{align}\label{mt 15}
\tnnm{r}^2\leq
C\bigg(\tnmp{h}^2+\int_{0}^L\ue^{-2\tv(y)}\Big(-G(y)\br{f\cos^2\psi,f\sin\phi}_{\phi,\psi}(y)+\br{S,f}_{\phi,\psi}(y)\Big)\ud{y}\bigg).
\end{align}
Since $\br{S,f}_{\phi,\psi}=\br{S,r}_{\phi,\psi}$ due to $\bar S=0$,
by Cauchy's inequality, we have
\begin{align}\label{mt 16}
\abs{\int_{0}^L\ue^{-2\tv(y)}\br{S,r}_{\phi,\psi}(y)\ud{y}}\leq&C'\tnnm{r}^2+C\tnnm{S}^2,
\end{align}
for some constant $C'>0$ sufficiently small. Also, from \eqref{mt 17}, we know
\begin{align}\label{mt 18}
\abs{-\int_{0}^L\ue^{-2\tv(y)}G(y)\br{f\cos^2\psi,f\sin\phi}_{\phi,\psi}(y)\ud{y}}\leq& C\abs{\int_{0}^LG(y)\br{f\cos^2\psi,f\sin\phi}_{\phi,\psi}(y)\ud{y}}\\
\leq&C\lnnm{G}\tnnm{f}^2\leq C\e\tnnm{f}^2.\no
\end{align}
Therefore, inserting \eqref{mt 16} and \eqref{mt 18} into \eqref{mt 15} and absorbing $C'\tnnm{r}^2$ to the left-hand side, we deduce
\begin{align}\label{mt 40}
\tnnm{r}^2\leq&C\left(
\tnmp{h}^2+\tnnm{S}^2+\e\tnnm{f}^2\right)\leq C\Big(1+\e\tnnm{f}^2\Big).
\end{align}
Note that this estimate is not closed since it depends on $f$.\\
\ \\
Step 2: Quasi-Orthogonality relation.\\
Multiplying $\cos\phi$ on both sides of \eqref{Milne problem} and integrating over $(\phi,\psi)\in\left[-\dfrac{\pi}{2},\dfrac{\pi}{2}\right]\times[-\pi,\pi]$ imply
\begin{align}
\frac{\ud{}}{\ud{\eta}}\br{\sin\phi,f}_{\phi,\psi}(\eta)=&-\tf\br{\cos\phi,\frac{\ud{f}}{\ud{\phi}}}_{\phi,\psi}(\eta)
-G\br{\cos\phi\cos^2\psi,\frac{\ud{f}}{\ud{\phi}}}_{\phi,\psi}(\eta)
+\bar S(\eta)\\
=&-2\tf\br{\sin\phi,f}_{\phi,\psi}(\eta)-2G\br{\sin\phi\cos^2\psi,f}_{\phi,\psi}(\eta),\no
\end{align}
where we use integration by parts in $\phi$ and the fact that $\bar S=0$. This is a first-order linear equation for $\br{\sin\phi,f}_{\phi,\psi}$. Note that the specular reflexive boundary
$f(L,\phi,\psi)=f(L,\rr[\phi],\psi)$ implies
$\br{\sin\phi,f}_{\phi,\psi}(L)=0$. We can solve it in $[\eta,L]$ to obtain that
\begin{align}\label{mt 19}
\br{\sin\phi,f}_{\phi,\psi}(\eta)=-2\int_{\eta}^L\ue^{2\tv(\eta)-2\tv(y)}G(y)\br{\sin\phi\cos^2\psi,f}_{\phi,\psi}(y)\ud{y}.
\end{align}
It is easy to check that
\begin{align}\label{mt 20}
\br{\sin\phi,q}_{\phi,\psi}(\eta)=\br{\sin\phi\cos^2\psi,q}_{\phi,\psi}(\eta)=0.
\end{align}
Hence, inserting \eqref{mt 20} into \eqref{mt 19}, we derive
\begin{align}\label{mt 21}
\br{\sin\phi,r}_{\phi,\psi}(\eta)=&-2\int_{\eta}^L\ue^{2\tv(\eta)-2\tv(y)}G(y)\br{\sin\phi\cos^2\psi,r}_{\phi,\psi}(y)\ud{y}.
\end{align}
\ \\
Step 3: Estimates of $q$.\\
Multiplying $\sin\phi\cos\phi$ on
both sides of \eqref{Milne problem} and
integrating over $(\phi,\psi)\in\left[-\dfrac{\pi}{2},\dfrac{\pi}{2}\right]\times[-\pi,\pi]$ lead to
\begin{align}
\label{mt 05}
\frac{\ud{}}{\ud{\eta}}\br{\sin^2\phi,f}_{\phi,\psi}(\eta)=&
-\br{\sin\phi,r}_{\phi,\psi}(\eta)-\tf(\eta)\br{\sin\phi\cos\phi,\frac{\p
f}{\p\phi}}_{\phi,\psi}(\eta)\\
&-G(\eta)\br{\sin\phi\cos\phi\cos^2\psi,\frac{\p
f}{\p\phi}}_{\phi,\psi}(\eta)+\br{\sin\phi,S}_{\phi,\psi}(\eta).\no
\end{align}
Integrate by parts in $\phi$ implies
\begin{align}
-\tf(\eta)\br{\sin\phi\cos\phi,\frac{\p
f}{\p\phi}}_{\phi,\psi}(\eta)=&\tf(\eta)\br{1-3\sin^2\phi,f}_{\phi,\psi}(\eta)\\
=&\tf(\eta)\br{1-3\sin^2\phi,r}_{\phi,\psi}(\eta),\no\\
-G(\eta)\br{\sin\phi\cos\phi\cos^2\psi,\frac{\p
f}{\p\phi}}_{\phi,\psi}(\eta)=&G(\eta)\br{1-3\sin^2\phi,f\cos^2\psi}_{\phi,\psi}(\eta)\\
=&G(\eta)\br{1-3\sin^2\phi,r\cos^2\psi}_{\phi,\psi}(\eta),\no
\end{align}
where we utilize the direct computation
\begin{align}
\br{1-3\sin^2\phi,q}_{\phi,\psi}(\eta)=\br{1-3\sin^2\phi,q\cos^2\psi}_{\phi,\psi}(\eta)=0.
\end{align}
Hence, we can simplify \eqref{mt 05} as
\begin{align}\label{mt 22}
\frac{\ud{}}{\ud{\eta}}\br{\sin^2\phi,f}_{\phi,\psi}(\eta)=&
-\br{\sin\phi,r}_{\phi,\psi}(\eta)+\tf(\eta)\br{1-3\sin^2\phi,r}_{\phi,\psi}(\eta)\\
&+G(\eta)\br{1-3\sin^2\phi,r\cos^2\psi}_{\phi,\psi}(\eta)+\br{\sin\phi,S}_{\phi,\psi}(\eta).\no
\end{align}
Let
\begin{align}
\beta(\eta)=\br{\sin^2\phi,f}_{\phi,\psi}(\eta).
\end{align}
Then \eqref{mt 22} can be rewritten as
\begin{align}\label{mt 06}
\frac{\ud{\beta}}{\ud{\eta}}=D(\eta),
\end{align}
where
\begin{align}\label{mt 23}
D(\eta)=&-\br{\sin\phi,r}_{\phi,\psi}(\eta)+\tf(\eta)\br{1-3\sin^2\phi,r}_{\phi,\psi}(\eta)\\
&+G(\eta)\br{1-3\sin^2\phi,r\cos^2\psi}_{\phi,\psi}(\eta)+\br{\sin\phi,S}_{\phi,\psi}(\eta).\no
\end{align}
Inserting the quasi-orthogonal relation \eqref{mt 21} into \eqref{mt 23}, we get
\begin{align}\label{mt 24}
D(\eta)=&2\int_{\eta}^L\ue^{2\tv(\eta)-\tv(y)}G(y)\br{\sin\phi\cos^2\psi,f}_{\phi,\psi}(y)\ud{y}+\tf(\eta)\br{1-3\sin^2\phi,r}_{\phi,\psi}(\eta)\\
&+G(\eta)\br{1-3\sin^2\phi,r\cos^2\psi}_{\phi,\psi}(\eta)+\br{\sin\phi,S}_{\phi,\psi}(\eta).\no
\end{align}
We can integrate over $[0,\eta]$ in \eqref{mt 06} to obtain
\begin{align}\label{mt 26}
\beta(\eta)-\beta(0)=\int_0^{\eta}D(z)\ud{z}.
\end{align}
Hence, plugging \eqref{mt 24} into \eqref{mt 26}, we deduce
\begin{align}\label{mt 25}
\\
\beta(\eta)-\beta(0)=&2\int_0^{\eta}\int_{z}^L\ue^{2\tv(z)-2\tv(y)}G(y)\br{\sin\phi\cos^2\psi,r}_{\phi,\psi}\ud{y}\ud{z}
+\int_0^{\eta}\tf(z)\br{1-3\sin^2\phi,r}_{\phi,\psi}(z)\ud{z}\no\\
&+\int_0^{\eta}G(z)\br{1-3\sin^2\phi,r\cos^2\psi}_{\phi,\psi}(z)\ud{z}+\int_0^{\eta}\br{\sin\phi,S}_{\phi,\psi}(z)\ud{z}.\no
\end{align}
For the boundary data $\beta(0)$ in \eqref{mt 25}, we use Cauchy's inequality to obtain
\begin{align}\label{mt 31}
\beta(0)=\br{\sin^2\phi,f}_{\phi,\psi}(0)\leq \bigg(\br{f,f\abs{\sin\phi}}_{\phi,\psi}(0)\bigg)^{\frac{1}{2}}\tnm{\sin\phi}^{\frac{3}{2}}\leq C
\bigg(\br{f,f\abs{\sin\phi}}_{\phi,\psi}(0)\bigg)^{\frac{1}{2}}.
\end{align}
We may decompose
\begin{align}\label{mt 30}
\br{f, f\abs{\sin\phi}
}_{\phi,\psi}(0)=\int_{\sin\phi>0}h^2(\phi)\sin\phi\cos\phi
\ud{\phi}-\int_{\sin\phi<0}\Big(f(0,\phi)\Big)^2\sin\phi\cos\phi\ud{\phi}.
\end{align}
Recall the definition of $\alpha(\eta)$ in \eqref{mt 27} and the estimate \eqref{mt 28} with $\bar S=0$, we have
\begin{align}
&\int_{\sin\phi>0} h^2(\phi)\sin\phi\cos\phi\ud{\phi}+\int_{\sin\phi<0}
\Big(f(0,\phi)\Big)^2\sin\phi\cos\phi\ud{\phi}=2\alpha(0)\\
\geq&
2\int_{0}^L\ue^{-2\tv(y)}\Big(G(y)\br{f\cos^2\psi,f\sin\phi}_{\phi,\psi}(y)-\br{S,r}_{\phi,\psi}(y)\Big)\ud{y}.\nonumber
\end{align}
This implies
\begin{align}\label{mt 29}
&-\int_{\sin\phi<0}
\Big(f(0,\phi)\Big)^2\sin\phi\cos\phi\ud{\phi}\\
\leq&\int_{\sin\phi>0}
h^2(\phi)\sin\phi\cos\phi\ud{\phi}-2\int_0^L\ue^{-2\tv(y)}\Big(G(y)\br{f\cos^2\psi,f\sin\phi}_{\phi,\psi}(y)-\br{S,r}_{\phi,\psi}(y)\Big)\ud{y}\no\\
\leq&
\tnmp{h}^2+C\abs{\int_0^L\Big(G(y)\br{f\cos^2\psi,f\sin\phi}_{\phi,\psi}(y)-\br{S,r}_{\phi,\psi}(y)\Big)\ud{y}}\nonumber.
\end{align}
Hence, inserting \eqref{mt 29} into \eqref{mt 30} and further \eqref{mt 31}, we obtain
\begin{align}\label{mt 32}
\\
\abs{\beta(0)}\leq& \abs{\br{f,f\abs{\sin\phi}}_{\phi,\psi}(0)}^{\frac{1}{2}}\leq
C\tnmp{h}+C\abs{\int_0^LG(y)\br{f\cos^2\psi,f\sin\phi}_{\phi,\psi}(y)\ud{y}}^{\frac{1}{2}}+C\abs{\int_0^L\br{S,r}_{\phi,\psi}(y)\ud{y}}^{\frac{1}{2}}.\no
\end{align}
From \eqref{mt 17}, we know
\begin{align}\label{mt 33}
\abs{\int_0^LG(y)\br{f\cos^2\psi,f\sin\phi}_{\phi,\psi}(y)\ud{y}}^{\frac{1}{2}}\leq C\lnnm{G}^{\frac{1}{2}}\tnnm{f}^2\leq C\e^{\frac{1}{2}}\tnnm{f}.
\end{align}
Using Cauchy's inequality, we have
\begin{align}\label{mt 34}
\abs{\int_0^L\br{S,r}_{\phi,\psi}(y)\ud{y}}^{\frac{1}{2}}\leq \tnnm{S}\tnnm{r}\leq C\tnnm{r}.
\end{align}
Inserting \eqref{mt 33} and \eqref{mt 34} into \eqref{mt 32}, we get
\begin{align}\label{mt 35}
\abs{\beta(0)}\leq&
C\bigg(\tnmp{h}+\e^{\frac{1}{2}}\tnnm{f}+\tnnm{r}\bigg)\leq C\bigg(1+\e^{\frac{1}{2}}\tnnm{f}+\tnnm{r}\bigg).
\end{align}
Now we turn to other terms in \eqref{mt 25}. Here we will repeated use the fact that $0<L<\e^{-n}$. Using Cauchy's inequality and \eqref{mt 17}, we obtain
\begin{align}\label{mt 36}
\abs{2\int_0^{\eta}\int_{z}^L\ue^{2\tv(z)-2\tv(y)}G(y)\br{\sin\phi\cos^2\psi,r}_{\phi,\psi}\ud{y}\ud{z}}\leq& 2\int_0^{\eta}\tnnm{G}\tnnm{r}\ud{z}\\
\leq& CL\tnnm{G}\tnnm{r}\leq C\e^{1-\frac{3n}{2}}\tnnm{r}.\no
\end{align}
Since $\tf(\eta)\in L^1[0,L]\cap L^2[0,L]$, using Cauchy's inequality, we know
\begin{align}\label{mt 37}
\abs{\int_0^{\eta}\tf(z)\br{1-3\sin^2\phi,r}_{\phi,\psi}(z)\ud{z}}\leq C\tnnm{\tf}\tnnm{r}\leq C\e^{1-\frac{n}{2}}\tnnm{r}.
\end{align}
Similarly, using Cauchy's inequality and \eqref{mt 17}, we may deduce
\begin{align}\label{mt 38}
\abs{\int_0^{\eta}G(z)\br{1-3\sin^2\phi,r\cos^2\psi}_{\phi,\psi}(z)\ud{z}}\leq \tnnm{G}\tnnm{r}\leq \e^{1-\frac{n}{2}}\tnnm{r}.
\end{align}
Considering the exponential decay of $S$ in \eqref{Milne decay}, we apply Cauchy's inequality to get
\begin{align}\label{mt 39}
\abs{\int_0^{\eta}\br{\sin\phi,S}_{\phi,\psi}(z)\ud{z}}\leq C\tnnm{\ue^{K\eta}S}\tnnm{\ue^{-K\eta}}\leq C.
\end{align}
Collecting all estimates in \eqref{mt 35}, \eqref{mt 36}, \eqref{mt 37}, \eqref{mt 38} and \eqref{mt 39}, and inserting them into \eqref{mt 25}, we have
\begin{align}
\abs{\beta(L)}\leq&C\Big(1+\e^{\frac{1}{2}}\tnnm{f}+\e^{1-\frac{3n}{2}}\tnnm{r}\Big).
\end{align}
Using \eqref{mt 40}, we further deduce that
\begin{align}\label{mt 41}
\abs{\beta(L)}\leq&C\Big(1+\e^{1-\frac{3n}{2}}\Big)\Big(1+\e^{\frac{1}{2}}\tnnm{f}\Big).
\end{align}
Define
\begin{align}\label{Milne limit}
f_L=q_L=\frac{\br{\sin^2\phi,f}_{\phi,\psi}(L)}{\tnm{\sin\phi}^2}=\frac{\beta(L)}{\tnm{\sin\phi}^2}.
\end{align}
We may rewrite \eqref{mt 41} as
\begin{align}
\abs{f_L}\tnm{\sin\phi}^2=\abs{\beta(L)}\leq&C\Big(1+\e^{1-\frac{3n}{2}}\Big)\Big(1+\e^{\frac{1}{2}}\tnnm{(f-f_L)+f_L}\Big)\\
\leq&C(1+\e^{1-\frac{3n}{2}})\left(1+\e^{\frac{1}{2}}\tnnm{f-f_L}+\e^{\frac{1}{2}}\tnnm{f_L}\right)\no\\
\leq&C(1+\e^{1-\frac{3n}{2}})\left(1+\e^{\frac{1}{2}}\tnnm{f-f_L}+\e^{\frac{1}{2}-\frac{n}{2}}\abs{f_L}\right)\no
\end{align}
Therefore, for $0<n<\dfrac{2}{3}$ and $\e$ sufficiently small, absorbing $\abs{f_L}$ into the left-hand side, we have
\begin{align}\label{mt 42}
\abs{f_L}\leq C\Big(1+\e^{\frac{1}{2}}\tnnm{f-f_L}\Big).
\end{align}
Thus, in \eqref{mt 40}, we may further estimate
\begin{align}
\tnnm{r}\label{mt 43}
\leq C\left(1+\e^{\frac{1}{2}}\tnnm{f-f_L}+\e^{\frac{1}{2}}\tnnm{f_L}\right)
\leq C\left(1+\e^{\frac{1}{2}}\tnnm{f-f_L}\right).
\end{align}
\ \\
Step 4: Estimates of $q-q_L$.\\
We can integrate over $[\eta,L]$ in \eqref{mt 06} to obtain
\begin{align}\label{mt 44}
\beta(L)-\beta(\eta)=&\int_{\eta}^LD(z)\ud{z}\\
=&\int_{\eta}^L\tf(z)\br{1-3\sin^2\phi,r}_{\phi,\psi}(z)\ud{z}
+\int_{\eta}^{L}\int_{z}^L\ue^{2\tv(z)-2\tv{y}}G(y)\br{\sin\phi\cos^2\psi,r}_{\phi,\psi}\ud{y}\ud{z}\no\\
&+\int_{\eta}^{L}G(z)\br{1-3\sin^2\phi,r\cos^2\psi}_{\phi,\psi}(z)\ud{z}+\int_{\eta}^{L}\br{\sin\phi,S}_{\phi,\psi}(z)\ud{z}.\no
\end{align}
Note that
\begin{align}
\beta(\eta)=&\br{\sin^2\phi,f}_{\phi,\psi}(\eta)=\br{\sin^2\phi,q}_{\phi,\psi}(\eta)+\br{\sin^2\phi,r}_{\phi,\psi}(\eta)\\
=&q(\eta)\tnm{\sin\phi}^2+\br{\sin^2\phi,r}_{\phi,\psi}(\eta),\no\\
\beta(L)=&q_L\tnm{\sin\phi}^2.
\end{align}
Hence, we have
\begin{align}\label{mt 54}
q(\eta)-q_L=\frac{\beta(\eta)-\beta(L)}{\tnm{\sin\phi}^2}-\frac{\br{\sin^2\phi,r}_{\phi,\psi}(\eta)}{\tnm{\sin\phi}^2}.
\end{align}
Then inserting \eqref{mt 44} into \eqref{mt 54}, we estimate
\begin{align}\label{mt 45}
\tnnm{q-q_L}^2\leq&C\Big(\tnnm{r}^2+\tnnm{\beta(\eta)-\beta(L)}^2\Big)\\
\leq&C\Bigg(\tnnm{r}^2+\int_0^L\abs{\int_{\eta}^{L}\int_{z}^L\ue^{2\tv(z)-2\tv{y}}G(y)\br{\sin\phi\cos^2\psi,r}_{\phi,\psi}\ud{y}\ud{z}}^2\ud{\eta}\no\\
&+\int_0^L\abs{\int_{\eta}^L\tf(z)\br{1-3\sin^2\phi,r}_{\phi,\psi}(z)\ud{z}}^2\ud{\eta}\no\\
&+\int_0^L\abs{\int_{\eta}^{L}G(z)\br{1-3\sin^2\phi,r\cos^2\psi}_{\phi,\psi}(z)\ud{z}}^2\ud{\eta}
+\int_0^L\abs{\int_{\eta}^{L}\br{\sin\phi,S}_{\phi,\psi}(z)\ud{z}}^2\ud{\eta}\Bigg).\no
\end{align}
We need to estimate each term on the right-hand side. Using Cauchy's inequality and \eqref{mt 17}, we obtain
\begin{align}\label{mt 46}
\int_0^L\abs{\int_{\eta}^{L}\int_{z}^L\ue^{2\tv(z)-2\tv{y}}G(y)\br{\sin\phi\cos^2\psi,r}_{\phi,\psi}\ud{y}\ud{z}}^2\ud{\eta}\leq& \int_0^L\abs{\int_{\eta}^{L}\tnnm{G}\tnnm{r}\ud{z}}^2\ud{\eta}\\
\leq& CL^3\tnnm{G}^2\tnnm{r}^2\leq C\e^{2-4n}\tnnm{r}^2.\no
\end{align}
Since $\tf(\eta)\in L^1[0,L]\cap L^2[0,L]$, using Cauchy's inequality, we know
\begin{align}\label{mt 47}
\int_0^L\abs{\int_{\eta}^L\tf(z)\br{1-3\sin^2\phi,r}_{\phi,\psi}(z)\ud{z}}^2\ud{\eta}\leq C\tnnm{r}^2\int_0^L\int_{\eta}^L\abs{\tf(z)}^2\ud{z}\ud{\eta}\leq C\e^{2-2n}\tnnm{r}^2.
\end{align}
Similarly, using Cauchy's inequality and \eqref{mt 17}, we may deduce
\begin{align}\label{mt 48}
\int_0^L\abs{\int_{\eta}^{L}G(z)\br{1-3\sin^2\phi,r\cos^2\psi}_{\phi,\psi}(z)\ud{z}}^2\ud{\eta}\leq CL\tnnm{G}^2\tnnm{r}^2\leq \e^{2-2n}\tnnm{r}.
\end{align}
Considering the exponential decay of $S$ in \eqref{Milne decay}, we apply Cauchy's inequality to get
\begin{align}\label{mt 49}
\int_0^L\abs{\int_{\eta}^{L}\br{\sin\phi,S}_{\phi,\psi}(z)\ud{z}}^2\ud{\eta}\leq \int_0^L\abs{\int_{\eta}^{L}\lnm{S(z)}\ud{z}}^2\ud{\eta}\leq \int_0^L\ue^{-2K\eta}\ud\eta\leq C.
\end{align}
Collecting all estimates in \eqref{mt 46}, \eqref{mt 47}, \eqref{mt 48} and \eqref{mt 49}, inserting them into \eqref{mt 45} and using \eqref{mt 43}, we obtain
\begin{align}\label{mt 50}
\tnnm{q-q_L}^2\leq&C\Big(1+\e^{2-4n}\Big)\Big(1+\tnnm{r}^2\Big)
\leq C\Big(1+\e^{2-4n}\Big)\left(1+\e\tnnm{f-f_L}^2\right).
\end{align}
Therefore, for $0<n<\dfrac{1}{2}$, we have
\begin{align}\label{mt 55}
\tnnm{q-q_L}\leq C\left(1+\e^{\frac{1}{2}}\tnnm{f-f_L}\right).
\end{align}
\ \\
Step 5: Synthesis.\\
For $0<n<\dfrac{1}{2}$, collecting \eqref{mt 43} and \eqref{mt 55}, we have
\begin{align}
\tnnm{r}\leq& C\left(1+\e^{\frac{1}{2}}\tnnm{f-f_L}\right),\\
\tnnm{q-q_L}\leq& C\left(1+\e^{\frac{1}{2}}\tnnm{f-f_L}\right).
\end{align}
On the other hand, based on the definition in \eqref{Milne limit} and above estimates,
\begin{align}\label{mt 53}
\\
\tnnm{f-f_L}=\tnnm{r+(q-q_L)}\leq \tnnm{r}+\tnnm{q-q_L}\leq C\left(1+\e^{\frac{1}{2}}\tnnm{f-f_L}\right).\no
\end{align}
Hence, for $\e$ sufficiently small, absorbing $\e^{\frac{1}{2}}\tnnm{f-f_L}$ to the left-hand side, we know
\begin{align}\label{mt 51}
\tnnm{f-f_L}\leq C.
\end{align}
Furthermore, using \eqref{mt 42}, we have
\begin{align}\label{mt 52}
\abs{f_L}\leq C.
\end{align}
\ \\
Step 6: Uniqueness.\\
Assume that there are
two solutions $f_1$ and $f_2$ to the equation \eqref{Milne problem} satisfying estimates \eqref{mt 51} and \eqref{mt 52}. Then $f'=f_1-f_2$ satisfies the equation
\begin{align}
\left\{
\begin{array}{l}\displaystyle
\sin\phi\frac{\p f'}{\p\eta}+F(\eta,\psi)\cos\phi\frac{\p
f'}{\p\phi}+f'-\bar f'=0,\\\rule{0ex}{2.0em}
f'(0,\phi,\psi)=0\ \ \text{for}\ \ \sin\phi>0,\\\rule{0ex}{2.0em}
f'(L,\phi,\psi)=f'(L,\rr[\phi],\psi).
\end{array}
\right.
\end{align}
Assume
\begin{align}
\abs{f'_L}\leq C,\quad \tnnm{f'-f'_L}\leq C.
\end{align}
Then we can repeat the proof in Step 1 to Step 5 and obtain the estimates similar to \eqref{mt 53},
\begin{align}
\tnnm{f'-f'_L}\leq C+C\e^{\frac{1}{2}}\tnnm{f'-f'_L}.
\end{align}
Note that in this proof, the $O(1)$ term $C$ purely comes from the boundary data and source term. Since all data are zero in $f'$ equation, we have
\begin{align}
\tnnm{f'-f'_L}\leq C\e^{\frac{1}{2}}\tnnm{f'-f'_L},
\end{align}
which implies $f=f_L$ is a constant. Then based on the zero boundary data, we must have $f=0$.

\end{proof}

\subsubsection{$\bar S\neq0$ Case}

\begin{lemma}\label{Milne finite LT.}
Assume $S$ and $h$ satisfy \eqref{Milne bounded} and \eqref{Milne decay}. Then there exists a solution $f(\eta,\phi,\psi)$ to the equation
\eqref{Milne problem}, satisfying that for $f_L=\dfrac{\br{\sin^2\phi,f}_{\phi,\psi}(L)}{\tnm{\sin\phi}^2}$,
\begin{align}\label{LT estimate'}
\abs{f_L}\leq C,\quad \tnnm{f-f_L}\leq C.
\end{align}
The solution is unique among functions such that \eqref{LT estimate'} holds.
\end{lemma}
\begin{proof}
We can utilize the superposition property for this linear problem, i.e. to
write $S=\bar S+(S-\bar S)=S_Q+S_R$. \\
\ \\
Step 1: Construction of auxiliary function $f^1$.\\
We first solve $f^1$ as the solution to
\begin{align}
\left\{
\begin{array}{l}\displaystyle
\sin\phi\frac{\p f^1}{\p\eta}+F(\eta)\cos\phi\frac{\p
f^1}{\p\phi}+f^1-\bar f^1=S_R(\eta,\phi,\psi),\\\rule{0ex}{2.0em}
f^1(0,\phi,\psi)=h(\phi,\psi)\ \ \text{for}\ \ \sin\phi>0,\\\rule{0ex}{2.0em}
f^1(L,\phi,\psi)=f^1(L,\rr[\phi],\psi).
\end{array}
\right.
\end{align}
Since $\bar S_R=0$, by Lemma \ref{Milne finite LT}, there
exists a unique solution $f^1$
satisfying the $L^2$ estimates \eqref{LT estimate'}.\\
\ \\
Step 2: Construction of auxiliary function $f^2$.\\
This is the most tricky step. We seek a function $f^{2}$ satisfying
\begin{align}\label{mt 07}
-\frac{1}{4\pi}\int_{-\pi}^{\pi}\int_{-\frac{\pi}{2}}^{\frac{\pi}{2}}\bigg(\sin\phi\frac{\p
f^{2}}{\p\eta}+F(\eta)\cos\phi\frac{\p
f^{2}}{\p\phi}\bigg)\cos\phi\ud{\phi}\ud{\psi}+S_Q=0.
\end{align}
The following analysis shows that this type of function can always be
found. Integration by parts in $\phi$ transforms the equation \eqref{mt 07} into
\begin{align}\label{mt 08}
-\int_{-\pi}^{\pi}\int_{-\frac{\pi}{2}}^{\frac{\pi}{2}}\frac{\p
f^{2}}{\p\eta}\sin\phi\cos\phi\ud{\phi}\ud{\psi}-2\int_{-\pi}^{\pi}\int_{-\frac{\pi}{2}}^{\frac{\pi}{2}}F(\eta)f^{2}\sin\phi
\cos\phi\ud{\phi}\ud{\psi}+4\pi S_Q=0.
\end{align}
Setting
\begin{align}
f^{2}(\phi,\eta)=a(\eta)\sin\phi,
\end{align}
and plugging this ansatz into \eqref{mt 08}, we have
\begin{align}
-\frac{\ud{a}}{\ud{\eta}}\int_{-\pi}^{\pi}\int_{-\frac{\pi}{2}}^{\frac{\pi}{2}}\sin^2\phi \cos\phi\ud{\phi}\ud{\psi}-2a\int_{-\pi}^{\pi}\int_{-\frac{\pi}{2}}^{\frac{\pi}{2}}F(\eta)\sin^2\phi\cos\phi\ud{\phi}\ud{\psi}+4\pi
S_Q=0.
\end{align}
Hence, we have
\begin{align}
\frac{\ud{a}}{\ud{\eta}}+\bar F(\eta)a(\eta)=6\pi S_Q,
\end{align}
where
\begin{align}
\bar F(\eta)=-2\pi\bigg(\dfrac{\e}{R_1-\e\eta}+\dfrac{\e}{R_2-\e\eta}\bigg).
\end{align}
This is a first order linear ordinary differential equation, which
possesses infinite solutions. The general solution is
\begin{align}
a(\eta)=\ue^{-\int_0^{\eta}\bar F(y)\ud{y}}\bigg(a(0)+6\pi\int_0^{\eta}\ue^{\int_0^y\bar F(z)\ud{z}}S_Q(y)\ud{y}\bigg).
\end{align}
We may take
\begin{align}
a(0)=-6\pi\int_0^{L}\ue^{\int_0^y\bar F(z)\ud{z}}S_Q(y)\ud{y}.
\end{align}
Based on the exponential decay of $S_Q$, we can directly verify that
$a(\eta)$ decays exponentially to zero as $\eta\rt L$ and $f^2$
satisfies the $L^2$ estimates \eqref{LT estimate'}.\\
\ \\
Step 3: Construction of auxiliary function $f^3$.\\
Based on above construction, we can directly verify that
\begin{align}\label{mt 09}
\int_{-\pi}^{\pi}\int_{-\frac{\pi}{2}}^{\frac{\pi}{2}}\bigg(-\sin\phi\frac{\p
f^{2}}{\p\eta}-F(\eta)\cos\phi\frac{\p f^{2}}{\p\phi}-f^{2}+\bar
f^{2}+S_Q\bigg)\cos\phi\ud{\phi}\ud{\psi}=0.
\end{align}
Then we can solve $f^3$ as the solution to
\begin{align}
\left\{
\begin{array}{l}\displaystyle
\sin\phi\frac{\p f^3}{\p\eta}+F(\eta)\cos\phi\frac{\p
f^3}{\p\phi}+f^3-\bar f^3=-\sin\phi\dfrac{\p
f^{2}}{\p\eta}-F(\eta)\cos\phi\dfrac{\p
f^{2}}{\p\phi}-f^{2}+\bar f^{2}+S_Q,\\\rule{0ex}{2.0em}
f^3(0,\phi,\psi)=-a(0)\sin\phi\ \ \text{for}\ \ \sin\phi>0,\\\rule{0ex}{2.0em}
f^3(L,\phi,\psi)=f^3(L,\rr[\phi],\psi).
\end{array}
\right.
\end{align}
By \eqref{mt 09}, we can apply Lemma \ref{Milne finite LT}
to obtain a unique solution $f^3$
satisfying the $L^2$ estimates \eqref{LT estimate'}.\\
\ \\
Step 4: Construction of auxiliary function $f^4$.\\
We now define $f^4=f^2+f^3$ and the superposition property implies that
\begin{align}
\left\{
\begin{array}{l}\displaystyle
\sin\phi\frac{\p f^4}{\p\eta}+F(\eta)\cos\phi\frac{\p
f^4}{\p\phi}+f^4-\bar f^4=S_Q(\eta,\phi,\psi),\\\rule{0ex}{2.0em}
f^4(0,\phi,\psi)=0\ \ \text{for}\ \ \sin\phi>0,\\\rule{0ex}{2.0em}
f^4(L,\phi,\psi)=f^4(L,\rr[\phi],\psi),
\end{array}
\right.
\end{align}
and $f^4$
satisfies the $L^2$ estimates \eqref{LT estimate'}.\\
\ \\
In summary, we deduce that $f^1+f^4$ is the solution to \eqref{Milne problem} and satisfies the $L^2$ estimates \eqref{LT estimate'}.
\end{proof}

In the above, we actually prove the $L^2$ estimates.
\begin{theorem}\label{Milne LT Theorem}
Assume $S$ and $h$ satisfy \eqref{Milne bounded} and \eqref{Milne decay}. Then there exists a solution $f(\eta,\phi,\psi)$ to the equation
\eqref{Milne problem}, satisfying that
\begin{align}
\abs{f_L}\leq C,\quad \tnnm{f-f_L}\leq C,
\end{align}
where
\begin{align}
f_L=\dfrac{\br{\sin^2\phi,f}_{\phi,\psi}(L)}{\tnm{\sin\phi}^2}.
\end{align}
The solution is unique among functions such that above estimates hold.
\end{theorem}

\subsection{$L^{\infty}$ Estimates}

\subsubsection{Formulation}

Consider the following $\e$-transport problem for $f(\eta,\phi,\psi)$ in $(\eta,\phi,\psi)\in[0,L]\times\left[-\dfrac{\pi}{2},\dfrac{\pi}{2}\right]\times[-\pi,\pi]$
\begin{align}\label{Milne transport}
\left\{
\begin{array}{l}\displaystyle
\sin\phi\frac{\p f}{\p\eta}+F(\eta,\psi)\cos\phi\frac{\p
f}{\p\phi}+f=H(\eta,\phi,\psi),\\\rule{0ex}{2.0em}
f(0,\phi,\psi)=h(\phi,\psi)\ \ \text{for}\ \ \sin\phi>0,\\\rule{0ex}{2.0em}
f(L,\phi,\psi)=f(L,\rr[\phi],\psi).
\end{array}
\right.
\end{align}
Comparing this with \eqref{Milne problem}, we actually have $H=\bar f+S$. Define a potential function $V(\eta,\psi)$ satisfying $V(0,\psi)=0$ and $\dfrac{\p V}{\p\eta}=-F(\eta,\psi)$. We may directly obtain
\begin{align}
V(\eta,\psi)=\ln\left(\frac{R_1}{R_1-\e\eta}\right)\sin^2\psi+\ln\left(\frac{R_2}{R_2-\e\eta}\right)\cos^2\psi.
\end{align}
Define the energy functional:
\begin{align}
E(\eta,\phi,\psi)=\ue^{-V(\eta,\psi)}\cos\phi.
\end{align}
\eqref{Milne transport} is a transport equation, so we may track the solution to the boundary data $h$ through the characteristic lines $\Big(\eta(s),\phi(s)\Big)$ for $s\in\r$, which satisfies
\begin{align}
\frac{\ud\eta}{\ud s}=\sin\phi,\quad \frac{\ud\phi}{\ud s}=F(\eta,\psi)\cos\phi.
\end{align}
We may check that along the characteristics, the energy $E$ is conserved and the equation can be simplified as follows:
\begin{align}
\frac{\ud{f}}{\ud{s}}+f=H,
\end{align}
or equivalently with $\phi=\phi(\eta)$,
\begin{align}
\sin\phi\frac{\ud{f}}{\ud{\eta}}+f=H.
\end{align}
Also, $\psi$ is a constant along the characteristics, so we may temporarily ignore $\psi$ dependence when there is no confusion. Note that all the estimates are uniform in $\psi$. \\
\ \\
If the characteristic touches $\phi=0$ line, let $\eta^+(\eta,\phi)$ satisfy
\begin{align}
E(\eta,\phi)=\ue^{-V(\eta^+)},
\end{align}
which means $(\eta^+,0)$ is on the same characteristics as $(\eta,\phi)$.\\
\ \\
Define the quantities for $0\leq\eta'\leq\eta^+$ as follows:
\begin{align}
\phi'(\phi,\eta;\eta')=&\cos^{-1}\Big(\ue^{V(\eta')-V(\eta)}\cos\phi\Big),\\
\rr[\phi'](\phi,\eta;\eta')=&-\cos^{-1}\Big(\ue^{V(\eta')-V(\eta)}\cos\phi\Big)=-\phi'(\phi,\eta;\eta'),
\end{align}
where the inverse trigonometric function can be defined
single-valued in the range $\left[0,\dfrac{\pi}{2}\right]$ and the quantities are always well-defined due to the monotonicity of $V$. In particular, no matter $\phi$ is positive or negative, $\phi'$ is always positive.\\
\ \\
Finally, we put
\begin{align}\label{mt 74}
G_{\eta,\eta'}(\phi)=&\int_{\eta'}^{\eta}\frac{1}{\sin\Big(\phi'(\phi,\eta;\xi)\Big)}\ud{\xi}.
\end{align}
Depending on whether the characteristics touch $\phi=0$ or $\eta=L$, we can rewrite the solution to the equation \eqref{Milne transport} as
\begin{align}
f(\eta,\phi)=\k[h](\eta,\phi)+\t[H](\eta,\phi),
\end{align}
where\\
\ \\
Region I: the characteristic touches neither $\phi=0$ nor $\eta=L$, i.e. $\sin\phi>0$.
\begin{align}\label{mt 71}
\k[h](\eta,\phi)=&h\Big(\phi'(\phi,\eta;0)\Big)\exp\Big(-G_{\eta,0}\Big),\\
\t[H](\eta,\phi)=&\int_0^{\eta}\frac{H\Big(\eta',\phi'(\phi,\eta;\eta')\Big)}{\sin\Big(\phi'(\phi,\eta;\eta')\Big)}\exp\Big(-G_{\eta,\eta'}\Big)\ud{\eta'}.\label{mt 71'}
\end{align}
\ \\
Region II: the characteristic touches, i.e. $\sin\phi<0$ and $E(\eta,\phi)\leq \ue^{-V(L)}$.
\begin{align}\label{mt 72}
\k[h](\eta,\phi)=&h\Big(\phi'(\phi,\eta;0)\Big)\exp\Big(-G_{L,0}-G_{L,\eta}\Big),\\
\t[H](\eta,\phi)=&\int_0^{L}\frac{H\Big(\eta',\phi'(\phi,\eta;\eta')\Big)}{\sin\Big(\phi'(\phi,\eta;\eta')\Big)}
\exp\Big(-G_{L,\eta'}-G_{L,\eta}\Big)\ud{\eta'}\label{mt 72'}\\
&+\int_{\eta}^{L}\frac{H\Big(\eta',\rr[\phi'](\phi,\eta;\eta')\Big)}{\sin\Big(\phi'(\phi,\eta;\eta')\Big)}\exp\Big(G_{\eta,\eta'}\Big)\ud{\eta'}.\no
\end{align}
\ \\
Region III: the characteristic touches, i.e. $\sin\phi<0$ and $E(\eta,\phi)\geq \ue^{-V(L)}$.
\begin{align}\label{mt 73}
\k[h](\eta,\phi)=&h\Big(\phi'(\phi,\eta;0)\Big)\exp\Big(-G_{\eta^+,0}-G_{\eta^+,\eta}\Big),\\
\t[H](\eta,\phi)=&\int_0^{\eta^+}\frac{H\Big(\eta',\phi'(\phi,\eta;\eta')\Big)}{\sin\Big(\phi'(\phi,\eta;\eta')\Big)}
\exp\Big(-G_{\eta^+,\eta'}-G_{\eta^+,\eta}\Big)\ud{\eta'}\label{mt 73'}\\
&+\int_{\eta}^{\eta^+}\frac{H\Big(\eta',\rr[\phi'](\phi,\eta;\eta')\Big)}{\sin\Big(\phi'(\phi,\eta;\eta')\Big)}\exp\Big(G_{\eta,\eta'}\Big)\ud{\eta'}.\no
\end{align}
Figure 1 depicts the characteristics of the equation and Figure 2 illustrates the domain decomposition.
\begin{figure}[H]
\begin{minipage}[t]{0.5\linewidth}
\centering
\includegraphics[width=3.2in]{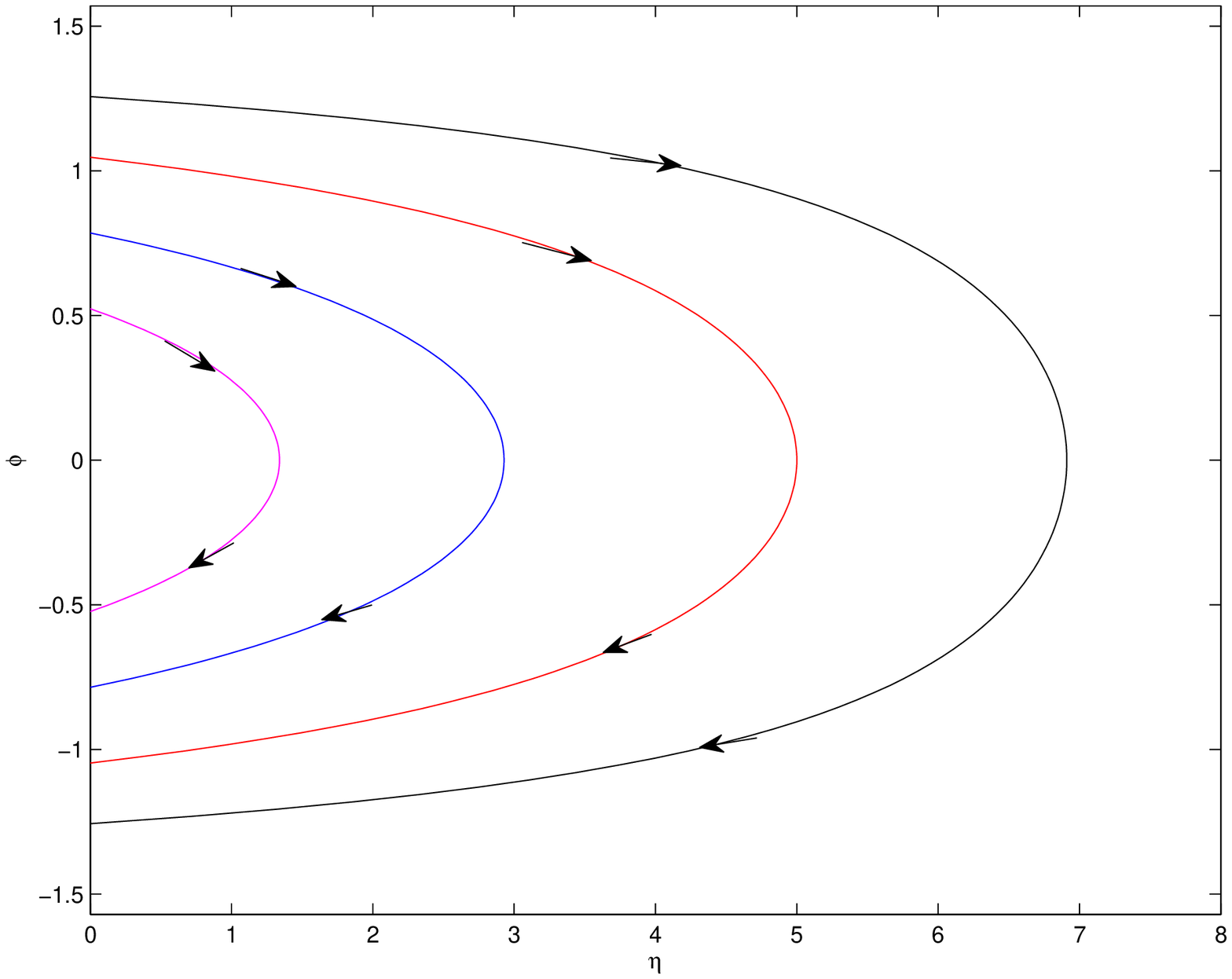}
\caption{Characteristics}
\end{minipage}%
\begin{minipage}[t]{0.5\linewidth}
\centering
\includegraphics[width=3.2in]{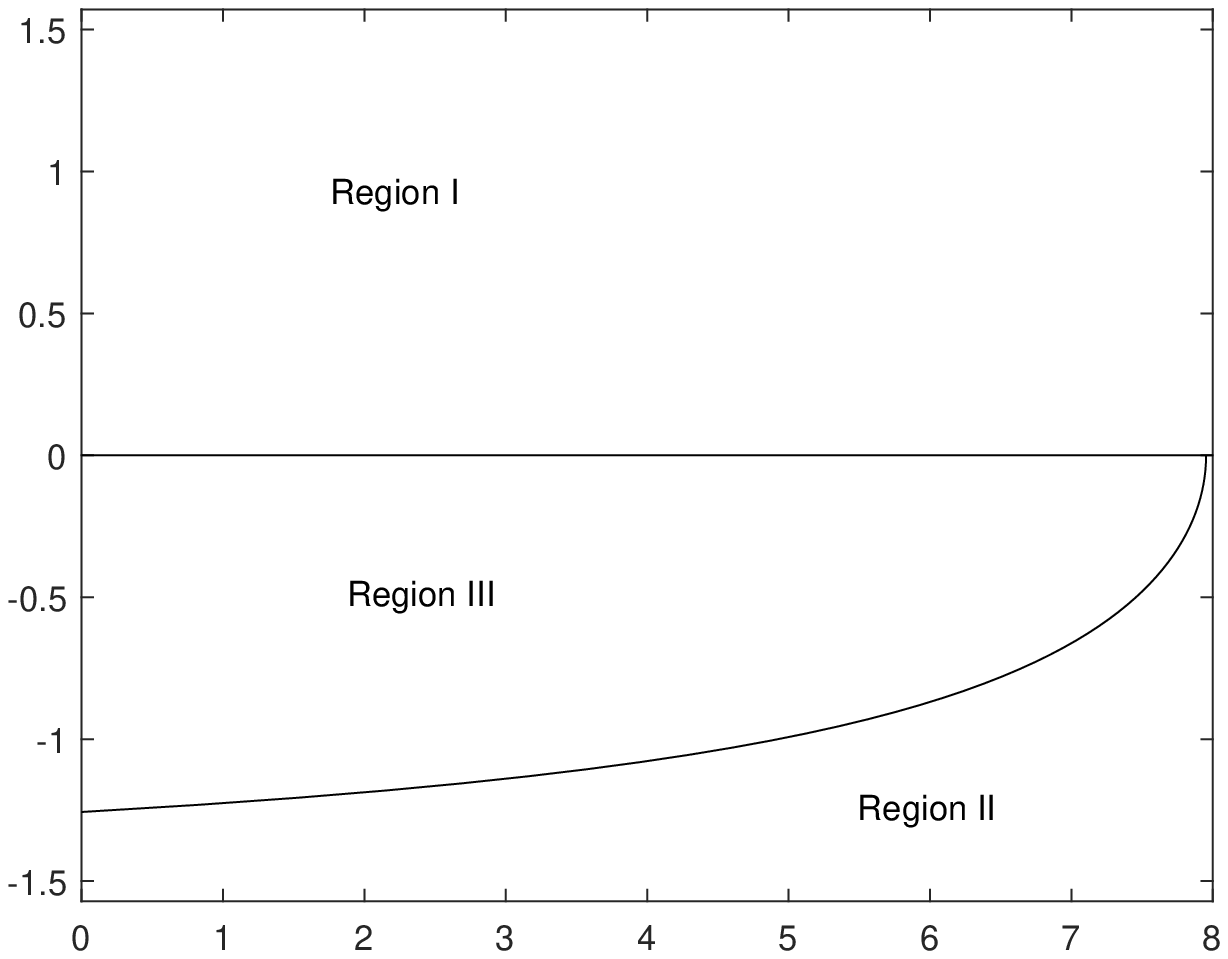}
\caption{Domain Decomposition}
\end{minipage}
\end{figure}

\subsubsection{Preliminaries}

\begin{lemma}\label{Milne lemma 1}
For any $0\leq\beta\leq1$, we have
\begin{align}
\lnnm{\ue^{\beta\eta}\k[h]}\leq \lnmp{h}.
\end{align}
In particular,
\begin{align}
\lnnm{\k[h]}\leq \lnmp{h}.
\end{align}
\end{lemma}
\begin{proof}
Since $\phi'$ is always in the domain $\left[0,\dfrac{\pi}{2}\right]$, we have
\begin{align}
0\leq\sin\Big(\phi'(\phi,\eta;\xi)\Big)\leq 1,
\end{align}
which further implies
\begin{align}
\frac{1}{\sin\Big(\phi'(\phi,\eta;\xi)\Big)}\geq 1.
\end{align}
Recall the definition of $G$ in \eqref{mt 74},
\begin{align}
G_{a,0}(\phi)=&\int_{0}^{a}\frac{1}{\sin\Big(\phi'(\phi,\eta;\xi)\Big)}\ud{\xi}\geq a.
\end{align}
Considering the fact that $L\geq\eta^+\geq\eta$, we deduce
\begin{align}
\exp\Big(-G_{\eta,0}\Big)\leq&\ue^{-\eta},\\
\exp\Big(-G_{L,0}-G_{L,\eta}\Big)\leq&\exp\Big(-G_{L,0}\Big)\leq \exp\Big(-G_{\eta,0}\Big)\leq\ue^{-\eta},\\
\exp\Big(-G_{\eta^+,0}-G_{\eta^+,\eta}\Big)\leq&\exp\Big(-G_{\eta^+,0}\Big)\leq\exp\Big(-G_{\eta,0}\Big)\leq
\ue^{-\eta}.
\end{align}
Hence, based on \eqref{mt 71}, \eqref{mt 72} and \eqref{mt 73}, our result follows.
\end{proof}

\begin{lemma}\label{Milne lemma 2}
The integral operator $\t$ satisfies
\begin{align}\label{mt 75}
\lnnm{\t[H]}\leq C\lnnm{H},
\end{align}
and for any $0\leq\beta\leq\dfrac{1}{2}$,
\begin{align}\label{mt 76}
\lnnm{\ue^{\beta\eta}\t[H]}\leq C\lnnm{\ue^{\beta\eta}H}.
\end{align}
\end{lemma}
\begin{proof}
We first prove \eqref{mt 75}. Based on a similar argument as in the proof of Lemma \ref{Milne lemma 1}, we know
\begin{align}
\frac{1}{\sin\Big(\phi'(\phi,\eta;\xi)\Big)}\geq 1.
\end{align}
When $\sin\phi>0$, using \eqref{mt 71'},
\begin{align}
\abs{\t[H]}\leq&\int_0^{\eta}\abs{H\Big(\eta',\phi'(\phi,\eta;\eta')\Big)}\frac{1}{\sin\Big(\phi'(\phi,\eta;\eta')\Big)}\exp\Big(-G_{\eta,\eta'}\Big)\ud{\eta'}\\
\leq&\lnnm{H}\int_0^{\eta}\frac{1}{\sin\Big(\phi'(\phi,\eta;\eta')\Big)}\exp\Big(-G_{\eta,\eta'}\Big)\ud{\eta'}\no.
\end{align}
Using substitution $z=G_{\eta,\eta'}$, we can directly estimate
\begin{align}
\int_0^{\eta}\frac{1}{\sin\Big(\phi'(\phi,\eta;\eta')\Big)}\exp\Big(-G_{\eta,\eta'}\Big)\ud{\eta'}\leq\int_0^{\infty}\ue^{-z}\ud{z}=1,
\end{align}
and thus \eqref{mt 75} follows. When $\sin\phi<0$ and
$E(\eta,\phi)\leq \ue^{-V(L)}$, using \eqref{mt 72'},
\begin{align}
\abs{\t[H]}\leq&\int_0^{L}\abs{H\Big(\eta',\phi'(\phi,\eta;\eta')\Big)}\frac{1}{\sin\Big(\phi'(\phi,\eta;\eta')\Big)}
\exp\Big(-G_{L,\eta'}\Big)\ud{\eta'}\\
&+\int_{\eta}^{\infty}\abs{H\Big(\eta',\phi'(\phi,\eta;\eta')\Big)}\frac{1}{\sin\Big(\phi'(\phi,\eta;\eta')\Big)}\exp\Big(G_{\eta,\eta'}\Big)\ud{\eta'}\no\\
\leq&\lnnm{H}\Bigg(\int_0^{L}\frac{1}{\sin\Big(\phi'(\phi,\eta;\eta')\Big)}
\exp\Big(-G_{L,\eta'}\Big)\ud{\eta'}+\int_{\eta}^{\infty}\frac{1}{\sin\Big(\phi'(\phi,\eta;\eta')\Big)}\exp\Big(G_{\eta,\eta'}\Big)\ud{\eta'}\Bigg)\nonumber.
\end{align}
The integral $\eta'\in[0,\eta]$ can be estimated as in $\sin\phi>0$ case, so we focus on $\eta'\in[\eta,L]$ integral.
Letting $z=G_{\eta,\eta'}$, we have
\begin{align}
\int_{\eta}^{\infty}\frac{1}{\sin\Big(\phi'(\phi,\eta;\eta')\Big)}\exp\Big(G_{\eta,\eta'}\Big)\ud{\eta'}\leq\int_{-\infty}^0\ue^{z}\ud{z}=1.
\end{align}
Hence, \eqref{mt 75} follows. When $\sin\phi<0$ and
$E(\eta,\phi)\geq \ue^{-V(L)}$, using \eqref{mt 73'}, we can use a similar argument to justify \eqref{mt 75}.\\
\ \\
Then we turn to \eqref{mt 76}. When $\sin\phi>0$, $\eta\geq\eta'$ and
$\beta<\dfrac{1}{2}$, since $G_{\eta,\eta'}\geq\eta-\eta'$, we have
\begin{align}
\beta(\eta-\eta')-G_{\eta,\eta'}\leq
\beta(\eta-\eta')-\half(\eta-\eta')-\half G_{\eta,\eta'}\leq -\half
G_{\eta,\eta'}.
\end{align}
Then using substitution $z=G_{\eta,\eta'}$, we know
\begin{align}
\int_0^{\eta}\frac{1}{\sin\Big(\phi'(\eta,\phi;\eta')\Big)}\exp\Big(\beta(\eta-\eta')-G_{\eta,\eta'}\Big)\ud{\eta'}
\leq&\int_0^{\eta}\frac{1}{\sin\Big(\phi'(\eta,\phi;\eta')\Big)}\exp\left(-\frac{G_{\eta,\eta'}}{2}\right)\ud{\eta'}\\
\leq&\int_0^{\infty}\ue^{-\frac{z}{2}}\ud{z}=2\no.
\end{align}
This leads to
\begin{align}
\abs{\ue^{\beta\eta}\t[H]}\leq&
\ue^{\beta\eta}\int_0^{\eta}\abs{H\Big(\eta',\phi'(\eta,\phi;\eta')\Big)}\frac{1}{\sin\Big(\phi'(\eta,\phi;\eta')\Big)}\exp\Big(-G_{\eta,\eta'}\Big)\ud{\eta'}\\
\leq&
\int_0^{\eta}\ue^{\beta(\eta-\eta')}\abs{\ue^{\beta\eta'}H\Big(\eta',\phi'(\eta,\phi;\eta')\Big)}
\frac{1}{\sin\Big(\phi'(\eta,\phi;\eta')\Big)}\exp\Big(-G_{\eta,\eta'}\Big)\ud{\eta'}\no\\
\leq&\lnnm{\ue^{\beta\eta}H}\int_0^{\eta}\frac{1}{\sin\Big(\phi'(\eta,\phi;\eta')\Big)}\exp\Big(\beta(\eta-\eta')-G_{\eta,\eta'}\Big)\ud{\eta'}\no\\
\leq&C\lnnm{\ue^{\beta\eta}H},\no
\end{align}
and \eqref{mt 76} follows. When $\sin\phi<0$ and
$E(\eta,\phi)\leq \ue^{-V(L)}$, note that $-G_{L,\eta'}-G_{L,\eta}\leq -G_{\eta,\eta'}$ and for $\eta'\geq\eta$
\begin{align}
\beta(\eta-\eta')+G_{\eta,\eta'}\leq
\beta(\eta-\eta')+\half(\eta-\eta')+\half G_{\eta,\eta'}\leq \half
G_{\eta,\eta'}.
\end{align}
Still the key is the integral for $\eta'\in[\eta,L]$.
Using substitution $z=G_{\eta,\eta'}$, we know
\begin{align}
\int_{\eta}^{\infty}\frac{1}{\sin\Big(\phi'(\eta,\phi;\eta')\Big)}\exp\Big(\beta(\eta-\eta')+G_{\eta,\eta'}\Big)\ud{\eta'}
\leq&\int_{\eta}^{\infty}\frac{1}{\sin\Big(\phi'(\eta,\phi;\eta')\Big)}\exp\left(\frac{G_{\eta,\eta'}}{2}\right)\ud{\eta'}\\
\leq&\int_{-\infty}^0\ue^{\frac{z}{2}}\ud{z}=2\no.
\end{align}
This yields
\begin{align}
&\ue^{\beta\eta}\int_{\eta}^{\infty}\abs{H\Big(\eta',\phi'(\eta,\phi;\eta')\Big)}\frac{1}{\sin\Big(\phi'(\eta,\phi;\eta')\Big)}\exp\Big(G_{\eta,\eta'}\Big)\ud{\eta'}\\
\leq&
\int_{\eta}^{\infty}\ue^{\beta(\eta-\eta')}\abs{\ue^{\beta\eta'}H\Big(\eta',\phi'(\eta,\phi;\eta')\Big)}
\frac{1}{\sin\Big(\phi'(\eta,\phi;\eta')\Big)}\exp\Big(G_{\eta,\eta'}\Big)\ud{\eta'}\no\\
\leq&\lnnm{\ue^{\beta\eta}H}\int_{\eta}^{\infty}\frac{1}{\sin\Big(\phi'(\eta,\phi;\eta')\Big)}\exp\Big(\beta(\eta-\eta')+G_{\eta,\eta'}\Big)\ud{\eta'}\no\\
\leq&C\lnnm{\ue^{\beta\eta}H}.\no
\end{align}
Then \eqref{mt 76} holds. The case $\sin\phi<0$ and $E(\eta,\phi)\geq
\ue^{-V(L)}$ can be shown in a similar fashion, so we
omit it here.
\end{proof}

\begin{lemma}\label{Milne lemma 3}
For any $\delta>0$ there is a constant $C(\delta)>0$ independent of
data such that
\begin{align}
\ltnm{\t[H]}\leq C(\delta)\tnnm{H}+\delta\lnnm{H}\label{mt 77}.
\end{align}
\end{lemma}
\begin{proof}
In the following, we use $\chi_i$ to represent certain indicator functions. Also, we let $m>0$ and $\sigma>0$ be some constants that are determined later. Since $\psi$ is always a constant along the characteristics and will not play a role in the estimates, we will omit it from the integrals and only highlight the relation between $\eta$ and $\phi$.\\
\ \\
Region I: $\sin\phi>0$.\\
Based on \eqref{mt 71'}, we have
\begin{align}
\t[H](\eta,\phi)=&\int_0^{\eta}\frac{H\Big(\eta',\phi'(\eta,\phi;\eta')\Big)}{\sin\Big(\phi'(\eta,\phi;\eta')\Big)}\exp\Big(-G_{\eta,\eta'}\Big)\ud{\eta'}.
\end{align}
We consider
\begin{align}
I=&\int_{\sin\phi>0}\abs{\t
[H](\eta,\phi)}^2\ud{\phi}=\int_{\sin\phi>0}\bigg(\int_0^{\eta}\frac{H\Big(\eta',\phi'(\eta,\phi;\eta')\Big)}{\sin\Big(\phi'(\eta,\phi;\eta')\Big)}
\exp(-G_{\eta,\eta'})\ud{\eta'}\bigg)^2\ud{\phi}=I_1+I_2.
\end{align}
Region I - Case I: $\chi_1: \sin\Big(\phi'(\eta,\phi;\eta')\Big)\geq m$.\\
By Cauchy's inequality, we get
\begin{align}\label{mt 78}
I_1\leq&\int_{\sin\phi>0}\bigg(\int_0^{\eta}\abs{H\Big(\eta',\phi'(\eta,\phi;\eta')\Big)}^2\ud{\eta'}\bigg)
\bigg(\int_0^{\eta}\chi_1\frac{\exp\Big(-2G_{\eta,\eta'}\Big)}{\sin^2\Big(\phi'(\eta,\phi;\eta')\Big)}
\ud{\eta'}\bigg)\ud{\phi}\\
\leq&\frac{1}{m}\int_{\sin\phi>0}\bigg(\int_0^{\eta}\abs{H\Big(\eta',\phi'(\eta,\phi;\eta')\Big)}^2\ud{\eta'}\bigg)
\bigg(\int_0^{\eta}\chi_1\frac{\exp\Big(-2G_{\eta,\eta'}\Big)}{\sin\Big(\phi'(\eta,\phi;\eta')\Big)}
\ud{\eta'}\bigg)\ud{\phi}\no\\
\leq&\frac{1}{m}\tnnm{H}^2\Bigg(\int_{\sin\phi>0}
\bigg(\int_0^{\eta}\chi_1\frac{\exp\Big(-2G_{\eta,\eta'}\Big)}{\sin\Big(\phi'(\eta,\phi;\eta')\Big)}
\ud{\eta'}\bigg)^2\ud{\phi}\Bigg)^{\frac{1}{2}}\nonumber\\
\leq&\frac{C}{m}\tnnm{H}^2\nonumber,
\end{align}
due to the substitution $z=G_{\eta,\eta'}$ which yields
\begin{align}
\int_0^{\eta}
\frac{1}{\sin\Big(\phi'(\eta,\phi;\eta')\Big)}\exp\Big(-2G_{\eta,\eta'}\Big)\ud{\eta'}\leq\int_{0}^{\infty}\ue^{-2z}\ud{z}=\frac{1}{2}.
\end{align}
\ \\
Region I - Case II: $\chi_2: \sin\Big(\phi'(\eta,\phi;\eta')\Big)\leq m$.\\
For $\eta'\leq\eta$, we can directly estimate
$\phi'(\eta,\phi;\eta')\geq\phi$. Hence, we have the relation
\begin{align}
\sin\phi\leq\sin\Big(\phi'(\eta,\phi;\eta')\Big).
\end{align}
Therefore, we can directly estimate $I_2$ as follows:
\begin{align}\label{mt 79}
I_2\leq&\lnnm{H}^2\int_{\sin\phi>0}
\bigg(\int_0^{\eta}\chi_2
\frac{1}{\sin\Big(\phi'(\eta,\phi;\eta')\Big)}\exp\Big(-G_{\eta,\eta'}\Big)\ud{\eta'}\bigg)^2\ud{\phi}\\
\leq&\lnnm{H}^2\int_{\sin\phi>0}\chi_2\ud{\phi}\no\\
\leq&Cm\lnnm{H}^2,\no
\end{align}
due to the substitution $z=G_{\eta,\eta'}$ which yields
\begin{align}
\int_0^{\eta}
\frac{1}{\sin\Big(\phi'(\eta,\phi;\eta')\Big)}\exp\Big(-G_{\eta,\eta'}\Big)\ud{\eta'}\leq\int_{0}^{\infty}\ue^{-z}\ud{z}=1.
\end{align}
Summing up \eqref{mt 78} and \eqref{mt 79}, for $m$
sufficiently small, we deduce \eqref{mt 77}.\\
\ \\
Region II: $\sin\phi<0$ and $E(\eta,\phi)\leq \ue^{-V(L)}$.\\
Based on \eqref{mt 72'}, we have
\begin{align}
\t[H](\eta,\phi)=&\int_0^{L}\frac{H\Big(\eta',\phi'(\eta,\phi;\eta')\Big)}{\sin\Big(\phi'(\eta,\phi;\eta')\Big)}
\exp\Big(-G_{L,\eta'}-G_{L,\eta}\Big)\ud{\eta'}
+\int_{\eta}^{L}\frac{H\Big(\eta',\rr[\phi'](\eta,\phi;\eta')\Big)}{\sin\Big(\phi'(\eta,\phi;\eta')\Big)}\exp\Big(G_{\eta,\eta'}\Big)\ud{\eta'}.
\end{align}
Since the integral $\eta'\in[0,\eta]$ can be estimated as in Region I, and for $\eta'\in[\eta,L]$, $-G_{L,\eta'}-G_{L,\eta}\leq -G_{\eta,\eta'}$, it suffices to estimate
\begin{align}
II=&\int_{\sin\phi<0}{\bf{1}}_{\{E(\eta,\phi)\leq
\ue^{-V(L)}\}}\bigg(\int_{\eta}^{L}
\frac{H\Big(\eta',\rr[\phi'](\eta,\phi;\eta')\Big)}{\sin\Big(\phi'(\eta,\phi;\eta')\Big)}
\exp\Big(G_{\eta,\eta'}\Big)\ud{\eta'}\bigg)^2\ud{\phi}
=II_1+II_2+II_3.
\end{align}
\ \\
Region II - Case I: $\chi_1: \sin\Big(\phi'(\eta,\phi;\eta')\Big)>m$.\\
Using Cauchy's inequality, we have
\begin{align}\label{mt 80}
II_1\leq&\int_{\sin\phi<0}\bigg(\int_{\eta}^{L}\abs{H\Big(\eta',\rr[\phi'](\eta,\phi;\eta')\Big)}^2\ud{\eta'}\bigg)
\bigg(\int_{\eta}^{L}\chi_1\frac{\exp\Big(2G_{\eta,\eta'}\Big)}{\sin^2\Big(\phi'(\eta,\phi;\eta')\Big)}
\ud{\eta'}\bigg)\ud{\phi}\\
\leq&\frac{1}{m}\int_{\sin\phi<0}\bigg(\int_{\eta}^{L}\abs{H\Big(\eta',\rr[\phi'](\eta,\phi;\eta')\Big)}^2\ud{\eta'}\bigg)
\bigg(\int_{\eta}^{L}\chi_1\frac{\exp\Big(2G_{\eta,\eta'}\Big)}{\sin\Big(\phi'(\eta,\phi;\eta')\Big)}
\ud{\eta'}\bigg)\ud{\phi}\no\\
\leq&\frac{1}{m}\tnnm{H}^2\int_{\sin\phi<0}
\bigg(\int_{\eta}^{L}\chi_1\frac{\exp\Big(2G_{\eta,\eta'}\Big)}{\sin\Big(\phi'(\eta,\phi;\eta')\Big)}
\ud{\eta'}\bigg)\ud{\phi}\no\\
\leq&\frac{C}{m}\tnnm{H}^2\no,
\end{align}
due to the substitution $z=G_{\eta,\eta'}$ which yields
\begin{align}
\int_{\eta}^L
\frac{1}{\sin\Big(\phi'(\eta,\phi;\eta')\Big)}\exp\Big(2G_{\eta,\eta'}\Big)\ud{\eta'}\leq\int^{0}_{-\infty}\ue^{2z}\ud{z}=\frac{1}{2}.
\end{align}
\ \\
Region II - Case II: $\chi_2: \sin\Big(\phi'(\eta,\phi;\eta')\Big)>m,\ \eta'-\eta\geq\sigma$.\\
We have
\begin{align}
II_2\leq&\lnnm{H}^2\int_{\sin\phi<0}
\bigg(\int_{\eta}^{L}\chi_2
\frac{\exp\Big(G_{\eta,\eta'}\Big)}{\sin\Big(\phi'(\eta,\phi;\eta')\Big)}\ud{\eta'}\bigg)^2\ud{\phi}.
\end{align}
Note that
\begin{align}
G_{\eta,\eta'}=\int_{\eta'}^{\eta}\frac{1}{\sin\Big(\phi'(\eta,\phi;y)\Big)}\ud{y}\leq-\frac{\eta'-\eta}{m}=-\frac{\sigma}{m}.
\end{align}
Then we can obtain
\begin{align}\label{mt 81}
II_2\leq&\lnnm{H}^2\int_{\sin\phi<0}
\bigg(\int^{-\frac{\sigma}{m}}_{-\infty}\ue^z\ud{z}\bigg)^2\ud{\phi}\leq C
\ue^{-\frac{2\sigma}{m}}\lnnm{H}^2.
\end{align}
\ \\
Region II - Case III: $\chi_3: \sin\Big(\phi'(\eta,\phi;\eta')\Big)>m,\ \eta'-\eta\leq\sigma$\\
For $II_3$, we can estimate as follows:
\begin{align}
I_3\leq&\lnnm{H}^2\int_{\sin\phi<0}
\bigg(\int_{\eta}^{L}\chi_3
\frac{\exp\Big(G_{\eta,\eta'}\Big)}{\sin\Big(\phi'(\eta,\phi;\eta')\Big)}\ud{\eta'}\bigg)^2\ud{\phi}\\
\leq&\lnnm{H}^2\int_{\sin\phi<0}\chi_3
\bigg(\int_{\eta}^{\eta+\sigma}
\frac{\exp\Big(G_{\eta,\eta'}\Big)}{\sin\Big(\phi'(\eta,\phi;\eta')\Big)}\ud{\eta'}\bigg)^2\ud{\phi}\no.
\end{align}
Note that due to the substitution $z=G_{\eta,\eta'}$, we have
\begin{align}
\int_{\eta}^{L}
\frac{1}{\sin\Big(\phi'(\eta,\phi;\eta')\Big)}\exp\Big(G_{\eta,\eta'}\Big)\ud{\eta'}\leq\int_{-\infty}^{0}\ue^{z}\ud{z}=1.
\end{align}
Then letting $\alpha=\ue^{V(\eta')-V(\eta)}$, we know
\begin{align}
1\leq\alpha:=\ue^{V(\eta')-V(\eta)}\leq
\ue^{V(\eta+\sigma)-V(\eta)}\leq 1+4\sigma.
\end{align}
Also, for $\eta'\in[\eta,\eta+\sigma]$,
\begin{align}
\sin\Big(\phi'(\eta,\phi;\eta')\Big)=\sin\bigg(\cos^{-1}(\alpha\cos\phi)\bigg).
\end{align}
Hence, $\sin\Big(\phi'(\eta,\phi;\eta')\Big)<m$ leads to
\begin{align}
\abs{\sin\phi}=&\sqrt{1-\cos^2\phi}=\sqrt{1-\frac{\cos^2\Big(\phi'(\eta,\phi;\eta')\Big)}{\alpha^2}}
=\frac{\sqrt{\alpha^2-\bigg(1-\sin^2\Big(\phi'(\eta,\phi;\eta')\Big)\bigg)}}{\alpha}\\
\leq&\frac{\sqrt{\alpha^2-1+m^2}}{\alpha}\leq\frac{\sqrt{(1+4\sigma)^2-1+m^2}}{\alpha}\leq
\sqrt{9\sigma+m^2}.\no
\end{align}
Hence, we can obtain
\begin{align}\label{mt 82}
I_3\leq&\lnnm{H}^2\int_{\sin\phi<0}\chi_3\ud{\phi}\leq\lnnm{H}^2
\int_{\sin\phi<0}{\bf{1}}_{\{\abs{\sin\phi}\leq \sqrt{9\sigma+m^2}\}}\ud{\phi}\leq C\sqrt{\sigma+m^2}.
\end{align}
Summarizing \eqref{mt 80}, \eqref{mt 81} and
\eqref{mt 82}, for sufficiently small $\sigma$,
we can always choose $m<<\sigma$ small enough to guarantee the relation \eqref{mt 77}.\\
\ \\
Region III: $\sin\phi<0$ and $E(\eta,\phi)\geq \ue^{-V(L)}$.\\
We have
\begin{align}
\t[H](\eta,\phi)=&\int_0^{\eta^+}\frac{H\Big(\eta',\phi'(\eta,\phi;\eta')\Big)}{\sin\Big(\phi'(\eta,\phi;\eta')\Big)}
\exp\Big(-G_{\eta^+,\eta'}-G_{\eta^+,\eta}\Big)\ud{\eta'}
+\int_{\eta}^{\eta^+}\frac{H\Big(\eta',\rr[\phi'(\eta,\phi;\eta')]\Big)}{\sin\Big(\phi'(\eta,\phi;\eta')\Big)}\exp\Big(G_{\eta,\eta'}\Big)\ud{\eta'}.
\end{align}
We can decompose $\t[H]$.
For the integral on $[0,\eta]$, we can apply a similar argument as in Region 1 and for
the integral on $[\eta,\eta^+]$, a similar argument as in Region 2 completes the proof.
\end{proof}

\subsubsection{$L^2-L^{\infty}$ Framework}

Consider the equation satisfied by $\v=f-f_L$ as follows:
\begin{align}\label{difference equation}
\left\{
\begin{array}{l}\displaystyle
\sin\phi\frac{\p \v}{\p\eta}+F(\eta,\psi)\cos\phi\frac{\p
\v}{\p\phi}+\v=\bar \v+S,\\\rule{0ex}{2em}
\v(0,\phi,\psi)=p(\phi,\psi):=h(\phi,\psi)-f_L\ \ \text{for}\ \ \sin\phi>0,\\\rule{0ex}{2em}
\v(L,\phi,\psi)=\v(L,\rr[\phi],\psi).
\end{array}
\right.
\end{align}
\begin{theorem}\label{Milne LI Theorem}
The unique solution $f(\eta,\phi,\psi)$ to the equation
\eqref{Milne problem} satisfies
\begin{align}\label{mt 87}
\lnnm{f-f_L}\leq C\bigg(1+\tnnm{f-f_L}\bigg).
\end{align}
\end{theorem}
\begin{proof}
We first show the following important facts:
\begin{align}
\tnnm{\bar \v}\leq&\tnnm{\v}\label{mt 83},\\
\lnnm{\bar \v}\leq&\ltnm{\v}\label{mt 84}.
\end{align}
We can directly derive them by Cauchy's inequality as follows:
\begin{align}
\tnnm{\bar
\v}^2=&\int_0^{\infty}\int_{-\pi}^{\pi}\int_{-\frac{\pi}{2}}^{\frac{\pi}{2}}\bigg(\frac{1}{4\pi}\bigg)^2
\bigg(\int_{-\pi}^{\pi}\int_{-\frac{\pi}{2}}^{\frac{\pi}{2}}\v(\eta,\phi,\psi)\cos\phi\ud{\phi}\ud\psi\bigg)^2\ud{\phi}\ud\psi\ud{\eta}\\
\leq&\int_0^{\infty}\int_{-\pi}^{\pi}\int_{-\frac{\pi}{2}}^{\frac{\pi}{2}}\bigg(\frac{1}{4\pi}\bigg)
\bigg(\int_{-\pi}^{\pi}\int_{-\frac{\pi}{2}}^{\frac{\pi}{2}}\v^2(\eta,\phi,\psi)\cos\phi\ud{\phi}\ud\psi\bigg)\ud{\phi}\ud\psi\ud{\eta}\no\\
=&\int_0^{\infty}\bigg(\int_{-\pi}^{\pi}\int_{-\frac{\pi}{2}}^{\frac{\pi}{2}}\v^2(\eta,\phi,\psi)\cos\phi\ud{\phi}\ud\psi\bigg)\ud{\eta}=\tnnm{\v}^2\no.\\
\lnnm{\bar \v}^2=&\text{esssup}_{\eta}\bar
\v^2(\eta)=\text{esssup}_{\eta}\bigg(\frac{1}{4\pi}\int_{-\pi}^{\pi}\int_{-\frac{\pi}{2}}^{\frac{\pi}{2}}\v(\eta,\phi,\psi)\cos\phi\ud{\phi}\ud\psi\bigg)^2\\
\leq&\text{esssup}_{\eta}\bigg(\frac{1}{4\pi}\bigg)^2\bigg(\int_{-\pi}^{\pi}\int_{-\frac{\pi}{2}}^{\frac{\pi}{2}}\v^2(\eta,\phi,\psi)\cos\phi\ud{\phi}\ud\psi\bigg)
\bigg(\int_{-\pi}^{\pi}\int_{-\frac{\pi}{2}}^{\frac{\pi}{2}}\cos\phi\ud{\phi}\ud\psi\bigg)\no\\
=&\text{esssup}_{\eta}\bigg(\int_{-\pi}^{\pi}\int_{-\frac{\pi}{2}}^{\frac{\pi}{2}}\v^2(\eta,\phi,\psi)\cos\phi\ud{\phi}\ud\psi\bigg)=\ltnm{\v}^2\no.
\end{align}
In \eqref{difference equation}, based on Lemma \ref{Milne lemma 1}, Lemma \ref{Milne lemma 2}, Lemma \ref{Milne lemma 3}, \eqref{mt 83} and
\eqref{mt 84}, $\v=\k[p]+\t[\bar \v]+\t[S]$ leads to
\begin{align}\label{mt 85}
\ltnm{\v}\leq&\ltnm{\k[p]}+\ltnm{\t[S]}+\ltnm{\t[\bar\v]}\\
\leq&\ltnm{\k[p]}+\ltnm{\t[S]}+C(\delta)\tnnm{\bar \v}+\delta\lnnm{\bar \v}\no\\
\leq&\lnmp{p}+\lnnm{S}+C(\delta)\tnnm{\v}+\delta\ltnm{\v}\no.
\end{align}
We can take $\delta=\dfrac{1}{2}$ to absorb $\delta\ltnm{\v}$ into the left-hand side and obtain
\begin{align}\label{mt 86}
\ltnm{\v}\leq C\bigg(\tnnm{\v}+\lnmp{p}+\lnnm{S}\bigg).
\end{align}
Therefore, in \eqref{difference equation}, based on Lemma \ref{Milne lemma 1}, Lemma \ref{Milne lemma 2} and \eqref{mt 86}, we can achieve
\begin{align}
\\
\lnnm{\v}\leq&\lnnm{\k[p]}+\lnnm{\t[S]}+\lnnm{\t[\bar \v]}
\leq
C\bigg(\lnmp{p}+\lnnm{S}+\lnnm{\bar \v}\bigg)\no\\
\leq&C\bigg(\lnmp{p}+\lnnm{S}+\ltnm{\v}\bigg)
\leq
C\bigg(\lnmp{p}+\lnnm{S}+\tnnm{\v}\bigg)\no.
\end{align}
Then \eqref{mt 87} follows.
\end{proof}

Combining Theorem \ref{Milne LT Theorem} and Theorem \ref{Milne LI Theorem}, we deduce the main theorem.
\begin{theorem}\label{Milne theorem 1}
Assume $S$ and $h$ satisfy \eqref{Milne bounded} and \eqref{Milne decay}. Then there exists a solution $f(\eta,\phi,\psi)$ to the equation
\eqref{Milne problem}, satisfying that
\begin{align}
\abs{f_L}\leq C,\qquad \tnnm{f-f_L}+\lnnm{f-f_L}\leq C,
\end{align}
where
\begin{align}
f_L=\dfrac{\br{\sin^2\phi,f}_{\phi,\psi}(L)}{\tnm{\sin\phi}^2}.
\end{align}
The solution is unique among functions such that above estimates hold.
\end{theorem}

\subsection{Exponential Decay}

In this section, we prove the spatial decay of the solution to the
$\e$-Milne problem with geometric correction.
\begin{theorem}\label{Milne theorem 2}
Assume $S$ and $h$ satisfy \eqref{Milne bounded} and \eqref{Milne decay}. Then there exists $K_0>0$ such that the solution $f(\eta,\phi,\psi)$ to the
equation \eqref{Milne problem} satisfies
\begin{align}
\tnnm{\ue^{K_0\eta}(f-f_L)}+\lnnm{\ue^{K_0\eta}(f-f_L)}\leq C.
\end{align}
\end{theorem}
\begin{proof}
We divide the proof into several steps:\\
\ \\
Step 1: $L^2$ Estimates.\\
We still use $\v=f-f_L$ as defined in \eqref{difference equation}. Assume $\bar S=0$. We consider the decomposition $\v=r_{\v}+q_{\v}$ as before. We continue using the notation $F=\tf+G$. Now we naturally have $(q_{\v})_{L}=0$.
%

The extra term is $K_0\v\sin\phi$. The quasi-orthogonal property \eqref{mt 21}
reveals
\begin{align}\label{mt 88}
\br{\v,\v\sin\phi}_{\phi,\psi}(\eta)=&\br{r_{\v},r_{\v}\sin\phi}_{\phi,\psi}(\eta)+2\br{r_{\v},q_{\v}\sin\phi}_{\phi,\psi}(\eta)+\br{q_{\v},q_{\v}\sin\phi}_{\phi,\psi}(\eta)\\
=&\br{r_{\v},r_{\v}\sin\phi}_{\phi,\psi}(\eta)-4q_{\v}(\eta)\int_{\eta}^L\ue^{2\tv(\eta)-2\tv(y)}G(y)\br{\sin\phi\cos^2\psi,r_{\v}}_{\phi,\psi}(y)\ud{y}.\no
\end{align}
Multiplying $\ue^{2K_0\eta}\v\cos\phi$ on both sides of equation \eqref{difference equation} and integrating over $(\phi,\psi)\in\left[-\dfrac{\pi}{2},\dfrac{\pi}{2}\right]\times[-\pi,\pi]$, considering \eqref{mt 88}, we obtain
\begin{align}\label{mt 91}
&\half\frac{\ud{}}{\ud{\eta}}\bigg(\ue^{2K_0\eta}\br{\v,\v\sin\phi}_{\phi}(\eta)\bigg)
+F(\eta)\bigg(\ue^{2K_0\eta}\br{\v,\v\sin\phi}_{\phi}(\eta)\bigg)\\
=&\ue^{2K_0\eta}K_0\br{\v,\v\sin\phi}_{\phi}(\eta)-\br{r_{\v},
r_{\v}}_{\phi}(\eta)
-G(\eta)\bigg(\ue^{2K_0\eta}\br{\v\cos^2\psi,\v\sin\phi}_{\phi}(\eta)\bigg)+\ue^{2K_0\eta}\br{S,r_{\v}}_{\phi}(\eta)\nonumber\\
=&\ue^{2K_0\eta}\bigg(K_0\br{r_{\v},r_{\v}\sin\phi}_{\phi}(\eta)-\br{r_{\v},
r_{\v}}_{\phi}(\eta)\bigg)\no\\
&+4\ue^{2K_0\eta}K_0q_{\v}(\eta)\int_{\eta}^L\ue^{2\tv(\eta)-\tv(y)}G(y)\br{\sin\phi\cos^2\psi,r_{\v}}_{\phi,\psi}(y)\ud{y}\no\\
&-G(\eta)\bigg(\ue^{2K_0\eta}\br{\v\cos^2\psi,\v\sin\phi}_{\phi}(\eta)\bigg)+\ue^{2K_0\eta}\br{S,r_{\v}}_{\phi}(\eta)\nonumber.
\end{align}
For $K_0<\min\left\{\dfrac{1}{2},K\right\}$, we have
\begin{align}\label{mt 92}
\frac{3}{2}\tnm{r_{\v}(\eta)}^2\geq-K_0\br{r_{\v},r_{\v}\sin\phi}_{\phi}(\eta)+\br{r_{\v},r_{\v}}_{\phi}(\eta)\geq
\half\tnm{r_{\v}(\eta)}^2.
\end{align}
Similar to the proof of Lemma \ref{Milne finite LT}, formula as \eqref{mt 91} and
\eqref{mt 92} imply
\begin{align}\label{mt 93}
\tnnm{\ue^{K_0\eta}r_{\v}}^2\leq&\abs{\int_0^L4\ue^{2K_0\eta}K_0q_{\v}(\eta)\int_{\eta}^L\ue^{2\tv(\eta)-\tv(y)}G(y)\br{\sin\phi\cos^2\psi,r_{\v}}_{\phi,\psi}(y)\ud{y}\ud{\eta}}\\
&+\abs{\int_0^LG(\eta)\bigg(\ue^{2K_0\eta}\br{\v\cos^2\psi,\v\sin\phi}_{\phi}(\eta)\bigg)\ud{\eta}}+\abs{\int_0^L\ue^{2K_0\eta}\br{S,r_{\v}}_{\phi}(\eta)\ud{\eta}}\no\\
\leq&CL\tnnm{G}\tnnm{\ue^{K_0\eta}r_{\v}}\tnnm{\ue^{K_0\eta}q_{\v}}+C\e\tnnm{\ue^{K_0\eta}\v}^2\no\\
&+C\tnnm{\ue^{K_0\eta}S}\tnnm{\ue^{K_0\eta}r_{\v}}\no\\
\leq&C\e\tnnm{\ue^{K_0\eta}q_{\v}}^2+C\e^{1-\frac{3n}{2}}\tnnm{\ue^{K_0\eta}r_{\v}}^2+\e\tnnm{\ue^{K_0\eta}\v}^2+C\tnnm{S}^2\no\\
\leq&C+C\tnnm{\ue^{K_0\eta}q_{\v}}^2+C\e^{1-\frac{3n}{2}}\tnnm{\ue^{K_0\eta}r_{\v}}^2+\e\tnnm{\ue^{K_0\eta}\v}^2.\no
\end{align}
Hence, for $\e$ sufficiently small, we know
\begin{align}\label{mt 95}
\tnnm{\ue^{K_0\eta}r_{\v}}^2\leq&C+C\e\tnnm{\ue^{K_0\eta}q_{\v}}^2+C\e\tnnm{\v}^2.
\end{align}
Then similar to the proof of Lemma \ref{Milne finite LT}, we deduce
\begin{align}
&\tnnm{\ue^{K_0\eta}q_{\v}}^2\\
\leq&\tnnm{\ue^{K_0\eta}r_{\v}}^2+\int_0^L\ue^{2K_0\eta}\abs{\int_{\eta}^L\tf(y)\br{1-3\sin^2\phi,r_{\v}}_{\phi,\psi}(y)\ud{y}}^2\ud{\eta}\no\\
&+\int_0^L\ue^{2K_0\eta}\abs{\int_{\eta}^{L}\int_{z}^L\ue^{2\tv(z)-2\tv{y}}G(y)\br{\sin\phi\cos^2\psi,r_{\v}}_{\phi,\psi}\ud{y}\ud{z}}^2\ud{\eta}\no\\
&+\int_0^L\ue^{2K_0\eta}\abs{\int_{\eta}^{L}G(y)\br{1-3\sin^2\phi,r_{\v}\cos^2\psi}_{\phi,\psi}(y)\ud{y}}^2\ud{\eta}
+\int_0^L\ue^{2K_0\eta}\abs{\int_{\eta}^{L}\br{\sin\phi,S}_{\phi,\psi}(y)\ud{y}}^2\ud{\eta}\no\\
\leq&C+C\tnnm{\ue^{K_0\eta}r_{\v}}^2+C\tnnm{\ue^{K_0\eta}r_{\v}}
\bigg(\int_0^{L}\int_{\eta}^{L}\ue^{2K_0(\eta-y)}F^2(y)\ud{y}\ud{\eta}\bigg)\no\\
&+L^3\tnnm{G}^2\tnnm{\ue^{K_0\eta}r_{\v}}^2+L\tnnm{G}^2\tnnm{\ue^{K_0\eta}r_{\v}}^2
+\int_0^{L}\ue^{2K_0\eta}\bigg(\int_{\eta}^{L}\lnm{S(y)}\ud{y}\bigg)^2\ud{\eta}
\nonumber\\
\leq&C+C(1+\e^{2-5n})\tnnm{\ue^{K_0\eta}r_{\v}}^2\no\\
\leq&C+C\tnnm{\ue^{K_0\eta}r_{\v}}^2\no\\
\leq&C+C\e\tnnm{\ue^{K_0\eta}q_{\v}}^2+C\e\tnnm{\ue^{K_0\eta}\v}^2,\no
\end{align}
which implies
\begin{align}\label{mt 96}
\tnnm{\ue^{K_0\eta}q_{\v}}^2\leq C+C\e\tnnm{\ue^{K_0\eta}\v}^2.
\end{align}
In summary, for $\e$ sufficiently small, \eqref{mt 95} and \eqref{mt 96} imply
\begin{eqnarray}
\tnnm{\ue^{K_0\eta}\v}^2\leq&\tnnm{\ue^{K_0\eta}q_{\v}}^2+\tnnm{\ue^{K_0\eta}r_{\v}}^2\leq C+C\e\tnnm{\ue^{K_0\eta}\v}^2,
\end{eqnarray}
which yields
\begin{eqnarray}\label{mt 94}
\tnnm{\ue^{K_0\eta}\v}\leq&C.
\end{eqnarray}
This is the desired result when $\bar S=0$.
By the method introduced in the proof of Lemma \ref{Milne finite LT.},
we can extend above $L^2$ estimates to the general $S$ case. Note
all the auxiliary functions
constructed in the proof of Lemma \ref{Milne finite LT.} satisfy the desired estimates. \\
\ \\
Step 2: $L^{\infty}$ Estimates.\\
Define $Z=\ue^{K_0\eta}\v$. Then $Z$ satisfies the equation
\begin{align}\label{decay equation}
\left\{
\begin{array}{l}\displaystyle
\sin\phi\frac{\p Z}{\p\eta}+F(\eta,\psi)\cos\phi\frac{\p
Z}{\p\phi}+Z=\bar Z+\ue^{K_0\eta}S+K_0\sin\phi Z,\\\rule{0ex}{2em}
Z(0,\phi,\psi)=p(\phi,\psi)=h(\phi,\psi)-f_L\ \ \text{for}\ \ \sin\phi>0,\\\rule{0ex}{2em}
Z(L,\phi,\psi)=Z(L,\rr[\phi],\psi).
\end{array}
\right.
\end{align}
In \eqref{decay equation}, based on Lemma \ref{Milne lemma 1}, Lemma \ref{Milne lemma 2}, Lemma \ref{Milne lemma 3}, \eqref{mt 83} and
\eqref{mt 84}, $Z=\k[p]+\t[\bar Z]+\t[\ue^{K_0\eta}S]+\t[K_0\sin\phi Z]$ leads to
\begin{align}\label{mt 85'}
\ltnm{Z}\leq&\ltnm{\k[p]}+\ltnm{\t[\ue^{K_0\eta}S]}+\ltnm{\t[K_0\sin\phi Z]}+\ltnm{\t[\bar\v]}\\
\leq&\ltnm{\k[p]}+\ltnm{\t[\ue^{K_0\eta}S]}+K_0\ltnm{Z}+C(\delta)\tnnm{\bar Z}+\delta\lnnm{\bar Z}\no\\
\leq&\lnmp{p}+\lnnm{\ue^{K_0\eta}S}+K_0\ltnm{Z}+C(\delta)\tnnm{Z}+\delta\ltnm{Z}\no.
\end{align}
We can take $\delta=\dfrac{1}{4}$ and $K_0=\min\left\{K,\dfrac{1}{4}\right\}$ to absorb $\delta\ltnm{Z}$ and $K_0\ltnm{Z}$ into the left-hand side and obtain
\begin{align}\label{mt 86'}
\ltnm{Z}\leq C\bigg(\tnnm{Z}+\lnmp{p}+\lnnm{\ue^{K_0\eta}S}\bigg).
\end{align}
Therefore, in \eqref{decay equation}, based on Lemma \ref{Milne lemma 1}, Lemma \ref{Milne lemma 2} and \eqref{mt 86'}, we can achieve
\begin{align}
\lnnm{Z}\leq&\lnnm{\k[p]}+\lnnm{\t[\ue^{K_0\eta}S]}+\ltnm{\t[K_0\sin\phi Z]}+\lnnm{\t[\bar Z]}\\
\leq&
C\bigg(\lnmp{p}+\lnnm{\ue^{K_0\eta}S}+K_0\lnnm{Z}+\lnnm{\bar Z}\bigg)\no\\
\leq&C\bigg(\lnmp{p}+\lnnm{\ue^{K_0\eta}S}+K_0\lnnm{Z}+\ltnm{Z}\bigg)\no\\
\leq&
C\bigg(\lnmp{p}+\lnnm{\ue^{K_0\eta}S}+K_0\lnnm{Z}+\tnnm{Z}\bigg)\no.
\end{align}
Absorbing $K_0\lnnm{Z}$ into the left-hand side, we obtain
\begin{align}\label{mt 93}
\lnnm{Z}\leq
C\bigg(\lnmp{p}+\lnnm{\ue^{K_0\eta}S}+\tnnm{Z}\bigg).
\end{align}
\ \\
Step 3: Synthesis.\\
Combining \eqref{mt 94} and \eqref{mt 93}, we
deduce the desired result
\begin{align}
\lnnm{\ue^{K_0\eta}(f-f_L)}\leq C\bigg(1+\tnnm{\ue^{K_0\eta}(f-f_L)}\bigg)\leq C.
\end{align}
\end{proof}

\subsection{Maximum Principle}

\begin{theorem}\label{Milne theorem 3}
The solution $f(\eta,\phi,\psi)$ to the equation \eqref{Milne problem} with $S=0$ satisfies the maximum principle, i.e.
\begin{align}
\min_{\sin\phi>0}h(\phi,\psi)\leq f(\eta,\phi,\psi)\leq
\max_{\sin\phi>0}h(\phi,\psi).
\end{align}
\end{theorem}
\begin{proof}
We claim that it suffices to show $f(\eta,\phi,\psi)\leq0$ whenever
$h(\phi,\psi)\leq0$. Suppose this claim is justified.
Denote $m=\min_{\sin\phi>0}h(\phi,\psi)$ and
$M=\max_{\sin\phi>0}h(\phi,\psi)$. Then $f^1=f-M$
satisfies the equation
\begin{align}
\left\{
\begin{array}{l}\displaystyle
\sin\phi\frac{\p f^1}{\p\eta}+F(\eta,\psi)\cos\phi\frac{\p
f^1}{\p\phi}+f^1-\bar f^1=0,\\\rule{0ex}{2em}
f^1(0,\phi,\psi)=h(\phi,\psi)-M\ \ \text{for}\ \ \sin\phi>0,\\\rule{0ex}{2em}
f^1(L,\phi,\psi)=f^1(L,\rr[\phi],\psi).
\end{array}
\right.
\end{align}
Hence, $h-M\leq0$ implies $f^1\leq0$ which is actually $f\leq M$. On
the other hand, $f^2=m-f$ satisfies the equation
\begin{align}
\left\{
\begin{array}{l}\displaystyle
\sin\phi\frac{\p f^2}{\p\eta}+F(\eta,\psi)\cos\phi\frac{\p
f^2}{\p\phi}+f^2-\bar f^2=0,\\\rule{0ex}{2em}
f^2(0,\phi,\psi)=m-h(\phi,\psi)\ \ \text{for}\ \ \sin\phi>0,\\\rule{0ex}{2em}
f^2(L,\phi,\psi)=f^2(L,\rr[\phi],\psi).
\end{array}
\right.
\end{align}
Thus, $m-h\leq0$ implies $f^2\leq0$ which further leads to $f\geq
m$. Therefore, the maximum principle is established.\\
\ \\
We now prove the claim that if $h(\phi,\psi)\leq0$, we have $f(\eta,\phi,\psi)\leq0$.
Assuming $h(\phi,\psi)\leq0$, we then consider the penalized Milne
problem for $f_{\l}(\eta,\phi,\psi)$
\begin{align}\label{mt 101}
\left\{
\begin{array}{l}\displaystyle
\l f_{\l}+\sin\phi\frac{\p f_{\l}}{\p\eta}+F(\eta,\psi)\cos\phi\frac{\p
f_{\l}}{\p\phi}+f_{\l}-\bar f_{\l}=0,\\\rule{0ex}{2em}
f_{\l}(0,\phi,\psi)=h(\phi,\psi)\ \ \text{for}\ \ \sin\phi<0,\\\rule{0ex}{2em}
f_{\l}(L,\phi,\psi)=f_{\l}(L,\rr[\phi],\psi).
\end{array}
\right.
\end{align}
In order to construct the solution of \eqref{mt 101}, we iteratively define the sequence
$\{f_{m}\}_{m=1}^{\infty}$ as $f_{0}=0$ and
\begin{align}
\left\{
\begin{array}{l}\displaystyle
\l f_{m}+\sin\phi\frac{\p
f_{m}}{\p\eta}+F(\eta,\psi)\cos\phi\frac{\p
f_{m}}{\p\phi}+f_{m}-\bar f_{m-1}=0,\\\rule{0ex}{2em}
f_{m}(0,\phi,\psi)=h(\phi,\psi)\ \ \text{for}\ \ \sin\phi<0,\\\rule{0ex}{2em}
f_{m}(L,\phi,\psi)=f_{m}(L,\rr[\phi],\psi).
\end{array}
\right.
\end{align}
Along the characteristics, it is easy to see we always have
$f_{m}<0$. The standard $L^{\infty}$ estimates reveals that $f_{m}$ converges strongly in
$L^{\infty}$ to $f_{\l}$ which satisfies
\eqref{mt 101}. Also, $f_{\l}$ satisfies
\begin{align}
\lnnm{f_{\l}}\leq \frac{1+\l}{\l}\lnmp{h}.
\end{align}
Naturally, we obtain $f_{\l}\leq0$. Similar to the proof of Lemma \ref{Milne finite LT.}, we know $f_{\l}$
is uniformly bounded in $L^{2}$ with respect
to $\l$, which implies we can take weakly convergent subsequence
$f_{\l}\rightharpoonup f$ as $\l\rt0$ with $f\in
L^2$.
Naturally, we have $f(\eta,\phi,\psi)\leq0$. This justifies the claim and completes the proof.
\end{proof}

\section{Regularity}

Using the notation as in the section of well-posedness and decay, $\v=f-f_L$ satisfies the $\e$-Milne problem with geometric correction
\begin{align}\label{Milne difference problem}
\left\{
\begin{array}{l}\displaystyle
\sin\phi\frac{\p \v}{\p\eta}+F(\eta,\psi)\cos\phi\frac{\p
\v}{\p\phi}+\v-\bar\v=S(\eta,\phi,\psi),\\\rule{0ex}{2.0em}
\v(0,\phi,\psi)=p(\phi,\psi)=h(\phi,\psi)-f_L\ \ \text{for}\ \ \sin\phi>0,\\\rule{0ex}{2.0em}
\v(L,\phi,\psi)=\v(L,\rr[\phi],\psi),
\end{array}
\right.
\end{align}
where
\begin{align}
F(\eta,\psi)=-\e\bigg(\dfrac{\sin^2\psi}{R_1-\e\eta}+\dfrac{\cos^2\psi}{R_2-\e\eta}\bigg).
\end{align}
Still, we will omit the dependence on $\e$ and $\iota_i$ for $i=1,2$ when there is no confusion. \\
\ \\
The potential function
\begin{align}\label{potential}
V(\eta,\psi)=\ln\left(\frac{R_1}{R_1-\e\eta}\right)\sin^2\psi+\ln\left(\frac{R_2}{R_2-\e\eta}\right)\cos^2\psi,
\end{align}
will play a significant role in this section. Recall that $V$ satisfies $V(0,\psi)=0$ and $\dfrac{\p V}{\p\eta}=-F(\eta,\psi)$. \\
\ \\
Define weight function $\zeta(\eta,\phi,\psi)$ as
\begin{align}\label{weight function}
\zeta(\eta,\phi,\psi)=\Bigg(1-\bigg(\ue^{-V(\eta,\psi)}\cos\phi\bigg)^2\Bigg)^{\frac{1}{2}}.
\end{align}
Note that $\zeta$ goes to zero as $(\eta,\phi,\psi)$ approaches the grazing set. Also, we have $0\leq\zeta\leq 1$.
Recall the energy functional
\begin{align}\label{rt 25}
E(\eta,\phi,\psi)=\ue^{-V(\eta,\psi)}\cos\phi=\cos\phi\bigg(\frac{R_1-\e\eta}{R_1}\bigg)^{\sin^2\psi}\bigg(\frac{R_2-\e\eta}{R_2}\bigg)^{\cos^2\psi}.
\end{align}
Then we know
\begin{align}\label{rt 26}
\zeta=(1-E^2)^{\frac{1}{2}}=\Bigg(1-\cos^2\phi\bigg(\frac{R_1-\e\eta}{R_1}\bigg)^{2\sin^2\psi}\bigg(\frac{R_2-\e\eta}{R_2}\bigg)^{2\cos^2\psi}\Bigg)^{\frac{1}{2}}.
\end{align}
\begin{lemma}\label{rt lemma 1}
We have
\begin{align}
\sin\phi\frac{\p \zeta}{\p\eta}+F(\eta,\psi)\cos\phi\frac{\p\zeta}{\p\phi}=0.
\end{align}
\end{lemma}
\begin{proof}
We may directly compute
\begin{align}
\frac{\p \zeta}{\p\eta}=&\frac{1}{2}\Bigg(1-\bigg(\ue^{-V(\eta,\psi)}\cos\phi\bigg)^2\Bigg){-\frac{1}{2}}\bigg(-2\ue^{-2V(\eta,\psi)}\cos^2\phi\bigg)F(\eta,\psi)
=-\frac{\ue^{-2V(\eta,\psi)}F(\eta,\psi)\cos^2\phi}{\zeta},\\
\frac{\p\zeta}{\p\phi}=&\frac{1}{2}\Bigg(1-\bigg(\ue^{-V(\eta,\psi)}\cos\phi\bigg)^2\Bigg){-\frac{1}{2}}\bigg(-2\ue^{-2V(\eta,\psi)}\cos\phi\bigg)(-\sin\phi)
=\frac{\ue^{-2V(\eta,\psi)}\cos\phi\sin\phi}{\zeta}.
\end{align}
Hence, we know
\begin{align}
\\
\sin\phi\frac{\p \zeta}{\p\eta}+F(\eta,\psi)\cos\phi\frac{\p\zeta}{\p\phi}&=&
\frac{-\sin\phi\bigg(\ue^{-2V(\eta,\psi)}F(\eta,\psi)\cos^2\phi\bigg)+F(\eta,\psi)\cos\phi\bigg(\ue^{-2V(\eta,\psi)}\cos\phi\sin\phi\bigg)}{\zeta}=0.\no
\end{align}
\end{proof}

\subsection{Mild Formulation}

Taking $\eta$ derivative in \eqref{Milne difference problem} and multiplying $\zeta$, we obtain the $\e$-transport problem for $\a=\zeta\dfrac{\p\v}{\p\eta}$ as
\begin{align}\label{regularity 1}
\left\{
\begin{array}{l}\displaystyle
\sin\phi\frac{\p\a}{\p\eta}+F(\eta,\psi)\cos\phi\frac{\p
\a}{\p\phi}+\a=\tilde\a+S_{\a},\\\rule{0ex}{1.5em}
\a(0,\phi,\psi)=p_{\a}(\phi,\psi)\ \ \text{for}\ \ \sin\phi>0,\\\rule{0ex}{1.5em}
\a(L,\phi,\psi)=\a(L,\rr[\phi],\psi),
\end{array}
\right.
\end{align}
where $p_{\a}$ and $S_{\a}$ will be specified later. Also,
\begin{align}
\tilde\a(\eta,\phi,\psi)=\frac{1}{2\pi}\int_{-\pi}^{\pi}\int_{-\frac{\pi}{2}}^{\frac{\pi}{2}}
\frac{\zeta(\eta,\phi,\psi)}{\zeta(\eta,\phi_{\ast},\psi_{\ast})}\a(\eta,\phi_{\ast},\psi_{\ast})\cos(\phi_{\ast})\ud{\phi_{\ast}}\ud\psi_{\ast}.
\end{align}
Here for clarity, we use $\phi_{\ast}$ and $\psi_{\ast}$ to represent the velocity dummy variables. Along the characteristics, the energy $E$ and the weight $\zeta$ are constants. The equation \eqref{regularity 1} can be simplified as
\begin{align}
\sin\phi\frac{\ud{\a}}{\ud{\eta}}+\a=\tilde\a+S_{\a}.
\end{align}
Also, $\psi$ is a constant along the characteristics, so we may temporarily ignore $\psi$ dependence when there is no confusion. Note that all the estimates are uniform in $\psi$. As in $L^{\infty}$ estimates of the $\e$-Milne problem with geometric correction, we define $\phi'$, $\rr[\phi']$, $\eta^+$ and $G_{\eta,\eta'}$\\
\ \\
Depending on whether the characteristics touch $\phi=0$ or $\eta=L$, we can rewrite the solution to the equation \eqref{regularity 1} as
\begin{align}
\a(\eta,\phi)=\k[p_{\a}]+\t[\tilde\a+S_{\a}],
\end{align}
where\\
\ \\
Region I:\\
For $\sin\phi>0$,
\begin{align}
\k[p_{\a}]=&p_{\a}\Big(\phi'(\eta,\phi;0)\Big)\exp\Big(-G_{\eta,0}\Big)\\
\t[\tilde\a+S_{\a}]=&\int_0^{\eta}\frac{(\tilde\a+S_{\a})\Big(\eta',\phi'(\eta,\phi;\eta')\Big)}{\sin\Big(\phi'(\eta,\phi;\eta')\Big)}\exp\Big(-G_{\eta,\eta'}\Big)\ud{\eta'}.
\end{align}
\ \\
Region II:\\
For $\sin\phi<0$ and $E(\eta,\phi)\leq \ue^{-V(L)}$,
\begin{align}
\k[p_{\a}]=&p_{\a}\Big(\phi'(\eta,\phi;0)\Big)\exp\Big(-G_{L,0}-G_{L,\eta}\Big)\\
\t[\tilde\a+S_{\a}]=&\int_0^{L}\frac{(\tilde\a+S)\Big(\eta',\phi'(\eta,\phi;\eta')\Big)}{\sin\Big(\phi'(\eta,\phi;\eta')\Big)}
\exp\Big(-G_{L,\eta'}-G_{L,\eta}\Big)\ud{\eta'}\\
&+\int_{\eta}^{L}\frac{(\tilde\a+S)\Big(\eta',\rr[\phi'](\eta,\phi;\eta')\Big)}{\sin\Big(\phi'(\eta,\phi;\eta')\Big)}\exp\Big(-G_{\eta',\eta}\Big)\ud{\eta'}.\no
\end{align}
\ \\
Region III:\\
For $\sin\phi<0$ and $E(\eta,\phi)\geq \ue^{-V(L)}$,
\begin{align}
\k[p_{\a}]=&p_{\a}\Big(\phi'(\eta,\phi;0)\Big)\exp\Big(-G_{\eta^+,0}-G_{\eta^+,\eta}\Big)\\
\t[\tilde\a+S_{\a}]=&\int_0^{\eta^+}\frac{(\tilde\a+S_{\a})\Big(\eta',\phi'(\eta,\phi;\eta')\Big)}{\sin\Big(\phi'(\eta,\phi;\eta')\Big)}
\exp\Big(-G_{\eta^+,\eta'}-G_{\eta^+,\eta}\Big)\ud{\eta'}\\&+
\int_{\eta}^{\eta^+}\frac{(\tilde\a+S_{\a})\Big(\eta',\rr[\phi'](\eta,\phi;\eta')\Big)}{\sin\Big(\phi'(\eta,\phi;\eta')\Big)}\exp\Big(-G_{\eta',\eta}\Big)\ud{\eta'}.\no
\end{align}
Then we need to estimate $\k[p_{\a}]$ and $\t[\tilde\a+S_{\a}]$ in each region. Assume $0<\d<<1$ and $0<\d_0<<1$ are small quantities which will be determined later.
Since we always assume that $(\eta,\phi)$ and $(\eta',\phi')$ are on the same characteristics, when there is no confusion, we simply write $\phi'$ or $\phi'(\eta')$ instead of $\phi'(\eta,\phi;\eta')$.

\subsection{Region I: $\sin\phi>0$}

We consider
\begin{align}
\k[p_{\a}]=&p_{\a}\Big(\phi'(\eta,\phi;0)\Big)\exp\Big(-G_{\eta,0}\Big)\\
\t[\tilde\a+S_{\a}]=&\int_0^{\eta}\frac{(\tilde\a+S_{\a})\Big(\eta',\phi'(\eta,\phi;\eta')\Big)}{\sin\Big(\phi'(\eta,\phi;\eta')\Big)}\exp\Big(-G_{\eta,\eta'}\Big)\ud{\eta'}.
\end{align}
Based on Lemma \ref{Milne lemma 1} and Lemma \ref{Milne lemma 2},
we can directly obtain
\begin{align}
\abs{\k[p_{\a}]}\leq&\lnmp{p_{\a}},\\
\abs{\t[S_{\a}]}\leq&\lnnm{S_{\a}}.
\end{align}
Hence, we only need to estimate
\begin{align}\label{mild 1}
I=\t[\tilde\a]=\int_0^{\eta}\frac{\tilde\a\Big(\eta',\phi'(\eta,\phi;\eta')\Big)}{\sin\Big(\phi'(\eta,\phi;\eta')\Big)}\exp\Big(-G_{\eta,\eta'}\Big)\ud{\eta'}.
\end{align}
Define a cut-off function $\chi\in C^{\infty}\left[-\dfrac{\pi}{2},\dfrac{\pi}{2}\right]$ satisfying $\chi(\phi_{\ast})\in[0,1]$ and
\begin{align}
\chi(\phi_{\ast})=\left\{
\begin{array}{ll}
1&\text{for}\ \ \abs{\sin\phi_{\ast}}\leq\d,\\
0&\text{for}\ \ \abs{\sin\phi_{\ast}}\geq2\d,
\end{array}
\right.
\end{align}
In the following, we will divide the estimate of $I$ into several cases based on the value of $\sin\phi$, $\cos\phi$, $\sin\phi'$, $\e\eta'$ and $\e(\eta-\eta')$. Let $\id$ denote the indicator function. We write
\begin{align}
I=&\int_0^{\eta}\id_{\{\sin\phi\geq\d_0\}}\id_{\{\cos\phi\geq\d_0\}}+\int_0^{\eta}\id_{\{0\leq\sin\phi\leq\d_0\}}\id_{\{\chi(\phi_{\ast})<1\}}\\
&+\int_0^{\eta}\id_{\{0\leq\sin\phi\leq\d_0\}}\id_{\{\chi(\phi_{\ast})=1\}}\id_{\{ \sqrt{\e\eta'}\geq\sin\phi'\}}\no\\
&+\int_0^{\eta}\id_{\{0\leq\sin\phi\leq\d_0\}}\id_{\{\chi(\phi_{\ast})=1\}}\id_{\{\sqrt{\e\eta'}\leq\sin\phi'\}}\id_{\{\sin^2\phi\leq\e(\eta-\eta')\}}\no\\
&+\int_0^{\eta}\id_{\{0\leq\sin\phi\leq\d_0\}}\id_{\{\chi(\phi_{\ast})=1\}}\id_{\{\sqrt{\e\eta'}\leq\sin\phi'\}}\id_{\{\sin^2\phi\geq\e(\eta-\eta')\}}\no\\
&+\int_0^{\eta}\id_{\{0\leq\cos\phi\leq\d_0\}}\no\\
=&I_1+I_2+I_3+I_4+I_5+I_6.\no
\end{align}
\ \\
Step 0: Preliminaries.\\
Using \eqref{rt 26}, we can directly obtain
\begin{align}\label{pt 01}
\zeta(\eta',\phi',\psi)=&\sqrt{1-\bigg(1-\frac{\e\eta'}{R_1}\bigg)^{2\sin^2\psi}\bigg(1-\frac{\e\eta'}{R_2}\bigg)^{2\cos^2\psi}\cos^2\phi'}
\\
=&\sqrt{1-\bigg(1-\frac{\e\eta'}{R_1}\bigg)^{2\sin^2\psi}\bigg(1-\frac{\e\eta'}{R_2}\bigg)^{2\cos^2\psi}
+\bigg(1-\frac{\e\eta'}{R_1}\bigg)^{2\sin^2\psi}\bigg(1-\frac{\e\eta'}{R_2}\bigg)^{2\cos^2\psi}\sin^2\phi'}
\no\\
\leq&\sqrt{1-\bigg(1-\frac{\e\eta'}{R_1}\bigg)^{2\sin^2\psi}\bigg(1-\frac{\e\eta'}{R_2}\bigg)^{2\cos^2\psi}}
+\sqrt{\bigg(1-\frac{\e\eta'}{R_1}\bigg)^{2\sin^2\psi}\bigg(1-\frac{\e\eta'}{R_2}\bigg)^{2\cos^2\psi}\sin^2\phi'}
\no\\
\leq& C\bigg(\sqrt{\e\eta'}+\sin\phi'\bigg),\no
\end{align}
and
\begin{align}\label{pt 02}
\zeta(\eta',\phi',\psi)\geq\sqrt{1-\bigg(1-\frac{\e\eta'}{R_1}\bigg)^{2\sin^2\psi}\bigg(1-\frac{\e\eta'}{R_2}\bigg)^{2\cos^2\psi}}\geq C\sqrt{\e\eta'}.
\end{align}
Also, we know for $0\leq\eta'\leq\eta$,
\begin{align}
\sin\phi'=&\sqrt{1-\cos^2\phi'}\leq\sqrt{1-\bigg(\frac{R_1-\e\eta}{R_1-\e\eta'}\bigg)^{2\sin^2\psi}\bigg(\frac{R_2-\e\eta}{R_2-\e\eta'}\bigg)^{2\cos^2\psi}\cos^2\phi}\\
=&\sqrt{\sin^2\phi+\Bigg(1-\bigg(1-\frac{\e(\eta-\eta')}{R_1-\e\eta'}\bigg)^{2\sin^2\psi}\bigg(1-\frac{\e(\eta-\eta')}{R_2-\e\eta'}\bigg)^{2\cos^2\psi}\Bigg)\cos^2\phi}.\no
\end{align}
Hence, we have
\begin{align}
\sin\phi\leq\sin\phi'
\leq\sqrt{\sin^2\phi+\e(\eta-\eta')},
\end{align}
which means
\begin{align}
\frac{1}{\sqrt{\sin^2\phi+\e(\eta-\eta')}}\leq\frac{1}{\sin\phi'}
\leq\frac{1}{\sin\phi}.
\end{align}
Therefore,
\begin{align}\label{pt 03}
-\int_{\eta'}^{\eta}\frac{1}{\sin\phi'(y)}\ud{y}\leq& -\int_{\eta'}^{\eta}\frac{1}{\sqrt{\sin^2\phi+\e(\eta-y)}}\ud{y}=\frac{1}{\e}\bigg(\sin\phi-\sqrt{\sin^2\phi+\e(\eta-\eta')}\bigg)\\
=&-\frac{\eta-\eta'}{\sin\phi+\sqrt{\sin^2\phi+\e(\eta-\eta')}}
\leq-\frac{\eta-\eta'}{2\sqrt{\sin^2\phi+\e(\eta-\eta')}}.\no
\end{align}
\ \\
Step 1: Estimate of $I_1$ for $\sin\phi\geq\d_0$ and $\cos\phi\geq\d_0$.\\
In this case, we do not need the mild formulation of $\a$. Instead, using \eqref{weight function}, we directly estimate
\begin{align}
\abs{\a}\leq\abs{\zeta}\abs{\frac{\p\v}{\p\eta}}\leq \abs{\frac{\p\v}{\p\eta}}.
\end{align}
We will estimate $I_1$ based on the characteristics of $\v$ itself instead of the derivative.
Here, we will use two formulations of the equation \eqref{Milne difference problem} along the characteristics:
\\
\ \\
Formulation I: $\eta$ is the principal variable, $\phi=\phi(\eta)$, and the equation can be rewritten as
\begin{align}
\sin\phi\frac{\ud{\v}}{\ud{\eta}}+\v=\bar\v+S.
\end{align}
Formulation II: $\phi$ is the principal variable, $\eta=\eta(\phi)$ and the equation can be rewritten as
\begin{align}
F(\eta)\cos\phi\frac{\ud{\v}}{\ud{\phi}}+\v=\bar\v+ S.
\end{align}
These two formulations are equivalent and can be applied to different regions of the domain.\\
\ \\
We may decompose $\v=\v_1+\v_2$ where $\v_1$ satisfies
\begin{align}\label{Milne difference problem 1}
\left\{
\begin{array}{l}\displaystyle
\sin\phi\frac{\p \v_1}{\p\eta}+F(\eta)\cos\phi\frac{\p
\v_1}{\p\phi}+\v_1=\bar \v,\\\rule{0ex}{1.5em}
\v_1(0,\phi)=p(\phi)\ \ \text{for}\ \ \sin\phi>0,\\\rule{0ex}{1.5em}
\v_1(L,\phi)=\v_1(L,\rr[\phi]),
\end{array}
\right.
\end{align}
and $\v_2$ satisfies
\begin{align}\label{Milne difference problem 2}
\left\{
\begin{array}{l}\displaystyle
\sin\phi\frac{\p \v_2}{\p\eta}+F(\eta)\cos\phi\frac{\p
\v_2}{\p\phi}+\v_2=S,\\\rule{0ex}{1.5em}
\v_2(0,\phi)=0\ \ \text{for}\ \ \sin\phi>0,\\\rule{0ex}{1.5em}
\v_2(L,\phi)=\v_2(L,\rr[\phi]).
\end{array}
\right.
\end{align}
Since $\v\in L^{\infty}$ is well-defined, then $\v_1\in L^{\infty}$ and $\v_2\in L^{\infty}$ are also well-defined.\\
\ \\
Using Formulation I, we rewrite the equation \eqref{Milne difference problem 1} along the characteristics as
\begin{align}\label{rt 01}
\v_1(\eta,\phi)=&\exp\Big(-G_{\eta,0}\Big)\Bigg(p\Big(\phi'(0)\Big)
+\int_0^{\eta}\frac{\bar\v(\eta')}{\sin\Big(\phi'(\eta')\Big)}
\exp\Big(G_{\eta',0}\Big)\ud{\eta'}\Bigg),
\end{align}
where $\Big(\eta',\phi'(\eta')\Big)$, $\Big(0,\phi'(0)\Big)$ and $(\eta,\phi)$ are on the same characteristic with $\sin\phi'\geq0$, and
\begin{align}
G_{t,s}=&\int_{s}^{t}\frac{1}{\sin\Big(\phi'(\xi)\Big)}\ud{\xi}.
\end{align}
Taking $\eta$ derivative on both sides of \eqref{rt 01}, we have
\begin{align}\label{rt 02}
\frac{\p\v_1}{\p\eta}=&X_1+X_2+X_3+X_4+X_5,
\end{align}
where
\begin{align}
X_1=&-\exp\Big(-G_{\eta,0}\Big)\frac{\p G_{\eta,0}}{\p\eta}\Bigg(p\Big(\phi'(0)\Big)
+\int_0^{\eta}\frac{\bar\v(\eta')}{\sin\Big(\phi'(\eta')\Big)}
\exp\Big(G_{\eta',0}\Big)\ud{\eta'}\Bigg),\\
X_2=&\exp\Big(-G_{\eta,0}\Big)\frac{\p p\Big(\phi'(0)\Big)}{\p\eta},\\
X_3=&\frac{\bar\v(\eta)}{\sin\phi},\\
X_4=&-\exp\Big(-G_{\eta,0}\Big)\int_0^{\eta}\bar\v(\eta')
\exp\Big(G_{\eta',0}\Big)
\frac{\cos\Big(\phi'(\eta')\Big)}{\sin^2\Big(\phi'(\eta')\Big)}\frac{\p\phi'(\eta')}{\p\eta}\ud{\eta'},\\
X_5=&\exp\Big(-G_{\eta,0}\Big)\int_0^{\eta}\frac{\bar\v(\eta')}{\sin\Big(\phi'(\eta')\Big)}
\exp\Big(G_{\eta',0}\Big)\frac{\p G_{\eta',0}}{\p\eta}\ud{\eta'}.
\end{align}
Then we need to estimate each term. This procedure is standard, so we omit the details. Note the fact that for $0\leq\eta'\leq\eta$, we have $\sin\phi'\geq\sin\phi\geq\d_0$ and
\begin{align}
\int_0^{\eta}\frac{1}{\sin\Big(\phi'(\eta')\Big)}
\exp\left(-G_{\eta,\eta'}\right)\ud{\eta'}\leq \int_0^{\infty}\ue^{-y}\ud{y}=1,
\end{align}
with the substitution $y=G_{\eta,\eta'}$. The estimates can be listed as below:
\begin{align}
\abs{X_1}=&\abs{\frac{\p G_{\eta,0}}{\p\eta}}\abs{\v_1}\leq\frac{C}{\d_0}\lnnm{\v},\\
\abs{X_2}\leq&\abs{\frac{\p p}{\p\phi}\Big(\phi'(0)\Big)}\abs{\frac{\p\phi'(0)}{\p\eta}}\leq\frac{C\e}{\d_0}\lnmp{\frac{\p p}{\p\phi}},\\
\abs{X_3}\leq&\frac{1}{\d_0}\lnnm{\v},\\
\abs{X_4}\leq&\abs{\bar\v(\eta')}\abs{\frac{\cos\Big(\phi'(\eta')\Big)}{\sin\Big(\phi'(\eta')\Big)}}\abs{\frac{\p\phi'(\eta')}{\p\eta}}\leq\frac{C\e}{\d_0^2}\lnnm{\v},\\
\abs{X_5}\leq&\abs{\bar\v(\eta')}\abs{\frac{\p G_{\eta',0}}{\p\eta}}\leq\frac{C\e}{\d_0^3}\lnnm{\v}.
\end{align}
In total, we have
\begin{align}\label{rt 03}
\abs{\frac{\p\v_1}{\p\eta}}\leq C\left(\frac{1}{\d_0}+\frac{\e}{\d_0^3}\right)\lnnm{\v}+\frac{C\e}{\d_0}\lnmp{\frac{\p p}{\p\phi}}.
\end{align}
\ \\
Using Formulation II, we rewrite the equation \eqref{Milne difference problem 2} along the characteristics as
\begin{align}\label{rt 01'}
\v_2(\eta,\phi)=&\exp\left(-H_{\phi,\phi_{\ast}}\right)\int_{\phi_{\ast}}^{\phi}\frac{S\Big(\eta'(\phi'),\phi'\Big)}{F\Big(\eta'(\phi')\Big)\cos\phi'}
\exp\left(H_{\phi',\phi_{\ast}}\right)\ud{\phi'},
\end{align}
where $\Big(\eta'(\phi'),\phi'\Big)$, $\Big(0,\phi_{\ast}\Big)$ and $(\eta,\phi)$ are on the same characteristic with $\sin\phi\geq\d_0$, and
\begin{align}
H_{t,s}=&\int_{s}^{t}\frac{1}{F\Big(\eta'(\xi)\Big)\cos\xi}\ud{\xi}.
\end{align}
Taking $\eta$ derivative on both sides of \eqref{rt 01'}, we have
\begin{align}\label{rt 02'}
\frac{\p\v_2}{\p\eta}=&Y_1+Y_2+Y_3+Y_4+Y_5,
\end{align}
where
\begin{align}
Y_1=&-\exp\left(-H_{\phi,\phi_{\ast}}\right)\frac{\p H_{\phi,\phi_{\ast}}}{\p\eta}\int_{\phi_{\ast}}^{\phi}\frac{S\Big(\eta'(\phi'),\phi'\Big)}{F\Big(\eta'(\phi')\Big)\cos\phi'}
\exp\left(H_{\phi',\phi_{\ast}}\right)\ud{\phi'},\\
Y_2=&-\frac{S(0,\phi_{\ast})}{F(0)\cos(\phi_{\ast})}\frac{\p\phi_{\ast}}{\p\eta},\\
Y_3=&-\exp\left(-H_{\phi,\phi_{\ast}}\right)\int_{\phi_{\ast}}^{\phi}S\Big(\eta'(\phi'),\phi'\Big)\frac{1}{F^2\Big(\eta'(\phi')\Big)\cos\phi'}\frac{\p F\Big(\eta'(\phi')\Big)}{\p\eta}
\exp\left(H_{\phi',\phi_{\ast}}\right)\ud{\phi'},\\
Y_4=&\exp\left(-H_{\phi,\phi_{\ast}}\right)\int_{\phi_{\ast}}^{\phi}\frac{S\Big(\eta'(\phi'),\phi'\Big)}{F\Big(\eta'(\phi')\Big)\cos\phi'}
\exp\left(H_{\phi',\phi_{\ast}}\right)\frac{\p H_{\phi',\phi_{\ast}}}{\p\eta}\ud{\phi'},\\
Y_5=&\exp\left(-H_{\phi,\phi_{\ast}}\right)\int_{\phi_{\ast}}^{\phi}\frac{\p_{\eta'}S\Big(\eta'(\phi'),\phi'\Big)}{F\Big(\eta'(\phi')\Big)\cos\phi'}\frac{\p\eta'(\phi')}{\p\eta}
\exp\left(H_{\phi',\phi_{\ast}}\right)\ud{\phi'}.
\end{align}
Then we need to estimate each term. Along the characteristics, for $0\leq\eta'\leq\eta$, we know
\begin{align}
\ue^{-V(\eta')}\cos\phi'=\ue^{-V(\eta)}\cos(\phi),
\end{align}
which implies
\begin{align}
\cos\phi'=&\ue^{V(\eta')-V(\eta)}\cos\phi\geq \ue^{V(0)-V(L)}\cos\phi\geq \ue^{V(0)-V(L)}\d_0.
\end{align}
Using \eqref{potential}, we can further deduce that
\begin{align}
\cos\phi'\geq\Big(1-C\e^{\frac{1}{2}}\Big)\d_0\geq \frac{\d_0}{2},
\end{align}
when $\e$ is sufficiently small. Also, we have
\begin{align}
\int_{\phi_{\ast}}^{\phi}\frac{1}{F\Big(\eta'(\phi')\Big)\cos\phi'}
\exp\left(H_{\phi,\phi'}\right)\ud{\phi'}\leq \int_0^{\infty}\ue^{-y}\ud{y}=1,
\end{align}
with the substitution $y=H_{\phi,\phi'}$. Similar to $X_i$ estimates, using standard argument, we may obtain
\begin{align}
\abs{Y_1}\leq&\abs{\v_2}\abs{\frac{\p H_{\phi,\phi_{\ast}}}{\p\eta}}\leq\frac{C}{\d_0^2}\lnnm{S},\\
\abs{Y_2}\leq&\frac{C}{\d_0}\abs{S(0,\phi_{\ast})}\abs{\frac{1}{F(0)}}\abs{\frac{\p\phi_{\ast}}{\p\eta}}\leq\frac{C}{\d_0^2}\lnnm{S},\\
\abs{Y_3}\leq&\abs{S\Big(\eta'(\phi'),\phi'\Big)}\abs{\frac{1}{F\Big(\eta'(\phi')\Big)}}\abs{\frac{\p F\Big(\eta'(\phi')\Big)}{\p\eta}}\leq\frac{C\e}{\d_0}\lnnm{S},\\
\abs{Y_4}\leq&\abs{S\Big(\eta'(\phi'),\phi'\Big)}\abs{\frac{\p H_{\phi',\phi_{\ast}}}{\p\eta}}\leq\frac{C}{\d_0^2}\lnnm{S},\\
\abs{Y_5}\leq&\abs{\p_{\eta'}S\Big(\eta'(\phi'),\phi'\Big)}\abs{\frac{\p\eta'(\phi')}{\p\eta}}\leq C\lnnm{\frac{\p S}{\p\eta}}.
\end{align}
In total, we have
\begin{align}\label{rt 04}
\abs{\frac{\p\v_2}{\p\eta}}\leq \frac{C}{\d_0^2}\lnnm{S}+C\lnnm{\frac{\p S}{\p\eta}}.
\end{align}
Combining \eqref{rt 03} and \eqref{rt 04}, we have
\begin{align}
\abs{\frac{\p\v}{\p\eta}}\leq&C\Bigg(\left(\frac{1}{\d_0}+\frac{\e}{\d_0^3}\right)\lnnm{\v}+\frac{\e}{\d_0}\lnmp{\frac{\p p}{\p\phi}}+\frac{1}{\d_0^2}\lnnm{S}+\lnnm{\frac{\p S}{\p\eta}}\Bigg).
\end{align}
Hence, noting that $\zeta(\eta,\phi)\geq\sin\phi\geq\d_0$, we know
\begin{align}\label{rt 07}
\abs{I_1}\leq&C\Bigg(\left(\frac{1}{\d_0}+\frac{\e}{\d_0^3}\right)\lnnm{\v}+\frac{\e}{\d_0}\lnmp{\frac{\p p}{\p\phi}}+\frac{1}{\d_0^2}\lnnm{S}+\frac{1}{\d_0}\lnnm{\zeta\frac{\p S}{\p\eta}}\Bigg).
\end{align}
\ \\
Step 2: Estimate of $I_2$ for $0\leq\sin\phi\leq\d_0$ and $\chi(\phi_{\ast})<1$.\\
We have
\begin{align}\label{rt 05}
\\
I_2=&\frac{1}{4\pi}\int_0^{\eta}\bigg(\int_{-\pi}^{\pi}\int_{-\frac{\pi}{2}}^{\frac{\pi}{2}}\frac{\zeta(\eta',\phi',\psi)}{\zeta(\eta',\phi_{\ast},\psi_{\ast})}
\Big(1-\chi(\phi_{\ast})\Big)
\a(\eta',\phi_{\ast},\psi_{\ast})\cos\phi_{\ast}\ud{\phi_{\ast}}\ud{\psi_{\ast}}\bigg)
\frac{1}{\sin\phi'}\exp\Big(-G_{\eta,\eta'}\Big)\ud{\eta'}\no\\
=&\frac{1}{4\pi}\int_0^{\eta}\bigg(\int_{-\pi}^{\pi}\int_{-\frac{\pi}{2}}^{\frac{\pi}{2}}\zeta(\eta',\phi',\psi)\Big(1-\chi(\phi_{\ast})\Big)
\frac{\p\v(\eta',\phi_{\ast},\psi_{\ast})}{\p\eta'}\cos\phi_{\ast}\ud{\phi_{\ast}}\ud{\psi_{\ast}}\bigg)\frac{1}{\sin\phi'}\exp\Big(-G_{\eta,\eta'}\Big)\ud{\eta'}.\no
\end{align}
Based on the equation \eqref{Milne difference problem} of $\v$
\begin{align}
\sin\phi_{\ast}\frac{\p\v(\eta',\phi_{\ast},\psi_{\ast})}{\p\eta'}+F(\eta',\psi_{\ast})\cos\phi_{\ast}\frac{\p\v(\eta',\phi_{\ast},\psi_{\ast})}{\p\phi_{\ast}}
+\v(\eta',\phi_{\ast},\psi_{\ast})-\bar\v(\eta')=S(\eta',\phi_{\ast},\psi_{\ast}),
\end{align}
we have
\begin{align}\label{rt 06}
\\
\frac{\p\v(\eta',\phi_{\ast},\psi_{\ast})}{\p\eta'}=-\frac{1}{\sin\phi_{\ast}}
\bigg(F(\eta',\psi_{\ast})\cos\phi_{\ast}\frac{\p\v(\eta',\phi_{\ast},\psi_{\ast})}{\p\phi_{\ast}}
+\v(\eta',\phi_{\ast},\psi_{\ast})-\bar\v(\eta')-S(\eta',\phi_{\ast},\psi_{\ast})\bigg).\no
\end{align}
Hence, inserting \eqref{rt 06} into \eqref{rt 05}, we have the quantity in the large paranthesis
\begin{align}
\tilde\a:=&\int_{-\pi}^{\pi}\int_{-\frac{\pi}{2}}^{\frac{\pi}{2}}\zeta(\eta',\phi',\psi')\Big(1-\chi(\phi_{\ast})\Big)
\frac{\p\v(\eta',\phi_{\ast},\psi_{\ast})}{\p\eta'}\cos\phi_{\ast}\ud{\phi_{\ast}}\ud{\psi_{\ast}}\\
=&-\int_{-\pi}^{\pi}\int_{-\frac{\pi}{2}}^{\frac{\pi}{2}}\zeta(\eta',\phi',\psi')\Big(1-\chi(\phi_{\ast})\Big)
\frac{1}{\sin\phi_{\ast}}
\bigg(\v(\eta',\phi_{\ast},\psi_{\ast})-\bar\v(\eta')-S(\eta',\phi_{\ast},\psi_{\ast})\bigg)\cos\phi_{\ast}\ud{\phi_{\ast}}\ud{\psi_{\ast}}\no\\
&-\int_{-\pi}^{\pi}\int_{-\frac{\pi}{2}}^{\frac{\pi}{2}}\zeta(\eta',\phi',\psi')\Big(1-\chi(\phi_{\ast})\Big)
\frac{1}{\sin\phi_{\ast}}
\bigg(F(\eta',\psi_{\ast})\cos\phi_{\ast}\frac{\p\v(\eta',\phi_{\ast},\psi_{\ast})}{\p\phi_{\ast}}\bigg)\cos\phi_{\ast}\ud{\phi_{\ast}}\ud{\psi_{\ast}}\no\\
:=&\tilde\a_1+\tilde\a_2.\no
\end{align}
We may directly obtain
\begin{align}
\abs{\tilde\a_1}\leq&\int_{-\pi}^{\pi}\int_{-\frac{\pi}{2}}^{\frac{\pi}{2}}
\frac{1}{\sin\phi_{\ast}}
\abs{\v(\eta',\phi_{\ast},\psi_{\ast})-\bar\v(\eta')-S(\eta',\phi_{\ast},\psi_{\ast})}\ud{\phi_{\ast}}\ud{\psi_{\ast}}\\
\leq&\frac{1}{\d}\int_{-\pi}^{\pi}\int_{-\frac{\pi}{2}}^{\frac{\pi}{2}}
\bigg(\abs{\v(\eta',\phi_{\ast},\psi_{\ast})}+\abs{\bar\v(\eta')}+\abs{S(\eta',\phi_{\ast},\psi_{\ast})}\bigg)\ud{\phi_{\ast}}\ud{\psi_{\ast}}\no\\
\leq&\frac{C}{\d}\bigg(\lnnm{\v}+\lnnm{S}\bigg).\no
\end{align}
On the other hand, integration by parts yields
\begin{align}
\tilde\a_2=&\int_{-\pi}^{\pi}\int_{-\frac{\pi}{2}}^{\frac{\pi}{2}}\frac{\p}{\p\phi_{\ast}}\bigg(\zeta(\eta',\phi',\psi')\Big(1-\chi(\phi_{\ast})\Big)
\frac{1}{\sin\phi_{\ast}}
F(\eta',\psi_{\ast})\cos^2\phi_{\ast}\bigg)\v(\eta',\phi_{\ast},\psi_{\ast})\ud{\phi_{\ast}}\ud{\psi_{\ast}},
\end{align}
which further implies
\begin{align}
\abs{\tilde\a_2}\leq&\frac{C\e}{\d^2}\lnnm{\v}.
\end{align}
Since we can use substitution to show
\begin{align}
\int_0^{\eta}\frac{1}{\sin\phi'}\exp\Big(-G_{\eta,\eta'}\Big)\ud{\eta'}\leq 1,
\end{align}
we have
\begin{align}\label{rt 08}
\abs{I_2}\leq&C\bigg(\frac{1}{\d}+\frac{\e}{\d^2}\bigg)\bigg(\lnnm{\v}+\lnnm{S}\bigg)\int_0^{\eta}\frac{1}{\sin\phi'}\exp\Big(-G_{\eta,\eta'}\Big)\ud{\eta'}\\
\leq&C\bigg(\frac{1}{\d}+\frac{\e}{\d^2}\bigg)\bigg(\lnnm{\v}+\lnnm{S}\bigg).\no
\end{align}
\ \\
Step 3: Estimate of $I_3$ for $0\leq\sin\phi\leq\d_0$, $\chi(\phi_{\ast})=1$ and $\sqrt{\e\eta'}\geq\sin\phi'$.\\
Based on \eqref{pt 01}, this implies
\begin{align}
\zeta(\eta',\phi',\psi)\leq C\sqrt{\e\eta'}.\no
\end{align}
Then combining this with \eqref{pt 02},
we know
\begin{align}
\abs{\frac{\zeta(\eta',\phi',\psi)}{\zeta(\eta',\phi_{\ast},\psi_{\ast})}}\leq C,
\end{align}
which yields
\begin{align}
\\
\int_{-\pi}^{\pi}\int_{-\frac{\pi}{2}}^{\frac{\pi}{2}}\frac{\zeta(\eta',\phi',\psi)}{\zeta(\eta',\phi_{\ast},\psi_{\ast})}\chi(\phi_{\ast})
\a(\eta',\phi_{\ast},\psi_{\ast})\cos\phi_{\ast}\ud{\phi_{\ast}}\ud{\psi_{\ast}}\leq&C\int_{-\pi}^{\pi}\int_{-\d}^{\d}
\a(\eta',\phi_{\ast},\psi_{\ast})\ud{\phi_{\ast}}\ud{\psi_{\ast}}\leq C\d\lnnm{\a}.\no
\end{align}
Hence, we have
\begin{align}\label{rt 09}
\abs{I_3}\leq&C\d\lnnm{\a}\int_0^{\eta}\frac{1}{\sin\phi'}\exp\Big(-G_{\eta,\eta'}\Big)\ud{\eta'}\leq C\d\lnnm{\a}.
\end{align}
\ \\
Step 4: Estimate of $I_4$ for $0\leq\sin\phi\leq\d_0$, $\chi(\phi_{\ast})=1$, $\sqrt{\e\eta'}\leq\sin\phi'$ and $\sin^2\phi\leq\e(\eta-\eta')$.\\
Based on \eqref{pt 01}, this implies
\begin{align}
\zeta(\eta',\phi',\psi)\leq C\sin\phi'.
\end{align}
Based on \eqref{pt 03} and using $\sin^2\phi\leq\e(\eta-\eta')$, we have
\begin{align}
-G_{\eta,\eta'}=-\int_{\eta'}^{\eta}\frac{1}{\sin\phi'(y)}\ud{y}\leq -\frac{\eta-\eta'}{2\sqrt{\sin^2\phi+\e(\eta-\eta')}}\leq&-\frac{\eta-\eta'}{2\sqrt{\e(\eta-\eta')}}\leq-C\sqrt{\frac{\eta-\eta'}{\e}}.
\end{align}
Hence, considering \eqref{pt 02}, i.e. $\zeta(\eta',\phi_{\ast},\psi_{\ast})\geq\sqrt{\e\eta'}$, we know
\begin{align}
\abs{I_4}\leq&C\int_0^{\eta}\bigg(\int_{-\pi}^{\pi}\int_{-\frac{\pi}{2}}^{\frac{\pi}{2}}\frac{\zeta(\eta',\phi',\psi)}{\zeta(\eta',\phi_{\ast},\psi_{\ast})}\chi(\phi_{\ast})
\a(\eta',\phi_{\ast},\psi_{\ast})\cos\phi_{\ast}\ud{\phi_{\ast}}\ud{\psi_{\ast}}\bigg)
\frac{1}{\sin\phi'}\exp\Big(-G_{\eta,\eta'}\Big)\ud{\eta'}\\
\leq&C\int_0^{\eta}\bigg(\int_{-\pi}^{\pi}\int_{-\d}^{\d}\frac{1}{\zeta(\eta',\phi_{\ast},\psi_{\ast})}
\a(\eta',\phi_{\ast},\psi_{\ast})\ud{\phi_{\ast}}\ud{\psi_{\ast}}\bigg)
\frac{\zeta(\eta',\phi',\psi)}{\sin\phi'}\exp\Big(-G_{\eta,\eta'}\Big)\ud{\eta'}\no\\
\leq&C\lnnm{\a}\int_0^{\eta}\bigg(\int_{-\pi}^{\pi}\int_{-\d}^{\d}\frac{1}{\zeta(\eta',\phi_{\ast},\psi_{\ast})}
\ud{\phi_{\ast}}\ud{\psi_{\ast}}\bigg)
\frac{\sin\phi'}{\sin\phi'}\exp\Big(-G_{\eta,\eta'}\Big)\ud{\eta'}\no\\
\leq&C\d\lnnm{\a}\int_0^{\eta}\frac{1}{\sqrt{\e\eta'}}\exp\Big(-G_{\eta,\eta'}\Big)\ud{\eta'}\no\\
\leq&C\d\lnnm{\a}\int_0^{\eta}\frac{1}{\sqrt{\e\eta'}}\exp\bigg(-C\sqrt{\frac{\eta-\eta'}{\e}}\bigg)\ud{\eta'}\no
\end{align}
Define $z=\dfrac{\eta'}{\e}$, which implies $\ud{\eta'}=\e\ud{z}$. Substituting this into above integral, we have
\begin{align}
\abs{I_4}\leq&C\d\lnnm{\a}\int_0^{\frac{\eta}{\e}}\frac{1}{\sqrt{z}}\exp\bigg(-C\sqrt{\frac{\eta}{\e}-z}\bigg)\ud{z}\\
=&C\d\lnnm{\a}\Bigg(\int_0^{1}\frac{1}{\sqrt{z}}\exp\bigg(-C\sqrt{\frac{\eta}{\e}-z}\bigg)\ud{z}
+\int_1^{\frac{\eta}{\e}}\frac{1}{\sqrt{z}}\exp\bigg(-C\sqrt{\frac{\eta}{\e}-z}\bigg)\ud{z}\Bigg).\no
\end{align}
We can estimate these two terms separately.
\begin{align}
\int_0^{1}\frac{1}{\sqrt{z}}\exp\bigg(-C\sqrt{\frac{\eta}{\e}-z}\bigg)\ud{z}\leq&\int_0^{1}\frac{1}{\sqrt{z}}\ud{z}=2.
\end{align}
\begin{align}
\int_1^{\frac{\eta}{\e}}\frac{1}{\sqrt{z}}\exp\bigg(-C\sqrt{\frac{\eta}{\e}-z}\bigg)\ud{z}\leq&\int_1^{\frac{\eta}{\e}}\exp\bigg(-C\sqrt{\frac{\eta}{\e}-z}\bigg)\ud{z}
\overset{t^2=\frac{\eta}{\e}-z}{\leq}2\int_0^{\infty}t\ue^{-Ct}\ud{t}<\infty.
\end{align}
Hence, we know
\begin{align}\label{rt 10}
\abs{I_4}\leq&C\d\lnnm{\a}.
\end{align}
\ \\
Step 5: Estimate of $I_5$ for $0\leq\sin\phi\leq\d_0$, $\chi(\phi_{\ast})=1$, $\sqrt{\e\eta'}\leq\sin\phi'$ and $\sin^2\phi\geq\e(\eta-\eta')$.\\
Based on \eqref{pt 01}, this implies
\begin{align}
\zeta(\eta',\phi',\psi)\leq C\sin\phi'.
\end{align}
Based on \eqref{pt 03}, we have
\begin{align}
-G_{\eta,\eta'}=-\int_{\eta'}^{\eta}\frac{1}{\sin\phi'(y)}\ud{y}\leq -\frac{\eta-\eta'}{2\sqrt{\sin^2\phi+\e(\eta-\eta')}}\leq&-\frac{C(\eta-\eta')}{\sin\phi}.
\end{align}
Hence, we have
\begin{align}
\abs{I_5}\leq& C\lnnm{\a}\int_0^{\eta}\bigg(\int_{-\pi}^{\pi}\int_{-\d}^{\d}\frac{1}{\zeta(\eta',\phi_{\ast},\psi_{\ast})}
\ud{\phi_{\ast}}\ud{\psi_{\ast}}\bigg)
\exp\left(-\frac{C(\eta-\eta')}{\sin\phi}\right)\ud{\eta'}.
\end{align}
Here, we use a different way to estimate the inner integral. We use substitution to find
\begin{align}
&\int_{-\d}^{\d}\frac{1}{\zeta(\eta',\phi_{\ast},\psi_{\ast})}
\ud{\phi_{\ast}}\\
=&\int_{-\d}^{\d}\frac{1}{\Bigg(1-\bigg(1-\dfrac{\e\eta'}{R_1}\bigg)^{2\sin^2\psi_{\ast}}\bigg(1-\dfrac{\e\eta'}{R_2}\bigg)^{2\cos^2\psi_{\ast}}
\cos^2\phi_{\ast}\Bigg)^{\frac{1}{2}}}
\ud{\phi_{\ast}}\no\\
\overset{\sin\phi_{\ast}\ small}{\leq}&C\int_{-\d}^{\d}\frac{\cos\phi_{\ast}}{\Bigg(1-\bigg(1-\dfrac{\e\eta'}{R_1}\bigg)^{2\sin^2\psi_{\ast}}\bigg(1-\dfrac{\e\eta'}{R_2}\bigg)^{2\cos^2\psi_{\ast}}
\cos^2\phi_{\ast}\Bigg)^{\frac{1}{2}}}
\ud{\phi_{\ast}}\no\\
=&C\int_{-\d}^{\d}\frac{\cos\phi_{\ast}}{\Bigg(1-\bigg(1-\dfrac{\e\eta'}{R_1}\bigg)^{2\sin^2\psi_{\ast}}\bigg(1-\dfrac{\e\eta'}{R_2}\bigg)^{2\cos^2\psi_{\ast}}
+\bigg(1-\dfrac{\e\eta'}{R_1}\bigg)^{2\sin^2\psi_{\ast}}\bigg(1-\dfrac{\e\eta'}{R_2}\bigg)^{2\cos^2\psi_{\ast}}\sin\phi_{\ast}^2\Bigg)^{\frac{1}{2}}}
\ud{\phi_{\ast}}\no\\
\overset{y=\sin\phi_{\ast}}{=}&C\int_{-\d}^{\d}\frac{1}{\Bigg(1-\bigg(1-\dfrac{\e\eta'}{R_1}\bigg)^{2\sin^2\psi_{\ast}}\bigg(1-\dfrac{\e\eta'}{R_2}\bigg)^{2\cos^2\psi_{\ast}}
+\bigg(1-\dfrac{\e\eta'}{R_1}\bigg)^{2\sin^2\psi_{\ast}}\bigg(1-\dfrac{\e\eta'}{R_2}\bigg)^{2\cos^2\psi_{\ast}}y^2\Bigg)^{\frac{1}{2}}}
\ud{y}.\no
\end{align}
Define
\begin{align}
p=&\sqrt{1-\bigg(1-\dfrac{\e\eta'}{R_1}\bigg)^{2\sin^2\psi_{\ast}}\bigg(1-\dfrac{\e\eta'}{R_2}\bigg)^{2\cos^2\psi_{\ast}}}\leq C\sqrt{\e\eta'},\\
q=&\bigg(1-\dfrac{\e\eta'}{R_1}\bigg)^{\sin^2\psi_{\ast}}\bigg(1-\dfrac{\e\eta'}{R_2}\bigg)^{\cos^2\psi_{\ast}}\geq C,\\
r=&\frac{p}{q}\leq C\sqrt{\e\eta'}.
\end{align}
Then we have
\begin{align}
\int_{-\d}^{\d}\frac{1}{\zeta(\eta',\phi_{\ast},\psi_{\ast})}\ud{\phi_{\ast}}\leq&C\int_{-\d}^{\d}\frac{1}{(p^2+q^2y^2)^{\frac{1}{2}}}\ud{y}\\
\leq&C\int_{-2}^{2}\frac{1}{(p^2+q^2y^2)^{\frac{1}{2}}}\ud{y}\leq C\int_{-2}^{2}\frac{1}{(r^2+y^2)^{\frac{1}{2}}}\ud{y}\no\\
\leq&C\int_{0}^{2}\frac{1}{(r^2+y^2)^{\frac{1}{2}}}\ud{y}=\bigg(\ln(y+\sqrt{r^2+y^2})-\ln(r)\bigg)\bigg|_0^{2}\no\\
\leq&C\bigg(\ln(2+\sqrt{r^2+4})-\ln{r}\bigg)\leq C\bigg(1+\ln(r)\bigg)\no\\
\leq&C\bigg(1+\abs{\ln(\e)}+\abs{\ln(\eta')}\bigg).\no
\end{align}
Hence, we know
\begin{align}
\abs{I_5}\leq&C\lnnm{\a}\int_0^{\eta}\bigg(1+\abs{\ln(\e)}+\abs{\ln(\eta')}\bigg)
\exp\left(-\frac{C(\eta-\eta')}{\sin\phi}\right)\ud{\eta'}
\end{align}
We may directly compute
\begin{align}
\abs{\int_0^{\eta}\bigg(1+\abs{\ln(\e)}\bigg)
\exp\left(-\frac{C(\eta-\eta')}{\sin\phi}\right)\ud{\eta'}}\leq C\sin\phi(1+\abs{\ln(\e)}).
\end{align}
Hence, we only need to estimate
\begin{align}
\abs{\int_0^{\eta}\abs{\ln(\eta')}
\exp\left(-\frac{C(\eta-\eta')}{\sin\phi}\right)\ud{\eta'}}.
\end{align}
If $\eta\leq 2$, using Cauchy's inequality, we have
\begin{align}
\abs{\int_0^{\eta}\abs{\ln(\eta')}
\exp\left(-\frac{C(\eta-\eta')}{\sin\phi}\right)\ud{\eta'}}
\leq&\bigg(\int_0^{\eta}\ln^2(\eta')\ud{\eta'}\bigg)^{\frac{1}{2}}\bigg(\int_0^{\eta}
\exp\left(-\frac{2C(\eta-\eta')}{\sin\phi}\right)\ud{\eta'}\bigg)^{\frac{1}{2}}\\
\leq&\bigg(\int_0^{2}\ln^2(\eta')\ud{\eta'}\bigg)^{\frac{1}{2}}\bigg(\int_0^{\eta}
\exp\left(-\frac{2C(\eta-\eta')}{\sin\phi}\right)\ud{\eta'}\bigg)^{\frac{1}{2}}\no\\
\leq&\sqrt{\sin\phi}.\no
\end{align}
If $\eta\geq 2$, we decompose and apply Cauchy's inequality to obtain
\begin{align}
&\abs{\int_0^{\eta}\abs{\ln(\eta')}
\exp\left(-\frac{C(\eta-\eta')}{\sin\phi}\right)\ud{\eta'}}\\
\leq&\abs{\int_0^{2}\abs{\ln(\eta')}
\exp\left(-\frac{C(\eta-\eta')}{\sin\phi}\right)\ud{\eta'}}+\abs{\int_2^{\eta}\ln(\eta')
\exp\left(-\frac{C(\eta-\eta')}{\sin\phi}\right)\ud{\eta'}}\no\\
\leq&\bigg(\int_0^{2}\ln^2(\eta')\ud{\eta'}\bigg)^{\frac{1}{2}}\bigg(\int_0^{2}
\exp\left(-\frac{2C(\eta-\eta')}{\sin\phi}\right)\ud{\eta'}\bigg)^{\frac{1}{2}}+\ln(L)\abs{\int_2^{\eta}
\exp\left(-\frac{C(\eta-\eta')}{\sin\phi}\right)\ud{\eta'}}\no\\
\leq&C\bigg(\sqrt{\sin\phi}+\abs{\ln(\e)}\sin\phi\bigg)\leq C\Big(1+\abs{\ln(\e)}\Big)\sqrt{\sin\phi}.\no
\end{align}
Hence, we have
\begin{align}\label{rt 11}
\abs{I_5}\leq C\Big(1+\abs{\ln(\e)}\Big)\sqrt{\d_0}\lnnm{\a}.
\end{align}
\ \\
Step 6: Estimate of $I_6$ for $\cos\phi<\d_0$.\\
We have
\begin{align}\label{rt 21}
I_6=&\frac{1}{4\pi}\int_0^{\eta}\bigg(\int_{-\pi}^{\pi}\int_{-\frac{\pi}{2}}^{\frac{\pi}{2}}\frac{\zeta(\eta',\phi',\psi)}{\zeta(\eta',\phi_{\ast},\psi_{\ast})}
\a(\eta',\phi_{\ast},\psi_{\ast})\cos\phi_{\ast}\ud{\phi_{\ast}}\ud\psi_{\ast}\bigg)
\frac{1}{\sin\phi'}\exp\Big(-G_{\eta,\eta'}\Big)\ud{\eta'}\\
=&\frac{1}{4\pi}\int_0^{\eta}\bigg(\int_{-\pi}^{\pi}\int_{-\frac{\pi}{2}}^{\frac{\pi}{2}}\zeta(\eta',\phi',\psi)\Big(1-\chi(\phi_{\ast})\Big)
\frac{\v(\eta',\phi_{\ast},\psi_{\ast})}{\p\eta'}\cos\phi_{\ast}\ud{\phi_{\ast}}\ud\psi_{\ast}\bigg)\frac{1}{\sin\phi'}\exp\Big(-G_{\eta,\eta'}\Big)\ud{\eta'}\no\\
&+\frac{1}{4\pi}\int_0^{\eta}\bigg(\int_{-\pi}^{\pi}\int_{-\frac{\pi}{2}}^{\frac{\pi}{2}}\chi(\phi_{\ast})
\frac{\zeta(\eta',\phi',\psi)}{\zeta(\eta',\phi_{\ast},\psi_{\ast})}
\a(\eta',\phi_{\ast},\psi_{\ast})\cos\phi_{\ast}\ud{\phi_{\ast}}\ud\psi_{\ast}\bigg)\frac{1}{\sin\phi'}\exp\Big(-G_{\eta,\eta'}\Big)\ud{\eta'}.\no
\end{align}
The first term in \eqref{rt 21} can be estimated as $I_2$.
\begin{align}
&\frac{1}{4\pi}\int_0^{\eta}\bigg(\int_{-\pi}^{\pi}\int_{-\frac{\pi}{2}}^{\frac{\pi}{2}}\zeta(\eta',\phi',\psi')\Big(1-\chi(\phi_{\ast})\Big)
\frac{\v(\eta',\phi_{\ast},\psi_{\ast})}{\p\eta'}\cos\phi_{\ast}\ud{\phi_{\ast}}\ud\psi_{\ast}\bigg)\frac{1}{\sin\phi'}\exp\Big(-G_{\eta,\eta'}\Big)\ud{\eta'}\\
\leq&C\bigg(\frac{1}{\d}+\frac{\e}{\d^2}\bigg)\bigg(\lnnm{\v}+\lnnm{S}\bigg).\no
\end{align}
It is easy to check that $\sqrt{\e\eta'}\leq\sin\phi\leq\sin\phi'$ and $\sin^2\phi\geq\e(\eta-\eta')$, so the second term in \eqref{rt 21} can be estimated as $I_5$.
\begin{align}
&\frac{1}{4\pi}\int_0^{\eta}\bigg(\int_{-\pi}^{\pi}\int_{-\frac{\pi}{2}}^{\frac{\pi}{2}}\chi(\phi_{\ast})
\frac{\zeta(\eta',\phi',\psi')}{\zeta(\eta',\phi_{\ast},\psi_{\ast})}
\a(\eta',\phi_{\ast},\psi_{\ast})\cos\phi_{\ast}\ud{\phi_{\ast}}\ud\psi_{\ast}\bigg)\frac{1}{\sin\phi'}\exp\Big(-G_{\eta,\eta'}\Big)\ud{\eta'}\\
\leq&C\Big(1+\abs{\ln(\e)}\Big)\sqrt{\sin\phi}\sup_{\sin\phi_{\ast}\leq\d}\abs{\a(\eta,\phi_{\ast},\psi_{\ast})}\leq C\Big(1+\abs{\ln(\e)}\Big)\sup_{\sin\phi_{\ast}\leq\d}\abs{\a(\eta,\phi_{\ast},\psi_{\ast})}.\no
\end{align}
Note that now we lose the smallness since $\sin\phi\geq\dfrac{1}{2}$, so we need a more detailed analysis. Actually, the value of $\abs{\a}$ for $\sin\phi_{\ast}\leq\d$, is covered in $I_2,I_3,I_4,I_5$ and the following $II_2,II_3,II_4,III$. Therefore, in fact, we get the estimate
\begin{align}
&\frac{1}{4\pi}\int_0^{\eta}\bigg(\int_{-\pi}^{\pi}\int_{-\frac{\pi}{2}}^{\frac{\pi}{2}}\chi(\phi_{\ast})
\frac{\zeta(\eta',\phi',\psi')}{\zeta(\eta',\phi_{\ast},\psi_{\ast})}
\a(\eta',\phi_{\ast},\psi_{\ast})\cos\phi_{\ast}\ud{\phi_{\ast}}\ud\psi_{\ast}\bigg)\frac{1}{\sin\phi'}\exp\Big(-G_{\eta,\eta'}\Big)\ud{\eta'}\\
\leq&C\Big(1+\abs{\ln(\e)}\Big)\bigg(\lnm{p_{\a}}+\lnnm{S_{\a}}\bigg)+C\Big(1+\abs{\ln(\e)}\Big)\bigg(\frac{1}{\d}+\frac{\e}{\d^2}\bigg)\bigg(\lnnm{\v}+\lnnm{S}\bigg)\no\\
&+C\Big(1+\abs{\ln(\e)}\Big)\bigg(\d+\Big(1+\abs{\ln(\e)}\Big)\sqrt{\d_0}\bigg)\lnnm{\a}.\no
\end{align}
Here, we do not need $I_6$ and $II_5$, so there is no logical loop. Therefore, we have
\begin{align}\label{rt 12}
\\
\abs{I_6}\leq&C\Big(1+\abs{\ln(\e)}\Big)\bigg(\lnm{p_{\a}}+\lnnm{S_{\a}}\bigg)+C\Big(1+\abs{\ln(\e)}\Big)\bigg(\frac{1}{\d}+\frac{\e}{\d^2}\bigg)\bigg(\lnnm{\v}+\lnnm{S}\bigg)\no\\
&+C\Big(1+\abs{\ln(\e)}\Big)\bigg(\d+\Big(1+\abs{\ln(\e)}\Big)\sqrt{\d_0}\bigg)\lnnm{\a}.\no
\end{align}
\ \\
Step 7: Synthesis.\\
Collecting all the terms \eqref{rt 07}, \eqref{rt 08}, \eqref{rt 09}, \eqref{rt 10}, \eqref{rt 11} and \eqref{rt 12}, we have proved
\begin{align}\label{rt 22}
\abs{I}\leq&C\Big(1+\abs{\ln(\e)}\Big)\bigg(\lnmp{p_{\a}}+\lnnm{S_{\a}}\bigg)\\
&+C\bigg(\lnmp{\e\frac{\p p}{\p\phi}}+\lnnm{\zeta\frac{\p S}{\p\eta}}\bigg)\no\\
&+C\Big(1+\abs{\ln(\e)}\Big)\bigg(\frac{1}{\d_0^3}+\frac{1}{\d}+\frac{\e}{\d^2}\bigg)\bigg(\lnnm{\v}+\lnnm{S}\bigg)\no\\
&+C\Big(1+\abs{\ln(\e)}\Big)\bigg(\d+\Big(1+\abs{\ln(\e)}\Big)\sqrt{\d_0}\bigg)\lnnm{\a}.\no
\end{align}

\subsection{Region II: $\sin\phi<0$ and $E(\eta,\phi)\leq \ue^{-V(L)}$}

We consider
\begin{align}
\k[p_{\a}]=&p_{\a}\Big(\phi'(\eta,\phi;0)\Big)\exp\Big(-G_{L,0}-G_{L,\eta}\Big)\\
\t[\tilde\a+S_{\a}]=&\int_0^{L}\frac{(\tilde\a+S)\Big(\eta',\phi'(\eta,\phi;\eta')\Big)}{\sin\Big(\phi'(\eta,\phi;\eta')\Big)}
\exp\Big(-G_{L,\eta'}-G_{L,\eta}\Big)\ud{\eta'}\\
&+\int_{\eta}^{L}\frac{(\tilde\a+S)\Big(\eta',\rr[\phi'](\eta,\phi;\eta')\Big)}{\sin\Big(\phi'(\eta,\phi;\eta')\Big)}\exp\Big(-G_{\eta',\eta}\Big)\ud{\eta'}.\no
\end{align}
Based on Lemma \ref{Milne lemma 1} and Lemma \ref{Milne lemma 2},
we can directly obtain
\begin{align}
\abs{\k[p_{\a}]}\leq&\lnmp{p_{\a}},\\
\abs{\t[S_{\a}]}\leq&\lnnm{S_{\a}}.
\end{align}
Hence, we only need to estimate
\begin{align}\label{mild 2}
II=\t[\tilde\a]=&\int_0^{L}\frac{\tilde\a\Big(\eta',\phi'(\eta,\phi;\eta')\Big)}{\sin\Big(\phi'(\eta,\phi;\eta')\Big)}
\exp\Big(-G_{L,\eta'}-G_{L,\eta}\Big)\ud{\eta'}\\
&+\int_{\eta}^{L}\frac{\tilde\a\Big(\eta',\rr[\phi'(\eta,\phi;\eta')]\Big)}{\sin\Big(\phi'(\eta,\phi;\eta')\Big)}\exp\Big(-G_{\eta',\eta}\Big)\ud{\eta'}\nonumber.
\end{align}
In particular,
since the integral $\displaystyle\int_0^{\eta}\cdots$ can be estimated as in Region I, so we only need to estimate the integral $\displaystyle\int_{\eta}^L\cdots$. Also, noting the fact that
\begin{align}
\exp(-G_{L,\eta'}-G_{L,\eta})\leq \exp\Big(-G_{\eta',\eta}\Big),
\end{align}
we only need to estimate
\begin{align}
\int_{\eta}^{L}\frac{\tilde\a\Big(\eta',\rr[\phi'(\eta,\phi;\eta')]\Big)}{\sin\Big(\phi'(\eta,\phi;\eta')\Big)}\exp\Big(-G_{\eta',\eta}\Big)\ud{\eta'}.
\end{align}
Here the proof is very similar to that in Region I, so we only point out the key differences.\\
\ \\
In the following, we will divide the estimate of $II$ into several cases based on the value of $\sin\phi$, $\cos\phi$, $\sin\phi'$ and $\e\eta'$. We write
\begin{align}
II=&\int_{\eta}^L\id_{\{\sin\phi\leq-\d_0\}}\id_{\{\cos\phi\geq\d_0\}}+\int_{\eta}^L\id_{\{-\d_0\leq\sin\phi\leq0\}}\id_{\{\chi(\phi_{\ast})<1\}}\\
&+\int_{\eta}^L\id_{\{-\d_0\leq\sin\phi\leq0\}}\id_{\{\chi(\phi_{\ast})=1\}}\id_{\{\sqrt{\e\eta'}\geq\sin\phi'\}}
+\int_{\eta}^L\id_{\{-\d_0\leq\sin\phi\leq0\}}\id_{\{\chi(\phi_{\ast})=1\}}\id_{\{\sqrt{\e\eta'}\leq\sin\phi'\}}\no\\
&+\int_{\eta}^L\id_{\{\cos\phi\leq\d_0\}}\no\\
=&II_1+II_2+II_3+II_4+II_5.\no
\end{align}
\ \\
Step 0: Preliminaries.\\
We need to update one key result. For $0\leq\eta\leq\eta'$,
\begin{align}
\sin\phi'=&\sqrt{1-\cos^2\phi'}\leq\sqrt{1-\bigg(\frac{R_1-\e\eta}{R_1-\e\eta'}\bigg)^{2\sin^2\psi}\bigg(\frac{R_2-\e\eta}{R_2-\e\eta'}\bigg)^{2\cos^2\psi}\cos^2\phi}
\leq\abs{\sin\phi}.
\end{align}
Then we have
\begin{align}\label{pt 04}
-\int_{\eta}^{\eta'}\frac{1}{\sin\phi'(y)}\ud{y}\leq&-\frac{\eta'-\eta}{\abs{\sin\phi}}.
\end{align}
\ \\
Step 1: Estimate of $II_1$ for $\sin\phi\leq-\d_0$ and $\cos\phi\geq\d_0$.\\
This step is very similar to the estimate of $I_1$, so we omit the details here. There is one key difference. In estimating $I_1$, we always have $\sin\phi'\geq\sin\phi$ when $0\leq\eta'\leq\eta$. However, now we allow $\eta\leq\eta'\leq L$, so we should update the bound of $\sin\phi'$.\\
\ \\
Along the characteristics, we know
\begin{align}
\ue^{-V(\eta')}\cos\phi'=\ue^{-V(\eta)}\cos\phi,
\end{align}
which implies
\begin{align}
\cos\phi'=&\ue^{V(\eta')-V(\eta)}\cos\phi\leq \ue^{V(L)-V(0)}\cos\phi= \ue^{V(L)-V(0)}\sqrt{1-\d_0^2}.
\end{align}
We can further deduce that
\begin{align}
\cos\phi'\leq \Big(1-\e^{\frac{1}{2}}\Big)^{-1}\sqrt{1-\d_0^2}.
\end{align}
Then we have
\begin{align}
\sin\phi'\geq\sqrt{1-\Big(1-\e^{\frac{1}{2}}\Big)^{-2}(1-\d_0^2)}\geq \d_0-\e^{\frac{1}{4}}>\frac{\d_0}{2},
\end{align}
when $\e$ is sufficiently small.\\
\ \\
Similar to Region I, we will use two formulations to handle different terms and we will decompose $\v=\v_1+\v_2$.\\
\ \\
Using Formulation I, we rewrite the $\v_1$ equation along the characteristics as
\begin{align}\label{rt 33}
\v_1(\eta,\phi)=&p\Big(\phi'(0)\Big)\exp(-G_{L,0}-G_{L,\eta})\\
&+\int_0^{L}\frac{\bar\v(\eta')}{\sin\Big(\phi'(\eta')\Big)}
\exp(-G_{L,\eta'}-G_{L,\eta})\ud{\eta'}+\int_{\eta}^{L}\frac{\bar\v(\eta')}{\sin\Big(\phi'(\eta')\Big)}\exp(-G_{\eta',\eta})\ud{\eta'},\no
\end{align}
where $(\eta',\phi')$ and $(\eta,\phi)$ are on the same characteristic with $\sin\phi'\geq0$.
Then taking $\eta$ derivative on both sides of \eqref{rt 33} yields
\begin{align}
\frac{\p\v_1}{\p\eta}=&X_1+X_2+X_3+X_4+X_5+X_6+X_7,
\end{align}
where
\begin{align}
X_1=&\frac{\p p\Big(\phi'(0)\Big)}{\p\eta}\exp(-G_{L,0}-G_{L,\eta}),\\
X_2=&-p\Big(\phi'(0)\Big)\exp(-G_{L,0}-G_{L,\eta})\bigg(\frac{\p G_{L,0}}{\p\eta}+\frac{\p G_{L,\eta}}{\p\eta}\bigg),\\
X_3=&-\int_0^{L}\bar\v(\eta')\frac{\cos\Big(\phi'(\eta')\Big)}{\sin^2\Big(\phi'(\eta')\Big)}\frac{\p\phi'(\eta')}{\p\eta}\exp(-G_{L,\eta'}-G_{L,\eta})\ud{\eta'},\\
X_4=&-\int_0^{L}\frac{\bar\v(\eta')}{\sin\Big(\phi'(\eta')\Big)}\exp(-G_{L,\eta'}-G_{L,\eta})\bigg(\frac{\p G_{L,\eta'}}{\p\eta}+\frac{\p G_{L,\eta}}{\p\eta}\bigg)\ud{\eta'},\\
X_5=&-\int_{\eta}^{L}\bar\v(\eta')\frac{\cos\Big(\phi'(\eta')\Big)}{\sin^2\Big(\phi'(\eta')\Big)}\frac{\p\phi'(\eta')}{\p\eta}\exp(-G_{\eta',\eta})\ud{\eta'},\\
X_6=&-\int_{\eta}^{L}\frac{\bar\v(\eta')}{\sin\Big(\phi'(\eta')\Big)}\exp(-G_{\eta',\eta})\frac{\p G_{\eta',\eta}}{\p\eta}\ud{\eta'},\\
X_7=&-\frac{\bar\v(\eta)}{\sin(\phi)}.
\end{align}
We need to estimate each term. The estimates are standard, so we only list the results:
\begin{align}
\abs{X_1}&\leq\abs{\frac{\p p}{\p\phi}\Big(\phi'(0)\Big)}\abs{\frac{\p\phi'(0)}{\p\eta}}\leq\frac{C\e}{\d_0}\lnmp{\frac{\p p}{\p\phi}},\\
\abs{X_2}&\leq\abs{p\Big(\phi'(0)\Big)}\abs{\frac{\p G_{L,0}}{\p\eta}+\frac{\p G_{L,\eta}}{\p\eta}}\leq\frac{C}{\d_0}\lnmp{p},\\
\abs{X_3}&\leq\abs{\bar\v(\eta')}\abs{\frac{\cos\Big(\phi'(\eta')\Big)}{\sin\Big(\phi'(\eta')\Big)}}\abs{\frac{\p\phi'(\eta')}{\p\eta}}\leq\frac{C\e}{\d_0^2}\lnnm{\v},\\
\abs{X_4}&\leq\abs{\bar\v(\eta')}\abs{\frac{\p G_{L,\eta'}}{\p\eta}+\frac{\p G_{L,\eta}}{\p\eta}}\leq\frac{C}{\d_0}\lnnm{\v},\\
\abs{X_5}&\leq\abs{\bar\v(\eta')}\abs{\frac{\cos\Big(\phi'(\eta')\Big)}{\sin\Big(\phi'(\eta')\Big)}}\abs{\frac{\p\phi'(\eta')}{\p\eta}}\leq\frac{C\e}{\d_0^2}\lnnm{\v},\\
\abs{X_6}&\leq\abs{\bar\v(\eta')}\abs{\frac{\p G_{\eta',\eta}}{\p\eta}}\leq\frac{C}{\d_0}\lnnm{\v},\\
\abs{X_7}&\leq\frac{1}{\d_0}\lnnm{\v}.
\end{align}
In total, we have
\begin{align}\label{rt 35}
\abs{\frac{\p\v_1}{\p\eta}}\leq C\left(\frac{1}{\d_0}+\frac{\e}{\d_0^3}\right)\lnnm{\v}+\frac{C\e}{\d_0}\lnmp{\frac{\p p}{\p\phi}}.
\end{align}
\ \\
Using Formulation II, we rewrite the $\v_2$ equation along the characteristics as
\begin{align}\label{rt 34}
\\
\v_2(\eta,\phi)=&\int_{\phi_{\ast}}^{\phi^{\ast}}\frac{S\Big(\eta'(\phi'),\phi'\Big)}{F\Big(\eta'(\phi')\Big)\cos(\phi')}
\exp(-H_{\phi^{\ast},\phi'}-H_{-\phi^{\ast},\phi})\ud{\phi'}
+\int_{\phi}^{-\phi^{\ast}}\frac{S\Big(\eta'(\phi'),\phi'\Big)}{F\Big(\eta'(\phi')\Big)\cos(\phi')}\exp(-H_{\phi',\phi})\ud{\phi'}\no
\end{align}
where $(\eta',\phi')$, $(0,\phi_{\ast})$, $(L,\phi^{\ast})$, $(L,-\phi^{\ast})$ and $(\eta,\phi)$ are on the same characteristic with $\sin\phi'\geq0$ and $\phi^{\ast}\geq0$.
Then taking $\eta$ derivative on both sides of \eqref{rt 34} yields
\begin{align}
\frac{\p\v_2}{\p\eta}=&Y_1+Y_2+Y_3+Y_4+Y_5+Y_6+Y_7+Y_8,
\end{align}
where
\begin{align}
Y_1=&\frac{S(L,\phi^{\ast})}{F(L)\cos(\phi^{\ast})}\exp(-H_{-\phi^{\ast},\phi})\frac{\p\phi^{\ast}}{\p\eta}-\frac{S(0,\phi_{\ast})}{F(0)\cos(\phi_{\ast})}
\exp(-H_{\phi^{\ast},\phi_{\ast}}-H_{-\phi^{\ast},\phi})\frac{\p\phi_{\ast}}{\p\eta},\\
Y_2=&-\int_{\phi_{\ast}}^{\phi^{\ast}}S\Big(\eta'(\phi'),\phi'\Big)\frac{1}{F^2\Big(\eta'(\phi')\Big)\cos(\phi')}\frac{\p F\Big(\eta'(\phi')\Big)}{\p\eta}
\exp(-H_{\phi^{\ast},\phi'}-H_{-\phi^{\ast},\phi})\ud{\phi'},\\
Y_3=&-\int_{\phi_{\ast}}^{\phi^{\ast}}\frac{S\Big(\eta'(\phi'),\phi'\Big)}{F\Big(\eta'(\phi')\Big)\cos(\phi')}
\exp(-H_{\phi^{\ast},\phi'}-H_{-\phi^{\ast},\phi})\bigg(\frac{\p H_{\phi^{\ast},\phi'}}{\p\eta}+\frac{\p H_{-\phi^{\ast},\phi}}{\p\eta}\bigg)\ud{\phi'},\\
Y_4=&\int_{\phi_{\ast}}^{\phi^{\ast}}\frac{\p_{\eta'}S\Big(\eta'(\phi'),\phi'\Big)}{F\Big(\eta'(\phi')\Big)\cos(\phi')}\frac{\p \eta'(\phi')}{\p\eta}
\exp(-H_{\phi^{\ast},\phi'}-H_{-\phi^{\ast},\phi})\ud{\phi'},\\
Y_5=&-\frac{S(L,-\phi^{\ast})}{F(L)\cos(-\phi^{\ast})}\exp(-H_{-\phi^{\ast},\phi})\frac{\p \phi^{\ast}}{\p\eta},\\
Y_6=&-\int_{\phi}^{-\phi^{\ast}}S\Big(\eta'(\phi'),\phi'\Big)\frac{1}{F^2\Big(\eta'(\phi')\Big)\cos(\phi')}\frac{\p F\Big(\eta'(\phi')\Big)}{\p\eta}\exp(-H_{\phi',\phi})\ud{\phi'},\\
Y_7=&-\int_{\phi}^{-\phi^{\ast}}\frac{S\Big(\eta'(\phi'),\phi'\Big)}{F\Big(\eta'(\phi')\Big)\cos(\phi')}\exp(-H_{\phi',\phi})\frac{\p H_{\phi',\phi}}{\p\eta}\ud{\phi'},\\
Y_8=&\int_{\phi}^{-\phi^{\ast}}\frac{\p_{\eta'}S\Big(\eta'(\phi'),\phi'\Big)}{F\Big(\eta'(\phi')\Big)\cos(\phi')}\frac{\p \eta'(\phi')}{\p\eta}\exp(-H_{\phi',\phi})\ud{\phi'}.
\end{align}
We need to estimate each term. The estimates are standard, so we only list the results:
\begin{align}
\abs{Y_1}&\leq\frac{C}{\d_0}\bigg(\abs{S(0,\phi_{\ast})}\abs{\frac{1}{F(0)}}\abs{\frac{\p\phi_{\ast}}{\p\eta}}
+\abs{S(L,\phi^{\ast})}\abs{\frac{1}{F(L)}}\abs{\frac{\p\phi^{\ast}}{\p\eta}}\bigg)\leq\frac{C}{\d_0^2}\lnnm{S},\\
\abs{Y_2}&\leq\abs{S\Big(\eta'(\phi'),\phi'\Big)}\abs{\frac{1}{F\Big(\eta'(\phi')\Big)}}\abs{\frac{\p F\Big(\eta'(\phi')\Big)}{\p\eta}}\leq\frac{C\e}{\d_0}\lnnm{S},\\
\abs{Y_3}&\leq\abs{S\Big(\eta'(\phi'),\phi'\Big)}\abs{\frac{\p H_{\phi^{\ast},\phi'}}{\p\eta}+\frac{\p H_{-\phi^{\ast},\phi}}{\p\eta}}\leq\frac{C}{\d_0^2}\lnnm{S},\\
\abs{Y_4}&\leq\abs{\p_{\eta'}S\Big(\eta'(\phi'),\phi'\Big)}\abs{\frac{\p\eta'(\phi')}{\p\eta}}\leq C\lnnm{\frac{\p S}{\p\eta}},\\
\abs{Y_5}&\leq\frac{C}{\d_0}\abs{S(L,\phi^{\ast})}\abs{\frac{1}{F(L)}}\abs{\frac{\p\phi^{\ast}}{\p\eta}}\leq\frac{C}{\d_0^2}\lnnm{S},\\
\abs{Y_6}&\leq\abs{S\Big(\eta'(\phi'),\phi'\Big)}\abs{\frac{1}{F\Big(\eta'(\phi')\Big)}}\abs{\frac{\p F\Big(\eta'(\phi')\Big)}{\p\eta}}\leq\frac{C\e}{\d_0}\lnnm{S},\\
\abs{Y_7}&\leq\abs{S\Big(\eta'(\phi'),\phi'\Big)}\abs{\frac{\p H_{\phi',\phi}}{\p\eta}}\leq\frac{C}{\d_0^2}\lnnm{S},\\
\abs{Y_8}&\leq\abs{\p_{\eta'}S\Big(\eta'(\phi'),\phi'\Big)}\abs{\frac{\p\eta'(\phi')}{\p\eta}}\leq C\lnnm{\frac{\p S}{\p\eta}}.
\end{align}
In total, we have
\begin{align}\label{rt 36}
\abs{\frac{\p\v_2}{\p\eta}}\leq \frac{C}{\d_0^2}\lnnm{S}+C\lnnm{\frac{\p S}{\p\eta}}.
\end{align}
Combining \eqref{rt 35} and \eqref{rt 36}, we have
\begin{align}
\abs{\frac{\p\v}{\p\eta}}\leq&C\Bigg(\left(\frac{1}{\d_0}+\frac{\e}{\d_0^3}\right)\lnnm{\v}+\frac{\e}{\d_0}\lnmp{\frac{\p p}{\p\phi}}+\frac{1}{\d_0^2}\lnnm{S}+\lnnm{\frac{\p S}{\p\eta}}\Bigg).
\end{align}
Hence, noting that $\zeta(\eta,\phi)\geq\sin\phi\geq\d_0$, we know
\begin{align}\label{rt 13}
\abs{II_1}\leq&C\Bigg(\left(\frac{1}{\d_0}+\frac{\e}{\d_0^3}\right)\lnnm{\v}+\frac{\e}{\d_0}\lnmp{\frac{\p p}{\p\phi}}+\frac{1}{\d_0^2}\lnnm{S}+\frac{1}{\d_0}\lnnm{\zeta\frac{\p S}{\p\eta}}\Bigg).
\end{align}
\ \\
Step 2: Estimate of $II_2$ for $-\d_0\leq\sin\phi\leq0$ and $\chi(\phi_{\ast})<1$.\\
This is similar to the estimate of $I_2$ based on the integral
\begin{align}
\int_{\eta}^{L}\frac{1}{\sin\phi'}\exp\Big(-G_{\eta',\eta}\Big)\ud{\eta'}\leq 1.
\end{align}
Then we have
\begin{align}\label{rt 14}
\abs{II_2}
\leq&\bigg(\frac{1}{\d}+\frac{\e}{\d^2}\bigg)\bigg(\lnnm{\v}+\lnnm{S}\bigg).
\end{align}
\ \\
Step 3: Estimate of $II_3$ for $-\d_0\leq\sin\phi\leq0$, $\chi(\phi_{\ast})=1$ and $\sqrt{\e\eta'}\geq\sin\phi'$.\\
This is similar to the estimate of $I_3$, we have
\begin{align}\label{rt 15}
\abs{II_3}\leq&C\d\lnnm{\a}.
\end{align}
\ \\
Step 4: Estimate of $II_4$ for $-\d_0\leq\sin\phi\leq0$, $\chi(\phi_{\ast})=1$ and $\sqrt{\e\eta'}\leq\sin\phi'$.\\
This step is different. We do not need to further decompose the cases.
Based on \eqref{pt 04}, we have
\begin{align}
-G_{\eta,\eta'}\leq&-\frac{\eta'-\eta}{\abs{\sin\phi}}.
\end{align}
Then following the same argument in estimating $I_5$, we obtain
\begin{align}
\abs{II_4}\leq&C\lnnm{\a}\int_{\eta}^{L}\bigg(1+\abs{\ln(\e)}+\abs{\ln(\eta')}\bigg)
\exp\left(-\frac{\eta'-\eta}{\abs{\sin\phi}}\right)\ud{\eta'}
\end{align}
If $\eta\geq 2$, we directly obtain
\begin{align}
\abs{\int_{\eta}^{L}\abs{\ln(\eta')}
\exp\left(-\frac{\eta'-\eta}{\abs{\sin\phi}}\right)\ud{\eta'}}\leq& \abs{\int_{2}^{L}\ln(\eta')
\exp\left(-\frac{\eta'-\eta}{\abs{\sin\phi}}\right)\ud{\eta'}}\\
\leq&\ln(2)\abs{\int_{2}^{L}
\exp\left(-\frac{\eta'-\eta}{\abs{\sin\phi}}\right)\ud{\eta'}}\no\\
\leq&C\sqrt{\abs{\sin\phi}}.\no
\end{align}
If $\eta\leq 2$, we decompose as
\begin{align}
&\abs{\int_{\eta}^{L}\abs{\ln(\eta')}
\exp\left(-\frac{\eta'-\eta}{\abs{\sin\phi}}\right)\ud{\eta'}}\\
\leq&\abs{\int_{\eta}^{2}\abs{\ln(\eta')}
\exp\left(-\frac{\eta'-\eta}{\abs{\sin\phi}}\right)\ud{\eta'}}+\abs{\int_{2}^{L}\abs{\ln(\eta')}
\exp\left(-\frac{\eta'-\eta}{\abs{\sin\phi}}\right)\ud{\eta'}}.\no
\end{align}
The second term is identical to the estimate in $\eta\geq2$. We apply Cauchy's inequality to the first term
\begin{align}
\abs{\int_{\eta}^{2}\abs{\ln(\eta')}
\exp\left(-\frac{\eta'-\eta}{\abs{\sin\phi}}\right)\ud{\eta'}}
\leq&\bigg(\int_{\eta}^{2}\ln^2(\eta')\ud{\eta'}\bigg)^{\frac{1}{2}}\bigg(\int_{\eta}^{2}
\exp\left(-\frac{2(\eta'-\eta)}{\abs{\sin\phi}}\right)\ud{\eta'}\bigg)^{\frac{1}{2}}\\
\leq&\bigg(\int_0^{2}\ln^2(\eta')\ud{\eta'}\bigg)^{\frac{1}{2}}\bigg(\int_{\eta}^{2}
\exp\left(-\frac{2(\eta'-\eta)}{\abs{\sin\phi}}\right)\ud{\eta'}\bigg)^{\frac{1}{2}}\no\\
\leq&C\sqrt{\abs{\sin\phi}}.\no
\end{align}
Hence, we have
\begin{align}\label{rt 16}
\abs{II_4}\leq C\Big(1+\abs{\ln(\e)}\Big)\sqrt{\d_0}\lnnm{\a}.
\end{align}
\ \\
Step 5: Estimate of $II_5$ for $\cos\phi<\d_0$.\\
This is similar to the estimate of $I_6$, we have
\begin{align}\label{rt 17}
\\
\abs{II_5}\leq&C\Big(1+\abs{\ln(\e)}\Big)\bigg(\lnm{p_{\a}}+\lnnm{S_{\a}}\bigg)+C\Big(1+\abs{\ln(\e)}\Big)\bigg(\frac{1}{\d}+\frac{\e}{\d^2}\bigg)\bigg(\lnnm{\v}+\lnnm{S}\bigg)\no\\
&+C\Big(1+\abs{\ln(\e)}\Big)\bigg(\d+\Big(1+\abs{\ln(\e)}\Big)\sqrt{\d_0}\bigg)\lnnm{\a}.\no
\end{align}
\ \\
Step 6: Synthesis.\\
Collecting all the terms \eqref{rt 13}, \eqref{rt 14}, \eqref{rt 15}, \eqref{rt 16} and \eqref{rt 17}, we have proved
\begin{align}\label{rt 23}
\abs{II}\leq&C\Big(1+\abs{\ln(\e)}\Big)\bigg(\lnmp{p_{\a}}+\lnnm{S_{\a}}\bigg)\\
&+C\bigg(\lnmp{\e\frac{\p p}{\p\phi}}+\lnnm{\zeta\frac{\p S}{\p\eta}}\bigg)\no\\
&+C\Big(1+\abs{\ln(\e)}\Big)\bigg(\frac{1}{\d_0^3}+\frac{1}{\d}+\frac{\e}{\d^2}\bigg)\bigg(\lnnm{\v}+\lnnm{S}\bigg)\no\\
&+C\Big(1+\abs{\ln(\e)}\Big)\bigg(\d+\Big(1+\abs{\ln(\e)}\Big)\sqrt{\d_0}\bigg)\lnnm{\a}.\no
\end{align}

\subsection{Region III: $\sin\phi<0$ and $E(\eta,\phi)\geq \ue^{-V(L)}$}

We consider
\begin{align}
\k[p_{\a}]=&p_{\a}\Big(\phi'(\eta,\phi;0)\Big)\exp\Big(-G_{\eta^+,0}-G_{\eta^+,\eta}\Big)\\
\t[\tilde\a+S_{\a}]=&\int_0^{\eta^+}\frac{(\tilde\a+S_{\a})\Big(\eta',\phi'(\eta,\phi;\eta')\Big)}{\sin\Big(\phi'(\eta,\phi;\eta')\Big)}
\exp\Big(-G_{\eta^+,\eta'}-G_{\eta^+,\eta}\Big)\ud{\eta'}\\&+
\int_{\eta}^{\eta^+}\frac{(\tilde\a+S_{\a})\Big(\eta',\rr[\phi'](\eta,\phi;\eta')\Big)}{\sin\Big(\phi'(\eta,\phi;\eta')\Big)}\exp\Big(-G_{\eta',\eta}\Big)\ud{\eta'}.\no
\end{align}
Based on Lemma \ref{Milne lemma 1} and Lemma \ref{Milne lemma 2}, we still have
\begin{align}
\abs{\k[p_{\a}]}\leq&\lnmp{p_{\a}},\\
\abs{\t[S_{\a}]}\leq&\lnnm{S_{\a}}.
\end{align}
Hence, we only need to estimate
\begin{align}
III=\t[\tilde\a]=&\int_0^{\eta^+}\frac{\tilde\a\Big(\eta',\phi'(\eta,\phi;\eta')\Big)}{\sin\Big(\phi'(\eta,\phi;\eta')\Big)}
\exp\Big(-G_{\eta^+,\eta'}-G_{\eta^+,\eta}\Big)\ud{\eta'}\\
&+
\int_{\eta}^{\eta^+}\frac{\tilde\a\Big(\eta',\rr[\phi'(\eta,\phi;\eta')]\Big)}{\sin\Big(\phi'(\eta,\phi;\eta')\Big)}\exp\Big(-G_{\eta',\eta}\Big)\ud{\eta'}\nonumber.
\end{align}
Note that $\abs{E(\eta,\phi)}\geq \ue^{-V(L)}$ implies
\begin{align}
\ue^{-V(\eta)}\cos\phi\geq \ue^{-V(L)}.
\end{align}
Hence, we can further deduce that
\begin{align}
\cos\phi\geq&\ue^{V(\eta)-V(L)}\geq \ue^{V(0)-V(L)}\geq \bigg(1-\e^{\frac{1}{2}}\bigg).
\end{align}
Hence, we know
\begin{align}
\abs{\sin\phi}\leq\sqrt{1-\bigg(1-\e^{\frac{1}{2}}\bigg)^2}\leq \e^{\frac{1}{4}}.
\end{align}
Hence, when $\e$ is sufficiently small, we always have
\begin{align}
\abs{\sin\phi}\leq \e^{\frac{1}{4}}\leq \d_0.
\end{align}
This means we do not need to bother with the estimate of $\sin\phi\leq-\d_0$ as Step 1 in estimating $I$ and $II$. Also, it is not necessary to discuss the case $\cos\phi<\d_0$.

Then the integral $\displaystyle\int_0^{\eta}(\cdots)$ is similar to the argument in Region I, and the integral $\displaystyle\int_{\eta}^{\eta^+}(\cdots)$ is similar to the argument in Region II.
Hence, combining the methods in Region I and Region II, we can show the desired result, i.e.
\begin{align}\label{rt 24}
\abs{III}\leq&C\Big(1+\abs{\ln(\e)}\Big)\bigg(\lnmp{p_{\a}}+\lnnm{S_{\a}}\bigg)+C\Big(1+\abs{\ln(\e)}\Big)\bigg(\frac{1}{\d}+\frac{\e}{\d^2}\bigg)\bigg(\lnnm{\v}+\lnnm{S}\bigg)\\
&+C\Big(1+\abs{\ln(\e)}\Big)\bigg(\d+\Big(1+\abs{\ln(\e)}\Big)\sqrt{\d_0}\bigg)\lnnm{\a}.\no
\end{align}

\subsection{Estimate of Normal and Velocity Derivative}

\begin{theorem}\label{regularity theorem 1}
The solution $\a$ to the equation \eqref{regularity 1} satisfies
\begin{align}
\lnnm{\a}\leq&C\abs{\ln(\e)}\bigg(\lnm{p_{\a}}+\lnnm{S_{\a}}\bigg)\\
&+C\abs{\ln(\e)}^8\bigg(\lnmp{\e\dfrac{\p p}{\p\phi}}+\lnnm{S}+\lnnm{\zeta\dfrac{\p S}{\p\eta}}+\lnnm{\v}\bigg).\no
\end{align}
\end{theorem}
\begin{proof}
Combining the estimates \eqref{rt 22}, \eqref{rt 23} and \eqref{rt 24}, and taking supremum over all $(\eta,\phi,\psi)$, we have
\begin{align}\label{pt 05}
\lnnm{\a}\leq&C\Big(1+\abs{\ln(\e)}\Big)\bigg(\lnmp{p_{\a}}+\lnnm{S_{\a}}\bigg)\\
&+C\bigg(\lnmp{\e\frac{\p p}{\p\phi}}+\lnnm{\zeta\frac{\p S}{\p\eta}}\bigg)\no\\
&+C\Big(1+\abs{\ln(\e)}\Big)\bigg(\frac{1}{\d_0^3}+\frac{1}{\d}+\frac{\e}{\d^2}\bigg)\bigg(\lnnm{\v}+\lnnm{S}\bigg)\no\\
&+C\Big(1+\abs{\ln(\e)}\Big)\bigg(\d+\Big(1+\abs{\ln(\e)}\Big)\sqrt{\d_0}\bigg)\lnnm{\a}.\no
\end{align}
Then we choose the constants $\d$ and $\d_0$ to perform absorbing argument. First we choose $\d=C_0\Big(1+\abs{\ln(\e)}\Big)^{-1}$ for $C_0>0$ sufficiently small such that
\begin{align}
C\d\leq\frac{1}{4}.
\end{align}
Then we take $\d_0=C_0\Big(1+\abs{\ln(\e)}\Big)^{-4}$ such that
\begin{align}
C\Big(1+\abs{\ln(\e)}\Big)^2\sqrt{\d_0}\leq \frac{1}{4},
\end{align}
for $\e$ sufficiently small. Note that this mild decay of $\d_0$ with respect to $\e$ also justifies the assumption in Case III that
\begin{align}
\e^{\frac{1}{4}}\leq \frac{\d_0}{2},
\end{align}
for $\e$ sufficiently small. Hence, we can absorb all the term related to $\lnnm{\a}$ on the right-hand side of \eqref{pt 05} to the left-hand side to obtain
\begin{align}
\lnnm{\a}\leq&C\abs{\ln(\e)}\bigg(\lnmp{p_{\a}}+\lnnm{S_{\a}}\bigg)\\
&+C\abs{\ln(\e)}^8\bigg(\lnmp{\e\dfrac{\p p}{\p\phi}}+\lnnm{S}+\lnnm{\zeta\dfrac{\p S}{\p\eta}}+\lnnm{\v}\bigg).\no
\end{align}
\end{proof}

\begin{theorem}\label{pt theorem 1}
The solution $\v$ to the difference equation \eqref{Milne difference problem} satisfies
\begin{align}
&\lnnm{\zeta\frac{\p\v}{\p\eta}}+\lnnm{F(\eta,\psi)\cos\phi\frac{\p\v}{\p\phi}}\\
\leq&C\abs{\ln(\e)}^8\bigg(\lnmp{p}+\lnmp{\e\dfrac{\p p}{\p\phi}}+\lnnm{S}+\lnnm{\zeta\dfrac{\p S}{\p\eta}}+\lnnm{\v}\bigg).\no
\end{align}
\end{theorem}
\begin{proof}
Based on Theorem \ref{regularity theorem 1}, we have
\begin{align}\label{pt 06}
\lnnm{\a}\leq&C\abs{\ln(\e)}\bigg(\lnmp{p_{\a}}+\lnnm{S_{\a}}\bigg)\\
&+C\abs{\ln(\e)}^8\bigg(\lnmp{\e\dfrac{\p p}{\p\phi}}+\lnnm{S}+\lnnm{\zeta\dfrac{\p S}{\p\eta}}+\lnnm{\v}\bigg).\no
\end{align}
Taking $\eta$ derivatives on both sides of \eqref{Milne difference problem} and multiplying $\zeta$, we have
\begin{align}
p_{\a}=&\e\cos\phi\frac{\p p}{\p\phi}-p+\bar\v(0)+S(0,\phi),\\
S_{\a}=&\frac{\p{F}}{\p{\eta}}\zeta\cos\phi\frac{\p\v}{\p\phi}+\zeta\frac{\p S}{\p\eta}.
\end{align}
Since $\abs{F(\eta,\psi)}\leq C\e$ and $\abs{\dfrac{\p F}{\p\eta}}\leq C\e F$, we may directly estimate
\begin{align}
\lnnm{p_{\a}}\leq&C\bigg(\lnmp{p}+\lnmp{\e\dfrac{\p p}{\p\phi}}+\lnnm{S}+\lnnm{\v}\bigg),\label{pt 07}\\
\lnnm{S_{\a}}\leq&C\bigg(\e\lnnm{F(\eta,\psi)\cos\phi\frac{\p\v}{\p\phi}}+\lnnm{\zeta\dfrac{\p S}{\p\eta}}\bigg).\label{pt 08}
\end{align}
Then inserting \eqref{pt 07} and \eqref{pt 08} into \eqref{pt 06}, we derive
\begin{align}
\lnnm{\a}\leq&C\e\lnnm{F(\eta,\psi)\cos\phi\frac{\p\v}{\p\phi}}\\
&+C\abs{\ln(\e)}^8\bigg(\lnmp{p}+\lnmp{\e\dfrac{\p p}{\p\phi}}+\lnnm{S}+\lnnm{\zeta\dfrac{\p S}{\p\eta}}+\lnnm{\v}\bigg).\no
\end{align}
This is actually
\begin{align}
\lnnm{\zeta\frac{\p\v}{\p\eta}}\leq&C\e\lnnm{F(\eta,\psi)\cos\phi\frac{\p\v}{\p\phi}}\\
&+C\abs{\ln(\e)}^8\bigg(\lnmp{p}+\lnmp{\e\dfrac{\p p}{\p\phi}}+\lnnm{S}+\lnnm{\zeta\dfrac{\p S}{\p\eta}}+\lnnm{\v}\bigg).\no
\end{align}
Since $\zeta(\eta,\phi)\geq\abs{\sin\phi}$, we know
\begin{align}\label{pt 10}
\lnnm{\sin\phi\frac{\p\v}{\p\eta}}\leq&C\e\lnnm{F(\eta,\psi)\cos\phi\frac{\p\v}{\p\phi}}\\
&+C\abs{\ln(\e)}^8\bigg(\lnmp{p}+\lnmp{\e\dfrac{\p p}{\p\phi}}+\lnnm{S}+\lnnm{\zeta\dfrac{\p S}{\p\eta}}+\lnnm{\v}\bigg).\no
\end{align}
Considering the equation \eqref{Milne difference problem}, and inserting \eqref{pt 10}, we have
\begin{align}
\lnnm{F(\eta,\psi)\cos\phi\frac{\p\v}{\p\phi}}\leq&\lnnm{\sin\phi\frac{\p\v}{\p\eta}}+\lnnm{\v}+\lnnm{\bar\v}+\lnnm{S}\\
\leq&C\e\lnnm{F(\eta,\psi)\cos\phi\frac{\p\v}{\p\phi}}\no\\
&+C\abs{\ln(\e)}^8\bigg(\lnmp{p}+\lnmp{\e\dfrac{\p p}{\p\phi}}+\lnnm{S}+\lnnm{\zeta\dfrac{\p S}{\p\eta}}+\lnnm{\v}\bigg).\no
\end{align}
Absorbing $\lnnm{F(\eta,\psi)\cos\phi\dfrac{\p\v}{\p\phi}}$ into the left-hand side, we obtain
\begin{align}
\\
\lnnm{F(\eta,\psi)\cos\phi\frac{\p\v}{\p\phi}}\leq&C\abs{\ln(\e)}^8\bigg(\lnmp{p}+\lnmp{\e\dfrac{\p p}{\p\phi}}+\lnnm{S}+\lnnm{\zeta\dfrac{\p S}{\p\eta}}+\lnnm{\v}\bigg).\no
\end{align}
Therefore, we further derive
\begin{align}
\\
\lnnm{\zeta\frac{\p\v}{\p\eta}}\leq&C\abs{\ln(\e)}^8\bigg(\lnmp{p}+\lnmp{\e\dfrac{\p p}{\p\phi}}+\lnnm{S}+\lnnm{\zeta\dfrac{\p S}{\p\eta}}+\lnnm{\v}\bigg).\no
\end{align}
\end{proof}

\begin{theorem}\label{pt theorem 2}
For $K_0>0$ sufficiently small, the solution $\v$ to the difference equation \eqref{Milne difference problem} satisfies
\begin{align}
&\lnnm{\ue^{K_0\eta}\zeta\frac{\p\v}{\p\eta}}+\lnnm{\ue^{K_0\eta}F(\eta,\psi)\cos\phi\frac{\p\v}{\p\phi}}\\
\leq&C\abs{\ln(\e)}^8\bigg(\lnmp{p}+\lnmp{\e\dfrac{\p p}{\p\phi}}+\lnnm{\ue^{K_0\eta}S}+\lnnm{\ue^{K_0\eta}\zeta\dfrac{\p S}{\p\eta}}+\lnnm{\ue^{K_0\eta}\v}\bigg).\no
\end{align}
\end{theorem}
\begin{proof}
This proof is almost identical to that of Theorem \ref{pt theorem 1}. The only difference is that $S_{\a}$ is added by $K_0\a\sin\phi$. When $K_0$ is sufficiently small, we can also absorb them into the left-hand side. Hence, this is obvious.
\end{proof}

\subsection{Local $\e$-Milne Problem with Geometric Correction}

In the analysis of regular boundary layer, we actually need to estimate the solution to the following modified $\e$-Milne problem with geometric correction:
\begin{align}\label{Milne difference problem.}
\left\{
\begin{array}{l}\displaystyle
\sin\phi\frac{\p g}{\p\eta}+F(\eta,\psi)\cos\phi\frac{\p
g}{\p\phi}+g=S(\eta,\phi,\psi),\\\rule{0ex}{2.0em}
g(0,\phi,\psi)=h(\phi,\psi)\ \ \text{for}\ \ \sin\phi>0,\\\rule{0ex}{2.0em}
g(L,\phi,\psi)=g(L,\rr[\phi],\psi).
\end{array}
\right.
\end{align}
Since there is no non-local term $\bar g$, the estimates is even simpler. We can always track along the characteristics back to the boundary data. Here $S$ is chosen such that the solution $g$ decays exponentially to a constant $g_L$.
Equivalently, $\w=g-g_L$ satisfies
\begin{align}\label{Milne difference problem'}
\left\{
\begin{array}{l}\displaystyle
\sin\phi\frac{\p\w}{\p\eta}+F(\eta,\psi)\cos\phi\frac{\p
\w}{\p\phi}+\w=S(\eta,\phi,\psi),\\\rule{0ex}{2.0em}
\w(0,\phi,\psi)=h(\phi,\psi)\ \ \text{for}\ \ \sin\phi>0,\\\rule{0ex}{2.0em}
\w(L,\phi,\psi)=\w(L,\rr[\phi],\psi).
\end{array}
\right.
\end{align}

\begin{theorem}\label{Milne theorem 2'}
Assume $S$ and $h$ satisfy \eqref{Milne bounded} and \eqref{Milne decay}. Then there exists $K_0>0$ such that the solution $g(\eta,\phi,\psi)$ to the
equation \eqref{Milne difference problem.} satisfies
\begin{align}
\tnnm{\ue^{K_0\eta}\w}+\lnnm{\ue^{K_0\eta}\w}\leq C.
\end{align}
\end{theorem}

\begin{theorem}\label{pt theorem 2'}
For $K_0>0$ sufficiently small, the solution $\v$ to the difference equation \eqref{Milne difference problem'} satisfies
\begin{align}
&\lnnm{\ue^{K_0\eta}\zeta\frac{\p\w}{\p\eta}}+\lnnm{\ue^{K_0\eta}F(\eta,\psi)\cos\phi\frac{\p\w}{\p\phi}}\\
\leq&C\abs{\ln(\e)}^8\bigg(\lnmp{p}+\lnmp{\e\dfrac{\p p}{\p\phi}}+\lnnm{\ue^{K_0\eta}S}+\lnnm{\ue^{K_0\eta}\zeta\dfrac{\p S}{\p\eta}}+\lnnm{\ue^{K_0\eta}\w}\bigg).\no
\end{align}
\end{theorem}

\newpage


\bibliographystyle{siam}
\bibliography{Reference}

\begin{thebibliography}{10}

\bibitem{Bensoussan.Lions.Papanicolaou1979}
{\sc A.~Bensoussan, J.-L. Lions, and G.~C. Papanicolaou}, {\em Boundary layers
  and homogenization of transport processes}, Publ. Res. Inst. Math. Sci., 15
  (1979), pp.~53--157.

\bibitem{Cercignani.Illner.Pulvirenti1994}
{\sc C.~Cercignani, R.~Illner, and M.~Pulvirenti}, {\em The mathematical theory
  of dilute gases}, Springer-Verlag, New York, 1994.

\bibitem{Cercignani.Marra.Esposito1998}
{\sc C.~Cercignani, R.~Marra, and R.~Esposito}, {\em The {Milne} problem with a
  force term}, Transport Theory Statist. Phys., 27 (1998), pp.~1--33.

\bibitem{Esposito.Guo.Kim.Marra2013}
{\sc R.~Esposito, Y.~Guo, C.~Kim, and R.~Marra}, {\em Non-isothermal boundary
  in the {Boltzmann} theory and {Fourier} law}, Comm. Math. Phys., 323 (2013),
  pp.~177--239.

\bibitem{Guo.Kim.Tonon.Trescases2013}
{\sc Y.~Guo, C.~Kim, D.~Tonon, and A.~Trescases}, {\em Regularity of the
  {Boltzmann} equation in convex domain}, Inventiones Mathematicae, 207 (2016),
  pp.~115--290.

\bibitem{Guo.Nguyen2011}
{\sc Y.~Guo and T.~Nguyen}, {\em A note on the {Prandtl} boundary layers},
  Comm. Pure Appl. Math., 64 (2011), pp.~1416--1438.

\bibitem{AA007}
{\sc Y.~Guo and L.~Wu}, {\em Geometric correction in diffusive limit of neutron
  transport equation in {2D} convex domains}, Arch. Rational Mech. Anal., 226
  (2017), pp.~321--403.

\bibitem{AA009}
\leavevmode\vrule height 2pt depth -1.6pt width 23pt, {\em Regularity of
  {Milne} problem with geometric correction in {3D}}, Math. Models Methods
  Appl. Sci., 27 (2017), pp.~453--524.

\bibitem{Larsen1974=}
{\sc E.~W. Larsen}, {\em A functional-analytic approach to the steady,
  one-speed neutron transport equation with anisotropic scattering}, Comm. Pure
  Appl. Math., 27 (1974), pp.~523--545.

\bibitem{Larsen1974}
\leavevmode\vrule height 2pt depth -1.6pt width 23pt, {\em Solutions of the
  steady, one-speed neutron transport equation for small mean free paths}, J.
  Mathematical Phys., 15 (1974), pp.~299--305.

\bibitem{Larsen1975}
\leavevmode\vrule height 2pt depth -1.6pt width 23pt, {\em Neutron transport
  and diffusion in inhomogeneous media {I}.}, J. Mathematical Phys., 16 (1975),
  pp.~1421--1427.

\bibitem{Larsen1977}
\leavevmode\vrule height 2pt depth -1.6pt width 23pt, {\em Asymptotic theory of
  the linear transport equation for small mean free paths {II}.}, SIAM J. Appl.
  Math., 33 (1977), pp.~427--445.

\bibitem{Larsen.D'Arruda1976}
{\sc E.~W. Larsen and J.~D'Arruda}, {\em Asymptotic theory of the linear
  transport equation for small mean free paths {I}.}, Phys. Rev., 13 (1976),
  pp.~1933--1939.

\bibitem{Larsen.Habetler1973}
{\sc E.~W. Larsen and G.~J. Habetler}, {\em A functional-analytic derivation of
  {Case}'s full and half-range formulas}, Comm. Pure Appl. Math., 26 (1973),
  pp.~525--537.

\bibitem{Larsen.Keller1974}
{\sc E.~W. Larsen and J.~B. Keller}, {\em Asymptotic solution of neutron
  transport problems for small mean free paths}, J. Mathematical Phys., 15
  (1974), pp.~75--81.

\bibitem{Larsen.Zweifel1974}
{\sc E.~W. Larsen and P.~F. Zweifel}, {\em On the spectrum of the linear
  transport operator}, J. Mathematical Phys., 15 (1974), pp.~1987--1997.

\bibitem{Larsen.Zweifel1976}
\leavevmode\vrule height 2pt depth -1.6pt width 23pt, {\em Steady,
  one-dimensional multigroup neutron transport with anisotropic scattering}, J.
  Mathematical Phys., 17 (1976), pp.~1812--1820.

\bibitem{Li.Lu.Sun2017}
{\sc Q.~Li, J.~Lu, and W.~Sun}, {\em Validity and regularization of classical
  half-space equations}, J. Stat. Phys., 166 (2017), pp.~398--433.

\bibitem{Sone2002}
{\sc Y.~Sone}, {\em Kinetic theory and fluid dynamics.}, Birkhauser Boston,
  Inc., Boston, MA, 2002.

\bibitem{Sone2007}
\leavevmode\vrule height 2pt depth -1.6pt width 23pt, {\em Molecular gas
  dynamics. Theory, techniques, and applications.}, Birkhauser Boston, Inc.,
  Boston, MA, 2007.

\bibitem{AA013}
{\sc L.~Wu}, {\em Boundary layer of {Boltzmann} equation in {2D} convex
  domains}, To appear in Analysis\&PDE.

\bibitem{AA004}
\leavevmode\vrule height 2pt depth -1.6pt width 23pt, {\em Hydrodynamic limit
  with geometric correction of stationary {Boltzmann} equation}, J.
  Differential Equations, 260 (2016), pp.~7152--7249.

\bibitem{AA005}
\leavevmode\vrule height 2pt depth -1.6pt width 23pt, {\em Diffusive limit with
  geometric correction of unsteady neutron transport equation}, Kinet. Relat.
  Models, 10 (2017), pp.~1163--1203.

\bibitem{AA012}
\leavevmode\vrule height 2pt depth -1.6pt width 23pt, {\em Asymptotic analysis
  of unsteady neutron transport equation}, Math. Methods Appl. Sci., 42 (2019),
  pp.~2544--2585.

\bibitem{AA014}
\leavevmode\vrule height 2pt depth -1.6pt width 23pt, {\em Boundary layer of
  transport equation with in-flow boundary}, Arch. Rational Mech. Anal., 235
  (2020), pp.~2085--2169.

\bibitem{AA003}
{\sc L.~Wu and Y.~Guo}, {\em Geometric correction for diffusive expansion of
  steady neutron transport equation}, Comm. Math. Phys., 336 (2015),
  pp.~1473--1553.

\bibitem{AA006}
{\sc L.~Wu, X.~Yang, and Y.~Guo}, {\em Asymptotic analysis of transport
  equation in annulus}, J. Stat. Phys., 165 (2016), pp.~585--644.

\end{thebibliography}

\end{document}